\documentclass[a4paper,11pt]{amsart}
\usepackage{graphicx, amssymb}  
\usepackage{amsmath,amscd,mathrsfs,amsthm}
\usepackage{multirow, makecell, longtable} 
\usepackage[numbers, sort&compress]{natbib}
\usepackage{float}
\usepackage{url,hyperref} 
\usepackage[inner=2.5cm,outer=2.5cm,top=2.5cm,bottom=2.5cm,includehead,includefoot]{geometry}
\hypersetup{citecolor=red, linkcolor=blue, colorlinks=true}

\newtheorem{thm}{Theorem}[section]
\newtheorem{prop}[thm]{Proposition}
\newtheorem{cor}[thm]{Corollary}
 \newtheorem{lemma}[thm]{Lemma}
 
\theoremstyle{definition}

\newtheorem{example}[thm]{Example}
\newtheorem{remark}[thm]{Remark}
\newtheorem{hyp}[thm]{Hypothesis}
\newtheorem{definition}[thm]{Definition}
\newtheorem*{notation}{Notation  and remarks for Examples \ref*{example:(I)dep}--\ref*{example:(I)E}}

\renewcommand\leq{\leqslant} 
\renewcommand\geq{\geqslant} 
\renewcommand\nleq{\nleqslant} 

\renewcommand\mod{\bmod} 
\newcommand\cprod{\mathbin{\Box}}

\DeclareMathOperator{\Aut}{Aut}
\DeclareMathOperator{\Sym}{Sym}

\DeclareMathOperator{\End}{End}

\DeclareMathOperator{\AG}{AG}
\DeclareMathOperator{\PG}{PG}
\DeclareMathOperator{\diag}{diag}
\DeclareMathOperator{\M}{M}
\newcommand*{\nonsplit}{{}^{\displaystyle .}}

\DeclareMathOperator{\GL}{GL}
\DeclareMathOperator{\SL}{SL}
\DeclareMathOperator{\PSL}{PSL}
\DeclareMathOperator{\PGL}{PGL}
\DeclareMathOperator{\AGL}{AGL}
\DeclareMathOperator{\PGammaL}{P\Gamma L}
\DeclareMathOperator{\AGammaL}{A\Gamma L}
\DeclareMathOperator{\GammaL}{\Gamma L}
\DeclareMathOperator{\Sp}{Sp}

\DeclareMathOperator{\GSp}{GSp}

\DeclareMathOperator{\SU}{SU}
\DeclareMathOperator{\PSU}{PSU}
\DeclareMathOperator{\GammaU}{\Gamma U}
\DeclareMathOperator{\GO}{O}
\DeclareMathOperator{\SO}{SO}
\DeclareMathOperator{\spin}{Spin}
\DeclareMathOperator{\POmega}{P\Omega}
\DeclareMathOperator{\Suz}{Suz}
\DeclareMathOperator{\Sz}{Sz}
\DeclareMathOperator{\f}{\mathbf{f}}

\title{Partial linear spaces with a rank~$3$ affine primitive  group of automorphisms}
\author[J. Bamberg, A. Devillers,  J. B. Fawcett, C. E. Praeger]{John Bamberg, Alice Devillers,  Joanna B. Fawcett,   Cheryl E. Praeger}

\address[Bamberg, Devillers, Praeger]{Centre for the Mathematics of Symmetry and Computation,
Department of Mathematics and Statistics,
The University of Western Australia, 35 Stirling Highway, 
Crawley, W.A. 6009, Australia}
\email{john.bamberg@uwa.edu.au, alice.devillers@uwa.edu.au, cheryl.praeger@uwa.edu.au}

\address[Fawcett]{Department of Mathematics,  Imperial College London, South Kensington Campus, 
London, SW7 2AZ, United Kingdom}
\email{j.fawcett@imperial.ac.uk}


\thanks{We sincerely thank the anonymous referees for their helpful and insightful comments. This work forms part of the Australian Research Council Discovery Project grant DP130100106 of the first, second and fourth authors. The third author was supported by this same grant, the London Mathematical Society and  the European Union's Horizon 2020 research and innovation programme under the Marie Sk\l{}odowska-Curie grant agreement No.\ 746889”.  
}


\begin{document}

\vspace*{-0.05cm}
\maketitle
\vspace*{-0.5cm}

\begin{abstract}
A partial linear space is a pair $(\mathcal{P},\mathcal{L})$ where $\mathcal{P}$ is a non-empty set of points and $\mathcal{L}$ is a collection of subsets of $\mathcal{P}$ called lines such that any two distinct points are contained in at most one line, and every line contains at least two points. A partial linear space is proper when it is not a  linear space or a graph. 
A group of automorphisms $G$ of a proper partial linear space acts transitively on ordered pairs of distinct collinear points and ordered pairs of distinct non-collinear points precisely when $G$ is transitive of rank~$3$ on points.
 In this paper, we classify the finite proper partial linear spaces that admit  rank~$3$ affine primitive automorphism groups, except for certain families of small groups, including  subgroups of $\AGammaL_1(q)$. Up to these exceptions, this completes the classification of the finite proper partial linear spaces admitting rank~$3$  primitive automorphism groups. We also provide a more detailed version of the classification of the rank~$3$ affine primitive permutation groups, which may be of independent interest.
\end{abstract}

\section{Introduction}
Partial linear spaces are a class of incidence structures that generalise both graphs and linear spaces. Examples include polar spaces and generalised quadrangles. Specifically, a \textit{partial linear space} $\mathcal{S}$ is a pair $(\mathcal{P},\mathcal{L})$, where $\mathcal{P}$ is a non-empty set of \textit{points} and $\mathcal{L}$ is a collection of subsets of $\mathcal{P}$ called \textit{lines} such that any two distinct points are contained in at most one line, and every line contains at least two points. \emph{In this paper, the set $\mathcal{P}$ will always be finite.} A partial linear space $\mathcal{S}$ is a \textit{linear space} when any two distinct points  are contained in exactly one line, and  a \textit{graph} when every line contains exactly two points. A partial linear space is \textit{proper} when it is neither a  linear space nor a graph;  in particular, every proper partial linear space contains at least one line.  If every line of $\mathcal{S}$ contains exactly $k$ points, then we say that $\mathcal{S}$ has \textit{line-size} $k$; dually, if every point of $\mathcal{S}$ lies on exactly $\ell$ lines, then we say that $\mathcal{S}$ has \textit{point-size} $\ell$. The \textit{automorphism group} $\Aut(\mathcal{S})$ of $\mathcal{S}$ consists of those permutations of $\mathcal{P}$ that preserve $\mathcal{L}$. 

Highly symmetric linear spaces  have been studied extensively. For example, Kantor~\cite{Kan1985} classified the $2$-transitive linear spaces, in which some automorphism group acts transitively on ordered pairs of distinct points; these include  the projective space $\PG_n(q)$ and the affine space $\AG_n(q)$. More generally, the flag-transitive  linear spaces $\mathcal{S}$---in which some automorphism group $G$ acts transitively on  point-line incident pairs---have also been classified~\cite{BueDelDoyKleLieSax1990}, except for the case where $\mathcal{S}$ has $q$ points and $G\leq \AGammaL_1(q)$ for some prime power $q$. 
The most natural way to generalise the concept of $2$-transitivity to arbitrary partial linear spaces is to consider those partial linear spaces for which some automorphism group  acts transitively on ordered pairs of  distinct collinear points, as well as ordered pairs of  distinct non-collinear points. Such  partial linear spaces are flag-transitive, and when they have non-empty line sets and are not linear spaces,  they are precisely the partial linear spaces $\mathcal{S}$ for which $\Aut(\mathcal{S})$ is transitive of rank~$3$ on points (see Lemma~\ref{lemma:rank3}  and Remark~\ref{remark:PLSrank3}). 

A permutation group $G$ acting on a finite set $\Omega$ has \textit{rank} $r$ when $G$ is transitive on $\Omega$ and has $r$ orbits on $\Omega\times \Omega$; we also say that $G$ is 
\textit{primitive} if it is transitive on $\Omega$ and there are no non-trivial $G$-invariant equivalence relations on $\Omega$. Using the classification of the finite simple groups (CFSG), the primitive permutation groups of rank~$3$  have been classified  (see~\cite{LieSax1986} for references), and as an immediate consequence of this classification, the graphs with a transitive automorphism group of rank~$3$ are known. However, we cannot make such a deduction for proper partial linear spaces, for the lines of such geometries are not necessarily determined by their collinearity relations. We therefore wish to classify the proper partial linear spaces with a rank~$3$ automorphism group $G$, and we will focus on the case where $G$ is  primitive (see Remark~\ref{remark:imprim} for some comments on the imprimitive case). The primitive permutation groups of rank~$3$ are of almost simple, grid or affine type (see Proposition~\ref{prop:primrank3}), and Devillers has  classified the partial linear spaces with a rank~$3$ primitive automorphism group of almost simple~\cite{Dev2005} or grid type~\cite{Dev2008}, so it remains to consider those of affine type. 

A primitive permutation group $G$ is \textit{affine} when the socle of $G$ is (the additive group of) a vector space $V:=V_d(p)$ for some $d\geq 1$ and prime $p$, in which case we may view $G$ as a subgroup of $\AGL_d(p)$ acting on $V$. Moreover, we may view  $V$ as the translation group of $\AGL_d(p)$,  in which case $G=V{:}G_0$, where $G_0$ denotes  the stabiliser of the zero vector in $V$, and $G_0$ is an irreducible subgroup of $\GL_d(p)$. Note that if $G_0\leq \GammaL_a(r)$ where $r^a=p^d$, then we may view $V$ as a vector space $V_a(r)$. The affine primitive permutation groups of rank~$3$ were classified by Foulser~\cite{Fou1969} in the soluble case and   Liebeck~\cite{Lie1987} in general,  and we provide a more detailed version of this classification in this paper (see Theorem~\ref{thm:rank3}). 
If $G$ is such a permutation group, then $G_0$ has two orbits on $V^*:=V\setminus \{0\}$, and if $\mathcal{S}$ is a partial linear space with  $G\leq \Aut(\mathcal{S})$, then we may identify the points of $\mathcal{S}$ with $V$, in which case the set of vectors in $V$ that are collinear with $0$, denoted by $\mathcal{S}(0)$, is one of the two orbits of $G_0$ on $V^* $. 

Before we state the main theorem of this paper, we first describe some infinite families of examples that arise. Throughout these examples, $p$ is a prime and $d$ is a positive integer.  See \S\ref{s:prelim} for any unexplained terminology. We will see that most of the proper partial linear spaces admitting rank~$3$ affine primitive automorphism groups have the form of Example~\ref{example:AG}.
 
\begin{example}
\label{example:AG}
Let $(G,a,r)$ satisfy the following properties:  $G$ is an affine primitive permutation group of rank~$3$ with socle $V:=V_d(p)$ such that $G_0\leq \GammaL_a(r)$  and $G_0$ has two orbits $\Delta_1$ and $\Delta_2$  on the points of $\PG_{a-1}(r)$,  where  $r^a=p^d$, $a\geq 2$  and $r>2$.  Let $\mathcal{L}_i:=\{\langle u\rangle + v: \langle u\rangle\in \Delta_i, v\in V\}$ for $i\in\{1,2\}$. Then $\mathcal{S}_i:=(V,\mathcal{L}_i)$ is a proper partial linear space  such that $G\leq \Aut(\mathcal{S}_i)$ for $i\in\{1,2\}$. Note that $(V,\mathcal{L}_1\cup \mathcal{L}_2)$ is the linear space $\AG_a(r)$, and $G$ has orbits $\mathcal{L}_1$ and $\mathcal{L}_2$  on the lines of $\AG_a(r)$. 
 Geometrically, if we embed  $\PG_{a-1}(r)$ as a hyperplane $\Pi$ in $\PG_{a}(r)$ and view $V$ as the set of points in $\AG_a(r)$ (i.e., the set of points in $\PG_a(r)$ that are not in $\Pi$), then $\mathcal{L}_i$ is the set of affine lines of $\AG_a(r)$ whose completions meet $\Pi$ in a point of $\Delta_i$. 
 For most rank~$3$ affine primitive permutation groups $G$, there are   various pairs $(a,r)$ such that  $G_0\leq \GammaL_a(r)$  and $G_0$ has two orbits on $\PG_{a-1}(r)$; see Hypothesis~\ref{hyp:AGgroups} and Corollary~\ref{cor:AGgroups} (and 
 Theorem~\ref{thm:rank3}). 
 \end{example}

\begin{example}
\label{example:tensor}
Let $U:=V_2(q)$ and $W:=V_m(q)$ where $q^{2m}=p^d$ and $m\geq 2$. 
Let $V:=U\otimes W$. Let $\Sigma_U:=\{U\otimes w:w\in W^*\}$ and $\Sigma_W:=\{u\otimes W : u\in U^*\}$. 
 For $X\in \{U,W\}$, let $\mathcal{L}_X:=\{Y+v : Y\in \Sigma_X,v\in V\}$. 
Then $\mathcal{S}_X:=(V,\mathcal{L}_X)$ is a proper partial linear space,
and $\mathcal{S}_U\simeq \mathcal{S}_W$ when $m=2$.  Geometrically,
if we  embed $\PG_{2m-1}(q)$ as a hyperplane $\Pi$ in $\PG_{2m}(q)$ and view $V$ as the set of points in $\AG_{2m}(q)$ (i.e., the set of points in $\PG_{2m}(q)$ that are not in $\Pi$), then for $(X,n)\in \{(U,2),(W,m)\}$, we may view 
 $\Sigma_X$ as a set of projective  $(n-1)$-subspaces  of $\Pi$, whence $\mathcal{L}_X$ is the set of affine $n$-subspaces  of $\AG_{2m}(q)$ whose completions  meet $\Pi$ in an element of $\Sigma_X$.  
  For $\mathcal{S}_U$, the line-size is $q^2$ and the point-size is $(q^m-1)/(q-1)$, while for $\mathcal{S}_W$, the line-size  is $q^m$ and the point-size is $q+1$. By Proposition~\ref{prop:tensoraut}, $\Aut(\mathcal{S}_U)=\Aut(\mathcal{S}_W)=V{:}(\GL_2(q)\circ  \GL_m(q)){:}\Aut(\mathbb{F}_q)$, a rank~$3$  affine primitive  group.  See \S\ref{s:(T)}  for more details.
\end{example}

\begin{example}
\label{example:grid}
Let $d=2n$ where $p^n\neq 2$. The $p^n\times p^n$ grid is a proper partial linear space with point set $V:=V_n(p)\oplus V_n(p)$ whose line set is the  union of $\{\{(v,w) : v \in V_n(p)\} : w \in V_n(p)\}$ and $\{\{(w,v) : v \in V_n(p)\} : w \in V_n(p)\}$. The $p^n\times p^n$ grid has line-size $p^n$ and point-size $2$,
 and its full automorphism group  is $S_{p^n}\wr S_2$, which contains the rank~$3$ affine primitive group $V{:}(\GL_n(p)\wr S_2)$. See \S\ref{s:(I)}  for more details.
\end{example}

Now we state the main result of this paper. 

\begin{thm}
\label{thm:main}  
Let $\mathcal{S}$ be a finite proper partial linear space, and let $G\leq \Aut(\mathcal{S})$ such that $G$ is an affine  primitive permutation group of rank~$3$ with socle $V:=V_d(p)$ where $d\geq 1$ and $p$ is prime. Then one of the following holds.
\begin{itemize}
\item[(i)] $\mathcal{S}$ is isomorphic to a partial linear space  from Example~\emph{\ref{example:AG}} with respect to  a triple $(H,a,r)$ satisfying Hypothesis~\emph{\ref{hyp:AGgroups}} such that $H$ is  primitive  with socle $V$ and $r^a=p^d$.
\item[(ii)] $\mathcal{S}$ is  described in Examples~\emph{\ref{example:tensor}} or~\emph{\ref{example:grid}}.
 \item[(iii)] $\mathcal{S}$ is described in Table~\emph{\ref{tab:main}} where $k$ and $\ell$ are the line- and point-size of $\mathcal{S}$, respectively, and $\Aut(\mathcal{S})=p^d{:}\Aut(\mathcal{S})_0$.
 \item[(iv)]  One of the following holds. 
 \begin{itemize}
 \item[(a)] $G_0\leq \GammaL_1(p^d)$.
 \item[(b)] $V=V_n(p)\oplus V_n(p)$ and $G_0\leq \GammaL_1(p^n)\wr S_2$ where $\mathcal{S}(0)=V_n(p)^*\times V_n(p)^*$.
 \item[(c)] $V=V_2(q^3)$ and  $\SL_2(q)\unlhd G_0\leq \GammaL_2(q^3)$ where 
 $|\mathcal{S}(0)|=q(q^3-1)(q^2-1)$.
 \end{itemize}
\end{itemize}
\end{thm}

 \begin{table}[!h]
\renewcommand{\baselinestretch}{1.1}\selectfont
\centering
\begin{tabular}{ c c c c  c c c c  }
\hline
$p^d$ & $k$ & $\ell$ & $|\mathcal{S}(0)|$ & $\Aut(\mathcal{S})_0$ &  No.  & Ref. & Notes \\
\hline 
$2^8$ & $16$ & $9$ & $135$ & $A_9$ &$1$ &\ref{example:(AS)A9} &  
\multirowcell{7}{$\bigg\uparrow$\\ not obtained\\ from any rank~$3$\\  $2$-$(p^d,k,1)$ design\\ $\bigg\downarrow$} 
\\
$3^4$ & $6$ & $12$ & $60$ & $M_{10}$ & $1$&\ref{example:(R2)} & \\
$3^5$ & $12$ & $12$ & $132$ & $M_{11}$ &$1$ &\ref{example:(AS)M11} & \\
$3^8$ & $9$ & $180$ & $1440$ & $(Q_8\circ Q_8\circ Q_8)\nonsplit \SO_6^-(2)$ &$1$& \ref{example:(E)nonplane} & \\
$3^{12}$ & $27$ & $56$ & $1456$ & $\SL_2(13)\wr S_2 $ &$1$& \ref{example:(I)goodHering} & \\
 & $9$ & $8190$ & $65520$ & $2\nonsplit G_2(4).2^-$ &$1$& \ref{example:(AS)G24} & \\
$5^6$ & $25$ & $315$ & $7560$ & $(2\nonsplit J_2\circ 4)\nonsplit 2$ &$1$ &\ref{example:(AS)J2} & \\
\hline
  $3^4$ & $9$ & $5$ & $40$ & $2\nonsplit S_5^+$ & $1$&\ref{example:(AS)nearfield} &  \multirowcell{2}{from the nearfield \\ plane of order $9$} \\
 & $9$ & $8$ & $64$ & $(Q_8\wr S_2){:}S_3$ &$1$& \ref{example:(I)reg} &   \\
 \hline 
$5^4$ & $25$ & $24$ & $24^2$ & $(\SL_2(3)\wr S_2){:}4$ & $1$&\ref{example:(I)reg}   &\multirowcell{7}{$\bigg\uparrow$\\ from an irregular \\ nearfield plane \\ of order $p^2$\\ $\bigg\downarrow$} \\
$7^4$ & $49$ & $48$ & $48^2$ & $(2\nonsplit S_4^-\wr S_2){:}3$ &$1$& \ref{example:(I)reg} & \\
$11^4$ & $121$ & $120$ & $120^2$ & $((\SL_2(3)\times 5)\wr S_2){:}2$ & $1$&\ref{example:(I)reg} & \\
$23^4$ & $23^2$ & $528$ & $528^2$ & $(2\nonsplit S_4^-\times 11)\wr S_2$ &$1$& \ref{example:(I)reg} & \\
$11^4$ & $121$ & $120$ & $120^2$ & $(\SL_2(5)\wr S_2){:}5$ &$1$& \ref{example:(I)reg} & \\
$29^4$ & $29^2$ & $840$ & $840^2$ & $((\SL_2(5)\times 7)\wr S_2){:}2$ &$1$ &\ref{example:(I)reg} &\\
$59^4$ & $59^2$ & $3480$ & $3480^2$ & $(\SL_2(5)\times 29)\wr S_2$ &$1$& \ref{example:(I)reg} &\\
\hline
$5^4$ & $25$ & $10$ & $240$ & $(\SL_2(5)\circ D_8\circ  Q_8\circ 4)\nonsplit 2$ & $1$ &\ref{example:(E)walker} &\multirowcell{8}{$\Bigg\uparrow$\\ from a rank~$3$ \\ affine plane\\ of order $p^2$\\ $\Bigg\downarrow$} \\
 & $25$ & $16$ & $384$ & $(\SL_2(5)\circ D_8\circ  Q_8\circ 4)\nonsplit 2$ &$1$ &\ref{example:(E)walker} & \\
 & $25$ & $6$ & $144$ & $ (2\nonsplit A_6\circ  4)\nonsplit 2 $ & $1$&\ref{example:(AS)walker} & \\
 & $25$ & $20$ & $480$ & $ (2\nonsplit A_6\circ 4)\nonsplit 2 $ & $1$&\ref{example:(AS)walker} & \\
$7^4$ & $49$ & $48$ & $48^2$ & $((\SL_2(3)\times \SL_2(3))\nonsplit 4){:}3$ &$1$& \ref{example:(I)spor} &\\
 & $49$ & $10$ & $480$ & $2\nonsplit S_5^+\circ D_8\circ Q_8\circ 6$  & $1$&\ref{example:(E)mason-ostrom} & \\
 & $49$ & $40$ & $1920$ & $2\nonsplit S_5^+\circ D_8\circ Q_8\circ 6$    &$1$ &\ref{example:(E)mason-ostrom} &\\
 & $49$ & $20$ & $960$ & $2\nonsplit S_5^-\circ 24$ &$1$& \ref{example:(AS)korch} &\\
 \hline 
$3^8$ & $9$ & $20$ & $160$ & $(D_8\circ Q_8\circ 2\nonsplit S_5^-)\wr S_2$ &$1$& \ref{example:(I)goodnearfield} & \multirowcell{7}{$\Bigg\uparrow$\\ from a rank~$3$ \\  $2$-$(p^d,k,1)$ design\\ $\Bigg\downarrow$} \\ 
& $9$ & $800$ & $6400$ & $(((D_8\circ Q_8).D_{10})\wr S_2)\nonsplit 2$ &$1$& \ref{example:(I)E} & \\
 & $9$ & $800$ & $6400$ & $((D_8\circ Q_8).D_{10})\wr S_2$ &$1$& \ref{example:(I)E} & \\
 & $9$ & $800$ & $6400$ & $((2\nonsplit S_5^-{:}2)\wr S_2){:}2$ &$1$& \ref{example:(I)sl25} & \\
 & $9$ & $800$ & $6400$ & $(2\nonsplit S_5^-{:}2)\wr S_2$ & $1$&\ref{example:(I)sl25} & \\
 
$3^{12}$  & $9$ & $182$ & $1456$ & $\SL_2(13)\wr S_2 $ & $2$ & \ref{example:(I)goodHering} &   \\
& $9$ & $66248$ & $529984$ & $\SL_2(13)\wr S_2 $ & $4$ & \ref{example:(I)sl213}   &\\
\hline 
\end{tabular}
\caption{The exceptional partial linear spaces}
\label{tab:main}
\end{table}

For each exceptional partial linear space $\mathcal{S}$ in Table~\ref{tab:main},  $\Aut(\mathcal{S})$ is an affine primitive group and $\Aut(\mathcal{S})_0$, the stabiliser in $\Aut(\mathcal{S})$ of the $0$ vector, is an irreducible subgroup of $\GL_d(p)$. We also  list the following in Table~\ref{tab:main}:  the number (No.)\ of partial linear spaces that satisfy the given parameters up to isomorphism; a reference (Ref.)\ for the  definition of the partial linear space and its automorphism group; and some additional notes that will be explained in Remark~\ref{remark:PLSfromLinear}. Observe that $\mathcal{S}$  has $p^d\ell/k$ lines and that $|\mathcal{S}(0)|=\ell(k-1)$ 
(see Lemma~\ref{lemma:lines}). 

 We caution the reader  that there are affine primitive permutation  groups $G$ of rank~$3$ and partial linear spaces $\mathcal{S}$ with $G\leq \Aut(\mathcal{S})$ such that $\mathcal{S}$ satisfies the conditions of Theorem~\ref{thm:main}(i) with respect to some triple $(H,a,r)$ but 
 not $(G,a,r)$ (see Remark~\ref{remark:Hexample}).

 \begin{remark}
 \label{remark:PLSfromLinear}
 One way of constructing a partial linear space  is to remove lines from a linear space, as we did in Example~\ref{example:AG} with the linear space $\AG_a(r)$. 
In particular, for a rank~$3$ permutation group $G$ on $\mathcal{P}$, if  $\mathcal{S}:=(\mathcal{P},\mathcal{L})$ is a linear space with at least three points on every line such that $G\leq \Aut(\mathcal{S})$  and $G$ has  exactly two orbits $\mathcal{L}_1$ and $\mathcal{L}_2$ on $\mathcal{L}$, then $(\mathcal{P},\mathcal{L}_1)$ and $(\mathcal{P},\mathcal{L}_2)$ are proper partial linear spaces that admit  $G$ and have disjoint collinearity relations; in fact, the converse of this statement also holds  (see Lemma~\ref{lemma:linearspace}).
 Observe that $\mathcal{S}$ is a $2$-$(v,k,1)$ design---that is, a linear space with $v$ points and line-size $k$---precisely when  $(\mathcal{P},\mathcal{L}_1)$ and $(\mathcal{P},\mathcal{L}_2)$ both have line-size $k$.
   Those $2$-$(v,k,1)$ designs admitting a rank~$3$ automorphism group  $G$  with two orbits on lines have been studied in the special case of affine planes (e.g.,~\cite{BilJoh2001}) and in general when $G$ is an affine primitive group~\cite{BilMonFra2015,Mon2015}.  For each  partial linear space $\mathcal{S}$ in Table~\ref{tab:main}, we state whether $\mathcal{S}$ can be obtained from a   $2$-$(v,k,1)$ design  using a rank~$3$ group as above; when this design is an affine plane (i.e., when $v=k^2$), we state this instead, and when this affine plane is well known, we give its name. More details may be found at the given reference or in~\S\ref{s:proof}.  
\end{remark}

It therefore follows from Theorem~\ref{thm:main} that  there are proper partial linear spaces with rank~$3$ affine primitive automorphism groups that  cannot be obtained  from any  $2$-$(v,k,1)$ design using a rank~$3$ group as   described in Remark~\ref{remark:PLSfromLinear}.  In fact, there are infinitely many such structures: we prove that  the partial linear space $\mathcal{S}_W$ from Example~\ref{example:tensor}  cannot be obtained  from a   $2$-$(v,k,1)$ design using any  rank~$3$ group for $m\geq 4$ and $(m,q)\neq (5,2)$; see  Proposition~\ref{prop:tensornodesign}.
 However, if  $\mathcal{S}$ is $\mathcal{S}_U$ from Example~\ref{example:tensor} for $m\geq 2$ (respectively, $\mathcal{S}_W$ for $m=3$), then $\mathcal{S}$ can be obtained from $\AG_m(q^2)$ (respectively, $\AG_2(q^3)$)  using a rank~$3$ affine primitive subgroup $G$ of $\Aut(\mathcal{S})$ (see~\S\ref{s:extra}); in other words, $\mathcal{S}$ is described in Example~\ref{example:AG} with respect to the triple $(G,m,q^2)$ (respectively, $(G,2,q^3)$), but we choose not to omit $\mathcal{S}$ from Example~\ref{example:tensor} because  $\Aut(\mathcal{S})$ is itself a rank~$3$ affine primitive group. Similarly, if $\mathcal{S}$ is the $p^n\times p^n$ grid of Example~\ref{example:grid}, then $\mathcal{S}$ can be obtained from the affine plane $\AG_2(p^n)$ using the rank~$3$ affine primitive group  $V{:}(\GL_1(p^n)\wr S_2)$, but this group is considerably  smaller than $\Aut(\mathcal{S})$.

\begin{remark}
	We are unable to classify the partial linear spaces that satisfy the conditions of Theorem~\ref{thm:main}(iv). In fact, the groups of   Theorem~\ref{thm:main}(iv)(a) were also   omitted from the classification of the flag-transitive linear spaces~\cite{BueDelDoyKleLieSax1990}.
There are partial linear spaces from Example~\ref{example:AG} for which $G$  satisfies  the conditions of Theorem~\ref{thm:main}(iv)(a); see Example~\ref{example:R0} for more details.   
There are also partial linear spaces $\mathcal{S}$  for which $G$ and $\mathcal{S}(0)$ satisfy the conditions of Theorem~\ref{thm:main}(iv)(b); see Examples~\ref{example:(I)dep} and \ref{example:(I)reg}.
We completely classify the  proper partial linear spaces $\mathcal{S}$  for which  $G$ and $\mathcal{S}(0)$ satisfy the conditions of Theorem~\ref{thm:main}(iv)(c) under the extra assumption that $\GL_2(q)\circ Z(\GL_2(q^3))\leq G_0$ (see Proposition~\ref{prop:(S0)bad}); several infinite families arise, including  some  from Example~\ref{example:AG}. 
 However,  we believe that the situation is much more complicated in general. We illustrate this by providing a complete classification of the partial linear spaces that arise when $q=4$ (see Example~\ref{example:hardS0}).
\end{remark}

All of the partial linear spaces of Examples~\ref{example:AG},~\ref{example:tensor} and~\ref{example:grid} have the property that their lines are affine subspaces of $V_d(p)$. Moreover, this turns out to be true for all of the partial linear spaces in Table~\ref{tab:main} except when $\Aut(\mathcal{S})=3^4{:}M_{10}$ or $3^5{:}M_{11}$, in which case neither partial linear space has this  property since the line-size is not a power of $p$. We  suspect 
that these are the only two such partial linear spaces, and we prove that  any other such partial linear space must  satisfy the conditions of Theorem~\ref{thm:main}(iv)(b) as well as the constraints in (iii) below. 

\begin{cor}
\label{cor:affinesub}
Let $\mathcal{S}$ be a finite proper partial linear space, and let $G\leq \Aut(\mathcal{S})$ such that $G$ is an affine  primitive permutation group of rank~$3$ with socle $V:=V_d(p)$ where $d\geq 1$ and $p$ is prime. If the lines of $\mathcal{S}$ are not affine subspaces of $V_d(p)$, then one of the following holds. 
\begin{itemize}
 \item[(i)] $p^d=3^4$, $\mathcal{S}$ has line-size $6$ and $\Aut(\mathcal{S})=3^4{:}M_{10}$. 
 \item[(ii)] $p^d=3^5$, $\mathcal{S}$ has line-size $12$ and $\Aut(\mathcal{S})=3^5{:}M_{11}$.
 \item[(iii)]  $V=V_n(p)\oplus V_n(p)$ and $G_0\leq \GammaL_1(p^n)\wr S_2$ where  $\mathcal{S}(0)=V_n(p)^*\times V_n(p)^*$  and $n\geq 2$. Further, all of the following hold for any  line $L$ of $\mathcal{S}$ such that $0\in L$.
 \begin{enumerate}
\item The prime $p$ is odd, and $-1\notin G_0$. In particular, $|G_0\cap Z(\GL_{d}(p))|$ is odd.
 \item If $k$ is the line-size of $\mathcal{S}$, then $k(k-1)$ divides $(p^n-1)^2$, so $k$ is coprime to $p$.
 \item $L\cap \{\lambda u :\lambda\in \mathbb{F}_p\}=\{0,u\}$ for all   $u\in L^*$. 
 \item  $L=\{(v,v^\alpha):v\in M\}$ for some $M\subseteq V_n(p)$ and injective map $\alpha:M\to V_n(p)$. 
\item  For  $g\in G_0$, there exists $v\in L$ such that $v^g\neq -v$.   
 \item There exist $g_1,g_2\in \GammaL_1(p^n)$ such that  $(g_1,-1)\in G_0$ and $(-1,g_2)\in G_0$. 
 \item If $H\times  K\leq G_0$ for some  $H,K\leq \GammaL_1(p^n)$, then $H$ or $K$ is not transitive on $V_n(p)^*$.
 \end{enumerate}
 \end{itemize}
\end{cor}

We saw in Example~\ref{example:grid} that the $p^n\times p^n$ grid is a proper partial linear space  with a rank~$3$ affine primitive group of automorphisms whose full automorphism group is not affine. Using~\cite{Pra1990}, we prove that this is the only such example. 

\begin{thm}
\label{thm:primrank3plus}
Let $\mathcal{S}$ be a finite proper partial linear space, and let $G\leq \Aut(\mathcal{S})$ such that $G$ is an affine  primitive permutation group of rank~$3$ with socle $V:=V_d(p)$ where $d\geq 1$ and $p$ is prime.
Then one of the following holds.
\begin{itemize}
\item[(i)] $\mathcal{S}$ is isomorphic to the $p^n\times p^n$ grid and $\Aut(\mathcal{S})=S_{p^n}\wr S_2$, where  $d=2n$ and $n\geq 1$.
\item[(ii)] $\Aut(\mathcal{S})$ is an affine primitive permutation group of rank~$3$ with socle $V$.
\end{itemize}
\end{thm}

This paper is organised as follows. In \S\ref{s:prelim}, we give some preliminaries  and  prove Theorem~\ref{thm:primrank3plus}. In \S\ref{s:rank3}, we state and prove a modified  version of Liebeck's classification~\cite{Lie1987} of the affine primitive permutation groups of rank~$3$ (see Theorem~\ref{thm:rank3});  in particular, we provide more detailed information about the possible rank~$3$ groups $G$ when $G_0$ is imprimitive or stabilises a tensor product decomposition. We then use Theorem~\ref{thm:rank3} to state Hypothesis~\ref{hyp:AGgroups}. 
The partial linear spaces of Example~\ref{example:AG}   have a property  that we term  \textit{dependence} (see \S\ref{s:dep}), and  we prove in \S\ref{s:dep} that, under very general assumptions, the dependent partial linear spaces all have the form of Example~\ref{example:AG}; we also provide more details about our  strategy for proving Theorem~\ref{thm:main}. 
 In \S\ref{s:(R1)}--\ref{s:(S0)}, we classify the independent proper  partial linear spaces for various classes of rank~$3$ affine primitive groups, as defined by Theorem~\ref{thm:rank3}, and in \S\ref{s:(I)}--\ref{s:(AS)}, we consider the remaining rank~$3$ groups.  In \S\ref{s:proof}, we prove Theorem~\ref{thm:main}, and  in \S\ref{s:affinesub}, we prove  Corollary~\ref{cor:affinesub}. In \S\ref{s:extra}, we consider  when a partial linear space   from Example~\ref{example:tensor}  can be obtained from a $2$-$(v,k,1)$ design using a rank~$3$ group.

\section{Preliminaries}
\label{s:prelim}

All groups and incidence structures in this paper are finite, and all group actions are written on the right. Basic terminology and results in permutation group theory or representation theory may be found in~\cite{DixMor1996} or~\cite{Isa1994}, respectively. 
The notation used to denote the
finite simple groups (and their automorphism groups) is consistent with~\cite{KleLie1990}.  We  use the algebra software {\sc Magma}~\cite{Magma} and {\sf GAP}~\cite{GAP4} for  a variety of computations.  In particular,  we use the {\sf GAP} package FinInG~\cite{FinInG}, as well as 
nauty and Traces \cite{NautyTraces} underneath the {\sf GAP} package Grape~\cite{Grape}.

 This section is organised as follows. In \S\ref{ss:basicsactions}--\ref{ss:affine}, we review some general notation, definitions and basic results. In \S\ref{ss:pls}, we consider   some elementary properties of partial linear spaces. In \S\ref{ss:plsnew}, we describe several ways of constructing new partial linear spaces from given ones. In \S\ref{ss:affinePLS}, we investigate some properties of partial linear spaces admitting affine automorphism groups. In \S\ref{ss:nearfield}, we describe a family of affine planes called nearfield planes.
In \S\ref{ss:2trans}, we state  the  classifications of the $2$-transitive affine  groups~\cite{Her1985trans,Hup1957}  and  the linear spaces admitting  $2$-transitive affine automorphism groups~\cite{Kan1985}. In \S\ref{ss:primrank3plus}, we prove Theorem~\ref{thm:primrank3plus}.
 
\subsection{Group actions}
\label{ss:basicsactions}
Let $G$ be a group acting on a (finite) set $\Omega$. We denote  the orbit of $x\in \Omega$ by $x^G$ and the pointwise stabiliser in $G$  of $x\in \Omega$ by $G_x$. We denote the  setwise and pointwise stabilisers of $X=\{x_1,\ldots,x_n\}\subseteq \Omega$ in $G$ by $G_X$ and  $G_{(X)}=G_{x_1,\ldots,x_n}$, respectively.  We denote the permutation group induced by $G_X$ on $X$ by $G^X_X=G_X/G_{(X)}$, and the symmetric group on $\Omega$ by $\Sym(\Omega)$.  The \textit{degree} of $G$ is $|\Omega|$.

A \textit{block} of $G$ is a non-empty subset $B$ of $\Omega$ such that for each $g\in G$, either $B^g=B$ or $B^g\cap B=\varnothing$.  If $G$ is transitive on $\Omega$, then a block $B$ is \textit{non-trivial} if it is neither a singleton nor $\Omega$, in which case $\{B^g: g\in G\}$ is a \textit{system of imprimitivity} for $G$ on $\Omega$, and $G$ is  \textit{imprimitive}; recall from the introduction that $G$ is primitive if no such system of imprimitivity exists. We will  use the following  observation  throughout this paper (see \cite[Theorem 1.5A]{DixMor1996}): if $G$ is transitive on $\Omega$ and $x\in \Omega$, then the set of blocks of $G$ containing $x$ is in one-to-one correspondence with the set of subgroups of $G$ containing $G_x$; under this correspondence, a block $B$ containing $x$ is mapped to $G_B$, and a subgroup $H$ containing $G_x$ is mapped to $x^{H}$.  

\subsection{Notation and definitions for groups}
\label{ss:basicsgroups}

For groups $G$ and $H$, we denote a split extension  of $G$ by $H$ by  $G{:}H$; a non-split extension  by $G\nonsplit H$, an arbitrary extension   by $G.H$, and the central product of $G$ and $H$ (with respect to some common central subgroup) by $G\circ H$.  For  $S\leq S_n$, we denote the wreath product $G^n{:}S$ by $G\wr S$. We denote the cyclic group of order $n$ by $C_n$ or just $n$,  the elementary abelian group $C_p^n$ by $p^n$, the dihedral group of order $n$ by $D_n$, and the quaternion group by $Q_8$. For $H\leq G$ and $K\leq G$, we denote the centraliser and normaliser of $H$ in $K$ by $C_K(H)$ and $N_K(H)$, respectively. We denote the centre of $G$ by $Z(G)$ and the derived subgroup of $G$ by $G'$.

The \textit{socle} of a group $G$ is the subgroup  generated by the minimal normal subgroups of $G$. The group $G$ is \textit{almost simple} if its socle is a non-abelian simple group $T$; equivalently, $G$ is almost simple if $T\leq G\leq \Aut(T)$. The group $G$ is \textit{quasisimple} if $G$ is perfect (i.e., $G=G'$) and $G/Z(G)$ is a simple group. A \textit{covering group} of $G$ is a group $L$ such that $L/Z(L)\simeq G$ and $Z(L)\leq L'$. The symmetric group $S_n$ has two covering groups $2\nonsplit S_n^+$ and $2\nonsplit S_n^-$ for $n\geq 4$, both of which contain the covering group $2\nonsplit A_n$ of $A_n$; in $2\nonsplit S_n^+$, transpositions lift to involutions, whereas in $2\nonsplit S_n^-$, transpositions lift to elements of order $4$. The almost simple group $G_2(4).2$ also has two covering groups $2\nonsplit G_2(4).2^+$ and $2\nonsplit G_2(4).2^-$, where $2\nonsplit G_2(4).2^+$ is the  group whose character table is given in \cite[p.99]{Atlas}. Note that $2\nonsplit S_4^+\simeq \GL_2(3)$, $2\nonsplit A_4\simeq \SL_2(3)$ and $2\nonsplit A_5\simeq \SL_2(5)$.

\subsection{Fields, vector spaces and representation theory}
\label{ss:basicsvs}

Let $q$ be a power of a prime $p$. We denote the finite field of order $q$  by $\mathbb{F}_q$ and an $n$-dimensional vector space over $\mathbb{F}_q$ by $V_n(q)$. If $W\subseteq V_n(q)$, then we define $W^*:=W\setminus \{0\}$. For a subfield  $F$ of $\mathbb{F}_q$, we write $\langle x_1,\ldots,x_m\rangle_F$ for the $F$-span of the vectors $x_1,\ldots,x_m\in V_n(q)$;  when $F=\mathbb{F}_q$ and the context permits, we omit  $F$ from this notation. We write $\sigma_q$ for the Frobenius automorphism $x\mapsto x^p$ of $\mathbb{F}_q$, so that $\Aut(\mathbb{F}_q)=\langle\sigma_q\rangle$, and we adopt the following convention: whenever we write $\GammaL_n(q)=\GL_n(q){:}\langle \sigma\rangle$, we mean that  $\langle\sigma\rangle\simeq \langle \sigma_q\rangle$ and $\GammaL_n(q)$  acts on $V_n(q)$ with respect to some basis $\{v_1,\ldots, v_n\}$ such that  $(\sum_{i=1}^n \lambda_i v_i)^{\sigma}=\sum_{i=1}^n \lambda_i^{\sigma} v_i$ for all $\lambda_i\in\mathbb{F}_q$. Any subgroup of $\GammaL_n(q)$ is $\mathbb{F}_q$\textit{-semilinear}, and  $g\in \GammaL_n(q)$ is $\sigma$\textit{-semilinear} when $\sigma\in\langle\sigma_q\rangle$ and $(\lambda v)^g=\lambda^\sigma v^g$ for all $\lambda\in\mathbb{F}_q$ and $v\in V_n(q)$.   We write $\diag(\lambda_1,\ldots,\lambda_n)$ for the diagonal $n\times n$  matrix with diagonal entries $\lambda_1,\ldots,\lambda_n$. Now $Z(\GL_n(q))=\{\diag(\lambda,\ldots,\lambda):\lambda\in\mathbb{F}_q^*\}$, and with a slight abuse of notation,  we write  $\lambda$ for  $\diag(\lambda,\ldots,\lambda)$ and $\mathbb{F}_q^*$ for $Z(\GL_n(q))$. We also write $\zeta_q$ for  a generator of the multiplicative group of $\mathbb{F}_q$. Note that  $-1$ denotes the  central involution of $\GL_n(q)$ when $p$ is odd, but $-1=1$ when $p$ is even. 

For a field $F$ and group $G$, we denote the group algebra of $G$ over $F$ by $FG$. Note that an $F$-vector space $V$ is a faithful $FG$-module if and only if $G\leq \GL(V)$. To emphasise that $V$ is a vector space over  $F$, we write $\GL(V,F)$.  An $FG$-module $V$ or a subgroup $G$ of $\GL(V,F)$ is \textit{irreducible} if there are no proper $FG$-submodules of $V$. 
An irreducible subgroup $G$ of $\GL(V,F)$ is  \textit{absolutely irreducible} if $G$ is irreducible when viewed as a subgroup of $\GL(V,E)$ for all  field extensions $E$ of $F$. An irreducible subgroup $G$ of  $\GL(V,F)$ is absolutely irreducible if and only if  $C_{\GL(V,F)}(G)=F^*$ by \cite[2.10.1]{KleLie1990}. We will  use the following  observation  throughout this paper: if $G$ is an irreducible subgroup of $\GL(V)$ and $Z(G)$ contains an involution $z$, then $z=-1$. 

Let $U$ and $W$ be vector spaces over the field $\mathbb{F}_q$, and let $V$ be the tensor product $U\otimes W$.  For $g\in  \GL(U)$ and $h\in  \GL(W)$, let $(u\otimes w)^{g\otimes h}:=u^g\otimes w^h$ for all $u\in U$ and $w\in W$. Now $g\otimes h$ extends to a linear map of $V$. For $S\leq \GL(U)$ and $T\leq \GL(W)$, define $S\otimes T:=\{g\otimes h : g\in S,h\in T\}$. Then $S\otimes T\leq \GL(V)$ and $S\otimes T\simeq S\circ T$. There are natural actions of $\Aut(\mathbb{F}_q)$ on $V$, $U$ and $W$ for which $(u\otimes w)^\sigma = u^\sigma \otimes w^\sigma$ for all $u\in U$, $w\in W$ and $\sigma\in\Aut(\mathbb{F}_q)$, so that  $(\GL(U)\otimes \GL(W)){:}\Aut(\mathbb{F}_q)$ stabilises the tensor decomposition of $V$.

\subsection{Affine and projective planes and spaces}
\label{ss:basicsplane}
 
 Recall from the introduction that a $2$-$(v,k,1)$ \textit{design} is a linear space with $v$ points and line-size $k$. 
 For $n\geq 2$, a (finite) \textit{affine plane of order} $n$ is a  $2$-$(n^2,n,1)$ design, and a (finite) \textit{projective plane of order} $n$ is a $2$-$(n^2+n+1,n+1,1)$ design. Given  a projective plane of order $n$, we obtain an affine plane of order $n$  by removing one line and all of its points. Conversely, given an affine plane $\mathcal{A}$ of order $n$, we obtain a projective plane of order $n$, called the \textit{completion} of $\mathcal{A}$, by adding  a \textit{point at infinity}  for every parallel class of lines, 
 and defining the union of these new points to be the \textit{line at infinity}, denoted by $\ell_\infty$.
 
For  $m\geq 0$ and a prime power $q$, the \textit{affine space} $\AG_m(q)$
is a linear space with points $V_m(q)$ and lines $\{\langle u\rangle +v : u\in V_m(q)^*,v\in V_m(q)\}$, while the \textit{projective space} $\PG_m(q)$ or $\PG(V_{m+1}(q))$ is a linear space whose points and lines are, respectively,  the one- and two-dimensional subspaces of $V_{m+1}(q)$. An \textit{affine} ($k$-)\textit{subspace} of $\AG_m(q)$ is a translation of a  $k$-dimensional subspace of $V_m(q)$, and a \textit{projective} ($k$-)\textit{subspace} of $\PG_m(q)$ is a  $(k+1)$-dimensional subspace of $V_{m+1}(q)$; in particular, 
affine $2$-subspaces are called \textit{planes}, and 
projective $(m-1)$-subspaces are called \textit{hyperplanes}.  The affine space $\AG_m(q)$ may be obtained from $\PG_m(q)$ by removing a hyperplane $\mathcal{H}$ and its points and lines. With this viewpoint of $\AG_m(q)$, for $1\leq k\leq m$, any affine $k$-subspace  $W$  of $\AG_m(q)$ is the intersection of a projective $k$-subspace of $\PG_m(q)$---the \textit{completion} of $W$---with the complement of $\mathcal{H}$; further, the completion of $W$ meets $\mathcal{H}$ in a projective $(k-1)$-subspace.  Note that $\AG_2(q)$ is the  \textit{Desarguesian} affine plane  of order $q$, and $\PG_2(q)$ is  the  \textit{Desarguesian} projective plane  of order $q$.
 
\subsection{Affine permutation groups}
\label{ss:affine}

In the introduction, we defined an affine primitive permutation group to be a primitive group whose socle is a vector space over a field of prime order. In this section, for convenience, we expand this definition to include (certain) transitive groups.

First we require some notation. Let $q$ be a power of a prime $p$, and let $n$ be a positive integer. We denote the affine general  linear group and affine semilinear group by $\AGL_n(q)$ and $\AGammaL_n(q)$, respectively.  For $v\in V:=V_n(q)$, define $\tau_v:V\to V$ to be the translation $x\mapsto x+v$ for all $x\in V$. With some abuse of notation, we denote the group of translations of $V$ by $V$, so that  $\AGL_n(q)=V{:}\GL_n(q)$ and $\AGammaL_n(q)=V{:}\GammaL_n(q)$.

A group  $G$ is  an \textit{affine permutation group  on} $V:=V_n(q)$ whenever $V\leq G\leq \AGammaL_n(q)$.  If $G$ is such a group, then 
 $G=V{:}G_0$ and $G_0\leq \GammaL_n(q)$, where $G_0$ is the stabiliser of the zero vector. Note that if $q^n=p^d$, then $\GammaL_n(q)\leq \GL_d(p)$ and  
  we may view $V$ as $V_d(p)$, so that $\AGammaL_n(q)\leq \AGL_d(p)$. The proof of the following is routine; see \cite[Proposition 6.1.1]{FawPhD}. 

\begin{lemma}
\label{lemma:primitiveirreducible}
Let $G$ be an affine permutation group on $V:=V_d(p)$, where  $d\geq 1$ and $p$ is prime. Then the following are equivalent.
\begin{itemize}
\item[(i)] $G$ is primitive on $V$.
\item[(ii)] $V$ is an irreducible $\mathbb{F}_pG_0$-module.
\end{itemize}
\end{lemma}

Note that if an affine permutation group $G$ on $V$ is primitive, then  $V$ is  the socle of $G$  by~\cite[Theorem~4.3B]{DixMor1996}. Thus the definition from the introduction of an affine primitive permutation group  coincides with the  definition in this section of an affine permutation group that is primitive.

If $G$ is an affine permutation group on $V$, then $G$ has rank~$3$  if and only if $G_0$ has two orbits on $V^*=V\setminus \{0\}$. This observation leads us to the following useful version of Lemma~\ref{lemma:primitiveirreducible}.

\begin{lemma}
\label{lemma:primitive}
Let $G$ be an affine permutation group of rank~$3$ on $V:=V_d(p)$, where $d\geq 1$ and $p$ is prime. Let $X$ and $Y$ be the orbits of $G_0$ on $V^*$. Then the following are equivalent. 
\begin{itemize}
\item[(i)] $G$ is primitive on $V$.
\item[(ii)] Neither $X\cup \{0\}$ nor $Y\cup \{0\}$ is an $\mathbb{F}_p$-subspace of $V$.
\end{itemize}
\end{lemma}

\begin{proof}
If $X\cup \{0\}$  is an $\mathbb{F}_p$-subspace of $V$, then $X\cup \{0\}$  is an $\mathbb{F}_pG_0$-submodule of $V$, so $G$ is imprimitive on $V$ by Lemma~\ref{lemma:primitiveirreducible}. Conversely, if $G$ is imprimitive on $V$,  then $V$ has a proper non-zero $\mathbb{F}_pG_0$-submodule $W$ by Lemma~\ref{lemma:primitiveirreducible}, but $W^*$ is preserved by $G_0$, so $W^*=X$ or $Y$. 
\end{proof}

Thus if $G$ and $H$ are affine permutation groups of rank~$3$ on $V:=V_d(p)$ where $G$ and $H$ have the same orbits on $V^*$, then $G$ is primitive on $V$ if and only if $H$ is primitive on $V$.

We conclude this section with an  observation that we will use frequently  to prove that some $\mathbb{F}_p$-subspace of $V$ is a line of a partial linear space (cf.~Lemma~\ref{lemma:sufficient}).

\begin{lemma}
\label{lemma:transitive}
Let $G$ be an affine permutation group on $V:=V_d(p)$, where $d\geq 1$ and $p$ is prime.  If $L$ is an $\mathbb{F}_p$-subspace of $V$, then $G_L$ is transitive on $L$.
\end{lemma}

\begin{proof}
Let $x\in L^*$. Then $0^{\tau_x}=x$ and $\tau_x\in G$, so it suffices to show that $L^{\tau_x}\subseteq L$. If $y\in L$, then $y^{\tau_x}=y+x\in L$, so $L^{\tau_x}\subseteq L$, as desired.
\end{proof}

\subsection{Partial linear spaces}
\label{ss:pls}

Let  $\mathcal{S}:=(\mathcal{P},\mathcal{L})$ be a partial linear space. Distinct points $x,y\in\mathcal{P}$ are \textit{collinear} if there exists a line $L\in \mathcal{L}$ containing $x$ and $y$; we also say that $x$ and $y$ lie on the line $L$, and so on. 
For $x\in\mathcal{P}$, let $\mathcal{L}_x$ denote the set of lines in $\mathcal{L}$ that contain the point $x$, and let $\mathcal{S}(x)$ denote the set of points in $\mathcal{P}$ that are collinear with $x$. The \textit{collinearity} relation of $\mathcal{S}$ is the set of $(x,y)\in \mathcal{P}\times \mathcal{P}$  such that $x$ and $y$ are collinear, and the \textit{non-collinearity} relation of $\mathcal{S}$ is the   set of $(x,y)\in \mathcal{P}\times \mathcal{P}$  such that $x$ and $y$ are distinct and  not collinear. The \textit{collinearity graph} of $\mathcal{S}$ is the graph $(\mathcal{P},\{\{x,y\} : (x,y)\in \mathcal{R}\})$ where $\mathcal{R}$ is the collinearity relation of $\mathcal{S}$.

An \textit{isomorphism} $\varphi:\mathcal{S}\to\mathcal{S}'$ between $\mathcal{S}$ and a  partial linear space $\mathcal{S}':=(\mathcal{P}',\mathcal{L}')$ is a  bijection $\varphi:\mathcal{P}\to \mathcal{P}'$ such that $\mathcal{L}'=\mathcal{L}^\varphi$, where $\mathcal{L}^\varphi:=\{L^\varphi:L\in\mathcal{L}\}$ and $L^\varphi:=\{x^\varphi : x\in L\}$ for $L\in\mathcal{L}$.
When such an isomorphism exists, we say that $\mathcal{S}$ and $\mathcal{S}'$ are \textit{isomorphic}.
The \textit{automorphism group} $\Aut(\mathcal{S})$  of $\mathcal{S}$ is  $\{g\in \Sym(\mathcal{P}) : \mathcal{L}^g=\mathcal{L}\}$.
For $g\in \Sym(\mathcal{P})$, the pair  $(\mathcal{P},\mathcal{L}^g)$ is a partial linear space, which we denote by $\mathcal{S}^g$. Observe that $\mathcal{S}$ and $\mathcal{S}^g$ are isomorphic, and $\Aut(\mathcal{S}^g)=\Aut(\mathcal{S})^g$ for all $g\in \Sym(\mathcal{P})$.

A \textit{flag} of $\mathcal{S}$ is a pair $(x,L)$ where $x$ is a point on a line $L$, and $\mathcal{S}$ is \textit{flag-transitive} if $\Aut(\mathcal{S})$ acts transitively on the flags of $\mathcal{S}$. Similarly, $\mathcal{S}$ is \textit{point-} or \textit{line-transitive} if $\Aut(\mathcal{S})$ acts transitively on $\mathcal{P}$ or $\mathcal{L}$, respectively. Observe that if $\mathcal{S}$ is point- or line-transitive, then the point- or line-size of $\mathcal{S}$ is defined, respectively. Further, a flag-transitive partial linear space (with no isolated points) is point- and line-transitive. Lastly, if $\Aut(\mathcal{S})$ acts transitively on its collinearity relation, then $\mathcal{S}$ is flag-transitive. 

 We will use the following fundamental result throughout this paper; its proof is routine. 

\begin{lemma}
\label{lemma:rank3}
Let $\mathcal{S}:=(\mathcal{P},\mathcal{L})$ be a  partial linear space with collinearity relation $\mathcal{R}_1$ and non-collinearity relation $\mathcal{R}_2$ where $\mathcal{R}_1$ and $\mathcal{R}_2$ are non-empty.   Let $G\leq \Aut(\mathcal{S})$. Then the following are equivalent.
\begin{itemize}
\item[(i)] $G$ is transitive of rank~$3$ on $\mathcal{P}$.
\item[(ii)] The orbits of $G$ on $\mathcal{P}\times \mathcal{P}$ are $\{(u,u): u\in\mathcal{P}\}$, $\mathcal{R}_1$ and $\mathcal{R}_2$. 
\item[(iii)] $G$ is transitive on $\mathcal{P}$ and  for $u\in \mathcal{P}$, the orbits of $G_u$ on $\mathcal{P}$ are $\{u\}$, $\mathcal{S}(u)=\{v\in \mathcal{P} : (u,v)\in \mathcal{R}_1\}$ and $\{v \in\mathcal{P} :(u,v)\in \mathcal{R}_2\}$.
\end{itemize}
\end{lemma}

\begin{remark}
\label{remark:PLSrank3}
If $\mathcal{S}$ is a proper partial linear space with collinearity relation $\mathcal{R}_1$ and non-collinearity relation $\mathcal{R}_2$, then $\mathcal{R}_2$ is non-empty (or else $\mathcal{S}$ is  a linear space) and $\mathcal{R}_1$ is non-empty (or else $\mathcal{S}$ is  a graph with no edges). In particular, for a proper partial linear space $\mathcal{S}:=(\mathcal{P},\mathcal{L})$ with $G\leq \Aut(\mathcal{S})$ such that $G$ has rank~$3$ on $\mathcal{P}$, Lemma~\ref{lemma:rank3} implies that $\mathcal{S}$ is   flag-transitive. Moreover,  the orbitals of $G$ are self-paired, and $\Aut(\mathcal{S})$ has rank $3$ on~$\mathcal{P}$. 
\end{remark} 
 
The following  is a collection of necessary conditions for the existence of a partial linear space  whose  automorphism group is transitive on points and pairs of collinear points.

\begin{lemma}
\label{lemma:necessary}
Let $\mathcal{S}:=(\mathcal{P},\mathcal{L})$ be a partial linear space with collinearity relation $\mathcal{R}$.  Let $G\leq \Aut(\mathcal{S})$ 
where $G$ is transitive on $\mathcal{P}$ and $\mathcal{R}$. Let $L\in\mathcal{L}$ and $u\in L$. Let $B:=L\setminus \{u\}$.   Then the following hold.
\begin{itemize}
\item[(i)] $B$ is a block of $G_u$ on $\mathcal{S}(u)$.
\item[(ii)] $G_L$ is $2$-transitive on $L$.
\item[(iii)] If $\mathcal{S}$ is proper and $G$ is primitive on $\mathcal{P}$, then $B$ is a non-trivial block of $G_u$ on $\mathcal{S}(u)$.
\end{itemize}
\end{lemma}

\begin{proof}
(i) Suppose that $v\in B\cap B^g$ for some $g\in G_u$. Now $u$ and $v$ are distinct points on the lines $L$ and $L^g$, so $L=L^g$, in which case $B=B^g$.

(ii) Let $x,y,v,w\in L$ where $x\neq y$ and $v\neq w$. Now $(x,y),(v,w)\in \mathcal{R}$, so there exists $g\in G$ such that $x^g=v$ and $y^g=w$. In particular, $v,w\in L\cap L^g$, so $L=L^g$.

(iii) Suppose that $\mathcal{S}$ is proper. Then $|B|\geq 2$ and $|\mathcal{L}|\geq 2$.  If $B=\mathcal{S}(u)$, then there is a unique line on $u$, so there is a unique line on every point of $\mathcal{S}$, but then $\mathcal{L}$ is a system of imprimitivity for $G$ on $\mathcal{P}$, so $G$ is imprimitive on $\mathcal{P}$. 
\end{proof}

Next we provide some  sufficient conditions for the existence of a point-transitive partial linear space.

\begin{lemma}
\label{lemma:sufficient}
Let $G$ be a transitive permutation group on   $\mathcal{P}$. Let $u\in \mathcal{P}$,  and let $B$ be a  block of  $G_u$  on $X$, where $X$ is an orbit of $G_u$ on $\mathcal{P}\setminus \{u\}$. Let  $L:=B\cup\{u\}$, $\mathcal{L}:=L^G$  and $\mathcal{S}:=(\mathcal{P},\mathcal{L})$, and suppose that  $G_L$ is transitive on $L$. Then  the following hold.
\begin{itemize} 
\item[(i)] $\mathcal{S}$ is a partial linear space with $G\leq \Aut(\mathcal{S})$ and $\mathcal{S}(u)=X$.
\item[(ii)] $\mathcal{S}$ is a linear space if and only if $G$ is $2$-transitive on $\mathcal{P}$.
\item[(iii)] If $B$ is a non-trivial block, then $\mathcal{S}$ has line-size at least $3$ and point-size at least $2$.
\item[(iv)] If $B$ is a non-trivial block and $G$ has rank~$3$ on $\mathcal{P}$, then $\mathcal{S}$ is proper.
\end{itemize}
\end{lemma}

\begin{proof} 
(i) The proof that $\mathcal{S}$ is a partial linear space is routine (see~\cite[Theorem 2.3]{Dev2005} for details). Clearly $G\leq \Aut(\mathcal{S})$. If $v\in\mathcal{S}(u)$, then $v\in M$ for some $M\in \mathcal{L}_u$, and there exists $g\in G_u$ such that $M=L^g$ since $G_L$ is transitive on $L$, so  $v\in B^g\subseteq X^g=X$. Conversely, if $v\in X$, then $v\in B^g$ for some $g\in G_u$, so $v\in\mathcal{S}(u)$.

(ii) $\mathcal{S}$ is a linear space if and only if $\mathcal{S}(u)=\mathcal{P}\setminus \{u\}$, so (ii) follows from (i).
 
 (iii) Note that $\mathcal{S}$ is point- and line-transitive. If $B$ is non-trivial, then $|L|\geq 3$ and there are at least two lines on $u$, so (iii) holds.  
 
 (iv) This follows from  (ii) and (iii).
\end{proof}

In particular, if $G$ is a primitive permutation group of rank~$3$ on $\mathcal{P}$, then $\mathcal{S}:=(\mathcal{P},\mathcal{L})$ is a proper partial linear space with $G\leq \Aut(\mathcal{S})$ if and only if  $\mathcal{L}=L^G$ for some $L\subseteq \mathcal{P}$ such that $G_L$ is transitive on $L$ and  $L\setminus \{u\}$ is a non-trivial block of $G_u$ on $v^{G_u}$ for some  distinct $u,v\in L$. 

The following result  is standard and can be proved by counting flags.

\begin{lemma}
\label{lemma:lines}
Let $\mathcal{S}:=(\mathcal{P},\mathcal{L})$ be a partial linear space with line-size $k$ and point-size~$\ell$. Then $|\mathcal{L}|k=|\mathcal{P}|\ell$ and $|\mathcal{S}(u)|=\ell(k-1)$ for all $u\in \mathcal{P}$.
\end{lemma}

In general, the automorphism group of a proper partial linear space can be primitive of rank~$3$ and have  imprimitive subgroups. However, we now see that   all rank~$3$  subgroups are primitive.

\begin{lemma}
\label{lemma:primrank3}
Let $\mathcal{S}:=(\mathcal{P},\mathcal{L})$ be a  partial linear space with collinearity relation $\mathcal{R}_1$ and non-collinearity relation $\mathcal{R}_2$ where $\mathcal{R}_1$ and $\mathcal{R}_2$ are non-empty.   Let $G\leq \Aut(\mathcal{S})$ where $G$ is transitive of rank~$3$ on $\mathcal{P}$. Then $G$ is primitive on $\mathcal{P}$ if and only if $\Aut(\mathcal{S})$ is primitive on $\mathcal{P}$.
\end{lemma}

\begin{proof}
Suppose that $G$ is imprimitive on $\mathcal{P}$. Then there is a $G$-invariant equivalence relation $\mathcal{R}$ on $\mathcal{P}$ for which  $\{(u,u):u\in \mathcal{P}\}\subsetneq \mathcal{R}\subsetneq\mathcal{P}\times \mathcal{P}$.
By Lemma~\ref{lemma:rank3},
$\mathcal{R}\setminus \{(u,u):u\in \mathcal{P}\}$ is $\mathcal{R}_1$ or $\mathcal{R}_2$. Since $\mathcal{R}_1$ and $\mathcal{R}_2$ are $\Aut(\mathcal{S})$-invariant,  $\Aut(\mathcal{S})$ is imprimitive. The converse is routine.
\end{proof}

To finish this section, we make  some brief observations  about classifying 
partial linear spaces whose automorphism groups are imprimitive of rank $3$.

\begin{remark}
\label{remark:imprim}
Let $\mathcal{S}:=(\mathcal{P},\mathcal{L})$ be a  partial linear space with collinearity relation $\mathcal{R}_1$ and non-collinearity relation $\mathcal{R}_2$ where $\mathcal{R}_1$ and $\mathcal{R}_2$ are non-empty. 
 Let $\Gamma$ be the collinearity graph of~$\mathcal{S}$. Then 
 $\Gamma$ has the same collinearity and non-collinearity relations as~$\mathcal{S}$.   Let $G\leq \Aut(\mathcal{S})$ where $G$ is imprimitive  of  rank~$3$ on $\mathcal{P}$. By the proof of Lemma~\ref{lemma:primrank3},  $\mathcal{R}_i\cup \{(u,u):u\in \mathcal{P}\}$ is an  equivalence relation with $m$ classes of size $n$ for some  $i\in \{1,2\}$  and $m,n\geq 2$.
 If $i=1$, then $\Gamma$ is a disjoint union of $m$ complete graphs of size $n$, so $\mathcal{S}$ is a disjoint union of isomorphic $2$-transitive linear spaces (and such linear spaces were  classified in~\cite{Kan1985}). Otherwise,  $i=2$ and $\Gamma$ is a complete multipartite graph $K_{m[n]}$ with $m$ parts of size $n$. Thus, in order to classify the proper partial linear spaces admitting  rank~$3$ imprimitive automorphism groups, it suffices to consider those  with  collinearity graph   $K_{m[n]}$. 
 This difficult problem has been studied in certain cases (see~\cite{DevHal2006,PeaPra2013}) but  remains open in general.
\end{remark}

\subsection{Constructing new partial linear spaces from given ones}
\label{ss:plsnew}

Let $\mathcal{S}:=(\mathcal{P},\mathcal{L})$ be a linear space with  at least three points on every line, let $G\leq \Aut(\mathcal{S})$ such that $G$ has rank~$3$ on $\mathcal{P}$, and let $u\in \mathcal{P}$. Now $G_u$ has two orbits $X$ and $Y$ on $\mathcal{P}\setminus \{u\}$. If there exists $L\in\mathcal{L}_u$ such that $X\cap L$ and $Y\cap L$ are non-empty, then it is not hard to see that 
 $\mathcal{S}$ is flag-transitive. Otherwise, $G$ has exactly two orbits on $\mathcal{L}$, and, as we discussed in Remark~\ref{remark:PLSfromLinear}, we can use these orbits to construct partial linear spaces that admit  $G$. We  consider this construction in more detail now. 

\begin{lemma}
\label{lemma:linearspace}
Let $G$ be a rank~$3$ permutation group on  $\mathcal{P}$. Then the following are equivalent.
\begin{itemize}
\item[(i)] $\mathcal{S}:=(\mathcal{P},\mathcal{L}_1\cup\mathcal{L}_2)$ is a linear space with at least three points on every line such that $G\leq\Aut(\mathcal{S})$ and $G$ has distinct orbits $\mathcal{L}_1$ and $\mathcal{L}_2$ on $\mathcal{L}_1\cup\mathcal{L}_2$.
\item[(ii)] $\mathcal{S}_1:=(\mathcal{P},\mathcal{L}_1)$ and $\mathcal{S}_2:=(\mathcal{P},\mathcal{L}_2)$  are proper partial linear spaces with $G\leq \Aut(\mathcal{S}_1)$ and $G\leq \Aut(\mathcal{S}_2)$ such that the collinearity relations of $\mathcal{S}_1$ and $\mathcal{S}_2$ are disjoint.
\end{itemize}
\end{lemma}

\begin{proof}
First suppose that (i) holds. Clearly $\mathcal{S}_1$ and $\mathcal{S}_2$ are proper partial linear spaces  such that $G\leq \Aut(\mathcal{S}_1)$ and $G\leq \Aut(\mathcal{S}_2)$. If there exist $u,v\in \mathcal{P}$ such that $u$ and $v$ are collinear in both $\mathcal{S}_1$ and $\mathcal{S}_2$, then there exist $L_1\in \mathcal{L}_1$ and $L_2\in\mathcal{L}_2$ such that $u,v\in L_1\cap L_2$, but then $L_1=L_2$, so $\mathcal{L}_1$ and $\mathcal{L}_2$ are not disjoint, a contradiction. Thus (ii) holds.

Conversely, suppose that (ii) holds, and let $\mathcal{R}_i$ be the collinearity relation of $\mathcal{S}_i$ for $i\in\{1,2\}$. Since $\mathcal{R}_1$ and $\mathcal{R}_2$ are  disjoint, $\mathcal{L}_1$ and $\mathcal{L}_2$ are disjoint, and $\mathcal{S}$ is a partial linear space. Observe that $\mathcal{L}_1$ and $\mathcal{L}_2$ are both non-empty since $\mathcal{S}_1$ and $\mathcal{S}_2$ are not graphs. Now $\mathcal{R}_1$ and $\mathcal{R}_2$ are both non-empty,  so the orbits of $G$ on $\mathcal{P}\times\mathcal{P}$ are $\{(u,u):u\in\mathcal{P}\}$, $\mathcal{R}_1$ and $\mathcal{R}_2$. If $u$ and $v$ are distinct elements of $\mathcal{P}$, then $(u,v)\in \mathcal{R}_i$ for some $i$, so there exists $L\in\mathcal{L}_i$ such that $u,v\in L$. Thus $\mathcal{S}$ is a linear space. Clearly $G\leq\Aut(\mathcal{S})$, and $G$ is transitive on $\mathcal{L}_1$ and $\mathcal{L}_2$, so (i) holds.
\end{proof}

Note  that the partial linear spaces  $\mathcal{S}_1$ and $\mathcal{S}_2$ of Lemma~\ref{lemma:linearspace} have the same line-size if and only if  $\mathcal{S}$ is a $2$-$(v,k,1)$ design for some $v$ and $k$.

\begin{remark}
\label{remark:lineartopls}
Observe that if $\mathcal{S}:=(\mathcal{P},\mathcal{L})$ and $\mathcal{S}':=(\mathcal{P}',\mathcal{L}')$ are isomorphic linear spaces that satisfy the conditions of Lemma~\ref{lemma:linearspace}(i), then any isomorphism $\varphi:\mathcal{S}\to \mathcal{S}'$ naturally determines isomorphisms between the corresponding partial linear spaces: indeed, if $\mathcal{L}_1$ and $\mathcal{L}_2$ are the orbits of the rank~$3$ group $G$ on $\mathcal{L}$, then $\mathcal{L}_1':=\mathcal{L}_1^\varphi$ and $\mathcal{L}_2':=\mathcal{L}_2^\varphi$ are the orbits of $G^\varphi$ on $\mathcal{L}'$, so $\varphi$ is an isomorphism between the partial linear spaces $(\mathcal{P},\mathcal{L}_1)$ and $(\mathcal{P}',\mathcal{L}_1')$, as well as  $(\mathcal{P},\mathcal{L}_2)$ and $(\mathcal{P}',\mathcal{L}_2')$.  However,  the converse is not true in general: see Remark~\ref{remark:linearspace}. 
\end{remark}

Next we define the intersection of two partial linear spaces. Let $\mathcal{S}_1:=(\mathcal{P}_1,\mathcal{L}_1)$ and $\mathcal{S}_2:=(\mathcal{P}_2,\mathcal{L}_2)$ be partial linear spaces. Define $\mathcal{S}_1\cap \mathcal{S}_2$ to be the pair $(\mathcal{P}_1\cap \mathcal{P}_2,\mathcal{L})$ where $\mathcal{L}:=\{L_1\cap L_2: L_1\in\mathcal{L}_1,L_2\in\mathcal{L}_2,|L_1\cap L_2|\geq 2\}$. Further, if $\mathcal{P}\subseteq \mathcal{P}_1$ such that $|\mathcal{P}|>1$, then $\mathcal{S}:=(\mathcal{P},\{\mathcal{P}\})$ is a linear space, and  we write $\mathcal{S}_1\cap \mathcal{P}$ for $\mathcal{S}_1\cap \mathcal{S}$. The proof of the following is routine.

\begin{lemma}
\label{lemma:intersect}
Let $\mathcal{S}_1:=(\mathcal{P}_1,\mathcal{L}_1)$ and $\mathcal{S}_2:=(\mathcal{P}_2,\mathcal{L}_2)$ be partial linear spaces. 
\begin{itemize}
\item[(i)] If $\mathcal{P}_1\cap \mathcal{P}_2\neq \varnothing$, then $\mathcal{S}_1\cap\mathcal{S}_2$ is a partial linear space. 
\item[(ii)] If $\mathcal{P}\subseteq \mathcal{P}_1$ such that $|\mathcal{P}|>1$ and $G\leq \Aut(\mathcal{S}_1)$, then $G_{\mathcal{P}}^\mathcal{P}\leq \Aut(\mathcal{S}_1\cap \mathcal{P})$.
\end{itemize}
\end{lemma}

Our final construction of this section is a natural generalisation of the cartesian product of graphs. 
Let $\mathcal{S}_1:=(\mathcal{P}_1,\mathcal{L}_1)$ and $\mathcal{S}_2:=(\mathcal{P}_2,\mathcal{L}_2)$ be partial linear spaces. Define the \textit{cartesian product} $\mathcal{S}_1\cprod\mathcal{S}_2$ of $\mathcal{S}_1$ and $\mathcal{S}_2$ to be $(\mathcal{P}_1\times \mathcal{P}_2,\mathcal{L}_1\cprod\mathcal{L}_2)$ where $\mathcal{L}_1\cprod\mathcal{L}_2$ 
 is the union of  $\{\{(x_1,x_2):x_1\in L\}:x_2\in\mathcal{P}_2,L\in\mathcal{L}_1\}$ and $\{\{(x_1,x_2):x_2\in L\}:x_1\in\mathcal{P}_1,L\in\mathcal{L}_2\}$. 
 
 \begin{lemma}
 \label{lemma:cartesian}
Let $\mathcal{S}_1:=(\mathcal{P}_1,\mathcal{L}_1)$ and $\mathcal{S}_2:=(\mathcal{P}_2,\mathcal{L}_2)$ be partial linear spaces. 
  \begin{itemize}
\item[(i)]   $\mathcal{S}_1\cprod\mathcal{S}_2$ is a partial linear space, and  $\Aut(\mathcal{S}_1)\times\Aut(\mathcal{S}_2)\leq\Aut(\mathcal{S}_1\cprod\mathcal{S}_2)$. 
\item[(ii)]  $\Aut(\mathcal{S}_1)\wr \langle \tau\rangle\leq \Aut(\mathcal{S}_1\cprod\mathcal{S}_1)$ where $(x_1,x_2)^\tau:=(x_2,x_1)$ for all $x_1,x_2\in\mathcal{P}_1$. 
\item[(iii)] If $|\mathcal{P}_1|> 1$, then $ \mathcal{S}_1\cprod\mathcal{S}_1$ is not a linear space.
\end{itemize}
 \end{lemma}
 
 \begin{proof} 
 Suppose that there exist distinct $(x_1,x_2),(y_1,y_2)\in L\cap M$ for some $L,M\in \mathcal{L}_1\cprod\mathcal{L}_2$. We may assume that $x_1\neq y_1$. Then $x_2=y_2$, so $L=\{(u,x_2): u\in L'\}$ and $M=\{(v,x_2):v\in M'\}$ for some $L',M'\in\mathcal{L}_1$ such that $x_1,y_1\in L'\cap M'$. Now $L'=M'$, so $L=M$. Thus 
 $\mathcal{S}_1\cprod\mathcal{S}_2$ is a partial linear space. The remaining claims of  (i) and (ii) are straightforward. If there exist distinct $x_1,x_2\in\mathcal{P}_1$, then $(x_1,x_2)$ and $(x_2,x_1)$ are not collinear in $ \mathcal{S}_1\cprod\mathcal{S}_1$, so (iii) holds.
  \end{proof}

\subsection{Affine partial linear spaces}
\label{ss:affinePLS}

A partial linear space $S:=(\mathcal{P},\mathcal{L})$ is  \textit{$G$-affine}  if $G\leq\Aut(\mathcal{S})$ and $G$ is an affine permutation group  on $V:=V_d(p)=\mathcal{P}$ (see~\S\ref{ss:affine}) for some prime $p$ and $d\geq 1$. We also say that $\mathcal{S}$ is \textit{affine} if it is $G$-affine for some $G$. 
Recall that for $x\in V$, we write $\mathcal{L}_x$ for the set of lines of $\mathcal{S}$ that contain $x$, and $\mathcal{S}(x)$ for the set of points of $\mathcal{S}$ that are collinear with $x$. 
In particular, $\mathcal{L}_0$ denotes the set of lines on the vector $0$, and $\mathcal{S}(0)$ denotes the set of vectors that are collinear with $0$. Recall also that if $L\subseteq V$, then $L^*:=L\setminus \{0\}$. 

The first result of this section consists of two elementary observations that  are the key tools for determining---or rather limiting---the structure of a  $G$-affine proper partial linear space. In particular, there is a useful tension between Lemmas~\ref{lemma:basic}(i)  and~\ref{lemma:primitive}.

\begin{lemma}
\label{lemma:basic}
Let $V:=V_d(p)$ where $d\geq 1$ and $p$ is prime, and let $\mathcal{S}:=(V,\mathcal{L})$  be a  $G$-affine proper  partial linear space where $G$ has  rank~$3$ on $V$. Let $L\in \mathcal{L}_0$.
\begin{itemize}
\item[(i)] If $x,y\in L^*$ and $x\neq y$, then  $y-x\in x^{G_0}$.
\item[(ii)] If $x,y\in L^*$ and $x\neq y$, then $y^{G_{0,x}}\subseteq L^*$.
\end{itemize}
\end{lemma}

\begin{proof}
(i) By assumption, $V\leq G$, and  $x$ and $y$ are collinear, so  $x^{\tau_{-x}}$ and $y^{\tau_{-x}}$ are also collinear. Thus  $y-x\in \mathcal{S}(0)$, and $\mathcal{S}(0)=x^{G_0}$ by Lemma~\ref{lemma:rank3} since $x\in \mathcal{S}(0)$ and $G$ has rank~$3$ on $V$.

(ii) If $z=y^g$ for some $g\in G_{0,x}$, then $L^g=L$, so $z\in (L^*)^g=L^*$.
\end{proof}

 A subset $L$ of a vector space $V$ is an \textit{affine subspace} of $V$ if $L=W+v$ for some subspace $W$ of $V$ and $v\in V$.   In the following, we generalise~\cite[Lemma 2.3]{Lie1998} and  the proof of~\cite[Lemma~2.2]{Lie1998}. 

\begin{lemma}
\label{lemma:affine}
Let $V:=V_d(p)$ where  $d\geq 1$ and $p$ is prime. Let $\mathcal{S}:=(V,\mathcal{L})$  be a  $G$-affine   partial linear space  where $G$ is transitive on $\mathcal{L}$ and $\mathcal{S}$ has  line-size at least $3$. Suppose that one  of the following holds.
\begin{itemize}
\item[(i)] There exists $g\in G_0$ and $L\in \mathcal{L}_0$  such that $v^g=-v$ for all $v\in  L$.
\item[(ii)] The prime $p=2$.
\item[(iii)] There exists $x\in V^*$ and $L\in \mathcal{L}$ such that $L^{\tau_x}=L$.
\end{itemize}
Then every line of $\mathcal{S}$ is an affine subspace of $V$. 
\end{lemma}

\begin{proof}
By assumption, $V\leq G$. Suppose that (i) holds.  Let $ y,z\in L^*$ be distinct. 
 Applying $g$ to $0$, $y$ and $z$, we see that $0$, $-y$ and $-z$ lie on a line, and we may translate by $y$, so $y$, $0$ and $y-z$ lie on a line.   
Now $L$ is the  unique line on $0$ and $y$, so $y-z\in L^*$. Similarly, $z-y\in L^*$. Since $z-y\neq z$, we also have $-y=(z-y)-z\in L^*$. It follows that $u-v\in L$ for all $u,v\in L$,  and since $V$ is an $\mathbb{F}_p$-vector space, $L$ is a subspace of $V$. Hence every line of $\mathcal{S}$ is an affine subspace of $V$, for if $M\in \mathcal{L}$, then $M=L^h+v$ for some $h\in G_0$ and $v\in V$.
 
If the condition of (ii) holds, then the condition of (i) is satisfied with $g=1$ for any line $L$, so every line is an affine subspace of $V$. 
 
Suppose that the condition of (iii) holds. First we claim that  $L^{\tau_{v-u}}=L$ for all $u,v\in L$.  Let $u,v\in L$. Now $u+\langle x\rangle$ and $v+\langle x\rangle$ are subsets of $L$, and $(u+\langle x\rangle)^{\tau_{v-u}}=v+\langle x\rangle$. Since $x\neq 0$, it follows  that $|L^{\tau_{v-u}}\cap L|\geq 2$, so  $L^{\tau_{v-u}}=L$, and the claim holds. Choose $y\in L$, and let $W:=\{u-y: u\in L\}$. Now $L=W+y$, and $W$ is a subspace of $V$ by the claim, so $L$ is an affine subspace of $V$. It follows that every line of $\mathcal{S}$ is an affine subspace of $V$.
\end{proof}

\begin{lemma}
\label{lemma:notaffine}
Let $V:=V_d(p)$ where  $d\geq 1$ and $p$ is prime, and let $\mathcal{S}:=(V,\mathcal{L})$  be a  $G$-affine proper  partial linear space where $G$ has  rank~$3$ on $V$. 
 Then one of the following holds.
\begin{itemize}
\item[(i)] Every line of $\mathcal{S}$ is an affine subspace of $V$.
\item[(ii)] The translation group $V$ acts semiregularly on $\mathcal{L}$ and $k(k-1)$ divides $|x^{G_0}|$, where $k$ is the line-size of $\mathcal{S}$ and $x\in \mathcal{S}(0)$. 
\end{itemize}
\end{lemma}

\begin{proof}
Suppose that (i) does not hold. Now the translation group $V$ acts semiregularly on $\mathcal{L}$ by Lemma~\ref{lemma:affine}. Hence $|V|$ divides $ |\mathcal{L}|$, so (ii) holds by Lemma~\ref{lemma:lines} since $\mathcal{S}(0)=x^{G_0}$.
\end{proof}

\subsection{Nearfield planes}
\label{ss:nearfield}
In Remark~\ref{remark:PLSfromLinear}, we described how some of the partial linear spaces in Table~\ref{tab:main} can be obtained from affine planes using  rank~$3$ groups. In this section, we construct a particular family of these affine planes: the nearfield planes.

A \textit{nearfield} $(Q,+,\circ)$ is a set $Q$ with binary operations $+$ and $\circ$ such that the following four axioms hold. 
\begin{itemize}
\item[(N1)] $(Q,+)$ is an abelian group.  
\item[(N2)] $(a+b)\circ c=a\circ c + b\circ c$ for all $a,b,c\in Q$. 
\item[(N3)] $(Q\setminus \{0\},\circ)$ is a group. 
\item[(N4)]  $a\circ 0=0$ for all $a\in Q$. 
\end{itemize} 
By (N2), $0\circ b=0$ for all $b\in Q$. It then follows from (N2) that $(-a)\circ b=-(a\circ b)$ for all $a,b\in Q$.

A nearfield plane is constructed using a nearfield $(Q,+,\circ)$ as follows. For $w\in Q$, let $L(w):=\{(v,v\circ w): v\in Q\}$, and let $L(\infty):=\{(0,v) : v \in Q\}$. 
The \textit{nearfield plane} on $Q$ has point set $Q\times Q$ and line set   
$\{L(w)+x : w\in Q\cup \{\infty\},x\in Q\times Q\}$. Note that $(L(w),+)\simeq (Q,+)$ for $w\in Q\cup \{\infty\}$. Further, for each $(a,b)\in (Q\times Q)\setminus \{(0,0)\}$, there is a unique $w\in Q\cup \{\infty\}$ such that $(a,b)\in L(w)$  by  (N3). Thus the nearfield plane on $Q$ is an affine plane of order $|Q|$.

Next we describe how to construct a nearfield using a sharply $2$-transitive group. Let $V:=V_n(q)$ where $n\geq 1$ and $q$ is a power of a prime $p$. Let $R$ be a subgroup of $\GammaL_n(q)$ that is regular on $V^*=V\setminus \{0\}$, so $V{:}R$ acts sharply $2$-transitively  on $V$. Let $x\in V^*$. For $w\in V^*$, let $r_w$ denote the unique element of $R$ for which $w=x^{r_w}$, and for  $v\in V$, let $v\circ w:=v^{r_w}$ and $v\circ 0:=0$. If $a,b\in V$ and $c\in V^*$, then $a\circ c+b\circ c = a^{r_c}+b^{r_c}=(a+b)^{r_c}=(a+b)\circ  c$.
Further, $r_vr_w=r_{v\circ w}$ for all $v,w\in V^*$, so $(V^*,\circ)$ is a group (and isomorphic to $R$).  Thus $(V,+,\circ)$ is a nearfield.
 
  The possibilities for $R$ are given by  Zassenhaus's  classification~\cite{Zas1936}
 of the sharply $2$-transitive  groups (see~\cite[\S 7.6]{DixMor1996}): 
 either $n=1$, or $n=2$ and  $q=p\in\{5,7,11,23,29,59\}$, in which case  there are seven possibilities for $R$ (two of which occur for  $p=11$), and  the  corresponding nearfield plane is an \textit{irregular nearfield plane of order $p^2$}. See~\cite[\S 7]{Lun1980} for more details.
  Note that $\AG_2(q)$ is the nearfield plane   that arises by taking $R=\GL_1(q)$. The (non-Desarguesian) nearfield plane of order $9$ arises by taking $R= Q_8\leq \GammaL_1(9)$; see~\cite[\S 8]{Lun1980} for more details.

\subsection{Affine $2$-transitive groups and linear spaces}
\label{ss:2trans}

The classification of the $2$-transitive affine permutation groups was obtained  by Huppert~\cite{Hup1957} in the soluble case (see also~\cite{Fou1969}), and by Hering~\cite{Her1985trans} in general. Liebeck provides another proof of Hering's result in~\cite[Appendix~1]{Lie1987}. The statement we give is similar to~\cite{Lie1987}, except that we  add a description of the affine $2$-transitive subgroups of $\AGammaL_2(q)$ and $\AGammaL_3(q)$ for  the proof of  Theorem~\ref{thm:tensorgroups}.

\begin{thm}[\cite{Her1985trans,Lie1987}]
\label{thm:2trans}
Let $G$ be a $2$-transitive affine permutation group with socle $V:=V_d(p)$ where $d\geq 1$ and $p$ is prime. Then the stabiliser $G_0$ belongs to one of the following classes.
\begin{itemize}
\item[(H0)] $G_0\leq \GammaL_1(q)$ where $q=p^d$.
\item[(H1)] $\SL_n(q)\unlhd G_0$ where $q^n=p^d$, $n\geq 2$ and $(n,q)\neq (2,2)$ or $(2,3)$.
\item[(H2)] $\Sp_{2n}(q)\unlhd G_0$ where $q^{2n}=p^d$ and $n\geq 2$. 
\item[(H3)] $G_2(q)'\unlhd G_0$ where $q^6=p^d$ and $q$ is even. 
\item[(H4)] $\SL_2(5)\unlhd G_0\leq \GammaL_2(q)$ where  $q^2=p^d$ and  $q\in \{9,11,19,29,59\}$.
\item[(H5)] $G_0=A_6$ or $A_7$ and $p^d=2^4$.
\item[(H6)] $G_0=\SL_2(13)$ and $p^d=3^6$.
\item[(H7)] $\SL_2(3)\unlhd G_0$ where $d=2$  and $p\in \{3,5,7,11,23\}$. Here $\SL_2(3)$ is irreducible on $V$.
\item[(H8)] $D_8\circ Q_8\unlhd G_0$ where $p^d=3^4$. Here $D_8\circ Q_8$ is irreducible on $V$.
\end{itemize}
Further, if $G_0\leq \GammaL_m(r)$ where $r^m=p^d$ and $m=2$ or $3$, then  one of the following holds: $G_0$ belongs to  \emph{(H0)};  $G_0$ belongs to   \emph{(H1)} with $n=m$; or $m=2$ and $G_0$ belongs to \emph{(H4)} or \emph{(H7)}.
\end{thm}

\begin{proof} Let $a$ be the minimal divisor of $d$ for which $G_0\leq \GammaL_a(p^{d/a})$. Let $s:=p^{d/a}$. If $a=1$, then (H0) holds, so we assume that $a\geq 2$. In particular, $(a,s)\neq (2,2)$ since  $\GL_2(2)=\GammaL_1(4)$. 
By~\cite[Appendix~1]{Lie1987}, one of the following holds: \textbf{(i)} $\SL_a(s)\unlhd G_0$; 
\textbf{(ii)} $\Sp_a(s)\unlhd G_0$; 
\textbf{(iii)} $G_2(s)'\unlhd G_0$ and $(a,p)=(6,2)$; 
\textbf{(iv)} $\SL_2(5)\unlhd G_0$, $a=2$ and $s\in \{9,11,19,29,59\}$; 
\textbf{(v)} $G_0=A_6$ or $A_7$ and $(a,s)=(4,2)$; 
\textbf{(vi)} $G_0=\SL_2(13)$ and $(a,s)=(6,3)$; \textbf{(vii)} $G_0$ normalises $E$ where $E\leq \GL_d(p)$ and either $E=Q_8$ and $p^d\in \{5^2,7^2,11^2,23^2\}$, or $E=D_8\circ Q_8$ and $p^d=3^4$. If (i) holds, then either (H1) holds with $n=a$, or  (H7) holds with $(a,s)=(2,3)$. If (ii) holds, then (H2) holds with $a=2n$ where $n\geq 2$ since $\Sp_2(q)=\SL_2(q)$. If 
one of  (iii), (iv), (v) or (vi) holds, then (H3), (H4), (H5) or (H6) holds, respectively. If (vii) holds with  $E=Q_8$, then $a=2$, and  (H7) holds by a computation in {\sc Magma}. Lastly, if (vii) holds and  $E=D_8\circ Q_8$, then by a computation in {\sc Magma}, either $a=2$ and (H4) holds, or $a=4$ and (H8) holds (cf. Remark~\ref{remark:2trans}).

Suppose that $G_0\leq \GammaL_m(r)$ where $r^m=p^d$ and $m=2$ or $3$. Then $a\leq m$.  If $a=1$ or $m$, then we are done by the observations above. Otherwise,  $a=2$ and $m=3$. 
Now $s^2=r^3$, so $s=t^3$ and $r=t^2$ where $t^6=p^d$. 
Further,  (H1) holds and $S:=\SL_2(t^3)\unlhd G_0$. Since $S$ is perfect,  $S\leq \GL_3(t^2)$, so
$C:=C_{\GL_6(t)}(S)$ contains both $Z:=Z(\GL_2(t^3))$ and $Z(\GL_3(t^2))$.  
However, $S$ is irreducible on $V=V_6(t)$, so by Schur's lemma (see~\cite[Lemma 1.5]{Isa1994}), $k:=\mathrm{Hom}_{\mathbb{F}_tS}(V,V)=C_{\End(V)}(S)$ is a  division ring.  Thus $k$ is a field by Wedderburn's theorem. Since $V$ is naturally a faithful $kS$-module, $t^6=|V|=|k|^i$ where $i\geq 2$, but $\mathbb{F}_{t^3}\subseteq k$, so $k=\mathbb{F}_{t^3}$ and $C=Z$, a contradiction. 
\end{proof}

Throughout this paper, we say that an affine $2$-transitive permutation group $G$ or its stabiliser $G_0$ belongs to (or lies in, is contained in, etc.) one of the classes (H0)--(H8) whenever $G_0$ satisfies the given description. Note that the groups in class (H7) are all soluble.

Next we state Kantor's classification~\cite{Kan1985} of the $2$-transitive affine linear spaces. 
Four exceptional linear spaces arise: the nearfield plane of order $9$ (see~\S\ref{ss:nearfield} or~\cite[\S8]{Lun1980}), the Hering plane of order $27$ (see~\cite[p.236]{Dem1997})  and two Hering spaces on $3^6$ points with line-size $9$~\cite{Her1985}.

\begin{thm}[\cite{Kan1985}]
\label{thm:Kantor}
Let $G$ be a $2$-transitive affine permutation group with socle $V:=V_d(p)$ where $d\geq 1$ and $p$ is  prime. Then $\mathcal{S}$ is a $G$-affine linear space with  line-size at least $3$ if and only if one of the following holds.
\begin{itemize}
\item[(i)] $\mathcal{S}$ has line set $\{V\}$ and $p^d\neq 2$.
\item[(ii)] $G_0\leq \GammaL_m(q)$ where $q^m=p^d$ and $\mathcal{S}$ has line set $\{\langle u\rangle_F +v: u\in V^*,v\in V\}$, where $F$ is a subfield of $\mathbb{F}_q$ and $|F|\neq 2$.
\item[(iii)] $G_0\leq N_{\GL_4(3)}(D_8\circ Q_8)$ where $p^d=3^4$ and one of the following holds, where $\mathcal{N}$ is the nearfield plane of order $9$.
\begin{itemize}
\item[(a)] $D_8\circ Q_8\unlhd G_0$ and $\mathcal{S}=\mathcal{N}$.
\item[(b)] $G_0=\SL_2(5)\langle \zeta_9\sigma_9\rangle\leq \GammaL_2(9)$  and   $\mathcal{S}=\mathcal{N}$ or $\mathcal{N}^{z}$ where $z:=\zeta_9^2$.
\end{itemize}
\item[(iv)] $G_0=\SL_2(13)$ where $p^d=3^6$ and $\mathcal{S}$ is the Hering plane of order $27$ or either  of the two Hering spaces with line-size $9$.
\end{itemize}
Moreover, if $G_0$ lies in~\emph{(H0)}--\emph{(H3)}, then (i) or (ii) holds, and if $G_0$ lies in~\emph{(H5)}, then (i) holds. 
\end{thm}

\begin{proof}
This follows from~\cite[\S 4]{Kan1985} and, in cases (iii) and (iv), a computation in {\sc Magma} using Lemmas~\ref{lemma:necessary} and~\ref{lemma:affine}. 
\end{proof}

\begin{remark}
\label{remark:2trans}
 The automorphism group of the nearfield plane of order $9$ is $V_4(3){:}N_0$   where $N_0:=N_{\GL_4(3)}(D_8\circ Q_8)=(D_8\circ Q_8).S_5=D_8\circ Q_8\circ (\SL_2(5)\langle \zeta_9\sigma_9\rangle)$ with $\SL_2(5)\leq \SL_2(9)$. Further, if $G_0\leq N_0$  such that $G_0$ is transitive on $V_4(3)^*$, then either $D_8\circ Q_8\unlhd G_0$, or $G_0=\SL_2(5)\langle \zeta_9\sigma_9\rangle$. Note that $\SL_2(5)\langle \zeta_9\sigma_9\rangle\simeq 2\nonsplit S_5^-$ (see \S\ref{ss:basicsgroups} for the definition of $2\nonsplit S_5^-$). 
\end{remark}

\subsection{Proof of Theorem~\ref{thm:primrank3plus}}
\label{ss:primrank3plus}

First we state the well-known structure theorem for primitive permutation groups of rank~$3$, which follows from the O'Nan-Scott Theorem (see~\cite{LiePraSax1988,Lie1987}). The groups in (ii) are said to be of \textit{grid type}.
 
\begin{prop}
\label{prop:primrank3}
For  a primitive permutation group $G$ of rank~$3$ with degree $n$, one of the following holds.
\begin{itemize}
\item[(i)] $G$ is almost simple.
\item[(ii)] $T\times T\unlhd G\leq K\wr S_2$ and $n=m^2$, where $K$ is a $2$-transitive group  of degree $m$ that is almost simple with socle $T$, and $G$ has subdegrees $1$, $2(m-1)$ and $(m-1)^2$.
\item[(iii)] $G$ is affine.
\end{itemize}
\end{prop}

 We will  use Proposition~\ref{prop:primrank3} to determine the  automorphism groups of certain affine  proper partial linear spaces, together with the following result.
 
 \begin{thm}[\cite{Gur1983}]
 \label{thm:Guralnick}
 Let $T$ be a non-abelian simple group that is transitive of degree $p^d$ for some prime $p$ and positive integer $d$. Then either $T$ is $2$-transitive, or $T\simeq \PSU_4(2)$ and 
 $p^d=27$, in which case $T$ is primitive of rank $3$ with subdegrees $1$, $10$ and $16$. 
 \end{thm}

Any non-trivial normal subgroup of a primitive group is transitive, so if $G$ is an almost simple primitive  group whose degree is a prime power not equal to $27$, then $G$ is $2$-transitive by Theorem~\ref{thm:Guralnick}. 

\begin{proof}[Proof of Theorem~\emph{\ref{thm:primrank3plus}}]
Let $H:=\Aut(\mathcal{S})$.
 Since $\mathcal{S}$ is proper and $G$ is a primitive permutation group of rank~$3$ on $V$, Remark~\ref{remark:PLSrank3} implies that $H$ is a primitive permutation group of rank~$3$ on $V$. Note that $H$ has degree $p^d$. If $H$ is almost simple with socle $T$, then $T$ is transitive of degree $p^d$.  Since $H$ is not $2$-transitive, $T\simeq \PSU_4(2)$ and $p^d=27$ by Theorem~\ref{thm:Guralnick}, but then $H$ has no affine primitive subgroups by a computation in~{\sc Magma}. (This also  
 follows from \cite[Proposition~5.1]{Pra1990}, except that
 the one possibility described for $G$ in \cite[Table 2]{Pra1990} is  not primitive as claimed, and this  is  corrected in~\cite[Proposition~2.1]{PraSax1992}.)
If $H$ is affine, then by~\cite[Proposition~5.1]{Pra1990}, the socle of $G$ is equal to the socle of $H$, so (ii) holds. Otherwise, by Proposition~\ref{prop:primrank3}, $T\times T\unlhd H\leq K\wr S_2$ and $d=2n$, where $K$ is a $2$-transitive group of degree $p^n$ that is almost simple with socle $T$. If  $T=A_{p^n}$, then  it is routine to verify that
 (i) holds (see~\cite{Dev2008} for details). If $T=\PSU_4(2)$ and $p^n=27$, then the order of $K$ is not divisible by $p^n-1$, so $K$ is not $2$-transitive, a contradiction. Thus, by~\cite[Proposition~5.1]{Pra1990}, $K$ has degree $p$ and is listed in~\cite[Table 2]{Pra1990}.

Since  $G$ is an affine primitive subgroup of $K\wr S_2$  with rank $3$, the subgroup $G\cap (K\times K)$ has index $2$ in $G$, and the projection $G^1$ of $G\cap (K\times K)$ onto the first coordinate of $K\times K$ is an affine $2$-transitive subgroup of $K$. Since $G^1$ has degree $p$, it is therefore  isomorphic to $\AGL_1(p)$. In particular, $G^1$ has order $p(p-1)$.
 By~\cite[Table 2]{Pra1990}, no such subgroup of $K$  exists when $(K,p)$ is one of $(\PSL_2(11),11)$, $(M_{11},11)$ or $(M_{23},23)$, so  $T=\PSL_m(q)$ and $K\leq \PGammaL_m(q)$ for some prime power $q$ and $m\geq 2$ such that $p=(q^m-1)/(q-1)$, and $G^1$ is a subgroup of the normaliser of a Singer cycle.  Now $q=r^f$ for some prime $r$, and  it follows from~\cite[p.187]{Hup1967} that the normaliser of a Singer cycle in $\PGammaL_m(q)$ has order $(q^m-1)/(q-1)fm=pfm$. Since $G^1$ has order $p(p-1)$, we conclude that $p-1$ divides $fm$. In particular,  
 $$fm\geq p-1=(q^m-1)/(q-1)-1\geq q^{m-1}\geq 2^{f(m-1)}\geq 2f(m-1).$$
 Thus $m=2$, and all of the inequalities above are equalities. Since $q\neq 2$ (or else $T$ is soluble), we  conclude that $q=4$.  Hence $T\simeq A_5$ and $p=5$, and as we saw above, (i) holds. \end{proof}

\section{The affine primitive permutation groups of rank~$3$}
\label{s:rank3}

In this section, we state and prove a modified  version of Liebeck's classification~\cite{Lie1987} of the affine primitive permutation groups of rank~$3$.  In particular, we provide more details about the rank~$3$ groups $G$ when  $G_0$  is  imprimitive or stabilises a  tensor product decomposition. Throughout this section, recall the definition of $\zeta_q$ and $\sigma_q$ from~\S\ref{ss:basicsvs}.

\begin{thm}[\cite{Lie1987}]
\label{thm:rank3}
Let $G$ be a finite affine primitive permutation group of rank~$3$ with socle $V:=V_d(p)$  where $d\geq 1$ and $p$ is prime. Then   $G_0$ belongs to one of the following classes.
\begin{itemize}
\item[(T)] The tensor product classes. Here $G_0$ stabilises a decomposition $V=V_2(q)\otimes V_m(q)$  where $q^{2m}=p^d$ and $m\geq 2$, and one of the following holds.
\begin{itemize}
\item[(1)] $\SL_2(q)\otimes \SL_m(q)\unlhd G_0$ where $q\geq 4$ or $(m,q)=(2,3)$.
\item[(2)] $\SL_2(5) \otimes \SL_m(q)\unlhd G_0$ where $q\in \{9,11,19,29,59\}$.  
\item[(3)] $1\otimes \SL_m(q)\unlhd G_0\leq N_{\GL_2(q)}(\SL_2(3))\otimes \GL_m(q)$ where $q\in \{3,5,7,11,23\}$   but $G_0\nleq \GammaL_1(q^2)\otimes \GL_m(q)$. Further,  $-1\otimes g\in G_0$ for some $g\in \GL_m(q)$.
\end{itemize}
\item[(S)] The subfield classes. Here  $q^{2n}=p^d$ and one of the following holds.
\begin{itemize}
\item[(0)] $V=V_2(q^n)$ and $\SL_2(q)\unlhd G_0\leq \GammaL_2(q^n)$ where $n=3$.
\item[(1)] $V=V_n(q^2)$ and $\SL_n(q)\unlhd G_0\leq \GammaL_n(q^2)$ where  $n\geq 2$.
\item[(2)] $V=V_n(q^2)$ and $A_7\unlhd  G_0 \leq \GammaL_n(q^2)$  where  $A_7\leq \SL_4(2)\simeq A_8$ and  $(n,q)=(4,2)$.
\end{itemize}

\item[(I)] The imprimitive classes.  Here $G_0$ stabilises a decomposition $V=V_n(p)\oplus V_n(p)$ where $2n=d$. 
Further, 
$G_0\leq \GammaL_m(q)\wr  S_2$ where $q^m=p^n$  and one of the following holds. 
\begin{itemize}
\item[(0)] $G_0\leq \GammaL_1(q)\wr S_2$ and  $m=1$. 
\item[(1)] $\SL_m(q)\times \SL_m(q)\unlhd G_0$ where $m\geq 2$ and $(m,q)\neq (2,2),(2,3)$.
\item[(2)] $\Sp_m(q) \times \Sp_m(q) \unlhd G_0$ where $m$ is even and $m\geq 4$.
\item[(3)] $G_2(q)'\times G_2(q)' \unlhd G_0$ where $m=6$ and $q$ is even.
\item[(4)] $S\times S\unlhd G_0 \leq N_{\GammaL_m(q)}(S)\wr S_2$ where $S=\SL_2(5)$, $m=2$ and $q\in \{9,11,19,29,59\}$.
\item[(5)] $G_0=A_6\wr S_2$ or $A_7\wr S_2$ where $m=4$ and $q=2$.
\item[(6)] $G_0=\SL_2(13)\wr S_2$ where $m=6$ and $q=3$.
\item[(7)] $G_0\leq N_{\GL_m(q)}(\SL_2(3))\wr S_2$ where   $m=2$ and $q\in \{3,5,7,11,23\}$.
\item[(8)] $E\times E\unlhd G_0\leq N_{\GL_m(q)}(E)\wr S_2$ where $E=D_8\circ Q_8$, $m=4$ and $q=3$.
\end{itemize}
\item[(R)] The remaining  infinite classes.
\begin{itemize}
\item[(0)] $G_0\leq \GammaL_1(p^d)$.
\item[(1)] $\SU_n(q)\unlhd G_0$ where $q^{2n}=p^d$ and $n\geq 3$.
\item[(2)] $\Omega_{2m}^\varepsilon(q)\unlhd G_0$ where $q^{2m}=p^d$ and either $m\geq 3$ and $\varepsilon=\pm$, or $m=2$ and $\varepsilon=-$.
\item[(3)] $V=\bigwedge^2(V_5(q))$ and $\SL_5(q)\unlhd G_0$ where $q^{10}=p^d$.
\item[(4)] $\Sz(q)={}^2B_2(q)\unlhd G_0$ where   $q^4=p^d$ and  $q=2^{2n+1}$ where  $n\geq 1$.
\item[(5)] $V=V_n(q)$ is a spin module and $\spin_m^\varepsilon(q)\unlhd G_0$
where $q^n=p^d$ and $(m,n,\varepsilon)$ is one of $(7,8,\circ)$ or $(10,16,+)$.
\end{itemize}
\item[(E)] The extraspecial class. Here $E\unlhd G_0$  and $E$ is irreducible on $V_d(p)$ where  $(E,p^d,G)$ is given by Table~\emph{\ref{tab:E}}. 
\item[(AS)] The almost simple class. Here $S\unlhd G_0\leq \GammaL_n(q)$   and $S$ is irreducible on $V_d(p)$ where   $(S,p^d,G)$  is   given by  Table~\emph{\ref{tab:AS}},   $q^n=p^d$, and    $q$ is  given by the embedding of $S$ in $\SL_n(q)$. 
\end{itemize}
\end{thm}

 \begin{table}[!h]
\renewcommand{\baselinestretch}{1.1}\selectfont
\centering
\begin{tabular}{ c c c c    }
\hline
 $E$ &  $p^d$ & Subdegrees of $G$  & Conditions  \\
\hline
&$7^2$  & $24,24$  &  $G_0\nleq \GammaL_1(7^2)$\\
&$13^2$  & $72,96$  & \\
&$17^2$  & $96,192$  & \\
$Q_8$ &$19^2$  & $144,216$  &\\
&$23^2$  & $264,264$  &\\
&$29^2$ & $168,672$  &\\
&$31^2$ & $240,720$  &\\
&$47^2$ & $1104,1104$  &\\
\hline 
 $3^{1+2}$ & $2^6$ & $27,36$  &  $\SU_3(2)\not\leq G_0\leq \GammaU_3(2)$ \\
 $D_8\circ Q_8$ & $5^4$ & $240,384$  &    \\
 $D_8\circ Q_8\circ \langle \zeta_5\rangle$ & $5^4$ & $240,384$    & $G_0\nleq N_{\GL_4(5)}(D_8\circ Q_8)$ \\
 $D_8\circ Q_8$ & $7^4$ & $480,1920$   &  \\
 $Q_8\circ Q_8\circ Q_8$ & $3^8$  &  $1440,5120$    & \\
\hline
\end{tabular}
\caption{Extraspecial class (E)}
\label{tab:E}
\end{table}

\begin{table}[!h]
\renewcommand{\baselinestretch}{1.1}\selectfont
\centering
\begin{tabular}{ c c c c   }
\hline
 $S$ & $p^d$ &    Embedding of $S$ & Subdegrees of $G$   \\
\hline
 & $3^4$ & & $40,40$  \\
    & $31^2$ & & $360,600 $  \\
    & $41^2$ & & $480,1200 $  \\
 $2\nonsplit A_5$   & $7^4$ &  $S< \SL_2(p^{d/2})$ & $960,1440 $ \\
    & $71^2$ & & $840,4200 $  \\
    & $79^2$ & & $1560,4680 $  \\
    & $89^2$ & & $2640,5280 $  \\
\hline
$3\nonsplit A_6$ & $2^6$ & $S< \SL_3(4)$ & $18,45$  \\
 $2\nonsplit A_6$   & $5^4$  & $S< \Sp_4(5)$ & $144,480$\\
$2\nonsplit A_7$ & $7^4$  & $S< \Sp_4(7)$ & $720,1680$\\
$A_9$ & $2^8$  & $S< \Omega_8^+(2)$& $120,135$\\
$A_{10}$ & $2^8$  & $S< \Sp_8(2)$& $45,210$\\
$\PSL_2(17)$ & $2^8$  & $S< \Sp_8(2)$& $102,153$\\
$2\nonsplit\PSL_3(4)$ & $3^6$ & $S< \Omega_6^-(3)$& $224,504$ \\
$2\nonsplit\PSU_4(2)$ & $7^4$  & $S< \SL_4(7)$& $240,2160$ \\
$M_{11}$ & $3^5$  & $S< \SL_5(3)$ & $22,220$ \\
&&&  $110,132$\\
$M_{24}$ & $2^{11}$  & $S< \SL_{11}(2)$ & $276,1771$ \\
&&& $759,1288$ \\
$2\nonsplit\Suz$ & $3^{12}$ & $S< \Sp_{12}(3)$ & $65520,465920$  \\
$2\nonsplit G_2(4)$ &  $3^{12}$ & $S< 2\nonsplit\Suz< \Sp_{12}(3)$  & $65520,465920$ \\
$J_2$ & $2^{12}$  & $S< G_2(4)< \Sp_6(4)$ & $1575,2520$\\
$2\nonsplit J_2$ & $5^6$  & $S<\Sp_6(5)$& $7560,8064$\\  
\hline
\end{tabular}
\caption{Almost simple class (AS)}
\label{tab:AS}
\end{table}

\begin{table}[!h]
\renewcommand{\baselinestretch}{1.1}\selectfont
\centering
\begin{tabular}{ l l l   }
\hline
 Class of $G$ & $p^d$ & Subdegrees of $G$ \\
\hline
(T): tensor product & $q^{2m}$ & $(q+1)(q^m-1),q(q^m-1)(q^{m-1}-1)$\\
(S): subfield & $q^{2n}$ & $(q+1)(q^n-1),q(q^n-1)(q^{n-1}-1)$\\
(I): $G_0$ imprimitive & $p^{2n}$ & $2(p^n-1),(p^n-1)^2$\\
(R0): $G_0\leq \GammaL_1(p^d)$ & $p^d$ & given in~\cite[\S 3]{FouKal1978} \\
(R1): $\SU_n(q)\unlhd G_0$ & $q^{2n}$ & $(q^n-1)(q^{n-1}+1),q^{n-1}(q-1)(q^n-1)$, $n$ even\\
& & $(q^n+1)(q^{n-1}-1),q^{n-1}(q-1)(q^n+1)$, $n$ odd \\
(R2): $\Omega_{2m}^\varepsilon(q)\unlhd G_0$ & $q^{2m}$ & $(q^m-1)(q^{m-1}+1),q^{m-1}(q-1)(q^m-1)$, $\varepsilon=+$ \\
& & $(q^m+1)(q^{m-1}-1),q^{m-1}(q-1)(q^m+1)$, $\varepsilon=-$ \\
(R3): $\SL_5(q)\unlhd G_0$ & $q^{10}$ & $(q^5-1)(q^2+1),q^2(q^5-1)(q^3-1)$\\
(R4): $\Sz(q)\unlhd G_0$ & $q^4$ & $(q^2+1)(q-1),q(q^2+1)(q-1)$\\
(R5): $\spin_{m}^\varepsilon(q)\unlhd G_0$ & $q^8$ &  $(q^4-1)(q^3+1),q^3(q^4-1)(q-1)$, $(m,\varepsilon)=(7,\circ)$   \\
& $q^{16}$ &  $(q^8-1)(q^3+1),q^3(q^8-1)(q^5-1)$, $(m,\varepsilon)=(10,+)$   \\
\hline
\end{tabular}
\caption{Classes (T), (S), (I) and (R)}
\label{tab:subdegree}
\end{table}

Throughout this paper, we say that an affine  permutation group $G$ or its stabiliser $G_0$ belongs to (or lies in, is contained in, etc.) one of the classes (T1)--(T3), (S0)--(S2), (I0)--(I8), (R0)--(R5), (E) or (AS) if $G$ has rank~$3$ and $G_0$ satisfies the given description. We also say that  $G$  or  $G_0$ belongs to class (T), (S), (I) or (R), respectively.  We remark that if $G$ belongs to one of the classes   (T1)--(T3), (S0)--(S2), (I0)--(I8) with $p^n>2$, (R1)--(R5), (E) or (AS), then $G$ is primitive:    this holds for classes (E) and (AS) by Lemma~\ref{lemma:primitiveirreducible} and can otherwise be verified using 
Lemma~\ref{lemma:primitive} and the descriptions of the orbits of $G_0$ that are given in subsequent sections.

 In Table~\ref{tab:subdegree}, we list the subdegrees of the groups in classes (T), (S), (I) and (R), as given by~\cite[Table 12]{Lie1987}. Note that the groups in the classes (S0), (I0) and (R0) are precisely those that are listed in Theorem~\ref{thm:main}(iv). We caution the reader that the groups in class (I) are not necessarily  subgroups of $\GammaL_{2m}(q)$. For example, if $q$ is not prime, then
$\GammaL_m(q)\wr S_2$ lies in class (I0) or (I1) but is not a subgroup of $\GammaL_{2m}(q)$ (see \S\ref{s:(I)} for more details). We  also caution the reader that (R5) includes the case where $m=7$ and $q$ is even, in which case $\spin_7(q)\simeq \Sp_6(q)$. 

 In order to prove Theorem~\ref{thm:rank3}, we first 
 consider the case where $G_0$ stabilises a tensor product decomposition. In the following, recall the definition of  $\sigma_q$ from~\S\ref{ss:basicsvs}.

\begin{thm}
\label{thm:tensorgroups}
Let $G$ be an affine  permutation group of rank~$3$ on $V$ where $G_0$ stabilises a decomposition $V=V_2(q)\otimes V_m(q)$ for some $m\geq 2$ and  prime power $q$.   Then either $G_0$ belongs to one of the classes~\emph{(R0)},~\emph{(I0)},~\emph{(T1)},~\emph{(T2)} or~\emph{(T3)}, or 
one of the following holds.
\begin{itemize}
\item[(T4)] $1\otimes A_7\unlhd  G_0$ where  $(m,q)=(4,2)$. 
\item[(T5)] $1\otimes \SL_m(q)\unlhd G_0\leq (\GL_1(q^2)\otimes \GL_m(q)){:}\langle (t\otimes 1)\sigma_q \rangle$ for some $t\in \GL_2(q)$ where  $\GammaL_1(q^2)=\GL_1(q^2){:}\langle t\sigma_q\rangle$.
\item[(T6)] $\SL_2(q)\otimes 1\unlhd G_0\leq (\GL_2(q)\otimes \GL_1(q^3)){:}\langle (1\otimes t)\sigma_q \rangle$ for some $t\in \GL_m(q)$ where $m=3$ and  $\GammaL_1(q^3)=\GL_1(q^3){:}\langle t\sigma_q\rangle$.
\end{itemize}
\end{thm}

We will see in the proof of Theorem~\ref{thm:rank3} that if $G_0$ satisfies  (T4), (T5) or (T6), then $G_0$ lies in class (S2), (S1) or (S0),  respectively, and we will see in~\S\ref{s:(S1)}, \S\ref{s:(S0)} and~\S\ref{s:extra} that   these two different viewpoints of $G_0$ are important. Throughout this paper, we say that an affine permutation group $G$ or its stabiliser $G_0$ belongs to (or lies in, is contained in, etc.) one of the classes (T4)--(T6) if $G$ has rank~$3$ and $G_0$ satisfies the given description. We also enlarge the class (T) so that it contains the groups in classes (T4)--(T6). Note that the statement of  Theorem~\ref{thm:tensorgroups} is similar to that of~\cite[Lemma~55]{BilMonFra2015}, but~\cite[Lemma~55]{BilMonFra2015}  is proved as part  of an analysis of  $2$-$(v,k,1)$ designs admitting  rank~$3$ automorphism groups.

\begin{proof}[Proof of Theorem~\emph{\ref{thm:tensorgroups}}]
Let $U:=V_2(q)$, $W:=V_m(q)$ and $A:=  (\GL(U)\otimes \GL(W)){:}\Aut(\mathbb{F}_{q})$ (see~\S\ref{ss:basicsvs}). We write elements of $A$ as $g\otimes h$ for $g\in \GammaL(U)$ and  $h\in \GammaL(W)$ where $g$ and $h$ are $\sigma$-semilinear for some $\sigma\in \Aut(\mathbb{F}_q)$.  Let $H_0:=A\cap G_0$.   Either $H_0= G_0$, or  $m=2$ and  $G_0=H_0\langle (g_0\otimes h_0) \tau\rangle$ for some $g_0,h_0\in \GammaL(U)$ where $\tau\in \GL(V)$ is defined by $u\otimes w\mapsto w\otimes u$ for all $u,w\in U=W$; let $\alpha:=(g_0\otimes h_0) \tau$. The orbits of $G_0$ on $V^*$ are the simple tensors $\{u\otimes w: u\in U^*, w\in W^*\}$ and the non-simple tensors $\{u_1\otimes w_1+u_2\otimes w_2 : w_1,w_2\in W, \ \dim\langle w_1,w_2\rangle =2\}$ for any choice of basis $\{u_1,u_2\}$ of $U$. Let $p$ and $d$ be such that $q^{2m}=p^d$ where $p$ is prime.

Let $Z_U:=Z(\GL(U))$, $Z_W:=Z(\GL(W))$, $Z_V:=Z(\GL(V))$ and $Z_0:=Z_V\cap G_0=Z_V\cap H_0$.
Define $G_0^U$ to be the set of $g\in \GammaL(U)$ for which there exists  $h\in\GammaL(W)$ such that $g$ and $h$ are both $\sigma$-semilinear for some $\sigma\in\Aut(\mathbb{F}_q)$ and $ g\otimes h\in H_0$.  Define $G_0^W$ similarly. Then $G_0^U\leq \GammaL(U)$ and $G_0^W\leq \GammaL(W)$. Define $\varphi_U:A\to  \PGammaL(U) $ by  $g\otimes h  \mapsto  Z_Ug$. Define $\varphi_W$ similarly. Then $\varphi_U$ and $\varphi_W$  are homomorphisms with kernels $1\otimes \GL(W)$ and $\GL(U)\otimes 1$, respectively. Let $K_U:=\ker(\varphi_U)\cap H_0$ and $K_W:=\ker(\varphi_W)\cap H_0$. Note that $G_0^UZ_U/Z_U=H_0\varphi_U\simeq H_0/K_U$ and $G_0^WZ_W/Z_W=H_0\varphi_W\simeq H_0/K_W$. Now $K_U=1\otimes H_W$ for some $H_W\leq \GL(W)$ and $K_W=H_U\otimes 1$ for some $H_U\leq \GL(U)$. Note that $H_U\unlhd G_0^U$ and $H_W\unlhd G_0^W$. 

First we claim that if $H_0\neq G_0$, then $G_0^U$ and $G_0^W$ are conjugate in $\GammaL(U)$. If  $g\in G_0^U$, then $g\otimes h\in H_0$ for some $h\in \GammaL(U)$, and $h^{h_0}\otimes g^{g_0}=(g\otimes h)^\alpha\in H_0$, so $(G_0^U)^{g_0}\leq G_0^W$. Similarly, $(G_0^W)^{h_0}\leq G_0^U$. By comparing orders, it follows that  $(G_0^U)^{g_0}=G_0^W$.

Next we claim that $G_0^U$ is transitive on $\PG(U)$ and  $G_0^W$ is transitive on  $\PG(W)$. Since $G_0$ is transitive on the set of simple tensors of $V$, the claim holds when $G_0=H_0$. Suppose then that $G_0\neq H_0$.
Now $G_0^U$ and $G_0^W$ have the same number of orbits, say $r$, on $\PG(U)$ since they are conjugate, so  $H_0$ has at least $r^2$ orbits on the set of simple tensors of $V$. On the other hand, $G_0$ is transitive on the set of simple tensors and $[G_0:H_0]=2$, so $H_0$ has at most $2$ orbits on the set of simple tensors. Thus $r=1$, as desired. 

\textbf{Case 1: $G_0$ is soluble.}  
We may assume that $G_0$ does not belong to class (I0) or (R0). Then $(m,q)\neq (2,2)$,  or else $G_0\leq \GL_2(2)\wr S_2=\GammaL_1(4)\wr S_2$.  If $G_0$ stabilises a decomposition $V=V_\ell(p)\oplus V_\ell(p)$ where $2\ell=d$, then $q^{2m}=p^{2\ell}$ and $G_0$ has orbit sizes $2(p^\ell-1)$ and $(p^\ell-1)^2$ on $V^*$, but $G_0$ also has orbit sizes $(q+1)(q^m-1)$ and $q(q^m-1)(q^{m-1}-1)$, so $(m,q)=(2,2)$, a contradiction. 
By~\cite[Theorem~1.1]{Fou1969} and a consideration of the suborbit lengths of $G$, we conclude that $(m,q)=(2,3)$ or $(3,3)$.   Using {\sc Magma}, we determine that if $(m,q)=(2,3)$, then $G_0$ belongs to class (T1) or (T5), and if $(m,q)=(3,3)$, then  either $\GL(U)\otimes \GL_1(q^3)\unlhd G_0\leq \GL(U)\otimes \GammaL_1(q^3)$, in which case $G_0$ belongs to (T6), or $G_0=\GammaL_1(q^2)\otimes \GammaL_1(q^3)\leq \GammaL_1(q^6)$, a contradiction.
 
\textbf{Case 2: $G_0$ is insoluble.} Now $G_0^U$ and $G_0^W$ are not both soluble, for $K_U/Z_0=K_U/(K_U\cap K_W)\simeq K_UK_W/K_W\unlhd H_0/K_W\simeq G_0^WZ_W/Z_W$, and similarly $K_W/Z_0$ is isomorphic to a normal subgroup of $H_0/K_U\simeq G_0^UZ_U/Z_U$, so if $G_0^U$ and $G_0^W$ are both soluble, then $K_W$ and $K_U$ are both soluble, but then $H_0$ is soluble, so $G_0$ is soluble, a contradiction.

If $m\geq 4$, then $G_0^W$ is transitive on the lines of $\PG(W)$ since $G_0$ is transitive on the set of non-simple tensors of $V$,  so  either $G_0^W$ is $2$-transitive on $\PG(W)$, or  $G_0^W=\GammaL_1(2^5)$ with $m=5$ and $q=2$~\cite{Kan1973}. In this latter case,  both $G_0^U$ and $G_0^W$ are soluble, a contradiction. Hence $G_0^W$ is $2$-transitive on $\PG(W)$, so  either $\SL(W)\unlhd G_0^W$, or  $G_0^W=A_7$ with $m=4$ and $q=2$~\cite{CamKan1979}. If instead $m\leq 3$ and $G_0^W$ is insoluble, then since $G_0^WZ_W$ is transitive on $W^*$,  either $\SL(W)\unlhd G_0^W$, or  $\SL_2(5)\unlhd G_0^W$ with $m=2$ and $q\in \{9,11,19,29,59\}$  by Theorem~\ref{thm:2trans}.  Thus one of the following holds: \textbf{(a)}  $\SL(W)\unlhd G_0^W$ where $m\geq 2$ and $(m,q)\neq (2,2)$ or $(2,3)$; \textbf{(b)} $\SL_2(5)\unlhd G_0^W$  where $m=2$ and  $q\in \{9,11,19,29,59\}$; \textbf{(c)}  $G_0^W$ is soluble where $m=2$ or $3$; \textbf{(d)} $G_0^W=A_7$ where $m=4$ and $q=2$. Further, since $G_0^UZ_U$ is transitive on $U^*$, either $\SL(U)\unlhd G_0^U$, or $\SL_2(5)\unlhd G_0^U$  where  $q\in \{9,11,19,29,59\}$, or $G_0^U$ is soluble.

Suppose for a contradiction that $H_W\leq Z_W$. Then $K_U=Z_0$, so $H_0/Z_0\simeq G_0^UZ_U/Z_U\leq \PGammaL(U)$. First suppose that $G_0= H_0$.  Now $(q+1)(q^m-1)$ divides $|G_0|$, so $(q+1)(q^m-1)$ divides $q(q-1)(q^2-1)e$ where $q=p^e$. Then $q^{m-1}+\cdots +q+1$ divides $q(q-1)e$ and therefore $(q-1)e$. But $q\geq e$, so  $q^{m-1}+\cdots +q+1>(q-1)e$ for $m\geq 3$, a contradiction. Thus $m=2$. Since $(q+1,q-1)=1$ or $2$, it follows that $q+1$ divides $2e$, a contradiction. Now suppose that $G_0\neq H_0$. Then  $(q+1)(q^2-1)$ divides $|G_0|=2|H_0|$, so $q+1$ divides $2(q-1)e$. If $q$ is even, then $q+1$ divides $e$, a contradiction. Thus $q$ is odd, and $q+1$ divides $4e$, so $q=3$, but then $G_0^U$ and $G_0^W$ are soluble, a contradiction.

Thus $H_W$ is a non-central normal subgroup of $G_0^W$.  If (d) holds, then $H_W=A_7$, so $G_0$ lies in class  (T4). Hence we may assume that one of (a), (b) or (c) holds. Since $G_0^WZ_W/Z_W$ is almost simple when (a) or (b) holds,  $\SL(W)\unlhd H_W$ when (a) holds and $\SL_2(5)\unlhd H_W$ when (b) holds.

If $H_U\leq Z_U$, then $K_W=Z_0$, so $  G_0^WZ_W/Z_W \simeq H_0/Z_0 \unrhd K_U/Z_0 \simeq H_WZ_W/Z_W$, and since the quotient of $G_0^WZ_W/Z_W$ by $H_WZ_W/Z_W$ is soluble, it follows that  $H_0/K_U$ is soluble and therefore that $G_0^U$ is soluble. Hence when  $G_0^U$ is insoluble, either $\SL(U)\unlhd H_U$ and $q\geq 4$, or $\SL_2(5)\unlhd H_U$  where $q\in \{9,11,19,29,59\}$. 

If $G_0\neq H_0$, then since $G_0^U$ and $G_0^W$ are conjugate, we may assume  that $S\otimes  S\unlhd H_0$ where either $S=\SL(U)$, or $S=\SL_2(5)$ and $q\in \{9,11,19,29,59\}$.  Now $S$ is the unique subgroup of $G_0^U$ and $G_0^W$ that is  isomorphic to $S$, so 
$(S\otimes S)^\alpha=S^{h_0}\otimes S^{g_0}=S\otimes S$. Thus $S\otimes S\unlhd G_0$.

We conclude that one of the following holds: \textbf{(i)} $\SL(U)\otimes \SL(W)\unlhd G_0$ where $q\geq 4$; \textbf{(ii)} $\SL_2(5) \otimes \SL(W)\unlhd G_0=H_0$ where  $q\in\{9,11,19,29,59\}$; \textbf{(iii)} $1\otimes \SL(W)\unlhd G_0=H_0$ and $G_0^U$ is soluble; \textbf{(iv)} $\SL(U)\otimes \SL_2(5)\unlhd G_0=H_0$ where $m=2$ and $q\in\{9,11,19,29,59\}$; \textbf{(v)} $\SL_2(5)\otimes \SL_2(5)\unlhd G_0$ where $m=2$ and $q\in\{9,11,19,29,59\}$; \textbf{(vi)} $1\otimes \SL_2(5)\unlhd G_0=H_0$ and $G_0^U$ is soluble where $m=2$ and $q\in\{9,11,19,29,59\}$; \textbf{(vii)} $\SL(U)\otimes 1\unlhd G_0=H_0$ and $G_0^W$ is soluble where $m=2$ or $3$; \textbf{(viii)} $\SL_2(5)\otimes 1\unlhd G_0=H_0$ and $G_0^W$ is soluble where $m=2$ or $3$ and $q\in\{9,11,19,29,59\}$.    Moreover, if $m\leq 3$ and $G_0^W$ is soluble, then  since $G_0^WZ_W$ is  transitive on $W^*$, either $G_0^W\leq \GammaL_1(q^m)$, or $\SL_2(3)\unlhd  G_0^WZ_W$ where $m=2$ and $q\in \{3,5,7,11,23\}$  by Theorem~\ref{thm:2trans}. A similar result applies to $G_0^U$.

\textbf{Cases (i), (ii), (iv) and (v).}
If (i) or (ii) holds, then  $G_0$ lies in class (T1) or (T2), respectively. If (iv) holds, then $G_0$ lies in class (T2) by interchanging $U$ and $W$. Suppose  that (v) holds. If $q\neq 9$, then $G_0\leq (\SL_2(5)\otimes \SL_2(5) Z_V).2$, but $q$ divides $|G_0|$, a contradiction.  Thus $q=9$, so  $G_0\leq (\SL_2(5)\otimes \SL_2(5) Z_V).2^2$, but, using {\sc Magma}, we determine that this group is not transitive on the set of non-simple tensors of $V$, a contradiction.

\textbf{Case (vi).} Either $G_0^U\leq \GammaL_1(q^2)$, or $\SL_2(3)\unlhd  G_0^UZ_U$ and $q=11$, in which case $G_0^U\leq \GL_2(3)Z_U$. In particular, the order of $G_0^U$ divides $(q^2-1)2e$, or $240$, respectively. Further, $\SL_2(5)\unlhd H_W\leq \GL(W)$, so $H_W\leq \SL_2(5)\circ Z_W$. In particular, $1\otimes H_W=K_U$ has order dividing $60(q-1)$. Recall that   $G_0^UZ_U/Z_U\simeq G_0/K_U$,  and $q$ divides $|G_0|$. If $q\neq 9$, then $q$ divides $|G_0^U|$, a contradiction. Similarly, if $q=9$, then $3$ divides $|G_0^U|$, a contradiction.
 
\textbf{Case (viii).} We have just proved that $m\neq 2$, so $m=3$ and $G_0^WZ_W\leq \GammaL_1(q^3)$.  In particular, the order of $G_0^WZ_W$ divides $(q^3-1)3e$. Further, $\SL_2(5)\unlhd H_U\leq \GL(U)$, so $H_U\otimes 1=K_W$ has order dividing $60(q-1)$. Recall that   $G_0^WZ_W/Z_W\simeq G_0/K_W$,  and $q$ divides $|G_0|$.
 If $q\neq 9$, then $q$ divides $|G_0^W|$, a contradiction. Thus $q=9$. Now $q(q^2-1)(q^3-1)(q-1)$ divides $|G_0||Z_W|$, which divides $(q^3-1)3e\cdot 60(q-1)$, a contradiction.
    
\textbf{Case (iii).} Either $\SL_2(3)\unlhd  G_0^UZ_U$ where  $q\in \{3,5,7,11,23\}$ and $G_0^U\nleq  \GammaL_1(q^2)$, or $G_0^U\leq  \GammaL_1(q^2)\leq \GammaL(U)$.  If the former holds, then  $G_0\leq (N_{\GL(U)}(\SL_2(3)))\otimes \GL(W)$, and   $-1\in G_0^U$ since $Q_8=\SL_2(3)'\leq  G_0^U$ and $Q_8$ is irreducible on $U$. Thus $-1\otimes g\in G_0$ for some $g\in \GL(W)$ and $G_0$ lies in class (T3).  
 Otherwise, since $\GammaL_1(q^2)=\GL_1(q^2){:}\langle t\sigma_q\rangle$ for some $t\in \GL(U)$,  it follows that $G_0$ lies in class (T5). 
 
\textbf{Case (vii).} If $m=2$, then $G_0$ lies in class (T5) by interchanging $U$ and $W$, so we may assume that $m=3$. Now $G_0^W\leq \GammaL_1(q^3)\leq \GammaL(W)$, so $G_0$ lies in class (T6). 
 \end{proof}
 
Next we consider the case where $G_0$ is imprimitive.
 Note that our result in this case is similar to~\cite[Proposition~39]{BilMonFra2015}, but ~\cite[Proposition~39]{BilMonFra2015} is  proved as part of an analysis of  $2$-$(v,k,1)$ designs admitting  rank~$3$ automorphism groups and omits  the cases where $G_0$ is soluble or $(n,p)=(6,2)$.

\begin{thm}
\label{thm:A2groups}
Let $G$ be an affine permutation group of rank~$3$  on $V$ where $G_0$ stabilises a decomposition $V=V_n(p)\oplus V_n(p)$ for some $n\geq 1$ and prime $p$.  Then $G_0$ belongs to one of the classes~\emph{(I0)}--\emph{(I8)}. 
\end{thm}

\begin{proof}
The group $G_0$ is a subgroup of $\GL_n(p)\wr\langle \tau\rangle $ where   $\tau$ is the involution in $\GL_{2n}(p)$  defined by  $(u_1,u_2)^\tau =(u_2,u_1)$ for all $u_1,u_2\in V_n(p)$. Let $V_1:=\{(u,0): u \in V_n(p)\}$ and $V_2:=\{(0,u):u\in V_n(p)\}$. For each $i\in \{1,2\}$, let $\pi_i$ denote the projection of $G_{0,V_1}=G_0\cap (\GL_n(p)\times \GL_n(p))$ onto the $i$-th factor of $\GL_n(p)\times \GL_n(p)$, and let $K_i$ and $G_0^i$ denote the kernel and image of $\pi_i$, respectively.  The orbits of $G_0$ on $V^*$ are $V_1^*\cup V_2^*$ and $V_n(p)^*\times V_n(p)^*$. In particular,  $G_0^i$ is transitive on $V_n(p)^*$ for both $i$, and $G_{0,V_1}$ is an index $2$ subgroup of $G_0$. Now $G_0=G_{0,V_1}\langle (t,s)\tau\rangle$ for some $t,s\in\GL_n(p)$. Conjugating $G_0$ by $(s^{-1},1)$ if necessary, we may assume that $s=1$, in which case $G_0^1=G_0^2$ and $t\in G_0^1$ since $((t,1)\tau)^2=(t,t)$. Now $G_0\leq G_0^1\wr \langle\tau\rangle$, and there exists $H\unlhd G_0^1$ such that $K_1=1\times H$ and $K_2=H\times 1$.

By Theorem~\ref{thm:2trans}, $G_0^1$ belongs to one of the classes (H0)--(H8). In particular, $G_0^1\leq \GammaL_m(q)$ where $q^m=p^n$, and either $G_0^1$ belongs to (H5)--(H8) and $q=p$, or $G_0^1$ belongs to (H0)--(H4) and $q$ is specified  by Theorem~\ref{thm:2trans}. Now $G_0\leq \GammaL_m(q)\wr \langle \tau\rangle$. Write $q=p^e$ and let $Z:=Z(\GL_m(q))$.

If $G_0^1$ lies in class (H0) (i.e., if $m=1$), then $G_0$ lies in class (I0), so we assume that $m\geq 2$.  If $G_0^1$ lies in (H7), then $G_0$ lies in (I7), and if $G_0^1$ lies in (H8), then $G_0$ lies in class (I8) by a computation in {\sc Magma}, so we may assume that $G_0^1$ lies in one of the classes  (H1)--(H6). 

Suppose for a contradiction that $H\leq Z$. Recall that $(q^m-1)^2$ divides $|G_0|=2|G_{0,V_1}|=2|G_0^1||H|$. Hence $(q^m-1)^2$ divides $2|\GammaL_m(q)|(q-1)$. First suppose that $em\geq 3$ and $(em,p)\neq (6,2)$. By Zsigmondy's theorem~\cite{Zsi1892}, there exists a primitive prime divisor $r$ of $p^{em}-1$ (see~\cite[Theorem 5.2.14]{KleLie1990}). Let $f$ be the largest integer for which $r^f$ divides $q^m-1$. Now $r^{2f}$ divides $(q^m-1)^2$, and $r$ is odd, so $r^{2f}$ divides $|\GammaL_m(q)|(q-1)$, but $r$ is coprime to $q$ and does not divide $p^i-1$ for $i<em$, so $r$ must divide $e$. However, $r\equiv 1\mod em$ by~\cite[Proposition 5.2.15]{KleLie1990}, a contradiction. Thus either $em=2$ (so $m=2$ and $e=1$), or $(em,p)=(6,2)$. Suppose that $m=2$. Now $(q^2-1)^2$ divides $2q(q-1)^2(q^2-1)e$, so $q+1$ divides $2(q-1)e$. This is impossible when $p=2$ and $e=3$. Hence $e=1$. Now $p+1$ divides $2(p-1)$, but $(p+1,p-1)=(p-1,2)$, so $p=3$, in which case $G_0^1$ belongs to (H7), a contradiction. Further, if $q=4$ and $m=3$, then $(4^3-1)^2$ divides $6|\GammaL_3(4)|$, a contradiction. 

Thus $q=2$ and $m=6$. In particular, $H=1$, and $G_0^1$ is one of $G_2(2)'=\PSU_3(3)$, $G_2(2)=\PSU_3(3){:}2$, $\Sp_6(2)$ or $\SL_6(2)$. Since $(2^6-1)^2$ divides $2|G_0^1|$, we conclude that $G_0^1=\SL_6(2)$. Now  there exists a bijective map $\alpha: \SL_6(2)\to \SL_6(2)$ such that $G_{0,V_1}=\{(g,g^\alpha):g\in \SL_6(2)\}$. Moreover, $\alpha$ is a homomorphism, so $\alpha\in\Aut(\SL_6(2))=\SL_6(2){:}\langle \iota\rangle$ where $\iota$ maps each element of $\SL_6(2)$ to its inverse transpose. Since $G_{0,V_1}$ has index $2$ in $G_{0}$ and $(2^6-1)^2$ is odd, it follows that $G_{0,V_1}$ is transitive on $V_6(2)^*\times V_6(2)^*$. Further, $G_{0,V_1}$ is conjugate in $\SL_6(2)\times \SL_6(2)$ to $\{(g,g):g\in \SL_6(2)\}$ or $\{(g,g^\iota):g\in \SL_6(2)\}$. However, neither of these is transitive, for if we view $V_6(2)$ as $\mathbb{F}_2^6$ and define $v_1:=(1,0,0,0,0,0)$ and $v_2:=(0,1,0,0,0,0)$, then $(v_1,v_1)$ and $(v_1,v_2)$ lie in different orbits.

Hence $H$ is a non-central normal subgroup of $G_0^1$. If $G_0^1$ lies in (H5) or (H6), then $H=G_0^1$, so $G_0$ lies in (I5) or (I6), respectively. Let $N$ be $\SL_m(q)$, $\Sp_m(q)$, $G_2(q)'$ or $\SL_2(5)$ when $G_0^1$ lies in classes  (H1)--(H4),  respectively. Since $NZ/Z$ is the socle of the almost simple group $G_0^1Z/Z$, it follows that $N\unlhd H$. Hence $N\times N\unlhd G_0$, and $G_0$ lies in one of the classes (I1)--(I4).
\end{proof}

\begin{remark}
\label{remark:rank3thm}
Before we use~\cite{Lie1987} to prove Theorem~\ref{thm:rank3}, we  discuss three  minor oversights in the main theorem of~\cite{Lie1987}. 
This theorem states that if $G$ is an affine primitive permutation group of rank~$3$ with socle $V:=V_d(p)$ for some prime $p$, then $G_0$ belongs to one of $13$ classes, and these classes are labelled (A1)--(A11), (B) and  (C). The three issues are described below. 
\begin{enumerate}
\item In (A9), (A10) and (C), there is a non-abelian simple group $L$ such that $L\unlhd G_0/Z(G_0)$ where $L\leq \PSL_a(q)$   and $q^a=p^d$.  Here  $Z(G_0)$ should be replaced with $G_0\cap \mathbb{F}_q^*$.
 Indeed, except for the case $(L,p^d)=(A_7,2^8)$ of (C),  such $G_0$ are analysed in~\cite{Lie1987} as part of~\cite[p.485, Case (I)]{Lie1987}, where  the following hold: $G_0\leq \GammaL_a(q)$; $L$ is the socle of $G_0/(G_0\cap \mathbb{F}_q^*)$; and the  (projective) representation of $L$ on $V=V_a(q)$ is absolutely irreducible and cannot be realised over a proper subfield of $\mathbb{F}_q$. If $G_0\leq \GL_a(q)$, then  $Z(G_0)=G_0\cap \mathbb{F}_q^*$ by~\cite[Lemma~2.10.1]{KleLie1990}, but $G_0\cap \mathbb{F}_q^*$ is not a subgroup of $Z(G_0)$ in general.
 \item  The case $(L,p^d)=(A_7,2^8)$ of (C) is different. Here $A_7\leq \SL_4(4)$ and $q=4$, but this representation of $A_7$ on $V_4(4)$ can be realised over $\mathbb{F}_2$ since $A_7\leq A_8\simeq \SL_4(2)\leq \SL_4(4)$. In~\cite[p.483, Case (IIc)]{Lie1987},  Liebeck notes  that  $H_0:=\mathbb{F}_4^*\times A_7$ has two orbits on $V^*$, and since $H_0/\mathbb{F}_4^*$ is simple,  $H_0$ lies in (C). Here Liebeck is classifying  those $G$ for which $G_0\leq N_{\GammaL_4(4)}(\SL_4(2))$ and $G_0\cap \SL_4(2)=A_7$. Since $C_{\GammaL_4(4)}(\SL_4(2))=\mathbb{F}_4^*{:}\langle \sigma_4\rangle$, we may take $G_0$ to be $H_0\langle\sigma_4\rangle=\mathbb{F}_4^*{:}\langle \sigma_4\rangle \times A_7$, but the socle of $H_0\langle\sigma_4\rangle/\mathbb{F}_4^*$ is not simple, so $H_0\langle\sigma_4\rangle$ is not in (C). This is an omission of~\cite[p.483]{Lie1987} but not~\cite{Lie1987} since $H_0\langle\sigma_4\rangle\simeq \GammaL_1(4)\otimes A_7$, which stabilises a decomposition $V=V_2(2)\otimes V_4(2)$, so  $H_0$ and $H_0\langle\sigma_4\rangle$  are in class (A3), and there are no other possibilities for $G_0$ by a computation in {\sc Magma}. Note that    $H_0$ and $H_0\langle\sigma_4\rangle$  belong to class~(T4) of  Theorem~\ref{thm:tensorgroups}, and they also belong to class~(S2).
 \item When $L=A_9<\Omega_8^+(2)$, the subdegrees in~\cite[Table~14]{Lie1987} are  recorded as $105$ and $150$; these should be $120$ and $135$, which are, respectively, the numbers of non-singular and singular vectors with respect to the non-degenerate quadratic form preserved by $\Omega_8^+(2)$.
\end{enumerate}
 \end{remark}

\begin{proof}[Proof of Theorem~\emph{\ref{thm:rank3}}]
By~\cite{Lie1987} and Remark~\ref{remark:rank3thm}, one of the following holds.
\begin{enumerate}
\item \begin{itemize}\item[(a)] $G_0\leq \GammaL_1(p^d)$. 
\item[(b)] $V=V_n(q^2)$ and $\SL_n(q)\unlhd G_0\leq \GammaL_n(q^2)$ where  $q^{2n}=p^d$ and $n\geq 2$ (see~\cite[p.482]{Lie1987}). 
\item[(c)] $V=V_2(q^3)$ and $\SL_2(q)\unlhd G_0\leq \GammaL_2(q^3)$ where $q^6=p^d$ (see~\cite[p.483]{Lie1987}).
\item[(d)] $V=\bigwedge^2(V_5(q))$ and $\SL_5(q)\unlhd G_0$ where  $q^{10}=p^d$. 
\item[(e)] $\Sz(q)={}^2B_2(q)\unlhd G_0$ where   $q^4=p^d$ and  $q$ is an odd power of $2$.
 \end{itemize}
\item $\SU_n(q)\unlhd G_0$ where $q^{2n}=p^d$ and $n\geq 2$.
\item $\Omega_{2m}^\pm(q)\unlhd G_0$ where $q^{2m}=p^d$ and $m\geq 1$.
\item $G_0$ stabilises a decomposition $V=V_n(p)\oplus V_n(p)$ where $2n=d$.
\item $G_0$ stabilises a decomposition $V=V_2(q)\otimes V_m(q)$ where $q^{2m}=p^d$ and $m\geq 2$.
\item  $G_0\leq \GammaL_n(q)$   and the socle $L$ of $G_0/(G_0\cap \mathbb{F}_q^*)$ is   absolutely irreducible on $V=V_n(q)$ and cannot be realised over a proper subfield of $\mathbb{F}_q$. Further,   one of the following holds.

\begin{itemize}\item[(a)] $V$ is a spin module,  $L=\POmega_m^\varepsilon(q)$ and $(m,n,\varepsilon)$ is   $(7,8,\circ)$ or $(10,16,+)$.
 \item[(b)] $(L,p^d)$ and the embedding of $L$ in $\PSL_n(q)$ are given by~\cite[Table~2]{Lie1987},  where $(L,p^d)\neq (A_7,2^8)$. The subdegrees of $G$ are given by~\cite[Theorem 5.3]{FouKal1978} and \cite[Table~14]{Lie1987}, except for $L=A_9$, in which case the subdegrees are $120$ and $135$.
 \end{itemize} 
\item $G_0$ normalises $E$ where $E\leq \GL_d(p)$ and  $(E,p^d)$ is  given by Table~\ref{tab:Bclass} 
(see~\cite[\S4.6]{KleLie1990} for details on extraspecial groups and~\cite[p.483, Case (IIe)]{Lie1987}).
  The subdegrees of $G$ are given by~\cite[Theorem~1.1, 2(b) or 2(b$'$)]{Fou1969}   when   $E=D_8$ or $Q_8$, and~\cite[Table 13]{Lie1987} otherwise.
  \begin{table}[!h]
\renewcommand{\baselinestretch}{1.1}\selectfont
\centering
\begin{tabular}{ c c}
\hline 
$E$ & $p^d$ \\
\hline 
$3^{1+2}$ & $2^6$\\
$D_8$, $Q_8$ & $3^4,3^6,7^2,13^2,17^2,19^2,23^2,29^2,31^2,47^2$\\ 
$Q_8\circ Q_8$ & $3^4$\\$D_8\circ Q_8$ & $3^4,5^4,7^4$\\
$D_8\circ Q_8\circ \langle \zeta_5\rangle$ & $5^4$\\ 
  $Q_8\circ Q_8\circ Q_8$ & $3^8$\\
 \hline
\end{tabular}
\caption{}
\label{tab:Bclass}
\end{table}
\end{enumerate}

We now consider each of the cases described above in order to prove that $G_0$ belongs to one of the classes  in the statement of Theorem~\ref{thm:rank3}.

\textbf{Case (1).} If one of (a), (b), (c) or (d) holds, then $G_0$ lies in class (R0), (S1), (S0) or (R3), respectively. If (e) holds, then either $G_0$ lies in class (R4), or $q=2$ and $G_0$ lies in class (R0).

\textbf{Case (2).} 
Either $G_0$ lies in class (R1), or $n=2$, in which case $G_0$ lies in class (S1) since $\SL_2(q)$ and $\SU_2(q)$ are conjugate in $\GL_d(p)$. Indeed, let $\{x_1,x_2\}$ be a basis of $V_2(q^2)$ on which $\SL_2(q)$ acts naturally over $\mathbb{F}_q$, and choose $\mu\in \mathbb{F}_{q^2}^*$  such that $\mu+\mu^q=0$. Then $\SL_2(q)$  preserves the non-degenerate unitary form $\f$ for which $\f(x_1,x_2)=\mu$ and $\f(x_i,x_i)=0$ for $i=1,2$.

\textbf{Case (3).} Here one of the following holds: $G_0$ lies in class (R2);  $m=2$ and $\varepsilon=+$, in which case $\Omega_4^+(q)$ and $\SL_2(q)\otimes \SL_2(q)$ are conjugate in $\GL_d(p)$, so $G_0$ lies in class (T1) for $q>2$ and class (I0) for $q=2$; or $m=1$, in which case $\varepsilon=+$ and  $G_0$ lies in class (I0). 

\textbf{Case (4).} By Theorem~\ref{thm:A2groups}, $G_0$ lies in one of the classes (I0)--(I8). 

\textbf{Case (5).} By Theorem~\ref{thm:tensorgroups}, either $G_0$ belongs to one of the classes (R0), (I0), (T1), (T2) or (T3), in which case we are done, or $G_0$ lies in  (T4), (T5) or (T6). If $G_0$ lies in class (T5), then $V$ is naturally an $m$-dimensional $\mathbb{F}_{q^2}$-vector space of the form  $V_1(q^2)\otimes V_m(q)$, and $\SL_m(q)\unlhd G_0\leq \GammaL_m(q^2)$, so $G_0$ belongs to class (S1). Similarly, if $G_0$ lies in class (T4), then since  $\GL_2(2)=\GammaL_1(2^2)$, again $V$ is naturally a $4$-dimensional $\mathbb{F}_{2^2}$-vector space of the form $V_1(2^2)\otimes V_4(2)$, and $A_7\unlhd G_0\leq \GammaL_4(2^2)$, so $G_0$ belongs to class (S2). Lastly, if $G_0$ lies in class (T6), then $V$ is naturally a $2$-dimensional $\mathbb{F}_{q^3}$-vector space of the form  $V_2(q)\otimes V_1(q^3)$, and $\SL_2(q)\unlhd G_0\leq \GammaL_2(q^3)$, so $G_0$ belongs to class~(S0). 

\textbf{Case (6).}
View $L$ as a subgroup of $G_0\mathbb{F}_q^*/\mathbb{F}_q^*$, and define $M$ to be the subgroup of $\GL_n(q)$ for which $M/\mathbb{F}_q^*=L$. Now $\mathbb{F}_q^*=Z(M)$, so by~\cite[31.1]{Asc2000}, $M=M'\mathbb{F}_q^*$ and $M'$ is quasisimple.  Since $M\unlhd G_0\mathbb{F}_q^*$ and $M'$ is perfect, it follows that $M'\unlhd G_0$ and $M'\leq \SL_n(q)$. If (a) holds,  then $M'=\spin_m^\varepsilon(q)$ (see~\S\ref{s:(R5)} for more details), so $G_0$ belongs to class (R5). Suppose instead that (b) holds. Using the $p$-modular character tables in~\cite{BAtlas,Atlas,GAP4}, we determine that $M'=S$, where $S$ is given by Table~\ref{tab:AS}.  
Now $S$ is absolutely irreducible on $V$ and  cannot be realised over a proper subfield of $\mathbb{F}_q$, so $S$ is irreducible on $V_d(p)$ by~\cite[Theorems VII.1.16 and VII.1.17]{BlaHup1981}. Thus  $ G_0$ belongs to class (AS).

\textbf{Case (7).} 
Where necessary, we use {\sc Magma}  to verify the various claims made below.  
 If either $p^d=3^8$ and $E=Q_8\circ Q_8\circ Q_8$, or   $p^d=7^4$ and $E=D_8\circ Q_8$, then $G_0$ belongs to (E). Suppose that $p^d=5^4$. If $E=D_8\circ Q_8$, then $G_0$ belongs to (E). Otherwise, $E=D_8\circ Q_8\circ \langle \zeta_5\rangle$, and we may assume that $G_0\nleq N_{\GL_4(5)}(D_8\circ Q_8)$, in which case $G_0$ belongs to (E).  Suppose that $p^d=2^6$ and $E$ is an extraspecial group $3^{1+2}$. Now $N_{\GL_d(p)}(E)$ is conjugate to $\GammaU_3(2)$, so we may assume that $N_{\GL_d(p)}(E)=\GammaU_3(2)$. If $\SU_3(2)\leq G_0$, then $G_0$ belongs to (R1); otherwise, $G_0$ belongs to (E).
 
Suppose that $d=2$, in which case $E=D_8$ or $Q_8$ and $p\in \{7,13,17,19,23,29,31,47\}$.  For each such  $p$, the group $\GL_2(p)$ contains a unique conjugacy class of subgroups isomorphic to $E$.  If $p\neq 7$, then $N_{\GL_2(p)}(D_8)$ has at least $4$ orbits on $V_d(p)$, so $E=Q_8$, in which case $G_0$ belongs to (E). Suppose instead that $p=7$. Now $N_{\GL_2(7)}(Q_8)$ is transitive on $V^*$ and contains, up to conjugacy in $\GL_d(p)$, five subgroups with two orbits on $V^*$; these have orders $24$, $24$, $24$, $36$ and $72$. The group of order $36$ has orbits of size $12$ and $36$ on $V^*$ and lies in (I0). The remaining four  groups have two orbits of size $24$;  two of these are subgroups of $\GammaL_1(7^2)$  and therefore belong to (R0), and the other two  belong to (E). Lastly, if $E=D_8$, then $N_{\GL_2(7)}(E)\leq \GammaL_1(7^2)$, so $G_0$ belongs to (R0).

Suppose that $p^d=3^6$ and $E=D_8$ or $Q_8$. Now $E\leq \GammaL_2(27)$ by~\cite[Theorem 1.1]{Fou1969}. The group $\GammaL_2(27)$ has a unique conjugacy class of subgroups isomorphic to $E$, and each such group is conjugate in $\GL_6(3)$ to $E\otimes 1\leq \GL_2(3)\otimes \GL_3(3)$. The normalisers of $D_8\otimes 1$ and $Q_8\otimes 1$ in $\GL_6(3)$ are $\GammaL_1(9)\otimes \GL_3(3)$ and $\GL_2(3)\otimes \GL_3(3)$,  respectively, so  $G_0$  stabilises a decomposition $V=V_2(3)\otimes V_3(3)$, and we  considered such groups in case (5) above.  

Suppose that $p^d=3^4$ and $E=D_8\circ Q_8$. Now $N_{\GL_4(3)}(E)$ is transitive on $V^*$ and has, up to conjugacy in $\GL_d(p)$, exactly  $18$ subgroups with two orbits on $V^*$.  Eight of these groups have orbit sizes $32$ and $48$; these lie in (S1). Another eight have orbit sizes $16$ and $64$; these  stabilise a decomposition $V=V_2(3)\oplus V_2(3)$, and we considered such groups in case (4) above. This leaves two groups with orbit sizes $40$ and $40$. One is a subgroup of $\GammaL_1(3^4)$ and therefore lies in (R0). The other is isomorphic to $\SL_2(5)$ and lies in class (AS). 

It remains to consider the case where $p^d=3^4$ and $E$ is one of  $Q_8$, $D_8$ or $Q_8\circ Q_8$. If $E=D_8$ or $Q_8$, then $\GL_4(3)$ contains a unique conjugacy class of subgroups isomorphic to $E$ whose normalisers in $\GL_4(3)$ have $2$ orbits on $V^*$, and this conjugacy class contains $E\otimes 1\leq \GL_2(3)\otimes \GL_2(3)$. The normalisers of $D_8\otimes 1$ and $Q_8\otimes 1$ in $\GL_4(3)$ are $\GammaL_1(9)\otimes \GL_2(3)$ and $\GL_2(3)\otimes \GL_2(3)$,  respectively. Both of these normalisers are subgroups of $(\GL_2(3)\otimes \GL_2(3)){:}2$, which is the normaliser in $\GL_4(3)$ of $Q_8\otimes Q_8\simeq Q_8\circ Q_8$, and $\GL_4(3)$ contains a unique conjugacy class of subgroups isomorphic to $Q_8\circ Q_8$. Thus, for each possible $E$, the group $G_0$  stabilises a decomposition $V=V_2(3)\otimes V_2(3)$, and we considered such groups in case (5) above. 
\end{proof}

In \cite[Remark~3]{Lie1987}, Liebeck observes that if $G$ is an  affine primitive permutation group of rank~$3$ on $V$, then $G_0$ has exactly two orbits on the set of one-dimensional subspaces of $V$ with one exception: those groups  in class (AS) with $S=2\nonsplit A_5$ and degree $3^4$. However, Liebeck does not specify the field over which this occurs, and the choice of field matters; indeed, in the exceptional case, $2\nonsplit A_5\unlhd G_0\leq \GammaL_2(9)\leq \GL_4(3)$, and $G_0$ is transitive on the one-dimensional subspaces of $V_2(9)$ but has two orbits on the set of one-dimensional subspaces of $V_4(3)$. In the following, we give an explicit statement of \cite[Remark~3]{Lie1987}, which  we deduce from~\cite{Lie1987}.

\begin{cor}
\label{cor:remark3}
Let $G$ be an affine primitive permutation group of rank~$3$ on $V_d(p)$ where $d\geq 1$ and $p$ is prime, and let $s^c=p^d$ where  one of the following  holds  (in the notation of Theorem~\emph{\ref{thm:rank3}}). 
\begin{itemize}
\item[(i)] $G_0$ lies in class $\mathcal{C}$, where $\mathcal{C}$ and $(c,s)$ are given by Table~\emph{\ref{tab:ar}}.
\item[(ii)] $G_0$ lies in class \emph{(E)} and either $(c,s)=(3,4)$, or $p$ is odd and $(c,s)=(d,p)$.  
\item[(iii)] $G_0$ lies in class  \emph{(AS)}, and $(c,s)=(d,p)$ unless $(S,p^d)$ is one of $(2\nonsplit A_5,3^4)$, $(2\nonsplit A_5,7^4)$, $(3\nonsplit A_6,2^6)$ or $(J_2,2^{12})$, in which case $(c,s)=(d/2,p^2)$.
\end{itemize}
 Then $G_0\leq \GammaL_c(s)$. Further, $G_0$ has two orbits on  the points of $\PG_{c-1}(s)$ unless $G_0$ lies in  \emph{(R0)} or \emph{(AS)} with $(S,p^d)=(2\nonsplit A_5,3^4)$, in which case $G_0$  is  transitive on the points of $\PG_{c-1}(s)$. 
\end{cor}

\begin{table}[!h]
\renewcommand{\baselinestretch}{1.1}\selectfont
\centering
\begin{tabular}{ c | c c c c c c c c c c}
\hline 
$\mathcal{C}$ & (R0) & (R1) & (R2) & (R3) & (R4) & (R5)  & (T1)--(T3) & (S0) & (S1)--(S2) & (I)\\
 $(c,s)$ & $(1,p^d)$ & $(n,q^2)$ & $(2m,q)$ & $(10,q)$ & $(4,q)$ & $(n,q)$  & $(2m,q)$ & $(2,q^3)$ & $(n,q^2)$ & $(d,p)$ \\ 
 \hline
\end{tabular}
\caption{For $G_0$ in class $\mathcal{C}$, the  parameters $(c,s)$ where  $G_0\leq \GammaL_c(s)$}
\label{tab:ar}
\end{table}

We now use the pair $(c,s)$ from Corollary~\ref{cor:remark3} to  state  Hypothesis~\ref{hyp:AGgroups}, which was alluded to in Example~\ref{example:AG} and Theorem~\ref{thm:main}. 

\begin{hyp}
\label{hyp:AGgroups}
Let $(G,a,r)$ be defined as follows:  $G$  is an affine primitive permutation group of rank~$3$ on $V:=V_d(p)$ where $d\geq 2$ and  $p$ is prime, $r^a=p^d$ and one of the following holds,  where  $s^c=p^d$ and $(c,s)$ is given by Corollary~\ref{cor:remark3}.
\begin{itemize}
\item[(i)] $G_0$ lies in one of   (R1)--(R5), (T1)--(T3), (S0)--(S2) or (E), and $\mathbb{F}_r$ is a subfield of $\mathbb{F}_s$. 
\item[(ii)] $G_0$ lies in (AS), and $\mathbb{F}_r$ is a subfield of $\mathbb{F}_s$ unless $(S,p^d)=(2\nonsplit A_5,3^4)$, in which case $r=p$.
\item[(iii)] $V=V_b(r)\oplus V_b(r)$ and $G_0\leq (\GammaL_b(r)\wr S_2)\cap \GammaL_{2b}(r)$ where $a=2b$.
\end{itemize}
\end{hyp}

\begin{cor}
\label{cor:AGgroups}
If $(G,a,r)$ satisfies  Hypothesis~\emph{\ref{hyp:AGgroups}}, then $G_0\leq \GammaL_a(r)$ and $G_0$ has two orbits on the points of $\PG_{a-1}(r)$.
\end{cor}

\begin{proof}
Suppose that $(G,a,r)$ satisfies  Hypothesis~\ref{hyp:AGgroups}, and let $V:=V_d(p)$ where $p^d=r^a$ and $p$ is prime.  Now $G_0$ has two orbits on $V^*$. In particular, if $G_0\leq \GammaL_m(q)$ where $q^m=p^d$, then  $G_0$ has two orbits on the points of $\PG_{m-1}(q)$ if and only if $\langle u\rangle_{\mathbb{F}_q}^*\subseteq u^{G_0}$ for all $u\in V^*$. 
Note also that if $t^n=q^m$ where $\mathbb{F}_t$ is a subfield of $\mathbb{F}_q$, then  $\GammaL_m(q)\leq \GammaL_n(t)$. If $V=V_b(r)\oplus V_b(r)$ and $G_0\leq (\GammaL_b(r)\wr S_2)\cap \GammaL_{2b}(r)$ where $a=2b$, then clearly $G_0\leq \GammaL_a(r)$ and $G_0$ has two orbits on the points of $\PG_{a-1}(r)$.   If $G_0$ lies in one of   (R1)--(R5), (T1)--(T3), (S0)--(S2), (E) or (AS) with $(S,p^d)\neq (2\nonsplit A_5,3^4)$,  and if $s^c=p^d$ where $(c,s)$ is given by Corollary~\ref{cor:remark3}, then by Corollary~\ref{cor:remark3}, $G_0\leq \GammaL_c(s)$ and $G_0$ has two orbits on the points of  $\PG_{c-1}(s)$, so if  $\mathbb{F}_r$ is a subfield of $\mathbb{F}_s$, then $G_0\leq \GammaL_a(r)$ and $G_0$ has two orbits on the points of $\PG_{a-1}(r)$. 
Lastly, suppose that $G_0$ lies in class (AS) with $(S,p^d)= (2\nonsplit A_5,3^4)$ and $r=p$. Clearly $G_0\leq \GL_4(3)$, and $-1\in G_0$ since $2\nonsplit A_5\unlhd G_0$ and $2\nonsplit A_5$ is irreducible on $V$, so $\{u,-u\}\subseteq u^{G_0}$ for all $u\in V^*$. Thus $G_0$ has two orbits on the points of $\PG_{3}(3)$.
\end{proof}

\begin{remark}
\label{remark:hypfail}
Let $V:=V_b(r)\oplus V_b(r)$ and $p^d=r^{2b}$ where $p$ is prime, $r$ is not prime and $b\geq 1$. Let $G:=\AGammaL_b(r)\wr S_2$. Then $G_0=\GammaL_b(r)\wr S_2$, and $G$ is an affine primitive permutation group of rank~$3$ on $V_d(p)$ by Lemma~\ref{lemma:primitive}. However, we claim that $(G,2b,r)$ does not satisfy Hypothesis~\ref{hyp:AGgroups}. Suppose for a contradiction that it does. Now (iii) does not hold since $G_0$ is not a subgroup of $\GammaL_{2b}(r)$ (see~\S\ref{s:(I)} for more details), so (i) or (ii) holds. Recall from Table~\ref{tab:subdegree}  that $G$ has subdegrees $2(r^b-1)$ and $(r^b-1)^2$. 
If (ii) holds, then $G_0$ lies in class (AS) and $r=s=p^2$, so $(S,p^d)$ is  $(2\nonsplit A_5,7^4)$ or $(J_2,2^{12})$, but this is impossible by the subdegrees listed in Table~\ref{tab:AS}. Thus (i) holds. If $G_0$ lies in class (E), then $r=s=4$ and $p^d=2^6$, a contradiction. Thus $G_0$ lies in one of  (R1)--(R5), (T1)--(T3) or  (S0)--(S2). By Table~\ref{tab:subdegree}, there is a subdegree $t$ of $G$ that is divisible by~$p$. Then $t=2(r^b-1)$,  $p=2$, and $2$ is the highest power of $p$ that divides~$t$. Since  $r\leq s$ and $r$ is not prime,  it follows that $G_0$ lies in class (S1) with $(n,q)=(2,2)$ and $(b,r)=(1,4)$,  but then $\SL_2(2)\unlhd G_0=\SL_2(2)\wr S_2$, a contradiction.
\end{remark}

 \section{Dependent partial linear spaces}
\label{s:dep}

Let $V:=V_n(q)$ where $q$ is a prime power and $n\geq 1$. A partial linear space $\mathcal{S}:=(V,\mathcal{L})$ is $\mathbb{F}_q$\textit{-dependent} if $L\subseteq  \langle x\rangle_{\mathbb{F}_q}$ for every $L\in\mathcal{L}_0$ and $x\in L^*$, and $\mathbb{F}_q$\textit{-independent} otherwise. Observe that if $\mathcal{S}$ is $\mathbb{F}_q$-dependent, then it is $\mathbb{F}_{q^m}$-dependent for all divisors $m$ of $n$ since  we may view $V$ as a vector space $V_{n/m}(q^m)$.  In other words, if $\mathcal{S}$ is $\mathbb{F}_q$-independent, then it is $\mathbb{F}_r$-independent for all subfields $\mathbb{F}_r$ of $\mathbb{F}_q$.

Using Kantor's classification~\cite[Proposition~4.1]{Kan1985} of the  $G$-affine linear spaces for which $G$ is $2$-transitive and $G\leq \AGammaL_1(q)$ (see also  Theorem~\ref{thm:Kantor}), we prove the following.

\begin{lemma}
\label{lemma:Kantor}
Let $V:=V_n(q)$ where $n\geq 1$ and $q$ is a prime power. Let $(V,\mathcal{L})$ be  a $G$-affine  partial linear space where $G_0\leq \GammaL_n(q)$. Let $L\in \mathcal{L}_0$ and $u\in L^*$. If  $\langle u\rangle^* \subseteq u^{G_0}$, then $\{\lambda \in \mathbb{F}_q:\lambda u\in L\}$ is either $\{0,1\}$ or a subfield of $\mathbb{F}_q$.
\end{lemma}

\begin{proof}
  Let  $\mathcal{S}:=(V,\mathcal{L})$. Let $U:=\langle u\rangle$ and $H:=U{:}G_{0,U}^{U}\leq \AGammaL_1(q)$. By Lemma~\ref{lemma:intersect}, $\mathcal{S}\cap U$ is a partial linear space, and $H\leq \Aut(\mathcal{S}\cap U)$. If $\lambda u, \mu u\in U^*$, then by assumption there exists $g\in G_0$ such that $(\lambda u)^g=\mu u$, so $U^g=U$. Thus $H$ acts $2$-transitively on $U$.  Let $F:=\{\lambda \in \mathbb{F}_q:\lambda u\in L\}$ and $k:=|F|\geq 2$. If $k=2$, then $F=\{0,1\}$, so we may assume that $k\geq 3$. Now  $\mathcal{S}\cap U$ is  an $H$-affine linear space whose line set contains $L\cap U$, and $H$ lies in class (H0), so Theorem~\ref{thm:Kantor}(i) or (ii) holds. In either case, 
 $F$ is a subfield of $\mathbb{F}_q$. 
\end{proof}

Let $V:=V_n(q)$ where $n\geq 2$ and  $q$ is a prime power such that $q>2$. Let $G$ be an affine  permutation group of rank~$3$ on $V$, so that $G_0\leq \GammaL_n(q)$ (see~\S\ref{ss:affine}).  Observe that the following three statements are equivalent.
\begin{itemize}
\item[(i)] $G_0$ has two orbits on the points of $\PG_{n-1}(q)$.
\item[(ii)] $\langle u\rangle^*\subseteq u^{G_0}$ for all $u\in V^*$.
\item[(iii)] $\langle x\rangle^*\subseteq x^{G_0}$ for some $x\in V^*$.
\end{itemize} 
Now suppose that $G_0$ has two orbits $\Sigma_1$ and $\Sigma_2$ on the points of $\PG_{n-1}(q)$, and let $\langle x_1\rangle\in \Sigma_1$ and $\langle x_2\rangle\in \Sigma_2$. Let $X_1$ and $X_2$ be the orbits of $G_0$ on $V^*$ containing $x_1$ and $x_2$, respectively, and note that $V^*=X_1\cup X_2$. In Example~\ref{example:AG} (for $G$ primitive), we constructed two $G$-affine proper partial linear spaces $\mathcal{S}_i:=(V,\mathcal{L}_i)$, where $\mathcal{L}_i:=\{\langle u\rangle+v:\langle u\rangle\in\Sigma_i,v\in V\}=\{\langle u\rangle+v: u\in X_i,v\in V\}$ for $i\in\{1,2\}$. Observe that $\mathcal{S}_1$ and $\mathcal{S}_2$ are $\mathbb{F}_q$-dependent partial linear spaces. 
Let  $F:=\mathbb{F}_r$  be a subfield of $\mathbb{F}_q$ with $r> 2$ where $r^a=q^n$. Now $G_0\leq \GammaL_n(q)\leq \GammaL_a(r)$, and $G_0$ has two orbits $\Delta_1$ and $\Delta_2$ on the points of $\PG_{a-1}(r)$, where $\langle x_1\rangle_F\in \Delta_1$ and $\langle x_2\rangle_F\in \Delta_2$. Thus both $(V,\{\langle u\rangle_F+v: u\in X_1,v\in V\})$ and $(V,\{\langle u\rangle_F+v: u\in X_2,v\in V\})$ are $\mathbb{F}_q$-dependent $G$-affine proper   partial linear spaces from Example~\ref{example:AG} (when $G$ is primitive). Using Lemma~\ref{lemma:Kantor}, we now prove that every $\mathbb{F}_q$-dependent $G$-affine proper partial linear space has this form for some subfield $F$.

\begin{prop}
\label{prop:dep}
Let $V:=V_n(q)$ where  $n\geq 2$ and $q$ is a prime power, and let $G$ be an affine permutation group of rank~$3$ on $V$ where $\langle x\rangle^*\subseteq x^{G_0}$ for some $x\in V^*$. Let  $(V,\mathcal{L})$ be an $\mathbb{F}_q$-dependent $G$-affine proper partial linear space. Then there exists a subfield $F$ of $\mathbb{F}_q$ with $|F|> 2$ and an orbit $X$ of $G_0$ on $V^*$ 
 such that $\mathcal{L}=\{\langle u\rangle_{F}+v:u\in X,v\in V\}$.
\end{prop}

\begin{proof}
By assumption, $G_0\leq \GammaL_n(q)$. There exists $L\in \mathcal{L}_0$ and $w\in L^*$.  Since $\langle x\rangle^*\subseteq x^{G_0}$, it follows that $\langle w\rangle^*\subseteq w^{G_0}$. By assumption, $L\subseteq \langle w\rangle$ and $|L|\geq 3$, so by Lemma~\ref{lemma:Kantor}, there exists a subfield $F$ of $\mathbb{F}_q$ with $|F|> 2$ such that $L=\langle x\rangle_{F}$. Hence $\mathcal{L}=\{\langle u\rangle_{F}+v:u\in w^{G_0},v\in V\}$.
\end{proof}

Let $G$ be an affine primitive permutation group of rank~$3$ on $V:=V_d(p)$ where $d\geq 1$ and $p$ is prime. If $G_0$ belongs to one of  (R1)--(R5), (T1)--(T3) or  (S0)--(S2), then by Corollary~\ref{cor:remark3}, $V=V_c(s)$  and $G_0\leq \GammaL_c(s)$  where  $(c,s)$ is given by Table \ref{tab:ar}, and $\langle x\rangle_{\mathbb{F}_s}^*\subseteq x^{G_0}$ for all $x\in V^*$, 
so each $\mathbb{F}_s$-dependent $G$-affine proper partial linear space is given by Proposition~\ref{prop:dep} and therefore  Example~\ref{example:AG} with respect to   some triple $(G,\log_r(p^d),r)$ that satisfies Hypothesis~\ref{hyp:AGgroups}. Thus for those $G_0$ in classes (R1)--(R5), (T1)--(T3) and   (S0)--(S2), it remains to consider $\mathbb{F}_s$-independent $G$-affine proper partial linear spaces. This we do in \S\ref{s:(R1)}--\ref{s:(S0)}, and we will see that Lemma~\ref{lemma:basic}  imposes severe restrictions on the  possible examples; in particular,  for $G_0$ in classes (R1)--(R5), there is a unique $\mathbb{F}_s$-independent $G$-affine proper partial linear space, and  this example arises in class (R2) with $G=3^4{:}M_{10}\simeq 3^4{:}(\Omega_4^-(3).2)$. However, for  $G_0$ in classes (I0)--(I8), since $G_0$ need not be a subgroup of $\GammaL_{2m}(q)$,  it is no longer convenient to make the distinction between $\mathbb{F}_q$-dependent and $\mathbb{F}_q$-independent  $G$-affine proper partial linear spaces in our proofs; instead, in \S\ref{s:(I)}, we develop methods for building such partial linear spaces from $2$-transitive affine linear spaces. Similarly,  for $G_0$ in classes (E) and (AS), our methods are primarily computational (see \S\ref{s:(E)} and \S\ref{s:(AS)}), so it is again more convenient not to make the distinction between dependent and independent  $G$-affine proper partial linear spaces. 

The next two results are also  consequences of Lemma~\ref{lemma:Kantor}. When $q^n=p^d$ for a prime $p$, we say that a subset  $L$ of $V:=V_n(q)$ is an \textit{affine} $\mathbb{F}_p$\textit{-subspace} of $V$ if $L$ is an affine subspace of the $\mathbb{F}_p$-vector space $V$.

\begin{lemma}
\label{lemma:closed}
Let $V:=V_n(q)$ where  $n\geq 1$ and $q$ is a power of a prime $p$, and let $G$ be an affine permutation group of rank~$3$ on $V$ where $\langle x\rangle^*\subseteq x^{G_0}$ for some $x\in V^*$. Let  $\mathcal{S}:=(V,\mathcal{L})$ be a $G$-affine proper partial linear space whose lines are affine $\mathbb{F}_p$-subspaces of $V$. Then there exists a subfield $F$ of $\mathbb{F}_q$ such that each $L\in\mathcal{L}_0$ is an $F$-subspace of $V$ with the property that $F=\{\lambda\in\mathbb{F}_q:\lambda u\in L\}$ for all $u\in L^*$.
\end{lemma}

\begin{proof}
Let $L\in\mathcal{L}_0$ and  $B:=L^*$. By Lemma~\ref{lemma:necessary}, $B$ is a block of $G_0$ on $\mathcal{S}(0)$. For $u\in B$, let $F_u:=\{\lambda\in \mathbb{F}_{q}:\lambda u\in L\}$. First we claim that $F_u$ is a subfield of $\mathbb{F}_q$ for all $u\in B$. Let $u\in B$, and note that $\{0,1\}\subseteq F_u$.  If $|F_u|\geq 3$, then the claim holds by Lemma~\ref{lemma:Kantor}. Otherwise, $|F_u|=2$, but $|F_u|\geq p$ since $L$ is an $\mathbb{F}_p$-subspace of $V$, so $p=2$, and the claim follows. Let $u,v\in B$. Now $v=u^g$ for some  $g\in G_0$, so $B=B^g$, and $g$ is $\sigma$-semilinear for some $\sigma\in\Aut(\mathbb{F}_q)$, so   $F_v=F_u^\sigma=F_u$. Thus we may define $F:=F_u$ for $u\in B$, and $L$ is an $F$-subspace of $V$. If $L'\in\mathcal{L}_0$, then $L'=L^g$ for some $g\in G_0$, so $L'$ also has the desired structure.
\end{proof}

\begin{lemma}
\label{lemma:affineplus}
Let $V:=V_n(q)$ where  $n\geq 1$ and $q$ is a power of a prime $p$, and let $G$ be an affine permutation group of rank~$3$ on $V$ where $\langle x\rangle^*\subseteq x^{G_0}$ for some $x\in V^*$. Let  $\mathcal{S}:=(V,\mathcal{L})$ be a $G$-affine proper partial linear space. If $|L\cap \langle u\rangle|\geq 3$ for some $L\in \mathcal{L}_0$ and $u\in L^*$, then every line of $\mathcal{S}$ is an affine $\mathbb{F}_p$-subspace of $V$.
\end{lemma}

\begin{proof}
If $|L\cap \langle u\rangle|\geq 3$ for some $L\in \mathcal{L}_0$ and $u\in L^*$, then $\{\lambda\in\mathbb{F}_q:\lambda u\in L\}$ is a subfield of $\mathbb{F}_q$ by Lemma~\ref{lemma:Kantor}, so $(L\cap \langle u\rangle)^{\tau_u}=L\cap \langle u\rangle$. Thus $L^{\tau_u}=L$, and we may apply Lemma~\ref{lemma:affine}.
\end{proof}

Now we  construct some dependent $G$-affine proper partial linear spaces for $G$ in class (R0).

\begin{example}
\label{example:R0}
Let $q:=p^d$ and $V:=\mathbb{F}_q$  where  $p$ is prime and $d\geq 2$. Let $G$ be an affine primitive permutation group of rank~$3$ on $V$, so that  $G_0\leq \GammaL_1(q)$. Let $F:=\mathbb{F}_r$ be a proper subfield of $\mathbb{F}_q$ with $r>2$ such that  $\langle x\rangle_F^*\subseteq x^{G_0}$ for $x\in V^*$. Write $q=r^a$.   Then $G_0\leq \GammaL_a(r)$, and $G_0$ has two orbits on the points of $\PG_{a-1}(r)$, so $(V,\{\langle u\rangle_F+v: u\in x^{G_0},v\in V\})$ is a $G$-affine proper partial linear space from Example~\ref{example:AG} for $x\in V^*$. 

To see that such $G$ and $F$ exist, assume that  $p$ is odd and $d$  is even, and choose $F$ to be any subfield of $V=\mathbb{F}_q$ for which $[\mathbb{F}_q:F]$ is even. Let $\zeta :=\zeta_q$ and $\sigma:=\sigma_q$ (see~\S\ref{ss:basicsvs}). Then $\GammaL_1(q)=\langle \zeta, \sigma\rangle$, and the orbits of $\langle \zeta^2,\sigma\rangle $ on $V^*$ are the squares $\langle \zeta^2\rangle$ and the non-squares $\langle \zeta^2\rangle \zeta$ of $\mathbb{F}_q^*$.  
Let $G:=V{:}G_0$, where $G_0$ is any subgroup of $\langle \zeta^2,\sigma\rangle $ with two orbits on $V^*$. Then $G$ is an affine  permutation group of rank~$3$ on $V$, and $G$ is primitive by Lemma~\ref{lemma:primitive} since $|x^{G_0}\cup \{0\}|=(q+1)/2$ for $x\in V^*$. 
Further, since $[\mathbb{F}_q:F]$ is even, every element of $\langle 1\rangle_F^*=F^*$ is a square in $\mathbb{F}_q^*$, so $\langle x\rangle_F^*\subseteq x^{G_0}$ for $x\in V^*$. Thus $\mathcal{S}:=(V,\{\langle u\rangle_F+v: u\in \langle \zeta^2\rangle,v\in V\})$ is a $G$-affine proper partial linear space from Example~\ref{example:AG}. In fact, $\Aut(\mathcal{S})=V{:}\langle \zeta^2,\sigma\rangle$ since  $V{:}\langle \zeta^2,\sigma\rangle$  is the automorphism group of the Paley graph with vertex set $\mathbb{F}_q$, in which  vertices  $x$ and $y$ are adjacent if and only if $x-y$ is a square in $\mathbb{F}_q^*$ (see~\cite[Theorem~9.1]{Jon2020}).  
\end{example}

\section{Class (R1)}
\label{s:(R1)}

Let $G$ be an affine permutation group of rank~$3$ on $V:=V_n(q^2)$ for which $\SU_n(q)\unlhd G_0$ where $n\geq 3$. The  orbits of $G_0$ on $V^*$ are the non-zero isotropic vectors and the   non-isotropic vectors.  Note that $\SU_n(q)$ acts transitively on the set of  non-zero isotropic vectors for $n\geq 3$ by~\cite[Lemma 2.10.5]{KleLie1990}. In order to simplify the proofs in this section, we define a standard basis for a unitary space. Let $\f$ be the non-degenerate unitary form preserved by $\SU_n(q)$, and let $m$ be such that $n=2m$ or $n=2m+1$. By~\cite[Proposition~2.3.2]{KleLie1990},   $V$ has a \textit{standard basis}  $\{e_1,\ldots,e_m,f_1,\ldots,f_m\}$ or  $\{e_1,\ldots,e_m,f_1,\ldots,f_m,x_0\}$ when $n=2m$ or $n=2m+1$, respectively, where $\f(e_i,e_j)=\f(f_i,f_j)=0$ and $\f(e_i,f_j)=\delta_{i,j}$  for all $i,j$, and when $n$ is odd, $\f(x_0,x_0)=1$ and $\f(e_i,x_0)=\f(f_i,x_0)=0$  for all $i$. 

\begin{prop}
\label{prop:(R1)}
Let $V:=V_n(q^2)$ where $n\geq 3$ and $q$ is a prime power, and let $G$ be an affine permutation group of rank~$3$ on $V$ where $\SU_n(q)\unlhd G_0$. Then there is no $\mathbb{F}_{q^2}$-independent $G$-affine proper partial linear space. 
\end{prop}

\begin{proof} 
Let $F:=\mathbb{F}_{q^2}$ and $\overline{\lambda}:=\lambda^q$ for $\lambda\in F$. Let $T:F\to \mathbb{F}_q$ be the trace map $\lambda\mapsto \lambda+\overline{\lambda}$, and let $N:F\to\mathbb{F}_q$ be the norm map $\lambda \mapsto \lambda\overline{\lambda}$. Note that both $T$ and $N$ are surjective maps. Further, $T$ has kernel $K:=\mathbb{F}_q\tau$ for some $\tau\in F^*$. 

Suppose for a contradiction that  $(V,\mathcal{L})$ is an $F$-independent  $G$-affine proper partial linear space. Let $\f$ be the non-degenerate unitary form preserved by $\SU_n(q)$. Let $L\in\mathcal{L}_0$, let $B:=L^*$, and let $x\in L^*$.  By Lemma~\ref{lemma:necessary}, $B$ is a  block of $G_0$ on $x^{G_0}$.  Since $(V,\mathcal{L})$  is $F$-independent,  $B\setminus \langle x\rangle$ is non-empty. Let $y\in B\setminus \langle x\rangle$.

First suppose that $x$ is isotropic. Let $m$ be such that $n=2m$ or $n=2m+1$. Let $\{e_1,\ldots,e_m,f_1,\ldots,f_m\}$ or $\{e_1,\ldots,e_m,f_1,\ldots,f_m,x_0\}$ be a standard basis of $V$ when $n=2m$ or $n=2m+1$, respectively.  In either case, we denote this basis by $\mathcal{B}$. We may assume without loss of generality that $x=e_1$. 
For  $\lambda\in F$, define $V_{x,\lambda}:=\{w \in V^*: \f(w,w)=0\ \mbox{and}\ \f(x,w)=\lambda \}$. 

We claim that if  $\lambda\neq 0$, then   $\SU_n(q)_x$ acts transitively on $V_{x,\lambda}$. We may assume that $\lambda=1$. Note that $f_1\in V_{x,1}$, and let $u\in V_{x,1}$. Then $u=\varepsilon e_1+f_1+u_0$ for some $\varepsilon\in F$ and $u_0\in \langle e_1,f_1\rangle^\perp$. There exists $g\in \SL_n(q^2)$ such that $e_1^g=e_1$, $f_1^g=u$ and $w^g=-\f(w,u)e_1+w$ for $w\in \mathcal{B}\setminus \{e_1,f_1\}$. Now $\f(v_1^g,v_2^g)=\f(v_1,v_2)$ for all $v_1,v_2\in V$, so $g\in \SU_n(q)_x$ and the claim holds.

Next we claim that if  $n\geq 5$, then $\SU_n(q)_x$ acts transitively on $V_{x,0}\setminus \langle x\rangle$. Note that $e_2\in V_{x,0}\setminus \langle x\rangle$, and let $u\in V_{x,0}\setminus \langle x\rangle$. Then $u=\varepsilon e_1+u_0$ for some $\varepsilon\in F$ and  non-zero  isotropic vector $u_0\in \langle e_1,f_1\rangle^\perp$. Since   $n\geq 4$, there exists $g\in \SL_n(q^2)$ such that $e_2^g=\varepsilon e_1+e_2$, $f_1^g=f_1-\overline{\varepsilon} f_2$ and $w^g=w$ for $w\in \mathcal{B}\setminus \{e_2,f_1\}$. Now $\f(v_1^g,v_2^g)=\f(v_1,v_2)$ for all $v_1,v_2\in V$, so $g\in \SU_n(q)_x$.  Moreover, 
we may view $\SU_{n-2}(q)$ as a subgroup of $\SU_n(q)$ that fixes $\langle e_1,f_1\rangle$ pointwise and acts naturally on $\langle e_1,f_1\rangle^\perp$, 
 and  there  exists $h\in \SU_{n-2}(q)$ such that $e_2^h=u_0$ by~\cite[Lemma 2.10.5]{KleLie1990} since $n\geq 5$. Now $gh\in \SU_n(q)_x$ and $e_2^{gh}=u$, so the claim holds.

Choose $\mu\in F\setminus K$ and let $\lambda:=\f(x,y)$. First suppose that either  $n\geq 4$ and $\lambda \neq 0$, or  $n\geq 5$  and $\lambda=0$. Now $\SU_n(q)_x$ acts transitively on $V_{x,\lambda}\setminus \langle x\rangle$  by the claims above, so $V_{x,\lambda}\setminus \langle x\rangle\subseteq B$ by Lemma~\ref{lemma:basic}(ii). In particular, $\overline{\lambda}f_1+\mu e_2$ and $\overline{\lambda}f_1+f_2$ are elements of  $B$, but then $z:=\mu e_2 -f_2$ is isotropic by Lemma~\ref{lemma:basic}(i), a contradiction since $\f(z,z)=-(\mu+\overline{\mu})\neq 0$. 

Next suppose that $n=3$. Now $V_{x,0}=\langle x\rangle$, so $\lambda\neq 0$. Again,  $V_{x,\lambda}\subseteq B$ by Lemma~\ref{lemma:basic}(ii). Choose distinct $\delta,\varepsilon\in F$ such that $N(\delta)=N(\varepsilon)=-N(\lambda)T(\mu)$. Now 
$\mu\overline{\lambda}e_1+\overline{\lambda}f_1+\delta x_0$ and $\mu\overline{\lambda}e_1+\overline{\lambda}f_1+\varepsilon x_0$ are elements of $V_{x,\lambda}$ and therefore $B$, but then $(\delta-\varepsilon)x_0$ is isotropic by Lemma~\ref{lemma:basic}(i), a contradiction.

 Hence $n=4$ and $\lambda=0$.  Let $\lambda_0,\ldots,\lambda_{q}$ be a transversal for $\mathbb{F}_q^*$ in $F^*$. Recall that $\tau\in F^*$ and $\tau+\overline{\tau}=0$.   For $0\leq i\leq q$, let 
$\Lambda_i:=\langle \lambda_i e_2,\lambda_i\tau f_2\rangle_{\mathbb{F}_q}$ and $\Delta_i:=\{\delta e_1+u : \delta\in F,u\in \Lambda_i^*\}$.
It is straightforward to verify that the orbits of $\SU_4(q)_x$ on $V_{x,0}\setminus \langle x\rangle$ are $\Delta_0,\ldots,\Delta_q$ since  the orbits of $\SU_2(q)$ on the  isotropic vectors in $\langle e_2,f_2\rangle^*$ are $\Lambda_0^*,\ldots,\Lambda_q^*$. Thus  $\Delta_i\subseteq B$ for some $i$ by Lemma~\ref{lemma:basic}(ii), in which case $\lambda_i e_2$ and $\lambda_i\tau f_2$ are elements of $B$. There exists $g\in G_{0,B}$ such that $(\lambda_ie_2)^g=x$. Now $z:=(\lambda_i\tau f_{2})^g\in B\setminus \langle x\rangle$ and $\f(x,z)=\f(\lambda_i e_2,\lambda_i\tau f_2)\neq 0$,   but we have already seen (replacing $y$ by $z$ above) that this leads to a contradiction.

 Thus $x$ is  non-isotropic. Let $\{v_1,\ldots,v_n\}$ be an orthonormal basis of $V$ (which exists by~\cite[Proposition~2.3.1]{KleLie1990}). We may assume without loss of generality that $x=v_1$. Now $B$ consists of non-isotropic vectors, so $\delta:=\f(y,y)\in \mathbb{F}_q^*$. Let $\lambda:=\f(x,y)$ and  $\varepsilon:= \delta-\lambda\overline{\lambda}\in \mathbb{F}_q$. Define 
   $\Gamma:=\{v\in V\setminus \langle x\rangle: \f(v,v)=\delta,\f(x,v)=\lambda\}$. Let $W:=\langle x\rangle$, $U:=W^\perp$ and $U_\varepsilon:=\{u\in U^*: \f(u,u)=\varepsilon\}$. Now $\Gamma=\{\overline{\lambda}x+u : u\in U_\varepsilon\}$. Observe that if either $n\geq 4$, or $n=3$ and $\varepsilon\neq 0$, then   $\SU_n(q)_x$ acts transitively on $\Gamma$ since $\SU_{n-1}(q)$  acts transitively (as a subgroup of $\SU_n(q)_x$) on $U_\varepsilon$ by~\cite[Lemma 2.10.5]{KleLie1990}, so $\Gamma\subseteq B$ by Lemma~\ref{lemma:basic}(ii).
      
First suppose that $\varepsilon=0$. By the above observations, there exist $\lambda_2,\lambda_3\in F^*$ such that $N(\lambda_2)+N(\lambda_3)=0$ and $z_1:=\overline{\lambda}x+\lambda_2v_2+\lambda_3v_3\in B$ (we may take $z_1=y$ when $n=3$). Choose $\mu\in F^*$ such that $\mu\neq 1$ and $\mu\overline{\mu}=1$. Let 
$z_2:=\overline{\lambda}x+\mu\lambda_2v_2+\overline{\mu}\lambda_3v_3$. There exists $g\in \SU_n(q)_x$ such that $v_2^g=\mu v_2$, $v_3^g=\overline{\mu}v_3$ and $v_i^g=v_i$ for $i\geq 4$.  Now  $z_2=z_1^g\in B$ by Lemma~\ref{lemma:basic}(ii), but $z_2-z_1$ is isotropic, contradicting Lemma~\ref{lemma:basic}(i). 

Hence $\varepsilon\neq 0$ and $\Gamma\subseteq B$. It suffices to find distinct $z_1,z_2\in \Gamma$ such that $z_2-z_1$ is isotropic, for this will contradict Lemma~\ref{lemma:basic}(i). 
 If $q=2$, then $\varepsilon=1$, so we may  take $z_1$ and $z_2$ to be $\overline{\lambda}x+v_2$ and  $\overline{\lambda}x+v_3$. Thus  $q\geq 3$. Now $|K|\geq 3$, so there exists $\mu\in F^*$ such that $\mu\neq 1$ and  $\mu+\overline{\mu}=2$ (including when $q$ is even). Let $F^*=\langle \zeta\rangle$, and write  $\varepsilon=\zeta^{i(q+1)}$ and $\mu=\zeta^{j-i}$ where $1\leq i\leq q-1$ and $1\leq j-i< q^2-1$. Choose $\alpha\in F$ such that $\alpha\overline{\alpha}=\varepsilon-\zeta^j\zeta^{jq}$. Let $z_1:=\overline{\lambda}x+\zeta^iv_2$ and $z_2:=\overline{\lambda}x+\zeta^jv_2+\alpha v_3$. Now $z_1$ and $z_2$ are distinct elements of $\Gamma$ since $\zeta^i\neq \zeta^j$. Further,  $$\f(z_1,z_2)+\f(z_2,z_1)-2\lambda\overline{\lambda}=\zeta^i\zeta^{jq}+\zeta^j\zeta^{iq}=(\zeta^{(j-i)q}+\zeta^{j-i})\zeta^{i(q+1)}=2\varepsilon=2\delta-2\lambda \overline{\lambda},$$
so $\f(z_2-z_1,z_2-z_1)=2\delta-\f(z_2,z_1)-\f(z_1,z_2)=0$. Thus $z_2-z_1$ is isotropic.
\end{proof}

\section{Class (R2)}
\label{s:(R2)}

Let $G$ be an affine permutation group of rank~$3$ on $V:=V_{2m}(q)$ for which $\Omega_{2m}^\varepsilon(q)\unlhd G_0$ where  either $m\geq 3$ and $\varepsilon=\pm$, or $m=2$ and $\varepsilon=-$. The  orbits of $G_0$ on $V^*$ consist of  the non-zero singular vectors and the   non-singular vectors.  In order to simplify the proofs in this section, we define a standard basis for a quadratic space. Let $Q$ be the non-degenerate quadratic form preserved by $\Omega_{2m}^\varepsilon(q)$, and let  $\f$ be the non-degenerate symmetric bilinear form associated with $Q$ (so  $\f(x,y)=Q(x+y)-Q(x)-Q(y)$ for all $x,y\in V$). 
 By~\cite[Proposition~2.5.3]{KleLie1990}, if 
 $\varepsilon=+$, then $V$ has a basis  $\{e_1,\ldots,e_m,f_1,\ldots,f_m\}$ where  $Q(e_i)=Q(f_i)=\f(e_i,e_j)=\f(f_i,f_j)=0$ and $\f(e_i,f_j)=\delta_{i,j}$ for all $i,j$. Further, if $\varepsilon=-$, then 
 $V$ has a basis $\{e_1,\ldots,e_{m-1},f_1,\ldots,f_{m-1},x_0,y_0\}$ where  $Q(e_i)=Q(f_i)=\f(e_i,e_j)=\f(f_i,f_j)=\f(e_i,x_0)=\f(f_i,x_0)=\f(e_i,y_0)=\f(f_i,y_0)=0$ and $\f(e_i,f_j)=\delta_{i,j}$ for all $i,j$, and $Q(x_0)=1$, $\f(x_0,y_0)=1$ and $Q(y_0)=\alpha$ for some $\alpha\in\mathbb{F}_q$ such that the polynomial $X^2+X+\alpha$ is irreducible over $\mathbb{F}_q$;  we then define $e_m:=x_0$ and $f_m:=y_0$. In either case, we refer to $\{e_1,\ldots,e_m,f_1,\ldots,f_m\}$ as a \textit{standard basis} of $V$.

First we consider the case where $m\geq 3$.

\begin{prop}
\label{prop:(R2)mbig}
Let $V:=V_{2m}(q)$ where  $m\geq 3$ and $q$ is a prime power, and let $G$ be an affine permutation group of rank~$3$  on $V$  where $\Omega_{2m}^\varepsilon(q)\unlhd G_0$ and $\varepsilon=\pm$. Then there is no $\mathbb{F}_q$-independent $G$-affine proper partial linear space.
\end{prop}

\begin{proof} 
Let $Q$ be the non-degenerate quadratic form  preserved by $\Omega_{2m}^\varepsilon(q)$, let $\f$ be the  non-degenerate symmetric bilinear form associated with $Q$, and let $\{e_1,\ldots,e_m,f_1,\ldots,f_m\}$ be a standard basis for $V$. Suppose 
for a contradiction that  $(V,\mathcal{L})$ is an $\mathbb{F}_q$-independent  $G$-affine proper partial linear space. Let $L\in\mathcal{L}_0$, let $B:=L^*$, and let $x\in L^*$. By Lemma~\ref{lemma:necessary},  $B$ is a  block of $G_0$ on $x^{G_0}$.  By assumption, $B\setminus \langle x\rangle$ is non-empty.

Let $U:=\langle e_1,f_1\rangle$ and $W:=U^\perp$, and note that $V=U\oplus W$ since $U$ is non-degenerate. We may view  $\Omega^\varepsilon_{2m-2}(q)$ as a subgroup of $\Omega_{2m}^\varepsilon(q)$ that fixes $\langle e_1,f_1\rangle$ pointwise and acts naturally on $W$ by~\cite[Lemma 4.1.1]{KleLie1990}.  For $\lambda\in\mathbb{F}_q$, the group $\Omega^\varepsilon_{2m-2}(q)$ acts transitively on $W_\lambda:=\{w\in W^*: Q(w)=\lambda\}$  by~\cite[Lemma 2.10.5]{KleLie1990} since $m\geq 3$.  The following consequence of Lemma~\ref{lemma:basic}(ii) will be used repeatedly without reference below: if $x\in U$ and
$u+w\in B$ where $u\in U$ and $w\in W^*$,  then $u+z\in B$ for all $z\in W_{Q(w)}$.

First suppose that $x=e_1$, in which case $B$ consists of singular vectors. Choose $y\in B\setminus \langle x\rangle$, and write $y=u+w$ where  $u\in U$ and $w\in W$.  Since $y-x$ is singular by Lemma~\ref{lemma:basic}(i), it follows that  $\f(y,e_1)=0$. Then $\f(u,e_1)=0$, so $u\in \langle x\rangle$. Since $y\not\in\langle x\rangle$, it follows that $w$ is a non-zero singular vector. Thus $u+e_2$ and $u+f_2$ are elements of $B$, but then $e_2-f_2$ is singular by Lemma~\ref{lemma:basic}(i), a contradiction.

Thus we may assume that $x=e_1+f_1$, in which case $B$ consists of non-singular vectors.  We claim that $B \subseteq U$. Suppose for a contradiction that there exists $y\in B\setminus  U$. Write $y=u+w$ where  $u\in U$ and $w\in W^*$. Let $\lambda:=Q(w)$.  If either $\varepsilon=+$, or $\varepsilon=-$ and $m\geq 4$, then $u+\lambda e_2+f_2+e_3$ and $u+\lambda e_2+f_2$ are elements of $B$, but then $e_3$ is non-singular by Lemma~\ref{lemma:basic}(i), a contradiction. Thus $m=3$ and $\varepsilon=-$ (so $e_3$ and $f_3$ are non-singular).  If $\lambda\neq 0$, then there exist $\mu_1,\mu_2\in\mathbb{F}_q$ such that $Q(\mu_1 e_3+\mu_2 f_3)=\lambda$ (see the remark after~\cite[Proposition 2.5.3]{KleLie1990}), in which case $u+e_2+\mu_1 e_3+\mu_2 f_3$ and $u+\mu_1 e_3+\mu_2 f_3$ are elements of $B$, but then $e_2$ is non-singular by Lemma~\ref{lemma:basic}(i), a contradiction. Thus $\lambda=0$.  If $q>2$, then there exists $\mu\in \mathbb{F}_q^*\setminus \{1\}$, 
in which case $u+e_2$ and $u+\mu e_2$ are elements of $B$, but then $(\mu-1)e_2$ is non-singular by Lemma~\ref{lemma:basic}(i), a contradiction. Thus $q=2$. In particular, $L$ is a subspace of $V$ by Lemma~\ref{lemma:affine}.  Therefore, since $u+e_2$, $u+f_2$ and $u+e_2+f_2+e_3$ are elements of $B$, it follows  $u+e_3\in B$, but $Q(e_3)\neq 0$, so this is a contradiction, as above in the case $\lambda\neq 0$.
 
 Thus $B\subseteq U$. Choose $y\in B\setminus \langle x\rangle$. In order to obtain a contradiction, it suffices to find  $g\in \Omega_{2m}^\varepsilon(q)_x$ such that $y^g\notin U$. There exists  $g_1\in \GL_{2m}(q)$ such that $e_1^{g_1}=e_1+e_2$, $f_1^{g_1}=f_1-e_2$, $ e_2^{g_1}= e_1-f_1+e_2+f_2$, $f_2^{g_1}=e_2$ and $g_1$ fixes $\{e_3,\ldots,e_m,f_3,\ldots,f_m\}$ pointwise, and there exists   $g_2\in \GL_{2m}(q)$ that fixes $\{e_1,e_3,\ldots, e_m,f_1,f_3,\ldots,f_m\}$ pointwise and interchanges $e_2$ and $f_2$. Let 
$g:=g_1^{-1}g_2^{-1}g_1g_2$.  Now $Q(v^{g_i})=Q(v)$ for $v\in V$ and $i\in \{1,2\}$, so $g\in \Omega_{2m}^\varepsilon(q)_x$
(see the definition of $\Omega_{2m}^\varepsilon(q)$ in~\cite[\S 2.5, Descriptions~$1$ and~$2$]{KleLie1990}). Write $y=\delta_1 e_1+\delta_2 f_1$ where $\delta_1,\delta_2\in\mathbb{F}_q$, and note that $\delta_1\neq \delta_2$ since $y\notin\langle x\rangle$.
 Now $(\delta_1 e_1+\delta_2 f_1+(\delta_2-\delta_1)f_2)^{g_1}=y$, so $y^g=\delta_2 e_1+\delta_1 f_1+(\delta_2-\delta_1) e_2\notin U$, as desired.   
\end{proof}

In contrast to the case where $m\geq 3$,  an $\mathbb{F}_q$-independent $G$-affine proper partial linear space does exist when $m=2$ and $q=3$. We describe this example now.

\begin{example}
\label{example:(R2)}
Let $G_0:=M_{10}\simeq A_6.2\simeq \Omega_4^-(3).2$.   Let $V$ be one of the irreducible $\mathbb{F}_3G_0$-modules of dimension $4$. (There are two such modules, but the corresponding groups are conjugate in $\GL_4(3)$.) Let $G:=V{:}G_0$. Now $G$ has subdegrees $20$ and $60$, and the restriction of $V$ to $H_0:=A_6$ is irreducible, so it is isomorphic to the fully deleted permutation module $D=S/W$ where $S=\{(\lambda_1,\ldots,\lambda_6)\in \mathbb{F}_3^6{:} \sum_{i=1}^6 \lambda_i=0\}$ and $W=\{(\lambda,\ldots,\lambda):\lambda\in\mathbb{F}_3\}$. With this viewpoint, $H_0$ has orbits of sizes $20$, $30$ and $30$  containing $(1,1,1,0,0,0)+W$, $(1,-1,0,0,0,0)+W$ and $(1,-1,1,-1,0,0)+W$, respectively.  Let $x:=(1,0,0,0,0,-1)+W$ and $B:=x^{K_0}$ where $K_0$ is the subgroup of $H_0$ that fixes $6$. Let  $\mathcal{L}:=L^G$ where $L:=B\cup\{0\}$. We claim that $\mathcal{M}:=(V,\mathcal{L})$ is a $G$-affine proper partial linear space  with  line-size $6$ and point-size $12$: by Lemma~\ref{lemma:sufficient}, it suffices to show that $B$ is a block of $G_0$ and that there exists  $h\in H_L$ for which $0^h=x$ where $H:=V{:}H_0$.  Let $C$ be the orbit of $H_0$ containing $x$. Now $C$ is  a block of $G_0$, so $G_{0,x}\leq G_{0,C}=H_0$. Thus $G_{0,x}$ is the subgroup of $H_0$ that fixes $1$ and $6$. Since $G_{0,x}\leq K_0$, it follows that $B$ is a block of $G_0$. Define $h:=\tau_v (1 2 6)$ where  $v:=(0,-1,0,0,0,1)+W$.  Then $h\in H_L$ and $0^h=x$, so the claim holds. Using {\sc Magma}, we verify that $\Aut(\mathcal{M})=G$.  Note that if we instead choose $K_0$ to be the subgroup of $H_0$ that fixes $1$, then we obtain another  $G$-affine proper partial linear space (to see this, take $h:=\tau_v(162)$ where  $v:=(-1,1,0,0,0,0)+W$), but we will see in Proposition~\ref{prop:(R2)msmall} that these partial linear spaces are isomorphic. Note also that $G_0$ has no proper subgroups with two orbits on $V^*$.
\end{example}

\begin{prop}
\label{prop:(R2)msmall}
Let $V:=V_{4}(q)$ where $q$ is a prime power, and let $G$ be an affine permutation group of rank~$3$  on $V$  where $\Omega_{4}^-(q)\unlhd G_0$.
 There is an $\mathbb{F}_q$-independent $G$-affine proper partial linear space $\mathcal{S}$  if and only if $q=3$, $G_0=M_{10}$ and $\mathcal{S}=\mathcal{M}$ or $\mathcal{M}^g$ where $g\in N_{\GL_4(3)}(G_0)\setminus G_0$ and $\mathcal{M}$  is the partial linear space of Example~\emph{\ref{example:(R2)}}. 
\end{prop}

\begin{proof}
Let  $S:=\Omega_{4}^-(q)$, and note that $S\simeq \PSL_2(q^2)$. Let $Q$ be the non-degenerate quadratic form  preserved by $S$, let $\f$ be the  non-degenerate symmetric bilinear form associated with $V$, and let $\{e_1,e_2,f_1,f_2\}$ be a standard basis for $V$. By~\cite[Lemma 2.10.5]{KleLie1990}, the orbits of $S$ on $V^*$ are  $V_\lambda:=\{v\in V^*: Q(v)=\lambda\}$ for $\lambda\in\mathbb{F}_q$. In particular, since $V$ contains $q(q^2+1)(q-1)$ non-singular vectors, it follows that  $|V_\lambda|=q(q^2+1)$ for $\lambda\in\mathbb{F}_q^*$.

Let $(V,\mathcal{L})$ be an $\mathbb{F}_q$-independent  $G$-affine proper partial linear space. Let $L\in\mathcal{L}_0$, let $B:=L^*$, and let $x\in B$. By Lemma~\ref{lemma:primitive}, $G$ is primitive, so by Lemma~\ref{lemma:necessary},  $B$ is a non-trivial block of $G_0$ on $x^{G_0}$. By assumption, there exists $y\in B\setminus \langle x\rangle$. If $x$ is singular, then $y$ is singular and  $y\in \langle x\rangle^\perp$ by Lemma~\ref{lemma:basic}(i),  but then $y\in\langle x\rangle$, a contradiction.  Thus $x$ is non-singular, and $B$ consists of non-singular vectors. 

First we claim that $S_x\simeq \PSL_2(q)$. Note that $|S_x|=q(q^2-1)/(2,q-1)=|\PSL_2(q)|$.  
If $q$ is odd, then $\langle x\rangle$ is non-degenerate,  so $\Omega_3(q)$ is a subgroup of $S$ that fixes $\langle x\rangle$ pointwise by~\cite[Lemma 4.1.1(iii)]{KleLie1990}; thus $\PSL_2(q)\simeq \Omega_3(q)\leq S_x$, and the claim follows. 
Otherwise $q$ is even, in which case $S_x=S_{\langle x\rangle}$,  and $S_{\langle x\rangle}\simeq \Sp_2(q)$ by the proof of~\cite[Proposition 4.1.7]{KleLie1990} (which is only stated for orthogonal groups with larger dimension). Thus $S_x\simeq \PSL_2(q)$, as desired. 

If $q=2$,  then $S$ is transitive on non-singular vectors  and $S\simeq A_5$. By the claim, $S_x\simeq S_3$, so $S_x$ is maximal in $S$, in which case $B=x^S=x^{G_0}$,  a contradiction. Thus  $q>2$. 

Let $M$ be a maximal subgroup of $S$ containing $S_x$. We claim that either $[M:S_x]\leq 2$, or $q=3$ and $M\simeq A_5$. If $q\neq 3$, then since $S_x\simeq \PSL_2(q)$, it follows from~\cite[Table 8.17]{BraHolRon2013} (see also Dickson~\cite{Dic1901})  that  $M\simeq \PGL_2(q)$, in which case $[M:S_x]\leq 2$. If $q=3$, then $S\simeq A_6$ and $S_x\simeq A_4$, and the claim follows.
 
Let $\lambda\in\mathbb{F}_q^*$. We next claim that either $|B\cap V_\lambda|\leq 2$, or $q=3$ and $S_{B\cap V_\lambda}\simeq A_5$.  Let $z\in B\cap V_\lambda$. Now $V_\lambda$ is an orbit of $S$ and $B$ is a block of $G_0$, so $B\cap V_\lambda$ is a block of $S$ in its action on $V_\lambda$, and $S_z\leq S_{B\cap V_\lambda}$. If $V_\lambda\subseteq B$, then some conjugate of $B$ contains $V_1$, but $Q(f_1+e_2)=1=Q(e_2)$, while $f_1$ is singular, contradicting Lemma~\ref{lemma:basic}(i). Thus $S_{B\cap V_\lambda}$ lies in some maximal subgroup $M$ of $S$. If $q=3$ and $M\simeq A_5$, then since $ S_z\simeq A_4$, we conclude that $B\cap V_\lambda =\{z\}$ or  $S_{B\cap V_\lambda}\simeq A_5$, as desired. Otherwise, $|B\cap V_\lambda|\leq [M:S_z]\leq 2$ by the previous claim.

We may assume that $x=e_2$ (since $x$ is non-singular). Suppose that $q$ is even.  Let $U:=\langle e_1,f_1\rangle$ and $W:=U^\perp=\langle e_2,f_2\rangle$. Suppose that there exists $z\in B\setminus W$.  Write $z=u+w$ where $u\in U^*$ and $w\in W$.  Now $\{u^g+w:g \in \Omega_2^+(q)\}\subseteq B\cap V_{Q(z)}$ by Lemma~\ref{lemma:basic}(ii), but  the  orbits of $\Omega_2^+(q)$ on $U^*$ all have size $q-1$ by~\cite[Lemma 2.10.5]{KleLie1990}, so $|B\cap V_{Q(z)}|>2$, a contradiction. Hence $B\subseteq W$. Now $y=\delta_1e_2+\delta_2f_2$ for some $\delta_1,\delta_2\in \mathbb{F}_q$ where $\delta_2\neq 0$. For $e\in \{e_1+e_2,e_2\}$, the reflection $r_e$ is an isometry of $(V,Q)$ for which $v\mapsto v+\f(e,v)e$ for all $v\in V$. By~\cite[p.30]{KleLie1990}, 
$g:=r_{e_1+e_2}r_{e_2}\in S$ since $g$ is a product of two reflections. Now $x^g=x$, so $y^g\in B$, but $f_2^g=e_1+f_2$, so 
$y^g=\delta_2 e_1 + y\not\in W$,  a contradiction.

Thus $q$ is odd. Now $\Omega_3(q)=S_x$. Let
$U:=\langle x\rangle$ and $W:=U^\perp$. Write $y=u+w$ where  $u\in U$ and $w\in  W^*$. Let $\lambda:=Q(y)$. Now $\{u+w^g:g\in \Omega_3(q) \}\subseteq B\cap V_\lambda$ by Lemma~\ref{lemma:basic}(ii), but 
 the orbits of $\Omega_3(q)$ on $W^*$ have sizes $(q^2-1)/2$, $q(q+1)$ and $q(q-1)$ by~\cite[Lemma 2.10.5]{KleLie1990}, none of which are at most $2$. Thus $q=3$ and $S_{B\cap V_\lambda}\simeq A_5$. In particular, $|B\cap V_{\lambda}|=5$.
 Using Lemma~\ref{lemma:necessary} and {\sc Magma}, we determine that this is only possible if $G_0= M_{10}$ and  $|B|=5$. Further, $G_0$ has exactly two such blocks, say  $B_1$ and $B_2$, and by Lemma~\ref{lemma:sufficient}, both $(B_1\cup\{0\})^G$ and $(B_2\cup\{0\})^G$ are $G$-affine proper partial linear spaces, so one  must be $\mathcal{M}$, where $\mathcal{M}$ is the partial linear space of Example~\ref{example:(R2)}. Since $N_{\GL_4(3)}(G_0)=M_{10}.2$,  the other is $\mathcal{M}^g$ where $g\in N_{\GL_4(3)}(G_0)\setminus G_0$.  
\end{proof}

\section{Class (R3)}
\label{s:(R3)}

Recall that the exterior square $\bigwedge^2(W)$ of a vector space $W:=V_n(q)$ is the quotient of $W\otimes W$ by the ideal $I:=\langle u\otimes u: u\in W\rangle$, and we write $u\wedge v$ for $u\otimes v+I$. A non-zero vector of $\bigwedge^2(W)$ is \textit{simple} whenever it can be written in the form $u\wedge v$ for some $u,v\in W$. Note that $u\wedge u=0$ and $u\wedge v=-v\wedge u$ for all $u,v\in W$. Further, if $\{u_1,\ldots,u_n\}$ is a basis for $W$, then $\{u_i\wedge u_j:1\leq i<j\leq n\}$ is a basis for $\bigwedge^2(W)$. Lastly,  $\GammaL(W)$ acts on $\bigwedge^2(W)$ by $(u\wedge v)^g=u^g\wedge v^g$ for all $g\in \GammaL(W)$ and $u,v\in W$.   The kernel of this action is $\langle -1\rangle$ for $n\geq 3$.

Let $G$ be an affine permutation group of rank~$3$ on $V:=\bigwedge^2(V_5(q))$ for which $\SL_5(q)\unlhd G_0$.  The  orbits of $G_0$ on $V^*$ consist of  the set of  simple vectors and the set of non-simple vectors. 

\begin{prop}
\label{prop:(R3)}
Let $V:=\bigwedge^2(V_5(q))$ where $q$ is a prime power. Let $G$ be an affine permutation group of rank~$3$ on $V$ where $\SL_5(q)\unlhd G_0$. Then there is no $\mathbb{F}_q$-independent $G$-affine proper partial linear space. 
\end{prop}

\begin{proof}
 Suppose for a contradiction that  $(V,\mathcal{L})$ is an $\mathbb{F}_q$-independent  $G$-affine proper partial linear space. Let $L\in \mathcal{L}_0$, let $B:=L^*$, and let $x\in L^*$. By Lemma~\ref{lemma:necessary},  $B$ is a  block of $G_0$ on $x^{G_0}$. By assumption, there exists $y\in B\setminus \langle x\rangle$.

First suppose that $x$ is simple.  Then $x=x_1\wedge x_2$ for some linearly independent $x_1,x_2\in V_5(q)$ and $y=y_1\wedge y_2$ for some linearly independent $y_1,y_2\in V_5(q)$. If $ x_1,x_2,y_1,y_2$ are linearly independent, then $y-x$ is not simple, contradicting Lemma~\ref{lemma:basic}(i), so  $\mu_1x_1+\mu_2x_2 +\mu_3y_1+\mu_4y_2=0$ for some  $\mu_1,\ldots,\mu_4\in \mathbb{F}_q$ where $\mu_i\neq 0$ for some $i$. If $\mu_3=\mu_4=0$, then we have a contradiction, so without loss of generality, we may assume that $\mu_3\neq 0$. Hence $y=v_1\wedge v_2$  where $v_2:=y_2$  and $v_1:=\lambda_1 x_1+\lambda_2 x_2$ for some $\lambda_1,\lambda_2\in \mathbb{F}_q$. 
Now $x_1,x_2,v_2$ are linearly independent since $y\notin \langle x\rangle$, so we may choose $v_3\in V_5(q)$ such that $x_1,x_2,v_2,v_3$ are linearly independent, and there exists $g\in \SL_5(q)$ such that $x_1^g=x_1$,  $x_2^g=x_2$ and $v_2^g=v_3$. Then $x^g=x$, so $z:=v_1\wedge v_3=y^g\in B$ by Lemma~\ref{lemma:basic}(ii). Since $v_1,v_2,v_3$ are linearly independent, there exists $v_4\in V_5(q)$ such that $v_1,v_2,v_3,v_4$ are linearly independent, and there exists $h\in \SL_5(q)$ such that $v_1^h=v_3$, $v_3^h=-v_1$ and $v_2^h=v_4$. Then $z^h=z$, so $v_3\wedge v_4=(v_1\wedge v_2)^h\in B$ by Lemma~\ref{lemma:basic}(ii). Thus $v_1\wedge v_2-v_3\wedge v_4$ is simple by Lemma~\ref{lemma:basic}(i), a contradiction.

Thus $x$ is not simple, in which case $x=x_1\wedge x_2+x_3\wedge x_4$ for some linearly independent $x_1,x_2,x_3,x_4\in V_5(q)$ and $y=y_1\wedge y_2+y_3\wedge y_4$ for some linearly independent $y_1,y_2,y_3,y_4\in V_5(q)$. Suppose that $x_1,x_2,x_3,x_4,y_4$ are linearly independent. We may assume without loss of generality that $y_1,y_2,y_3\in \langle x_1,x_2,x_3,x_4\rangle$. Choose $x_0\in \langle x_1,x_2,x_3,x_4\rangle \setminus \langle y_3\rangle$. There exists $g\in \SL_5(q)$ such that $x_i^g=x_i$ for $1\leq i\leq 4$ and $y_4^g=x_0+y_4$. Then $x^g=x$ and  $y^g-y=y_3\wedge x_0\neq 0$, contradicting Lemma~\ref{lemma:basic}. By symmetry, we conclude that $y_i\in \langle x_1,x_2,x_3,x_4\rangle$ for $1\leq i\leq 4$. Now  $y=\sum_{1\leq i<j\leq 4}\lambda_{i,j} x_i\wedge x_j$ for some $\lambda_{i,j}\in\mathbb{F}_q$.  Let $x_0:=\lambda_{1,3}x_3+\lambda_{1,4}x_4$. There exists $h\in \SL_5(q)$ such that $x_1^h=x_1+x_2$ and $x_i^h=x_i$ for $2\leq i\leq 4$. Now $x^h=x$ and   $y^h-y= x_2\wedge x_0$, so by Lemma~\ref{lemma:basic}, $x_0=0$.  By exchanging the roles of $x_1$ and $x_2$ in this argument, we see that $y=\lambda x_1\wedge x_2+\mu x_3\wedge x_4$ for some $\lambda,\mu\in\mathbb{F}_q$. Note that $\lambda\neq \mu$ since $y\notin \langle x\rangle$. There exists $k\in \SL_5(q)$ such that $x_1^k=x_1+x_3$, $x_4^k=x_4-x_2$ and $x_i^k=x_i$ for $i=2,3$. Then $x^k=x$ and $y^k-y=(\lambda-\mu) x_3\wedge x_2\neq 0$, contradicting Lemma~\ref{lemma:basic}.
\end{proof}

\section{Class (R4)}
\label{s:(R4)}

Let $G$ be an affine permutation group of rank~$3$ on $V_4(q)$ for which $\Sz(q)\unlhd G_0$, where $q=2^{2n+1}$.   Recall that the group $G$ has suborbits of size $(q-1)(q^2+1)$ and $q(q-1)(q^2+1)$. 
 
\begin{prop}
\label{prop:(R4)}
Let $V:=V_4(q)$ where $q:=2^{2n+1}$ for some $n\geq 1$, and let $G$ be an affine permutation group of rank~$3$ on $V$ where $\Sz(q)\unlhd G_0$.  Then there is no $\mathbb{F}_q$-independent $G$-affine proper partial linear space. 
\end{prop}

\begin{proof}
We follow the notation of~\cite{Suz1962}.   Let $\theta\in \Aut(\mathbb{F}_q)$ where $\alpha^\theta=\alpha^{2^{n+1}}$ for all $\alpha\in\mathbb{F}_q$. For $\alpha,\beta\in \mathbb{F}_q$, define $f(\alpha,\beta):=\alpha^{\theta+2} +\alpha\beta+\beta^\theta$, 
$$
[\alpha,\beta]:=
\begin{pmatrix}
1 & 0 & 0 & 0 \\
\alpha & 1 & 0 & 0 \\
\alpha^{\theta+1}+\beta & \alpha^\theta & 1 & 0 \\
f(\alpha,\beta) & \beta & \alpha & 1 \\
\end{pmatrix},
\tau:=
\begin{pmatrix}
0 & 0 & 0 & 1 \\
0 & 0 & 1 & 0 \\
0 & 1 & 0 & 0 \\
1 & 0 & 0 & 0 \\
\end{pmatrix},
$$
and observe that $[\alpha,\beta][\gamma,\delta]=[\alpha+\gamma,\alpha \gamma^\theta+\beta+\delta]$ for all $\gamma,\delta\in\mathbb{F}_q$. 
We also define
\begin{align*}
Q(q)&:=\{[\alpha,\beta]:\alpha,\beta\in \mathbb{F}_q\},\\
K(q)&:=\{\diag(\kappa^{\theta^{-1}+1},\kappa^{\theta^{-1}},(\kappa^{\theta^{-1}})^{-1},(\kappa^{\theta^{-1}+1})^{-1}):\kappa\in \mathbb{F}_q^*\},\\
H(q)&:=\langle Q(q),K(q)\rangle,\\
\Sz(q)&:=\langle H(q),\tau\rangle.
\end{align*}
Clearly $Q(q)$ is a group of order $q^2$. Note that if $\kappa^{\theta+1}=1$ for some $\kappa\in\mathbb{F}_q^*$, then $\kappa=1$ since the integers $q-1$ and $2^{n+1}+1$ are coprime, so $\{\kappa^{\theta+1}:\kappa\in\mathbb{F}_q^*\}=\mathbb{F}_q^*$. Now $K(q)$ is a cyclic group of order $q-1$, and $K(q)$ normalises $Q(q)$, so $H(q)=Q(q){:}K(q)$ and  $|H(q)|=q^2(q-1)$.  Let $\Sz(q)$ act on $V$ with respect to the basis $\{v_1,v_2,v_3,v_4\}$.
By~\cite[Theorem 7]{Suz1962} and its proof, $\Sz(q)$ has order $q^2(q-1)(q^2+1)$ and every element of $\Sz(q)\setminus H(q)$ has the form $h_1\tau h_2$ for some $h_1\in H(q)$ and $h_2\in Q(q)$. It follows that $\Sz(q)_{v_1}=Q(q)$ and  
$$\Omega_1:=v_1^{\Sz(q)}=\{\lambda v_1:\lambda\in\mathbb{F}_q^*\}\cup\{\mu(f(\alpha,\beta)v_1+\beta v_2 +\alpha v_3+v_4):\mu\in\mathbb{F}_q^*,\alpha,\beta\in\mathbb{F}_q\},
$$
a set of size $(q-1)(q^2+1)$.
Moreover, $\Sz(q)_{v_2}=\{[0,\beta]:\beta\in\mathbb{F}_q\}$ and $\Omega_2:=v_2^{\Sz(q)}$ is a set of size $q(q-1)(q^2+1)$. Hence $V{:}\Sz(q)$ is itself a rank~$3$ permutation group on $V$.

Suppose for a contradiction that  $(V,\mathcal{L})$ is an $\mathbb{F}_q$-independent  $G$-affine proper partial linear space.  Then $(V,\mathcal{L})$ is also a  $(V{:}\Sz(q))$-affine partial linear space. Let $L\in \mathcal{L}_0$, let $B:=L^*$, and let $x\in B$. By Lemma~\ref{lemma:primitive}, $V{:}\Sz(q)$ is primitive, so by Lemma~\ref{lemma:necessary},  $B$ is a non-trivial block of  $\Sz(q)$ on $x^{\Sz(q)}$. 
By assumption, there exists $y\in B\setminus \langle x\rangle$. Now $y=\sum_{i=1}^4\lambda_i v_i$ for some $\lambda_i\in \mathbb{F}_q$. If $x=v_1$, then $y\in \Omega_1\setminus \langle v_1\rangle$, so $\lambda_4\neq 0$, but  
  $x^{[0,1]}=x$ and $y^{[0,1]}-y=(\lambda_3+\lambda_4)v_1+\lambda_4 v_2\notin x^{\Sz(q)}$,
  contradicting  Lemma~\ref{lemma:basic}. It follows that $x\not\in \Omega_1$, so $x\in \Omega_2$, and we may assume without loss of generality that $x=v_2$. 

Observe that   $\Sz(q)_x\leq \Sz(q)_B<\Sz(q)$. Now $q$  divides $|\Sz(q)_B|$, so    $\Sz(q)_B\leq H(q)^g$ for some $g\in \Sz(q)$ by~\cite[Theorem 9]{Suz1962}. If $g\notin H(q)$, then $H(q)^g= H(q)^{\tau h}$ for some $h\in Q(q)$, but $\Sz(q)_x=Z(Q(q))$, so $\Sz(q)_x \leq H(q)^\tau$,  a contradiction since $\tau$ conjugates any lower triangular matrix to an upper triangular matrix. Hence $g\in H(q)$, so   $\Sz(q)_B\leq H(q)$.

Let $C:=\langle v_1,v_2\rangle\setminus \langle v_1\rangle$. Now $\Sz(q)_C=H(q)$ and $C$ is a block of $\Sz(q)$ in its action on~$\Omega_2$. Thus $B\subseteq C$. In particular, $|B|$ divides $q(q-1)$. Moreover, $L$ is an $\mathbb{F}_2$-subspace of $V$ by Lemma~\ref{lemma:affine}, so $|B|$ divides $q-1$. Observe that $\Sz(q)_x$ is the kernel of the action of $H(q)$ on $C$, so $\overline{H}:=H(q)/\Sz(q)_x$ acts regularly on $C$. Thus $|\overline{H}_B|$ divides $q-1$. 

Let   $\overline{Q}:=Q(q)/\Sz(q)_x$ and  $\overline{K}:=K(q)\Sz(q)_x/\Sz(q)_x\simeq K(q)$. Note that $\overline{H}=\overline{Q}{:}\overline{K}$. Further,  $|\overline{Q}|=q$ and $|\overline{K}|=q-1$. Now $\overline{H}_B\cap \overline{Q}=1$. In particular, $\overline{H}_B$ is isomorphic to a subgroup of $\overline{H}/\overline{Q}\simeq \overline{K}$, so $\overline{H}_B$ is cyclic. Let $h$ be a generator of $\overline{H}_B$. We claim that $h\in \overline{K}^g$ for some $g\in \overline{Q}$.  Observe that $\overline{K}\cap \overline{K}^g=1$ for all $g\in \overline{Q}\setminus \{1\}$, so $$|\overline{H}\setminus (\bigcup_{g\in \overline{Q}}\overline{K}^g)|=|\overline{H}|-(|\overline{Q}|(|\overline{K}|-1)+1)=|\overline{Q}|-1.$$ Thus $\overline{Q}\setminus \{1\}=\overline{H}\setminus (\bigcup_{g\in \overline{Q}}\overline{K}^g)$, and the claim follows. 

Hence $\overline{H}_B\leq \overline{K}^g$ for some $g\in \overline{Q}$. It follows that $B\subseteq \{\alpha (\kappa^{\theta^{-1}+1}+\kappa^{\theta^{-1}})v_1+\kappa^{\theta^{-1}}v_2: \kappa\in \mathbb{F}_q^*\}$ for some $\alpha\in\mathbb{F}_q$. Note that $\alpha\neq 0$, or else $B\subseteq \langle x\rangle$. Since $|B|\geq 2$, there exists  $\kappa\in \mathbb{F}_q^*$ such that $\kappa\neq 1$ and $\alpha (\kappa^{\theta^{-1}+1}+\kappa^{\theta^{-1}})v_1+\kappa^{\theta^{-1}}v_2\in B$, and since $L$ is an $\mathbb{F}_2$-subspace of $V$, $\alpha (\kappa^{\theta^{-1}+1}+\kappa^{\theta^{-1}})v_1+(\kappa^{\theta^{-1}}+1)v_2\in B$. Note that $\kappa^{\theta^{-1}}+1=(\kappa+1)^{\theta^{-1}}$. Now
$$\kappa^{\theta^{-1}+1}+\kappa^{\theta^{-1}}=(\kappa+1)^{\theta^{-1}+1}+(\kappa+1)^{\theta^{-1}}=\kappa^{\theta^{-1}+1}+\kappa.$$ 
Thus $\kappa^\theta=\kappa$, but then $\kappa=\kappa^\theta=\kappa^{\theta^2}=\kappa^2$, so $\kappa=1$, a contradiction.
\end{proof}

\section{Class (R5)}
\label{s:(R5)}

In order to determine the $G$-affine proper partial linear spaces where $G$  belongs to (R5), we first outline the theory of spin modules and the quasisimple group $\spin_m^\varepsilon(q)$. Our exposition closely follows Chevalley's treatise on the subject~\cite{Che1996}.

Let $W:=V_m(q)$ where $m\geq 7$ and $q$ is a prime power.  Let $Q$ be a non-degenerate quadratic form on $W$ with associated  non-degenerate symmetric bilinear form $\f$. (In particular, if $m$ is odd, then $q$ is odd.) Let $\varepsilon$ be the type of $Q$, where $\varepsilon\in\{+,-,\circ\}$. The \textit{Clifford algebra} $C(W,Q)=C$  is defined to be the algebra $T/I$, where $T$ is the tensor algebra of $W$ (see~\cite{ShuSur2015}) and $I$ is the ideal generated by $v\otimes v-Q(v)$ for all $v\in W$. We view $W$ as a subspace of $C$ by identifying each $v\in W$ with $v+I$. Similarly, we view $\mathbb{F}_q$ as a subalgebra of $C$ by identifying each $\lambda\in \mathbb{F}_q$ with $\lambda+I$. We also suppress the notation~$\otimes$. Now $v^2=Q(v)$ and $vw+wv=\f(v,w)$ for all $v,w\in V$.   

By~\cite[II.1.2]{Che1996}, for any basis $\{v_1,\ldots,v_m\}$ of $W$, the Clifford algebra has a basis consisting of the  products  $v_{i_1}\cdots v_{i_k}$ for $1\leq k\leq m$ and  integers  $i_1,\ldots,i_k$ such that $1\leq i_1<\cdots<i_k\leq m$, as well as the empty product $1$. In particular, $C$ has dimension $2^m$. Further, $C=C_+\oplus C_-$ where $C_+$ (respectively, $C_-$) is the subspace of $C$ generated by the basis vectors with an even (respectively, odd) number of factors.   We call the elements of $C_+$ and $C_-$ \textit{even} and \textit{odd},  respectively.
Let $Z(C)$ denote the centre of $C$.  For $m$ even, $Z(C)=\mathbb{F}_q$ by~\cite[II.2.1]{Che1996}, and for $m$ odd,  $Z(C)=\mathbb{F}_q+\mathbb{F}_qz_0$ for some $z_0\in C_-$  by~\cite[II.2.6]{Che1996}. In particular, $Z(C)\cap C_+=\mathbb{F}_q$. For $R\subseteq C$, we write $R^\times$ for the set of invertible elements in~$R$.

The \textit{Clifford group}  $\Gamma(C)=\Gamma$ is defined to be $\{s\in C^\times: s^{-1}xs\in W\ \mbox{for all}\ x\in W\}$. Note that $\Gamma\cap W$ is the set of non-singular vectors of $W$, for such a vector $v$ has inverse $Q(v)^{-1}v$ and  $v^{-1}xv=-x+\f(x,v)Q(v)^{-1}v$ for all $x\in W$, while no singular vector of $W$ is invertible. The \textit{even Clifford group}  $\Gamma^+$ is defined to be $\Gamma\cap C_+$. There is a natural representation $\chi$ of $\Gamma$ on $W$ defined by $s\chi: x\mapsto s^{-1}xs$ for $s\in\Gamma$ and $x\in W$. By~\cite[II.3.1]{Che1996}, $\chi$ has kernel $Z(C)^\times$, and $\Gamma\chi$ is $\GO_m^\varepsilon(q)$ for $m$ even and $\SO_m^\varepsilon(q)$ for $m$ odd. By~\cite[II.3.3]{Che1996}, the group $\Gamma^+\chi$ has index $2$ in $\GO_m^\varepsilon(q)$ and, for $q$ odd,  equals $\SO_m^\varepsilon(q)$. 
 Note that $v\chi=-r_v$ for all non-singular vectors $v\in W$, where $r_v$ is the \textit{reflection} in $v$, defined by $x\mapsto x-\f(x,v)Q(v)^{-1}v$ for all $x\in W$.

The Clifford algebra has an  antiautomorphism $\alpha$ defined by $v_1v_2 \cdots v_k\mapsto v_k v_{k-1}\cdots v_1$ for all  $k\geq 1$ and $v_i\in W$ (see~\cite{ShuSur2015}). Now $\Gamma^\alpha=\Gamma$ and $s(s^\alpha)\in Z(C)^\times$ for $s\in\Gamma$ by~\cite[II.3.5]{Che1996}, so there is a homomorphism $\varphi:\Gamma\to Z(C)^\times$ defined by $s\mapsto s(s^\alpha)$ for $s\in \Gamma$. Note that $v\varphi=v^2=Q(v)$ for $v\in \Gamma\cap W$. Let $\Gamma_0:=\ker(\varphi)$ and $\Gamma_0^+:=\Gamma_0\cap C_+$. Let $\pi:\mathbb{F}_q^*\to \mathbb{F}_q^*/\mathbb{F}_q^{*2}$ be the natural homomorphism, where $\mathbb{F}_q^{*2}=\{\lambda^2:\lambda\in\mathbb{F}_q^*\}$, and note that $|\mathbb{F}_q^*/\mathbb{F}_q^{*2}|=\gcd(2,q-1)$. Since $\ker(\chi|_ {\Gamma^+})=\mathbb{F}_q^*\leq \ker(\varphi\pi)$, there is a homomorphism $\theta:\Gamma^+\chi\to \mathbb{F}_q^*/\mathbb{F}_q^{*2}$ called the \textit{spinor norm} defined by $s\chi \mapsto s\varphi\pi$ for $s\in \Gamma^+$. Now $(\ker(\varphi|_{\Gamma^+}))\chi\leq \ker(\theta)$, and $\ker(\theta)\leq (\ker(\varphi|_{\Gamma^+}))\chi$ since  $\ker(\pi)\leq  (\ker(\chi|_{\Gamma^+}))\varphi$. Thus $\Gamma_0^+\chi=\ker(\theta)$. It follows from the definition of $\Omega_m^\varepsilon(q)$ in~\cite[\S 2.5, Descriptions~$1$ and~$2$]{KleLie1990} that $\ker(\theta)=\Omega_m^\varepsilon(q)$. Thus $\Gamma_0^+\chi= \Omega_m^\varepsilon(q)$.    Further, $\Gamma_0^+\cap \ker(\chi)=\langle -1\rangle$  since $\ker(\chi|_ {\Gamma^+})=\mathbb{F}_q^*$. Hence $\Gamma_0^+/\langle -1\rangle\simeq \Omega_m^\varepsilon(q)$. The group $\Gamma_0^+$ is  called the \textit{spin group} and is also denoted by $\spin_m^\varepsilon(q)$. Since $\Omega_m^\varepsilon(q)$ is perfect, $\Gamma_0^+=(\Gamma_0^+)'\langle -1\rangle$. There exist $u,v\in W$ such that $Q(u)=Q(v)=1$ and $\f(u,v)=0$. Now  $uv\in \Gamma_0^+$, so $-1=(uv)^2\in (\Gamma_0^+)'$. Thus $\spin_m^\varepsilon(q)$ is perfect. The proof of the following is routine.

\begin{lemma}
\label{lemma:spinelements}
Let $Q$ be a non-degenerate quadratic form on $W:=V_{m}(q)$ with type $\varepsilon$ and associated  bilinear form $\f$ where $m\geq 7$ and $q$ is a prime power. Let $u,v\in W$ be linearly independent where $Q(u)=0$ and $\f(u,v)=0$. Then $(1-uv)(1+uv)=1$ and  $(1+uv)x(1-uv)=x+\f(x,v)u-\f(x,u)v-Q(v)\f(x,u)u$ for all $x\in W$. In particular, $1-uv\in \spin_{m}^\varepsilon(q)$.
\end{lemma}

Now suppose that $m$ is even and $\varepsilon=+$. Write $m=2r$. Let  $\{e_1,\ldots,e_r,f_1,\ldots,f_r\}$ be a standard basis of $W$ (see \S\ref{s:(R2)}). Let $E:=\langle e_1,\ldots,e_r\rangle$, and denote by $C^E$ the subalgebra of $C$ generated by $E$. For $S=\{i_1,\ldots,i_k\}$ in the power set $\mathbb{P}(\{1,\ldots,r\})$ where $i_1<\cdots <i_k$, let $e_S:=e_{i_1}\cdots e_{i_k}$  and $f_S:=f_{i_1}\cdots f_{i_k}$. Let $e_\varnothing:=1$ and $f_\varnothing:=1$. 
Then $\{e_S : S\in \mathbb{P}(\{1,\ldots,r\})\}$ is a basis for $C^E$, and $\{f_Te_S : T,S\in \mathbb{P}(\{1,\ldots,r\})\}$ is a basis for $C$.   Let $f^*:=f_1\cdots f_r=f_{\{1,\ldots, r\}}$, and note that $f^*f_i=0$ for  $1\leq i\leq r$.    
Using this observation and the above basis for $C$, it can be verified that, for each $a\in C^E$ and $c\in C$, there exists a unique $a'\in C^E$ such that $f^*ac=f^*a'$, and we write $a^c:=a'$. Further,   $C^E$ is a right $C$-module under this action.  
Define $C^E_+:=C^E\cap C^+$. Now $C^E_+$ is an irreducible $\mathbb{F}_q\Gamma^+_0$-module by~\cite[II.4.3]{Che1996} and the introduction to~\cite[Ch.\ III]{Che1996}; it is called the \textit{spin module} and has dimension $2^{r-1}$.
 The kernel of the action of $\Gamma^+$ on $C^E_+$ 
 has order $\gcd(2,q-1)$  by~\cite[III.6.1]{Che1996}. Let $z:=\Pi_{i=1}^r (e_i+f_i)(f_i-e_i)$. For $S\subseteq \{1,\ldots,r\}$, note that   $f^*e_S(e_i+f_i)(f_i-e_i)$ is $-f^*e_S$ when $i\in S$ and $f^*e_S$ otherwise,  whence $e_S^z=(-1)^{|S|}e_S$. Since those  $e_S$ with $|S|$ even form a basis of  $C^E_+$, it follows that $\langle z\rangle$ is the kernel of $\Gamma^+$ on $C^E_+$. In particular, for $r$  odd, $C^E_+$ is a faithful $\mathbb{F}_q\spin_{2r}^+(q)$-module. 
 
We claim that $\spin_{10}^+(q)$ is quasisimple with centre $\spin_{10}^+(q)\cap Z$ where $Z:=Z(\GL(C_+^E))$. Since $\Gamma_0^+/\langle -1\rangle\simeq\Omega_{10}^+(q)$, it suffices to consider the case where $\Omega_{10}^+(q)$ is not simple, so we assume that $q\equiv 1\mod 4$. Now there exists $\lambda\in \mathbb{F}_q^*$ such that  $\lambda^2=-1$, and $z\lambda\in \Gamma_0^+\cap Z$. Since $(\Gamma_0^+\cap Z)/\langle -1\rangle \leq Z(\Gamma_0^+)/\langle -1\rangle\leq Z(\Omega_{10}^+(q))$, which has order $2$, it follows that $\Gamma_0^+\cap Z=Z(\Gamma_0^+)$ and that $\Gamma_0^+$ is quasisimple.

Let $M$ be a maximal totally singular subspace of $W$. By~\cite[\S 3.1]{Che1996}, there exists a unique one-dimensional subspace $P_M$ of $C^E$ such that $v^w=0$ for all $v\in P_M$ and $w\in M$. Any non-zero element of $P_M$ is called a \textit{representative  spinor} of $M$; for example, $1$ is a representative spinor of $\langle f_1,\ldots,f_r\rangle$. Any element of $C^E$ that is the representative spinor of some maximal totally singular subspace is called a \textit{pure spinor}. If $u$ is a representative spinor of the maximal totally singular subspace $M'$, then $M'=\{v\in W : u^v=0\}$ by~\cite[III.1.4]{Che1996}. Thus there is a one-to-one correspondence between the maximal totally singular subspaces of $W$ and the  one-dimensional subspaces of $C^E$ spanned by pure spinors.
 If $u$ is a representative spinor of $M$ and $s\in \Gamma$, then $u^s$ is a representative spinor of $s^{-1}Ms$ by~\cite[III.1.3]{Che1996}. The group $\Gamma_0\chi$ acts transitively on the maximal totally singular subspaces  by~\cite[III.2.7]{Che1996}, so $\Gamma_0$ acts transitively on the one-dimensional subspaces of $C^E$ spanned by pure spinors.  Since $m$ is even, any element $s$  of  $\Gamma$ is either even or odd by~\cite[II.3.2]{Che1996}, so  $1^s$ is either even or odd. Thus  a pure spinor is either even or odd, and  $\Gamma_0^+$ acts transitively on the one-dimensional subspaces of $C_+^E$ spanned by pure spinors. Further, since  $W$ has exactly $2(q+1)(q^2+1)\cdots (q^{r-1}+1)$ maximal totally singular subspaces,  
 there are exactly $(q-1)(q+1)(q^2+1)\cdots (q^{r-1}+1)$ pure spinors in $C_+^E$.  The following is~\cite[III.1.12]{Che1996}.

\begin{lemma}
\label{lemma:evenspinor}
Let $M$ and $M'$ be distinct maximal totally singular subspaces of $V_{2r}(q)$ with respect to a non-degenerate quadratic form with type $+$ where $r\geq 4$ and $q$ is a prime power. Let  $u$ and $u'$ be  representative spinors of $M$ and $M'$, respectively.
Then $u+u'$ is pure if and only if $\dim(M\cap M')=r-2$.
\end{lemma}
 
Observe that $\{e_S : S\in \mathbb{P}(\{1,\ldots,r\}),|S|\equiv 0\mod 2\}$ is a basis for $C^E_+$ consisting entirely of pure spinors. Indeed, $e_S$ is a representative spinor for the maximal totally singular subspace  spanned by $e_i$ for $i\in S$ and $f_j$ for $j\in \{1,\ldots,r\}\setminus S$. 

\begin{lemma}
\label{lemma:eventrick}
Let $Q$ be a non-degenerate quadratic form on $W:=V_{2r}(q)$ with type $+$ where $r\geq 4$ and $q$ is a power of a prime $p$. Let $\{e_1,\ldots,e_r,f_1,\ldots,f_r\}$ be a standard basis of $W$, and let $E:=\langle e_1,\ldots,e_r\rangle$. Let $H\leq \spin_{2r}^+(q)$, and let $B$ be a  block of $H$ on $C^E_+$ such that $B\cup \{0\}$ is an $\mathbb{F}_p$-subspace of $C^E_+$. Let $x,y\in B$. Let $I:=\{S\in \mathbb{P}(\{1,\ldots,r\}): |S|\equiv 0\mod 2\}$ and write $y=\sum_{S\in I} \lambda_Se_{S}$ where $\lambda_S\in\mathbb{F}_q$ for $S\in I$.   Let $\ell,k\in \{1,\ldots,r\}$ where $\ell\neq k$, and let $I_{\ell,k}:=\{S \in I:\ell\in S,k\notin S\}$.  If $\lambda_T\neq 0$ for some $T\in I_{\ell,k}$ and  $1-f_\ell e_k\in H_x$,  then 
 $$\sum_{S\in I_{\ell,k}} (-1)^{i_S}\lambda_S e_{(S\setminus \{\ell\})\cup \{k\}}\in B$$ 
 for some $i_S\in \{0,1\}$.
\end{lemma}

\begin{proof}
Let $f^*:=f_1\cdots f_r$ and $s:=1-f_\ell e_k$. Now $f^*y^s=f^*y-\sum_{S\in I_{\ell,k}}\lambda_S f^*e_Sf_\ell e_k$ since $f^*e_Sf_\ell e_k=0$ for $S\in I\setminus I_{\ell,k}$. Further, for $S\in I_{\ell,k}$, $f^*e_Sf_\ell e_k=f^*(-1)^{i_S}e_{(S\setminus \{\ell\})\cup \{k\}}$ for some $i_S\in \{0,1\}$ since for any $A\in \mathbb{P}(\{1,\ldots,r\})$ with $\ell\notin A$, $f^*e_Ae_\ell f_\ell =f^*e_A(1-f_\ell e_\ell)=f^*e_A$. Let $w:=\sum_{S\in I_{\ell,k}} (-1)^{i_S}\lambda_S e_{(S\setminus \{\ell\})\cup \{k\}}\neq 0$. Now $x^s=x$, so $y-w=y^s\in B$. Thus $w\in B$.
\end{proof}

Now we consider class (R5) where $(m,n,\varepsilon)=(10,16,+)$. 

\begin{prop}
\label{prop:(R5)10}
Let $V$ be the spin module of $\spin_{10}^+(q)$ where $q$ is a prime power. Let $G$ be an affine permutation group of rank~$3$ on $V$ where $\spin_{10}^+(q)\unlhd G_0$. Then there is no $\mathbb{F}_q$-independent $G$-affine proper partial linear space. 
\end{prop}

\begin{proof}
Let $Q$ be a non-degenerate quadratic form on $W:=V_{10}(q)$ with type $+$ and standard basis $\{e_1,\ldots,e_5,f_1,\ldots,f_5\}$. Let $E:=\langle e_1,\ldots,e_5\rangle$ so that $V=C_+^E$,  and let $Z:=Z(\GL(C_+^E))$.   Recall that $\spin_{10}^+(q)/(\spin_{10}^+(q)\cap Z)\simeq \POmega_{10}^+(q)$ and
 that $\spin_{10}^+(q)Z$ acts transitively on the set of pure spinors in $C_+^E$. By~\cite[Lemma 2.9]{Lie1987},  $\spin_{10}^+(q)Z$ has two orbits on $(C_+^E)^*$. Thus one of the orbits of $G_0$  is the set of pure spinors.  Let $q=p^e$ where $p$ is prime.

Suppose for a contradiction that  $(C_+^E,\mathcal{L})$ is an $\mathbb{F}_q$-independent  $G$-affine proper partial linear space. Let $L\in \mathcal{L}_0$, let $B:=L^*$, and let $x\in B$.  By Lemma~\ref{lemma:necessary}, $B$ is a  block of $G_0$ on $x^{G_0}$, and   $L$ is an $\mathbb{F}_p$-subspace of $C_+^E$  by Lemma~\ref{lemma:affine}(i)  since $-1\in \spin_{10}^+(q)$. Let $I:=\{S\in \mathbb{P}(\{1,\ldots,5\}): |S|\equiv 0\mod 2\}$, and for distinct integers $\ell,k\in \{1,2,3,4,5\}$, let $I_{\ell,k}:=\{S \in I:\ell\in S,k\notin S\}$. By assumption, there exists $y\in B\setminus \langle x\rangle$. Write $y=\sum_{S\in I} \lambda_Se_{S}$ where $\lambda_S\in\mathbb{F}_q$ for $S\in I$.  

First suppose that $x=1$.   Since $y\notin\langle 1\rangle$, there exists  $T\in I\setminus \{\varnothing\}$ such that $\lambda_T\neq 0$. Since $T$ is a non-empty proper subset of $\{1,\ldots,5\}$, we may choose $k\in \{1,\ldots,5\}\setminus T$ and $\ell\in T$. Now $T\in I_{\ell,k}$ and $1-f_\ell e_k\in \spin_{10}^+(q)_x$ by Lemma~\ref{lemma:spinelements}, so Lemma~\ref{lemma:eventrick} implies that $w:=\sum_{S\in I_{\ell,k}} (-1)^{i_S}\lambda_S e_{(S\setminus \{\ell\})\cup \{k\}}\in B$ for some $i_S\in \{0,1\}$. Thus $w=\sum_{j\neq \ell,k} \mu_j e_{\{j,k\}} +\mu e_{A}$ for some $\mu_j,\mu\in \mathbb{F}_q$ where $A:=\{1,\ldots,5\}\setminus \{\ell\}$. First suppose that $w=\mu e_A$. Then $1+\mu e_A\in B$. However,  $1+\mu e_A$ is not pure by Lemma~\ref{lemma:evenspinor}, a contradiction. Thus $\mu_j\neq 0$ for some $j$  (where $j\neq \ell,k$). Choose $i\in \{1,\ldots,5\}\setminus \{\ell,k,j\}$, and note that $i\in A$. Let $\ell^*:=j$ and $k^*:=i$. Now $1-f_{\ell^*} e_{k^*}\in \spin_{10}^+(q)_x$ by Lemma~\ref{lemma:spinelements}, and $\{j,k\}$ is the only element of $I_{\ell^*,k^*}$ in  the decomposition of $w$, so by Lemma~\ref{lemma:eventrick}, $\langle e_{\{i,k\}}\rangle \cap B\neq\varnothing$. Without loss of generality, we may assume that $\{i,k\}=\{1,2\}$. Since $1-f_1e_3\in \spin_{10}^+(q)_x$  by Lemma~\ref{lemma:spinelements}, $\langle e_{\{2,3\}}\rangle \cap B\neq\varnothing$ by Lemma~\ref{lemma:eventrick}. By a similar argument, $\langle e_{\{3,4\}}\rangle \cap B\neq\varnothing$. Thus $\delta_1e_{\{1,2\}}+\delta_2e_{\{3,4\}}\in B$ for some $\delta_1,\delta_2\in \mathbb{F}_q^*$, contradicting Lemma~\ref{lemma:evenspinor}.

Thus $x$ is not a pure spinor, in which case we may assume that $x=1+e_{\{1,2,3,4\}}$ by Lemma~\ref{lemma:evenspinor}. Let $A:=\{1,2,3,4\}$. Suppose that there exists $T\in I$ such that $\lambda_T\neq 0$ and $5\in T$. Choose $k\in A\setminus T$. Now $T\in I_{5,k}$ and $1-f_5 e_k\in \spin_{10}^+(q)_x$ by Lemma~\ref{lemma:spinelements}, so Lemma~\ref{lemma:eventrick} implies that $w:=\sum_{S\in I_{5,k}} (-1)^{i_S}\lambda_S e_{(S\setminus \{5\})\cup \{k\}}\in B$ for some $i_S\in \{0,1\}$. Now $w=\sum_{j\neq 5,k} \mu_j e_{\{j,k\}} +\mu e_{A}$ for some $\mu_j,\mu\in \mathbb{F}_q$. Since $w$ is not pure, $\mu_j\neq 0$ for some $j$  (where $j\neq 5,k$). Choose $i\in A\setminus \{k,j\}$. Let $\ell^*:=j$ and $k^*:=i$. Now $1-f_{\ell^*} e_{k^*}\in \spin_{10}^+(q)_x$ by Lemma~\ref{lemma:spinelements}, and
$\{j,k\}$ is the only element of $I_{\ell^*,k^*}$ in  the decomposition of $w$, so by Lemma~\ref{lemma:eventrick}, $\langle e_{\{i,k\}}\rangle \cap B\neq\varnothing$, a contradiction. Thus $5\notin S$ for all $S\in I$ such that $\lambda_S\neq 0$. Suppose that $\lambda_{\{1,2\}}\neq 0$.  Now $1-f_1 e_3\in \spin_{10}^+(q)_x$ by Lemma~\ref{lemma:spinelements}, so $\langle e_{2,3},e_{3,4}\rangle\cap B\neq \varnothing$ by Lemma~\ref{lemma:eventrick}, contradicting Lemma~\ref{lemma:evenspinor}. It follows that $y=\lambda 1+\mu e_A$ for some $\lambda,\mu\in \mathbb{F}_q^*$ such that $\lambda\neq\mu$. Let $u:=e_2-f_4$ and $v:=f_3-e_1$, and observe that $s:=1-uv\in \spin_{10}^+(q)$ by Lemma~\ref{lemma:spinelements}. Now $1^s=1-e_{\{1,2\}}$ and $e_A^s=e_A+e_{\{1,2\}}$, so $x^s=x$. Thus $(\mu-\lambda)e_{\{1,2\}}= y^s-y\in B$, a contradiction.
\end{proof}

 We continue to assume that $W=V_m(q)$ where  $m=2r\geq 8$ and that $Q$ is a non-degenerate form on $W$ of type $\varepsilon=+$ with associated bilinear form $\f$. As above, $C=C(W,Q)$, $\Gamma=\Gamma(C)$, and 
 $W$ has a  standard basis $\{e_1,\ldots,e_r,f_1,\ldots,f_r\}$. 

First suppose that $q$ is even, and define $\spin_{2r-1}(q):=\{s\in \Gamma_0^+ : s^{-1}\langle e_r+f_r\rangle s=\langle e_r+f_r\rangle\}$. By~\cite[p.197]{KleLie1990}, $\spin_{2r-1}(q)\simeq \Sp_{2r-2}(q)$, and  the \textit{spin module} of $\Sp_{2r-2}(q)$ is defined to be $C_+^E$ restricted to $\spin_{2r-1}(q)$. 

Suppose instead that $q$ is odd. Let $\overline{W}$ be the subspace of $W$ spanned by $e_r+f_r$ and $e_1,\ldots,e_{r-1},f_1,\ldots,f_{r-1}$. Observe that the restriction of $f$ to $\overline{W}$ is non-degenerate. Let $\overline{C}:=C(\overline{W},Q)$ and $\overline{\Gamma}:=\Gamma(\overline{C})$. We may view $\overline{C}$ as the subalgebra of $C$ generated by $\overline{W}$ by~\cite[II.1.4]{Che1996}, in which case $\overline{C}_+=\overline{C}\cap C_+$. Further, $\spin_{2r-1}(q)= \overline{\Gamma}_0^+=\overline{\Gamma}^+\cap \Gamma_0^+$ by~\cite[II.6.1]{Che1996}.  The  \textit{spin module} of $\spin_{2r-1}(q)$  is defined to be the restriction of $C_+^E$ to $\overline{\Gamma}_0^+$ (see~\cite[II.5.1 and \S 3.8]{Che1996}). 

Now we consider class (R5) where $(m,n,\varepsilon)=(7,8,\circ)$. Note that $\spin_{7}(q)$ is quasisimple with centre $\spin_7(q)\cap Z=\langle -1\rangle$ where $Z:=Z(\GL(C_+^E))$ since $\Sp_6(q)$ ($q$ even) and $\Omega_7(q)$ ($q$ odd) are simple.

\begin{prop}
\label{prop:(R5)7}
Let $V$ be the spin module of $\spin_7(q)$ where $q$ is a prime power. Let $G$ be an affine permutation group of rank~$3$ on $V$  where $\spin_{7}(q)\unlhd G_0$. Then there is no $\mathbb{F}_q$-independent $G$-affine proper partial linear space. 
\end{prop}

\begin{proof}
Let $Q$ be a non-degenerate quadratic form on $W:=V_{8}(q)$ with type $+$ and  standard basis $\{e_1,\ldots,e_4,f_1,\ldots,f_4\}$. Let $E:=\langle e_1,\ldots,e_4\rangle$, so that $V=C_+^E$, and let $Z:=Z(\GL(C_+^E))$.   
    Recall that $\spin_7(q)/(\spin_7(q)\cap Z)\simeq \Omega_7(q)$. Recall also that $\spin_7(q)Z$ acts on the set $P$ of pure spinors in $C_+^E$ and that $|P|=(q^3+1)(q^4-1)$.  If $q$ is even, then by~\cite[Lemma 2.8]{Lie1987}, $\spin_7(q)Z$ has two orbits on $(C_+^E)^*$ with sizes $(q^3+1)(q^4-1)$ and $q^3(q-1)(q^4-1)$, so $P$ is an orbit of $\spin_7(q)Z$. If $q$ is odd, then by~\cite[p.494]{Lie1987}, $\spin_7(q)Z$ has three orbits on $(C_+^E)^*$, one with size $(q^3+1)(q^4-1)$ and two with size $q^3(q-1)(q^4-1)/2$, so $P$ is an orbit of $\spin_7(q)Z$. Thus one of the orbits of $G_0$  is the set of pure spinors. Let $q=p^e$ where $p$ is prime.

Suppose for a contradiction that  $(C_+^E,\mathcal{L})$ is an $\mathbb{F}_q$-independent  $G$-affine proper partial linear space. Let $L\in \mathcal{L}_0$, let $B:=L^*$, and let $x\in B$.  By Lemma~\ref{lemma:necessary}, $B$ is a  block of $G_0$ on $x^{G_0}$, and by Lemma~\ref{lemma:affine}, $L$ is an $\mathbb{F}_p$-subspace of $C_+^E$. Let $A:=\{1,2,3,4\}$ and $I:=\{S\in\mathbb{P}(A):|S|\equiv 0\mod 2\}$. For distinct integers $\ell,k\in A$, let $I_{\ell,k}:=\{S \in I:\ell\in S,k\notin S\}$.  Note that $1-f_\ell e_k\in \spin_7(q)$ for distinct $k,\ell\in \{1,2,3\}$ by Lemma~\ref{lemma:spinelements}. By assumption, there exists $y\in B\setminus \langle x\rangle$. Write $y=\sum_{S\in I} \lambda_Se_{S}$ where $\lambda_S\in\mathbb{F}_q$ for $S\in I$. 

First suppose that $x=1$.  Then  $1-f_\ell e_k$ fixes $x$ for all $k,\ell\in \{1,2,3\}$.  We claim that $\langle e_{\{i,j\}}\rangle\cap B$ is empty for all $i,j\in\{1,2,3\}$ such that $i\neq j$. Suppose otherwise. Then $\lambda e_{\{1,2\}},\mu e_{\{1,3\}}\in B$ for some  $\lambda,\mu\in\mathbb{F}_q^*$ by Lemma~\ref{lemma:eventrick}.  Now $s:=1-f_1(e_4+f_4)\in \spin_7(q)_x$ by Lemma~\ref{lemma:spinelements}, so $\mu(e_{\{1,3\}}+e_{\{3,4\}})=(\mu e_{\{1,3\}})^{s}\in B$. But then $\lambda e_{\{1,2\}}+\mu e_{\{3,4\}}\in B$, contradicting Lemma~\ref{lemma:evenspinor}. Thus the claim holds.   If $\lambda_T=0$ for all $T\in I$ such that $|T|=2$, then   $y=\lambda_Ae_A$ by Lemma~\ref{lemma:evenspinor}. However, this implies that $1+\lambda_A e_A\in B$, contradicting Lemma~\ref{lemma:evenspinor}. Thus $\lambda_T\neq 0$ for some $T\in I$ such that $|T|=2$. Choose $\ell\in T\setminus \{ 4\}$ and $k\in \{1,2,3\}\setminus T$. By Lemma~\ref{lemma:eventrick}, $\lambda e_{\{i,k\}}+\mu e_{\{k,4\}}\in B$ for some $\lambda,\mu\in\mathbb{F}_q$ where $\{i,k,\ell\}=\{1,2,3\}$.
If $\lambda\neq 0$, then $\langle e_{\{\ell,k\}}\rangle \cap B\neq\varnothing$ by Lemma~\ref{lemma:eventrick}, a contradiction. Thus $\mu e_{\{k,4\}}\in B$. Let $s:=1-e_ke_\ell\in \spin_7(q)$ by Lemma~\ref{lemma:spinelements}. Now $s$ fixes $\mu e_{\{k,4\}}$, so $1^s\in B$, but then $e_{\{k,\ell\}}\in B$, a contradiction.

Thus $x$ is not a pure spinor, in which case we may assume that $x=1+e_A$ by Lemma~\ref{lemma:evenspinor}.  Suppose that $\lambda_T\neq 0$ for some $T\in I$ such that $|T|=2$. Choose $\ell\in T\setminus \{ 4\}$ and $k\in \{1,2,3\}\setminus T$. Since $x^{1-f_\ell e_k}=x$, Lemma~\ref{lemma:eventrick} implies that  $\langle e_{\{i,k\}},e_{\{k,4\}}\rangle\cap B\neq \varnothing$ where $\{i,k,\ell\}=\{1,2,3\}$,
 contradicting Lemma~\ref{lemma:evenspinor}. Thus $y=\lambda 1+\mu e_A$ for some $\lambda,\mu\in \mathbb{F}_q^*$ such that $\lambda\neq \mu$. Let $u:=e_2-f_2-e_4-f_4$ and $v:=f_3-e_1$, and observe that $s:=1-uv\in \spin_7(q)$ by Lemma~\ref{lemma:spinelements}. Now $1^s=1-e_{\{1,2\}}+e_{\{1,4\}}$ and $e_A^s=e_A+e_{\{1,2\}}-e_{\{1,4\}}$, so $x^s=x$. Thus $(\mu-\lambda)(e_{\{1,2\}}-e_{\{1,4\}})= y^s-y\in B$,  contradicting Lemma~\ref{lemma:evenspinor}.
\end{proof}

\section{Tensor product class (T)}
\label{s:(T)}

Let $V:=U\otimes W$ where $U:=V_2(q)$ and $W:=V_m(q)$ with $q$ a prime power and $m\geq 2$. Recall from~\S\ref{ss:basicsvs} that for $g\in  \GL(U)$ and $h\in  \GL(W)$, the linear map $g\otimes h$ maps  $u\otimes w$ to $u^g\otimes w^h$ for all $u\in U$ and $w\in W$. Recall also that for $S\leq \GL(U)$ and $T\leq \GL(W)$, the group $S\otimes T=\{g\otimes h : g\in S,h\in T\}\leq \GL(V)$, and $\Aut(\mathbb{F}_q)$ acts naturally on $V$, $U$ and $W$ in such a way that $(u\otimes w)^\sigma = u^\sigma \otimes w^\sigma$ for all $u\in U$, $w\in W$ and $\sigma\in\Aut(\mathbb{F}_q)$. Either $m\geq 3$ and  the stabiliser of the decomposition of $V$ is
$ (\GL(U)\otimes \GL(W)){:}\Aut(\mathbb{F}_{q})$, or $m=2$, in which case we may assume that $U=W$, and  the stabiliser of the decomposition of $V$ is $(\GL(U)\otimes \GL(U)){:}(\Aut(\mathbb{F}_{q})\times \langle\tau\rangle)$ where $\tau\in \GL(V)$ is defined by $u\otimes w\mapsto w\otimes u$ for all $u,w\in U$. We adopt the convention of writing elements of $(\GL(U)\otimes \GL(W)){:}\Aut(\mathbb{F}_{q})$ as $g\otimes h$ for $g\in \GammaL(U)$ and  $h\in \GammaL(W)$ where $g$ and $h$ are $\sigma$-semilinear for the same $\sigma\in \Aut(\mathbb{F}_q)$. Let $G$ be an affine  permutation group of rank~$3$ on $V$ where  $G_0$ stabilises the decomposition of $V$. Recall that the orbits of $G_0$ on $V^*$ are the \textit{simple} tensors $\{u\otimes w: u\in U^*, w\in W^*\}$ and the \textit{non-simple} tensors  $\{u_1\otimes w_1+u_2\otimes w_2 : w_1,w_2\in W, \ \dim\langle w_1,w_2\rangle =2\}$ for any choice of basis $\{u_1,u_2\}$ of $U$. 

This section is devoted to the proof of the following result, which deals with the classes (T1), (T2) and (T3). However, in the process of proving Proposition~\ref{prop:(T1)--(T3)}, we will often prove results for more general families of groups in class (T) so that these results can be used for classes (S0), (S1) and (S2). 

\begin{prop}
\label{prop:(T1)--(T3)}
Let $G$ be an affine permutation group of rank~$3$ on $V:=U\otimes W$ where $U:=V_2(q)$, $W:=V_m(q)$, $m\geq 2$, $q$ is a prime power and $G_0$   belongs to one of the classes~\emph{(T1)},~\emph{(T2)} or~\emph{(T3)}. Then the following are equivalent. 
\begin{itemize}
\item[(i)] $(V,\mathcal{L})$ is an $\mathbb{F}_q$-independent $G$-affine proper partial linear space.
\item[(ii)]  $\mathcal{L}$ is  $\{U\otimes w+v: w \in W^*,v\in V\}$ or $\{u\otimes W+v: u \in U^*,v\in V\}$, and either $m\geq 3$, or $m=2$ and $G_0\leq (\GL(U)\otimes  \GL(W)){:}\Aut(\mathbb{F}_{q})$.
\end{itemize}
\end{prop}

Observe that the partial linear spaces of  Proposition~\ref{prop:(T1)--(T3)}(ii) are all described in Example~\ref{example:tensor}, though the case $q=2$ is omitted from Proposition~\ref{prop:(T1)--(T3)} by our restrictions on $G_0$. In Lemma~\ref{lemma:tensorsimple} below, we prove that the incidence structures $\mathcal{S}_U$ and $\mathcal{S}_W$ of Example~\ref{example:tensor} are indeed $G$-affine proper partial linear spaces for all $m$ and $q$, where $G:=V{:}(\GL(U)\otimes \GL(W)){:}\Aut(\mathbb{F}_q)$.  We will prove at the end of this section that $G=\Aut(\mathcal{S}_U)=\Aut(\mathcal{S}_W)$.

In order to prove Proposition~\ref{prop:(T1)--(T3)}, we first consider the orbit of simple tensors.

\begin{lemma}
\label{lemma:tensorsimple}
Let $G$ be an affine permutation group of rank~$3$ on $V:=U\otimes W$ where $U:=V_2(q)$, $W:=V_m(q)$, $m\geq 2$, $q$ is a prime power and $G_0$ stabilises the decomposition of $V$. 
\begin{itemize}
\item[(i)] Let $\mathcal{S}:=(V,\mathcal{L})$ be a  $G$-affine proper partial linear space in which $\mathcal{S}(0)$ is the set of simple tensors. Let $L\in \mathcal{L}_0$ and $u\otimes w\in L^*$. Then $L\subseteq U\otimes w$ or $L\subseteq u\otimes W$. Moreover, if $S\otimes T\leq G_0$ for some $S\leq \GL(U)$ and $T\leq \GL(W)$, then $L^*=C\otimes w$ or $u\otimes C$,  where $C$ is  a block of $S$ on $U^*$ or $T$ on $W^*$, respectively.
\item[(ii)] Let $H_0\leq G_0$. Then for any $u\in U^*$ and $w\in W^*$, both $(U\otimes w)^*$ and $(u\otimes W)^*$ are blocks of $H_0$ in its action on the simple tensors unless $m=2$ and $H_0$ contains an element interchanging $U$ and $W$, in which case neither are blocks.
\item[(iii)] Define $\mathcal{L}$ to be $\{U\otimes w+v: w \in W^*,v\in V\}$ or $\{u\otimes W+v: u\in U^*,v\in V\}$. Then $(V,\mathcal{L})$ is  an $\mathbb{F}_q$-independent  $G$-affine proper partial linear space if and only if either $m\geq 3$, or $m=2$ and $G_0\leq (\GL(U)\otimes  \GL(W)){:}\Aut(\mathbb{F}_{q})$. 
\item[(iv)] If $m=2$, then the two partial linear spaces described in (iii) are isomorphic. 
\end{itemize}
 \end{lemma}

\begin{proof}
(i) Suppose that  there exists $u_1\otimes w_1\in L\setminus U\otimes w$ and $u_2\otimes w_2\in L\setminus u\otimes W$. Then $\{w_1,w\}$ and $\{u_2,u\}$ are both linearly independent sets, and $u_1\otimes w_1-u\otimes w$ and $u_2\otimes w_2-u\otimes w$ are both simple by Lemma~\ref{lemma:basic}(i), so $u_1=\lambda u$ and $w_2=\mu w$ for some $\lambda,\mu\in \mathbb{F}_q^*$. Thus $\lambda u\otimes w_1- u_2\otimes \mu w$ is a non-simple tensor, contradicting Lemma~\ref{lemma:basic}(i). Hence $L\subseteq U\otimes w$ or $L\subseteq u\otimes W$.

Moreover, suppose that $S\otimes T\leq G_0$ and $L\subseteq U\otimes w$. Then  $L^*=C\otimes w$   for some $C\subseteq U^*$. Suppose that $v'=v^g$ for some $v',v\in C$ and $g\in S$. Then $v'\otimes w=(v\otimes w)^{g\otimes 1}$, so by Lemma~\ref{lemma:necessary},  $C\otimes w=(C\otimes w)^{g\otimes 1}=C^g\otimes w$. Thus $C=C^g$. The case where $L\subseteq u\otimes W$ is similar. 

(ii) First suppose that $m=2$ and $H_0$ contains an element $\nu$ interchanging $U$ and $W$, where we may assume that $W=U$. If $(U\otimes u)^*$ is a block for some $u\in U^*$, then  $(U\otimes u)^{\nu}=u'\otimes U$ is also a block for some $u'\in U^*$, but $u'\otimes u\in (U\otimes u)^* \cap (u'\otimes U)^*$, a contradiction. Similarly, $(u\otimes U)^*$ is not a block for any $u\in U^*$. 

Otherwise,  $H_0\leq (\GL(U)\otimes \GL(W)){:}\Aut(\mathbb{F}_{q})$. Let $w\in W^*$. Suppose that $u'\otimes w=(u\otimes w)^{g\otimes h}$ for some $u\otimes w,u'\otimes w\in (U\otimes w)^*$ and $g\otimes h\in H_0$. Then $u'=\lambda u^{g}$ and $w=\lambda^{-1} w^{h}$ for some $\lambda\in\mathbb{F}_q^*$. Hence $(U\otimes w)^{g\otimes h}=U^{g}\otimes \lambda w = U\otimes w$, so $(U\otimes w)^*$ is a block. The proof for $(u\otimes W)^*$ is similar.

(iii) This follows from (ii) and Lemmas~\ref{lemma:transitive},~\ref{lemma:necessary} and~\ref{lemma:sufficient}.

(iv) Suppose that $m=2$. Let $\tau\in \GL(V)$ be defined by $u\otimes w\mapsto w\otimes u$ for all $u,w\in U=W$. Now $\tau$ interchanges the line sets of the two partial linear spaces of (iii), so these two partial linear spaces are isomorphic. 
\end{proof}

Note that in (i), we are not necessarily assuming  that $S$ is transitive on $U^*$ nor that $T$ is transitive on $W^*$. Similarly,  we are not necessarily assuming in (ii) that $H_0$ is transitive on the set of simple tensors.

\begin{lemma}
\label{lemma:tensorsimpleSL}
Let $G$ be an affine permutation group of rank~$3$ on $V:=U\otimes W$ where $U:=V_2(q)$, $W:=V_m(q)$, $m\geq 2$, $q$ is a power of a prime $p$ and $G_0$ stabilises the decomposition of $V$.  Let $\mathcal{S}:=(V,\mathcal{L})$ be an $\mathbb{F}_q$-independent  $G$-affine proper partial linear space in which $\mathcal{S}(0)$ is the set of simple tensors. Let $L\in\mathcal{L}_0$.
\begin{itemize}
\item[(i)] If $\SL(U)\otimes 1\leq G_0$ and  $L\subseteq U\otimes w$ for some $w\in W^*$, then  $\mathcal{L}_0=\{U\otimes v: v \in W^*\}$.
\item[(ii)] If $1\otimes \SL(W)\leq G_0$ and  $L\subseteq u\otimes W$ for some $u\in U^*$, then  $\mathcal{L}_0=\{v\otimes W: v \in U^*\}$.
\item[(iii)] If   $q=p$ and  $L\subseteq U\otimes w$ for some $w\in W^*$, and if $-1\otimes g\in G_0$ for some $g\in \GL(W)_w$, then  $\mathcal{L}_0=\{U\otimes v: v \in W^*\}$.
\end{itemize}
\end{lemma}

\begin{proof} 
If $L=U\otimes w$ or $u\otimes W$, then $\mathcal{L}_0=\{U\otimes v: v \in W^*\}$ or $\{v\otimes W: v \in U^*\}$ by Lemmas~\ref{lemma:necessary} and~\ref{lemma:tensorsimple}(ii). It is routine to verify that if $C$ is a block of $\SL_n(q)$ on $V_n(q)^*$ where $n\geq 2$ and $x\in C$, then either $C\subseteq \langle x\rangle_{\mathbb{F}_q}$, or $C=V_n(q)^*$. Thus  (i) and (ii) follow from Lemma~\ref{lemma:tensorsimple}(i). For (iii), $L$ is an $\mathbb{F}_p$-subspace of $U\otimes w$ by Lemma~\ref{lemma:affine}, and $|U\otimes w|=p^2$, so $L=U\otimes w$.
\end{proof}

\begin{lemma}
\label{lemma:T34simple}
Let $G$ be an affine permutation group of rank~$3$ on $V:=U\otimes W$ where $U:=V_2(q)$, $W:=V_m(q)$, $m\geq 2$, $q$ is a prime power and $G_0$   belongs to one of the classes~\emph{(T1)},~\emph{(T3)} or~\emph{(T4)}.
Let $\mathcal{S}:=(V,\mathcal{L})$ be an $\mathbb{F}_q$-independent  $G$-affine proper partial linear space in which $\mathcal{S}(0)$ is the set of simple tensors. Then $\mathcal{L}_0$ is $\{U\otimes w: w \in W^*\}$ or $\{u\otimes W: u\in U^*\}$. 
 \end{lemma}
 
 \begin{proof}
Let $L\in \mathcal{L}_0$, let $B:=L^*$, and let $u\otimes w\in B$. By  Lemma~\ref{lemma:tensorsimple}(i), either $L\subseteq U\otimes w$, or  $L\subseteq u\otimes W$. Hence we are done when $G_0$ lies in (T1) by Lemma~\ref{lemma:tensorsimpleSL}(i)--(ii), so we assume that $G_0$ lies in (T3) or (T4). Now $(U\otimes w)^*$ and $(u\otimes W)^*$ are blocks of $G_0$ on $\mathcal{S}(0)$ by Lemma~\ref{lemma:tensorsimple}(ii). If $G_0$ lies in (T4), then we use {\sc Magma} to determine that  there are exactly two non-trivial blocks of $G_0$  on $\mathcal{S}(0)$ that contain $u\otimes w$, so  the desired result follows. Suppose that $G_0$ lies in  (T3), and   recall that $q$ is prime and $-1\otimes g\in G_0$ for some $g\in \GL_m(q)$. If $L\subseteq u\otimes W$, then  $\mathcal{L}_0=\{v\otimes W: v \in U^*\}$ by Lemma~\ref{lemma:tensorsimpleSL}(ii). Otherwise, $L\subseteq U\otimes w$.  Now there exists $h\in \SL(W)$ such that $(w^g)^h=w$, and $-1\otimes gh \in G_0$, so $\mathcal{L}_0=\{U\otimes v: v \in W^*\}$ by Lemma~\ref{lemma:tensorsimpleSL}(iii).
 \end{proof}

\begin{lemma}
\label{lemma:T2simple}
Let $G$ be an affine permutation group of rank~$3$ on $V:=U\otimes W$ where $U:=V_2(q)$, $W:=V_m(q)$, $m\geq 2$, $q$ is a prime power and $G_0$  belongs to~\emph{(T2)}. Let $\mathcal{S}:=(V,\mathcal{L})$ be an $\mathbb{F}_q$-independent  $G$-affine proper partial linear space in which $\mathcal{S}(0)$ is the set of simple tensors. Then $\mathcal{L}_0$ is $\{U\otimes w: w \in W^*\}$ or $\{u\otimes W: u\in U^*\}$. 
 \end{lemma}

\begin{proof}
Let $L\in \mathcal{L}_0$ and $B:=L^*$. If $L\subseteq u\otimes W$ for some $u\in U^*$, then  $\mathcal{L}_0=\{v\otimes W: v \in U^*\}$ by Lemma~\ref{lemma:tensorsimpleSL}(ii), as desired. Otherwise, by  Lemma~\ref{lemma:tensorsimple}(i), $L\subseteq U\otimes w$ for some $w\in W^*$, and $B=C\otimes w$, where $C$ is a block of $\SL_2(5)$ on $U^*$. If $C=U^*$, then $\mathcal{L}_0=\{U\otimes v: v \in W^*\}$, as desired, so we   suppose for a contradiction that $C\neq U^*$. Let $U=\langle u_1,u_2\rangle$ and $x:=u_1$. We may assume that $x\otimes w\in B$, so that $x\in C$. Since $\SL_2(5)$ is an irreducible subgroup of $\GL_2(q)$, its central involution must be the central involution of $\GL_2(q)$. Thus  $q=9$ by Lemma~\ref{lemma:tensorsimpleSL}(iii) and the definition of (T2), and $L$ is an $\mathbb{F}_3$-subspace of $U\otimes w$ by Lemma~\ref{lemma:affine}.

View  $\SL_2(3)$ as the subgroup of  $\SL_2(9)$ that acts naturally on the basis $\{u_1,u_2\}$ over $\mathbb{F}_3$. Assume that $\zeta^2=\zeta+1$ where $\zeta:=\zeta_9$ (see~\S\ref{ss:basicsvs}). There exists $s\in \GL_2(9)$ such that $u_1^s=\zeta^2 u_1$ and $u_2^s=-u_1+\zeta^6 u_2$. Let $S:=\langle \SL_2(3),s\rangle$. Now $S\simeq \SL_2(5)$ (this can be verified using {\sc Magma}), and $\GL_2(9)$ has a unique conjugacy class of subgroups isomorphic to $\SL_2(5)$, so we may assume that $S\otimes \SL(W)\unlhd G_0$ and that $C$ is a block of $S$.
 The group $S$ has $2$ orbits on $U^*$, each with size $40$, and $S$ acts transitively on the set of $1$-spaces of $U$. Let $\Omega_1$ and $\Omega_2$ be the two orbits of $S$ where $x\in \Omega_1$, and note that $\Omega_2=\zeta \Omega_1$. Since $|\SL_2(3)_x|=3$  and $|S_x|=3$, we conclude that $S_x\leq \SL_2(3)$. Then $S_x$ has orbits of size $3$ on $U\setminus \langle x\rangle$, and no such orbit contains both $u$ and $-u$.

Observe that $|L|$ divides $81$, while $|L|-1=|B|$ divides $|(U\otimes w)^*|=80$  since $(U\otimes w)^*$ is a block of $S\otimes \SL(W)$   in its transitive action on the set of simple tensors by Lemma~\ref{lemma:tensorsimple}(ii). Since $\mathcal{S}$ is $\mathbb{F}_9$-independent,   it follows that $|C|=8$.  If $C$ is not contained in $\Omega_1$, then by the observations made above concerning the orbits of $S_x$, we must have $|C\cap \Omega_2|=6$.   However, $C\cap \Omega_2$ is a block of $S$ on $\Omega_2$, while $6$ does not divide $40$, a contradiction. Thus $C\subseteq \Omega_1$. In particular,   $C=x^{S_C}$ and $|S_C|=24$. Recall that  $S=\langle \SL_2(3),s\rangle$ where $u_1^s=\zeta^2 u_1$ and $u_2^s=-u_1+\zeta^6 u_2$. 
There are exactly two subgroups of $S$ of order $24$ that contain $S_x$, namely $\SL_2(3)$ and $\SL_2(3)^s$ since $s$ normalises $S_x$. Hence $C$ is either $x^{\SL_2(3)}=\langle u_1,u_2\rangle_{\mathbb{F}_3}^*$ or $$x^{\SL_2(3)^s}=\{\pm u_1\}\cup \{ \zeta^iu_1+u_2 : i\in \{3,5,6\}\}\cup \{\zeta^ju_1-u_2:  j\in \{1,2,7\} \}.$$

Let  $\lambda:=\zeta^2$ and $\mu:=\lambda^{-1}$. There exists $h\in \SL(W)$ such that $w^h=\mu w$. Note that $x^s=\lambda x$. Now $x\otimes w=(x\otimes w)^{s\otimes h}$, so  $C\otimes w=(C\otimes w)^{s\otimes h}=C^{s\mu}\otimes w$. Thus $C=C^{s\mu}$. However, if $C=x^{\SL_2(3)}$, then $u_2\in C$, so $\zeta^2u_1-u_2=u_2^{s\mu}\in C$, a contradiction. Similarly, if $C=x^{\SL_2(3)^s}$, then $\zeta^2u_1-u_2\in C$, so $u_2=(\zeta^2u_1-u_2)^{s\mu}\in C$, a contradiction.
\end{proof}

 Next we consider the orbit of non-simple tensors.
 
 \begin{lemma}
\label{lemma:nonsimple}
Let $G$ be an affine permutation group of rank~$3$ on $V:=U\otimes W$ where $U:=V_2(q)$, $W:=V_m(q)$, $m\geq 2$, $q$ is a power of a prime $p$, $G_0$ stabilises the decomposition of $V$ and $1\otimes \SL(W)\unlhd G_0$. Let $\mathcal{S}:=(V,\mathcal{L})$ be a  $G$-affine proper partial linear space in which $\mathcal{S}(0)$ is the set of non-simple tensors. Let $L\in \mathcal{L}_0$ and $x\in L^*$. Write $x=u_1\otimes x_1+u_2\otimes x_2$ where $U=\langle u_1,u_2\rangle$ and $X:=\langle x_1,x_2\rangle \subseteq W$. Then $L$ is an $\mathbb{F}_p$-subspace of $U\otimes X$, and $|L|$ divides $q^2$. 
\end{lemma}

\begin{proof}
Suppose for a contradiction that there exists $y\in L\setminus (U\otimes X)$. Now $y=u_1\otimes y_1+u_2\otimes y_2$ for some $y_1,y_2\in W$, and without loss of generality, we may assume that $y_1\notin X$. We may therefore extend $\{x_1,x_2,y_1\}$ to a basis $\{x_1,x_2,y_1,v_4,\ldots,v_m\}$ of $W$. Now there exists $g\in \SL(W)_{x_1,x_2,v_4,\ldots,v_m}$ such that $y_1^g=x_1+y_1$, in which case $y_2^g-y_2=\lambda x_1$ for some $\lambda\in\mathbb{F}_q$. By Lemma~\ref{lemma:basic}, since $x^{1\otimes g}=x$, it follows that $y^{1\otimes g}\in L^*$  and therefore that the non-zero vector $y^{1\otimes g}-y=u_1\otimes x_1 + u_2\otimes \lambda x_1$ is not simple, a contradiction. Thus $L\subseteq u_1\otimes X+u_2\otimes X=U\otimes X$.  There exists $g\in \SL(W)$ such that $x_1^g=-x_1$ and $x_2^g=-x_2$, and $v^{1\otimes g}=-v$ for all $v\in L$, 
so $L$ is an $\mathbb{F}_p$-subspace of $U\otimes X$ by Lemma~\ref{lemma:affine}(i). If $u_1\otimes w_1+u_2\otimes w_2$ and $u_1\otimes w_1+u_2\otimes w_3$ are elements of $L$ for some $w_1,w_2,w_3\in X$, then $u_2\otimes (w_3-w_2)\in L$, so $w_2=w_3$. Thus $|L|\leq |X|=q^2$, and since $L$ is an $\mathbb{F}_p$-vector space, $|L|$ divides $q^2$. 
\end{proof}

Lemma~\ref{lemma:nonsimple} enables us to give a reduction to the case where $m=2$; we include groups satisfying (T5) for our analysis of class (S1) in \S\ref{s:(S1)}. 

\begin{lemma}
\label{lemma:tensorred}
Let $G$ be an affine permutation group of rank~$3$ on $V:=U\otimes W$ where $U:=V_2(q)=\langle u_1,u_2\rangle$, $W:=V_m(q)=\langle x_1,\ldots,x_m\rangle$, $m\geq 3$, $q$ is a prime power, $G_0$ stabilises the decomposition of $V$ and $1\otimes \SL(W)\unlhd G_0$. Let $X:=\langle x_1,x_2\rangle$, $Y:=U\otimes X$ and $H:=Y{:}G_{0,Y}^Y$. Then $H$ is an affine permutation group of rank~$3$ on $Y$ such that $H_0$ stabilises the decomposition of $Y$ and $1\otimes \GL(X)\unlhd H_0$. \begin{itemize}
\item[(i)] Let $\mathcal{S}:=(V,\mathcal{L})$ be an $\mathbb{F}_q$-independent  $G$-affine proper partial linear space for which $0$ and $u_1\otimes x_1+u_2\otimes x_2$ lie on a line $L$. Then $\mathcal{S}\cap Y=(Y,L^H)$ is an $\mathbb{F}_q$-independent  $H$-affine proper partial linear space. 
\item[(ii)] Let   $(Y,\mathcal{L})$ be an $\mathbb{F}_q$-independent  $H$-affine proper partial linear space for which $0$ and $u_1\otimes x_1+u_2\otimes x_2$ lie on a line $L$. Then $(V,L^G)$ is an $\mathbb{F}_q$-independent $G$-affine proper partial linear space.
\end{itemize}
Moreover, if  $G_0$ belongs to class~\emph{(Ti)} where $i\in \{1,2,3,5\}$, then  $H_0$ belongs to class~\emph{(Ti)}.
\end{lemma}

\begin{proof}
 Let $A:=(\GL(U)\otimes \GL(W)){:}\Aut(\mathbb{F}_q)$ where $\Aut(\mathbb{F}_q)$ fixes $x_1$ and $x_2$. Note that $G_0\leq A$ since $m\geq 3$.  Now $A_Y=(\GL(U)\otimes \GL(W)_X){:}\Aut(\mathbb{F}_q)$ and $A_{(Y)}=1\otimes \GL(W)_{(X)}$, so there is a natural permutation isomorphism  between $A_Y/A_{(Y)}=A_Y^Y$ and $(\GL(U)\otimes \GL(X)){:}\Aut(\mathbb{F}_q)$. In particular, $H_0$ stabilises the decomposition of $Y$ and $1\otimes \GL(X)\simeq (1\otimes \SL(W))_Y^Y\unlhd  H_0$. 
 
Clearly $H$ is an affine permutation group on $Y$. We claim that $H$ has rank~$3$ on $Y$.
Let $Y'$ be the set of non-simple tensors of $Y$. Since $1\otimes \GL(X)$ is faithful and semiregular on  $Y'$ and $|Y'|=q(q-1)(q^2-1)=|\GL(X)|$, it follows that $1\otimes \GL(X)$ is regular on  $Y'$. In particular, $H_0$ is transitive on $Y'$.  Let $u\otimes x,u'\otimes x'$ be simple tensors in $Y$. There exists $g\otimes h\in G_0$ such that $(u\otimes x)^{g\otimes h}=u'\otimes x'$. Now $u^g=\lambda u'$ and $x^h=\lambda^{-1}x'$ for some $\lambda\in\mathbb{F}_q^*$. Further, there exist $w,w'\in X$ such that $X=\langle x,w\rangle=\langle x',w'\rangle$, and since $m\geq 3$, there exists $k\in \SL(W)$ such that $(x')^k=x'$ and $(w^h)^k=w'$. Now  $g\otimes hk\in G_{0,Y}$ since  $X^{hk}=X$, and $(u\otimes x)^{g\otimes hk}= u'\otimes x'$,  so the claim holds.

Suppose that $G_0$ belongs to (Ti)  where $i\in \{1,2,3,5\}$. If $i=1$ or $2$, then $S\otimes 1\unlhd G_0$ where $S=\SL(U)$ or $\SL_2(5)$,  respectively, and $S\otimes 1\simeq (S\otimes 1)^Y_Y\unlhd H_0$, so $H$ belongs to (Ti). If $i=3$, then since $-1\otimes 1\in H_0$ and 
$K\otimes \GL(X)\simeq (K\otimes \GL(W))_Y^Y$ for any $K\leq\GL(U)$, it follows that $H$ belongs to (T3). Lastly, if $i=5$, then $H_0\leq ((\GL_1(q^2)\otimes \GL(W)){:}\langle (t\otimes 1)\sigma_q \rangle)_Y^Y\simeq (\GL_1(q^2)\otimes \GL(X)){:}\langle (t\otimes 1)\sigma_q \rangle$, so $H$ belongs to (T5).

(i)    Lemma~\ref{lemma:intersect} implies that $\mathcal{S}\cap Y$ is an $H$-affine partial linear space. Since $L\subseteq Y$ by Lemma~\ref{lemma:nonsimple}, $L$ is also a line of $\mathcal{S}\cap Y$, so $\mathcal{S}\cap Y$ is  $\mathbb{F}_q$-independent with line-size at least $3$, and $0$ and $u_1\otimes x_1$ are non-collinear points, so $\mathcal{S}\cap Y$ is a proper partial linear space. Also, $\mathcal{S}\cap Y=(Y,L^H)$.
 
 (ii) Let $B:=L^*$.  We claim that $G_L$ is transitive on $L$ and that $B$ is a  non-trivial block  of  $G_0$  on the  non-simple tensors of $V$,  for then $(V,\mathcal{L}^G)$ will be a (clearly $\mathbb{F}_q$-independent) $G$-affine proper partial linear space by Lemma~\ref{lemma:sufficient}.
 Since $-1\otimes 1\in H_0$, Lemmas~\ref{lemma:transitive} and~\ref{lemma:affine} imply that   $G_L$ is transitive on $L$. Suppose that  $(u_1\otimes w_1+u_2\otimes w_2)^{g\otimes h}=u_1\otimes w_1'+u_2\otimes w_2'$ for some $u_1\otimes w_1+u_2\otimes w_2,u_1\otimes w_1'+u_2\otimes w_2'\in B$ and $g\otimes h\in G_0$. Now $X^h=\langle w_1,w_2\rangle ^h=\langle w_1',w_2'\rangle=X$, so $g\otimes h\in G_{0,Y}$. By Lemma~\ref{lemma:necessary},  $B$ is a block of $H_0$ on the  non-simple tensors of $Y$, so $B^{g\otimes h}=B$.  Thus $B$ is a block of $G_0$ on the  non-simple tensors of $V$, and it is clearly non-trivial.
\end{proof}

\begin{lemma}
\label{lemma:T123nonsimple}
Let $G$ be an affine permutation group of rank~$3$ on $V:=U\otimes W$ where $U:=V_2(q)$, $W:=V_m(q)$, $m\geq 2$, $q$ is a prime power and  $G_0$  belongs to one of the classes~\emph{(T1)},~\emph{(T2)} or~\emph{(T3)}. Then there is no $\mathbb{F}_q$-independent $G$-affine proper partial linear space   in which  $0$ is collinear with a non-simple tensor.
\end{lemma}

\begin{proof}
By Lemma~\ref{lemma:tensorred}, we may assume that $m=2$. Let $\mathcal{S}:=(V,\mathcal{L})$ be an $\mathbb{F}_q$-independent $G$-affine proper partial linear space in which $\mathcal{S}(0)$ is the set of non-simple tensors. Let $q=p^e$ where $p$ is prime.
 Let $L\in\mathcal{L}_0$ and $B:=L^*$.   Now, by Lemma~\ref{lemma:nonsimple}, $L$ is an $\mathbb{F}_p$-subspace of $V$ and $|L|$ divides $q^2$. Further, $B$ is a  block of $G_0$ in its action on the  non-simple tensors by Lemma~\ref{lemma:necessary}. Using {\sc Magma}, we determine that there are no such blocks when any of the following cases hold: (T1) with $q=3$, (T2) or (T3).   Thus $G_0$ belongs to (T1)  and $q\geq 4$.

We may assume that   $U=W$. Now $S:=\SL_2(q)\otimes \SL_2(q)\unlhd G_0$. Let $x:=u_1\otimes u_1+u_2\otimes u_2\in B$, where $U=\langle u_1,u_2\rangle$. By assumption, there exists $y\in B\setminus \langle x\rangle$, and $y=u_1\otimes (\lambda_1 u_1+\lambda_2u_2)+u_2\otimes (\varepsilon_1u_1+\varepsilon_2u_2)$ for some $\lambda_1,\lambda_2,\varepsilon_1,\varepsilon_2\in\mathbb{F}_q$. For $g\in \SL_2(q)$, let $\tilde{g}$ denote the transpose of $g^{-1}$.  Now $S_x=\{g\otimes \tilde{g}: g \in \SL_2(q)\}$. Since $q\geq 4$, there exists $a\in \mathbb{F}_q^*\setminus \langle -1 \rangle$. There exist $g,h\in \SL_2(q)$ such that $u_1^g=au_1$,  $u_2^g=a^{-1}u_2$, $u_1^h=u_1$ and $u_2^h=u_1+u_2$. Then  $u_1^{\tilde{g}}=a^{-1}u_1$, $u_2^{\tilde{g}}=au_2$,   $u_1^{\tilde{h}}=u_1-u_2$ and $u_2^{\tilde{h}}=u_2$. Since $g\otimes \tilde{g}$ fixes $x$ and $L$ is an $\mathbb{F}_p$-subspace of $V$, Lemma~\ref{lemma:basic}(ii) implies that $z:=y^{g\otimes \tilde{g}}-y\in L$. Let $\lambda:=(a^2-1)\lambda_2$ and $\varepsilon:=(a^{-2}-1)\varepsilon_1$. Then $z=u_1\otimes \lambda u_2+u_2\otimes \varepsilon u_1$.  Suppose that $\varepsilon=0$. Then $\lambda=0$ since $B$ only contains non-simple tensors. Since $a\not\in \langle -1\rangle$, it follows that $\lambda_2=\varepsilon_1=0$, so $y=\lambda_1 u_1\otimes u_1 + \varepsilon_2 u_2\otimes u_2$. Now $(\varepsilon_2-\lambda_1)u_1\otimes u_2=y^{h\otimes \tilde{h}}-y\in L$ since $h\otimes \tilde{h}$ fixes $x$, so $\varepsilon_2=\lambda_1$, but then $y\in \langle x\rangle$, a contradiction. Thus  $\varepsilon\neq 0$. Let $w:=u_1\otimes \varepsilon(u_1-u_2)+u_2\otimes (-\varepsilon)u_2=z^{h\otimes \tilde{h}}-z\in B$. Then $u_1\otimes (1-a^2)\varepsilon u_2=w^{g\otimes \tilde{g}}-w\in B$, a contradiction.
\end{proof}

\begin{proof}[Proof of Proposition~\emph{\ref{prop:(T1)--(T3)}}]
If (i) holds, then $0$ is collinear with a simple tensor by Lemma~\ref{lemma:T123nonsimple}, so (ii) holds by Lemmas~\ref{lemma:T34simple},~\ref{lemma:T2simple} and~\ref{lemma:tensorsimple}(iii).   Conversely, if (ii) holds, then (i) holds by Lemma~\ref{lemma:tensorsimple}(iii). 
\end{proof}

To finish this section, we prove a result that enables us to determine the full automorphism group of a $G$-affine proper partial linear space when $G$ belongs to class (T). We then apply this result to the partial linear spaces of Example~\ref{example:tensor}. 

\begin{lemma}
\label{lemma:Taut}
Let $G$ be an affine permutation group of rank~$3$  on $V:=U\otimes W$ where $U:=V_2(q)$, $W:=V_m(q)$, $m\geq 2$, $q$ is a prime power and $G_0$ stabilises the decomposition of $V$. Let $\mathcal{S}$ be a $G$-affine proper partial linear space. Then $\Aut(\mathcal{S})$ is an affine permutation group on $V$ and $\Aut(\mathcal{S})_0$ stabilises the decomposition of $V$.
\end{lemma}

\begin{proof}
Let $\Gamma$ be the graph with vertex set $V$ in which distinct vectors $v$ and $v'$ are adjacent if and only if $v-v'$ is a simple tensor.  Then  $\Gamma$ is isomorphic to the  bilinear forms graph $H_q(2,m)$ (see~\cite[\S9.5A]{BroCohNeu1989}). By~\cite[Theorem~9.5.1]{BroCohNeu1989}, $\Aut(\Gamma)=V{:}\Aut(\Gamma)_0$, where $\Aut(\Gamma)_0$ is the full stabiliser of the  decomposition of $V$. By Lemma~\ref{lemma:rank3} and Remark~\ref{remark:PLSrank3}, the collinearity graph of~$\mathcal{S}$ is either $\Gamma$ or the complement of $\Gamma$, so $\Aut(\mathcal{S})\leq \Aut(\Gamma)$, and the result  holds since $V\leq G\leq \Aut(\mathcal{S})$.
\end{proof}

\begin{prop}
\label{prop:tensoraut}
Let $V:=U\otimes W$ where $U:=V_2(q)$, $W:=V_m(q)$, $m\geq 2$, and $q$ is a prime power. Let $\mathcal{S}:=(V,\mathcal{L})$ where $\mathcal{L}$ is $\{(U\otimes w) +v : w\in W^*,v\in V\} $ or  $\{(u\otimes W) +v : u\in U^*,v\in V\} $. Then $\Aut(\mathcal{S})=V{:}(\GL_2(q)\otimes  \GL_m(q)){:}\Aut(\mathbb{F}_q)\simeq V{:}(\GL_2(q)\circ \GL_m(q)){:}\Aut(\mathbb{F}_q)$.
\end{prop}
 
\begin{proof}
This follows from Lemmas~\ref{lemma:tensorsimple}(iii) and~\ref{lemma:Taut}.
\end{proof}

\section{Subfield classes (S1) and (S2)}
\label{s:(S1)}

Classes (S1) and (S2) consist of those  affine permutation groups $G$ of rank~$3$ on $V_n(q^2)$ for which either $\SL_n(q)\unlhd G_0$ where $n\geq 2$, or $A_7\unlhd G_0$ where $(n,q)=(4,2)$ and $A_7\leq \SL_4(2)\simeq A_8$. Let $r:=q^2$,  $K:=\mathbb{F}_{r}$ and $\zeta:=\zeta_r$ (see~\S\ref{ss:basicsvs}).  Let $\{v_1,\ldots,v_n\}$ be a basis of $V_n(r)$ on which $\GL_n(q)$ acts naturally over $\mathbb{F}_q$, and let $W:=\langle v_1,\ldots,v_n\rangle_{\mathbb{F}_q}$. By assumption,  $G_0\leq \GammaL_n(r)$, so    $G_0\leq (\GL_n(q)\circ K^*){:}\langle \sigma_r\rangle$ where  $\sigma_r$ (see~\S\ref{ss:basicsvs}) acts on  $V_n(r)$  with respect to  $\{v_1,\ldots,v_n\}$.
Representatives for the  orbits of  $G_0$ are $v_1$ and $v_1+\zeta v_2$, and  $v_1^{G_0}=\{\lambda v : \lambda\in K^*,v\in W^*\}$.  Let $\lambda W:=\{\lambda v : v \in W\}$ for $\lambda\in K^*$.   The orbits of $\SL_n(q)$ on $v_1^{G_0}$ are $(\lambda_1W)^*,\ldots,(\lambda_{s}W)^*$, where $\lambda_1,\ldots,\lambda_{s}$ is a transversal for $\mathbb{F}_q^*$ in~$K^*$ (so $s=q+1$).

Here is an equivalent definition for the classes (S1) and (S2). As above, let $r:=q^2$, $K:=\mathbb{F}_{r}$ and $\zeta:=\zeta_r$, and let $\{v_1,\ldots,v_n\}$ be a basis of $W:=V_n(q)$. The extension of scalars $K\otimes W$ is an $n$-dimensional $K$-vector space with basis $\{1\otimes v_1,\ldots,1\otimes v_n\}$ on which $(\GL_1(r)\otimes \GL_n(q)){:}\langle \sigma_r\rangle$ acts naturally. Now classes (S1) and (S2) consist of those affine permutation groups $G$ of rank~$3$ on $K\otimes W$ for which either $1\otimes \SL_n(q)\unlhd G_0$ where $n\geq 2$, or $1\otimes A_7\unlhd G_0$ where $(n,q)=(4,2)$ and $A_7\leq \SL_4(2)$. Representatives for the orbits of $G_0$ are $1\otimes v_1$ and $1\otimes v_1+\zeta\otimes v_2$. Further, we may view $K\otimes W$ as a $2n$-dimensional $\mathbb{F}_q$-vector space with basis $\{1\otimes v_1,\ldots,1\otimes v_n,\zeta\otimes v_1,\ldots,\zeta\otimes v_n\}$, in which case $G_0$ lies in class (T5) or (T4), respectively. (Note that in class (T5),  there exists $t\in \GL_2(q)$ such that $t\sigma_q$ acts as the Frobenius automorphism of $K$.) 

The orbits of $1\otimes \SL_n(q)$ on $(1\otimes v_1)^{G_0}$ are $\lambda_1\otimes W^*,\ldots,\lambda_{s} \otimes W^*$ where $\lambda_1,\ldots,\lambda_{s}$ is a transversal for $\mathbb{F}_q^*$ in $K^*$. If we define $\mathcal{L}$ to be the set of translations of $\lambda_i\otimes W$ for $1\leq i\leq q+1$, then $(K\otimes W,\mathcal{L})$ is one of the proper partial linear spaces described in Example~\ref{example:tensor} and Lemma~\ref{lemma:tensorsimple}(iii), and it is both $K$-independent and $\mathbb{F}_q$-independent. However, if we instead define $\mathcal{L}$ to be the set of translations of $K\otimes v=\langle 1\otimes v\rangle_K$ for $v\in W^*$, then $(K\otimes W,\mathcal{L})$ is not only the other proper partial linear space described in Example~\ref{example:tensor} and Lemma~\ref{lemma:tensorsimple}(iii), but also  a partial linear space from Example~\ref{example:AG} that is $K$-dependent and  $\mathbb{F}_q$-independent. 
 (In the notation of Example~\ref{example:AG}, for any $G_0$ in class (S1) or (S2), the partial linear space is $\mathcal{S}_1$, where $\Delta_1$ is the orbit  $\{\langle 1\otimes v\rangle_K : v\in W^*\}$ of $G_0$ on the points of $\PG_{n-1}(q^2)$.)

By combining these two viewpoints of classes (S1) and (S2), we obtain the following as an immediate consequence of Lemmas~\ref{lemma:tensorsimple},~\ref{lemma:tensorsimpleSL}(ii) and Lemma~\ref{lemma:T34simple}.

\begin{prop}
\label{prop:(S1)--(S2)good}
Let $G$ be an affine permutation group of rank~$3$ on $V:=V_n(q^2)$ where $q$ is a prime power and either $\SL_n(q)\unlhd G_0$ and  $n\geq 2$, or $A_7\unlhd G_0$  and   $(n,q)=(4,2)$. Let $K:=\mathbb{F}_{q^2}$. Let $\{v_1,\ldots,v_n\}$ be a basis of $V$ on which $\GL_n(q)$ acts naturally over $\mathbb{F}_q$, and let $W:=\langle v_1,\ldots,v_n\rangle_{\mathbb{F}_q}$. 
\begin{itemize}
\item[(i)] The only  $K$-independent  $G$-affine proper partial linear space  for which $0$ and $v_1$ are collinear  has line set  $\mathcal{L}:=\{\lambda W+v:\lambda\in K^*,v\in V\}$.
\item[(ii)] The partial linear space $(V,\mathcal{L})$ is described in Example~\emph{\ref{example:tensor}}.
\end{itemize}
\end{prop}

 Next we consider the orbit containing $v_1+\zeta v_2$.
 
 \begin{prop}
\label{prop:(S1)--(S2)bad}
Let $G$ be an affine permutation group of rank~$3$ on $V:=V_n(q^2)$ where $q$ is a prime power and  either $\SL_n(q)\unlhd G_0$ and  $n\geq 2$, or $A_7\unlhd G_0$ and $(n,q)=(4,2)$. Let $r:=q^2$, $K:=\mathbb{F}_{r}$ and $\zeta:=\zeta_{r}$ (where $\zeta^2=\zeta+1$ when $q=3$). Let $\{v_1,\ldots,v_n\}$ be a basis of $V$ on which $\GL_n(q)$ acts naturally over $\mathbb{F}_q$.  Let $x:=v_1+\zeta v_2$, $y_1:=\zeta^2v_1+v_2$ and $y_2:=-\zeta v_1+v_2$.
\begin{itemize}
\item[(i)] Let $A$ be the affine permutation group on $V$ for which $A_0=\GL_n(q)\langle \zeta^2,\zeta\sigma_r\rangle$, where   $\sigma_r$ acts on $V$   with respect to $\{v_1,\ldots,v_n\}$. Then $A$ has rank~$3$.

\item[(ii)] $(V,\mathcal{L})$ is a $K$-independent  $G$-affine proper partial linear space for which $0$ and $x$ are collinear if and only if $q=3$, $G\leq A$ and  $\mathcal{L}=L^A$ where   $L=\langle x,y_1\rangle_{\mathbb{F}_3}$ or $\langle x,y_2\rangle_{\mathbb{F}_3}$.

\item[(iii)] Let $q:=3$. Let $L_3:=\langle x\rangle_{K}$ and $L_i:=\langle x,y_i\rangle_{\mathbb{F}_3}$ for $i=1,2$. 
 Let $\mathcal{S}_i$ be the partial linear space $(V,L_i^A)$ for $1\leq i\leq 3$. Then $\mathcal{S}_3$ is $K$-dependent and $\mathcal{S}_1\simeq \mathcal{S}_2\simeq \mathcal{S}_3$.
\end{itemize}
\end{prop}

 When $\SL_n(q)\unlhd G_0$, we will use Lemma~\ref{lemma:tensorred} to reduce to the case $n=2$. In order to deal with the case $n=2$, we require the following 
technical result.

\begin{lemma}
\label{linear tech}
Let $q:=p^e$ where $p$ is a prime and $e\geq 1$, and let $s:=p^f$ where $1\leq f< 2e$  and $f$ divides $2e$. If  $(q+1)/\gcd(q+1,4e)$ divides $s-1$, 
then $(q,s)$ is $(3,3)$, $(7,7)$ or $(8,4)$.
\end{lemma}

\begin{proof}
Since $f$ is a proper divisor of $2e$, it follows that $f\leq e$. First suppose that $f=e$. Then $s=q$, so $q+1=a\gcd(q+1,4e)$ for some $a\in \{1,2\}$, in which case $q=3$ or $7$.

Now suppose that $f<e$. Since $f$ divides $2e$, it follows that $f\leq 2e/3$. Define $m$ to be $1$  when $q$ is even, $2$ when $q\equiv 1\mod 4$, and  $4$  when $q\equiv 3\mod 4$.  Now  $p^e/me<p^{2e/3}$,  so $p^e<(me)^3$.  

If $q$ is even, then $2\leq e\leq 9$. Since $(q+1)/\gcd(q+1,e)\leq p^{2e/3}-1$, it follows that $q$ is $2^3$ or $2^9$. Since $(q+1)/\gcd(q+1,e)$ divides $s-1$ and $f\mid 2e$ but $f<e$, it follows that $(q,s)=(8,4)$.

If $q$ is odd, then $2\leq e\leq 10$, so $p\leq 19$. Since $(q+1)/\gcd(q+1,4e)\leq p^{2e/3}-1$, it follows that $q$ is $3^3$, $5^3$, or $11^3$.  Then $f\in \{1,2\}$, but $(q+1)/\gcd(q+1,4e)\nmid (s-1)$. 
\end{proof}

Now we consider the case $n=2$.

\begin{prop}
\label{prop:A4red}
Let $G$ be an affine permutation group of rank~$3$ on $V:=V_2(q^2)$ where $\SL_2(q)\unlhd G_0$ and $q$ is a prime power. Let $r:=q^2$, $K:=\mathbb{F}_{r}$ and $\zeta:=\zeta_r$ (where $\zeta^2=\zeta+1$ when $q=3$). Let $\{v_1,v_2\}$ be a basis of $V$ on which $\GL_2(q)$ acts naturally over $\mathbb{F}_q$.  Let $A$ be the affine permutation group on $V$ for which $A_0=\GL_2(q)\langle \zeta^2,\zeta\sigma_r\rangle$, where   $\sigma_r$ acts on $V$ with respect to  $\{v_1,v_2\}$.  The following are equivalent.
 \begin{itemize}
 \item[(i)] $(V,\mathcal{L})$ is a $K$-independent  $G$-affine proper partial linear space for which $0$ and $v_1+\zeta v_2$ are collinear.
 \item[(ii)] $q=3$, $G\leq A$ and  $\mathcal{L}=L^A$ where   $L=\langle v_1+\zeta v_2,\zeta^2v_1+v_2\rangle_{\mathbb{F}_3}$ or $\langle v_1+\zeta v_2,-\zeta v_1+v_2\rangle_{\mathbb{F}_3}$.
 \end{itemize}
\end{prop}

\begin{proof}
If (ii) holds, then (i) holds by Lemmas~\ref{lemma:transitive} and~\ref{lemma:sufficient} and  a computation in {\sc Magma}. 

Conversely, suppose that (i) holds, and let $\mathcal{S}:=(V,\mathcal{L})$.  By Lemmas~\ref{lemma:nonsimple} and~\ref{lemma:closed}, there exists a subfield $\mathbb{F}_s$ of $K$ such that each $L\in\mathcal{L}_0$ is an $\mathbb{F}_s$-subspace of $V$ with the property that $\mathbb{F}_s=\{\lambda\in K:\lambda u\in L\}$ for all $u\in L^*$. Write $q=p^e$ and $s=p^f$ where $p$ is prime.  By Lemma~\ref{lemma:nonsimple}, $\mathcal{S}$ has line-size at most $q^2$.  Since $\mathcal{S}$ is $K$-independent,  $\mathbb{F}_s$ is a proper subfield of $K$, so $f$ is a proper divisor of $2e$.  If $q=3$, then (ii) holds by Lemma~\ref{lemma:necessary} and a computation in {\sc Magma}, so we assume that $q\neq 3$. 
We claim that $(q+1)/\gcd(q+1,4e)$ divides $s-1$. If so, then $(q,s)$ is $(7,7)$ or $(8,4)$ by Lemma~\ref{linear tech}, in which case we use {\sc Magma} and Lemma~\ref{lemma:necessary}(i) and (ii)   to obtain a contradiction. 

Since $\SL_2(q)$ and $\SU_2(q)$ are conjugate in $\GL(V)$ (see case (2) of the proof of Theorem~\ref{thm:rank3}),  we may assume  that $\SU_2(q)\unlhd G_0\leq \GammaU_2(q)$, in which case  the orbits of $G_0$ on $V^*$ consist of the non-zero isotropic vectors and the non-isotropic vectors, with sizes $(q^2-1)(q+1)$ and  $q(q-1)(q^2-1)$, respectively. Further, the  vector $v_1+\zeta v_2$ in (i) corresponds to a non-isotropic vector, so $0$ is collinear with a non-isotropic vector. Let $\f$ be the non-degenerate unitary form preserved by $\SU_2(q)$, and let 
$\{u_1,u_2\}$ be an orthonormal basis for $V$ (see~\cite[Proposition~2.3.1]{KleLie1990}).  Then $u_1$ is non-isotropic. Let $N$ be the kernel of the (multiplicative) norm map $\lambda\mapsto \lambda^{q+1}$ for $\lambda\in K^*$.
Let $D$ be the set of $g\in \GL_2(q^2)$ such that  $u_1^g=u_1$ and $u_2^g=\lambda u_2$ for some $\lambda\in N$, and observe that $D\simeq C_{q+1}$ since the norm map is surjective. It is routine to verify that $\GammaU_2(q)_{u_1}\cap \GL_2(q^2)=D$, so $\GammaU_2(q)_{u_1}/D$ is a group of order $2e$, and $|G_{0,u_1}/(G_{0,u_1}\cap D)|$ divides~$2e$.

Since the subdegrees of $G$ are  $(q^2-1)(q+1)$ and $|u_1^{G_0}|=q(q-1)(q^2-1)$, it follows that    the least common multiple of $q+1$ and $q-1$ divides $|G_0|/q(q^2-1)=(q-1)|G_{0,u_1}|$.  Thus $m:=(q+1)/\gcd(q+1,q-1)$ divides $|G_{0,u_1}|$, which divides $|G_{0,u_1}\cap D| 2e$, so $m/\gcd(m,2e)$ divides $|G_{0,u_1}\cap D|$. Now $\gcd(q+1,q-1)=\gcd(2,q-1)$, so $|G_{0,u_1}\cap D|$ is divisible by $(q+1)/\gcd(q+1,4e)$. Thus,   
for any non-isotropic vector $v\in V$, the order of $|G_{0,v}\cap \GL_2(q^2)|$ is divisible by $(q+1)/\gcd(q+1,4e)$. Moreover, $(q+1)\neq \gcd(q+1,4e)$ since $q\neq 3$, so there exists  $h\in G_{0,u_1}\cap D $ such that $h\neq 1$, in which case $u_2^h=\lambda u_2$ for some $\lambda\in N\setminus \{1\}$.

Let $L$ be the line of $\mathcal{S}$ on $0$ and $u_1$, and let $B:=L^*$. By Lemma~\ref{lemma:necessary}, $B$ is a  block of $G_0$ on the set of non-isotropic vectors.  By assumption, there exists $y\in B\setminus \langle u_1\rangle$. Write $y=\lambda_1u_1+\lambda_2u_2$ where $\lambda_1,\lambda_2\in K$. Now $\lambda_1u_1+\lambda \lambda_2u_2=y^h\in B$ by Lemma~\ref{lemma:basic}(ii), so $\delta u_2\in B$ where $\delta:=(\lambda-1)\lambda_2$ since $L$ is an $\mathbb{F}_s$-subspace of $V$.  Let $D'$ be the set of $g\in \GL_2(q^2)$ such that $u_2^g=u_2$ and $u_1^g=\alpha u_1$ for some $\alpha\in N$. Let $N'$ be the set of $\alpha\in N$ such that $u_1^g=\alpha u_1$ for some $g\in G_{0,\delta u_2}\cap D'$. Then $N'\leq N$, and $N'\simeq G_{0,\delta u_2}\cap D'=G_{0,\delta u_2}\cap \GL_2(q^2)$, so
$(q+1)/\gcd(q+1,4e)$ divides $|N'|$. 
If $\mu\in N'$, then  $\mu u_1\in B$ by Lemma~\ref{lemma:basic}(ii), so $\mu\in \mathbb{F}_s^*$. Thus $N'\leq \mathbb{F}_s^*$, and the claim follows.
\end{proof}

\begin{proof}[Proof of Proposition~\emph{\ref{prop:(S1)--(S2)bad}}] (i) This is routine. 

(ii) If $A_7\unlhd G_0$ and $(n,q)=(4,2)$, then there are no $K$-independent $G$-affine proper partial linear spaces for which $0$ and $x$ are collinear by Lemmas~\ref{lemma:necessary} and~\ref{lemma:affine} and a computation in {\sc Magma}, so we assume that $\SL_n(q)\unlhd G_0$. If $n=2$, then we are done by Proposition~\ref{prop:A4red}, so we assume that  $n\geq 3$.  Let $Y:=\langle v_1,v_2\rangle_K$ and $H:=W{:}G_{0,Y}^Y$. By Lemma~\ref{lemma:tensorred}, $\GL_2(q)\unlhd H_0$ and $H$ is transitive of rank~$3$. 

 First we claim that if $q=3$ and $H_0 \leq \GL_2(3)\langle \zeta^2,\zeta\sigma_9\rangle$, then $H_0=\GL_2(3)\langle \zeta^2,\zeta\sigma_9\rangle$. Since $H$ has rank~$3$, there exists $g\in H_0$ mapping $v_1$ to $\zeta^2 v_1$, and $g=h\zeta^{2i}(\zeta\sigma_9)^j$ for some $h\in \GL_2(3)$, $i\in \{0,1\}$ and $j\in \{0,1\}$, but this is impossible if $j=1$, so $j=0$, whence $i=1$. Since $\GL_2(3)\leq H_0$, it follows that  $\zeta^2\in H_0$. Similarly, there exists $g\in H_0$ mapping $v_1$ to $\zeta v_1$, so $\zeta\sigma_9\in H_0$, and the claim holds.

Let $\mathcal{S}:=(V,\mathcal{L})$ be a $K$-independent $G$-affine proper partial linear space for which $0$ and $x$ lie on a line $L$.  Then $\mathcal{S}$ is $\mathbb{F}_q$-independent, so by  Lemma~\ref{lemma:tensorred}(i), $\mathcal{S}\cap Y=(Y,L^H)$ is an  $H$-affine proper partial linear space. Since $\mathcal{S}\cap Y$ contains the line $L$, it is  $K$-independent, so by  Proposition~\ref{prop:A4red}  and the claim, 
$q=3$, $H_0=\GL_2(3)\langle \zeta^2,\zeta\sigma_9\rangle$ and $L=\langle x,y_1\rangle_{\mathbb{F}_3}$ or $\langle x,y_2\rangle_{\mathbb{F}_3}$. 
 Let $g\in G_0$. Then $Y^g=Y^h$ for some $h\in \GL_n(3)$. There exists $s\in \SL_n(3)$ such that $v_1^{hs}=v_1$ and $v_2^{hs}=v_2$, so  $gs\in G_{0,Y}\leq A_0$. Thus $G\leq A$. Further,  $L^{G}=L^{A}$ by Lemma~\ref{lemma:tensorred}(ii) since  $A$ is a rank~$3$ group and $A_{0,Y}^Y=H_0$. 

Conversely, suppose that $q=3$, $G\leq A$ and  $\mathcal{L}=L^A$ where   $L=\langle x,y_1\rangle_{\mathbb{F}_3}$ or $\langle x,y_2\rangle_{\mathbb{F}_3}$. 
Then $H_0\leq \GL_2(3)\langle \zeta^2,\zeta\sigma_9\rangle$, so
$H_0=\GL_2(3)\langle \zeta^2,\zeta\sigma_9\rangle$  by the claim.  Thus  $(V,\mathcal{L})$ is a $G$-affine proper partial linear space by  Lemma~\ref{lemma:tensorred}(ii) and Proposition~\ref{prop:A4red}, and it is clearly $K$-independent.

(iii) Clearly $\mathcal{S}_3$ is $K$-dependent. Since  $\zeta$ normalises $A$, it follows that $\zeta$ maps $\mathcal{L}_1$ to $\mathcal{L}_2$, so it suffices to show that $\mathcal{S}_2$ and $\mathcal{S}_3$ are isomorphic. Since $\{v_1,\ldots,v_n,\zeta v_1,\ldots,\zeta v_n\}$ is an $\mathbb{F}_3$-basis of $V$, there exists $t\in \GL_{2n}(3)$ such that $v_i^t=\zeta^2 v_i$ and $(\zeta v_i)^t=-v_i$ for $1\leq i\leq n$. Now $t$ centralises $\GL_n(3)$, and $t^{-1}\zeta\sigma_9 t=\zeta^3\sigma_9$ and $t^{-1}\zeta^2t=\zeta\sigma_9$, so $t$ normalises $A_0$. Further, $x^t=\zeta^2v_1-v_2$ and $y_2^t=v_1+\zeta^2 v_2$, so  $y_2^t=\zeta^6 x^t$, in which case $L_2^t=\langle x^t\rangle_{K}$. Thus $(L_2^A)^t=(\langle x^t\rangle_{K})^{A}=L_3^A$.
\end{proof}

In Remark~\ref{remark:lineartopls}, we observed that if $\mathcal{S}:=(\mathcal{P},\mathcal{L})$ and $\mathcal{S}':=(\mathcal{P}',\mathcal{L}')$ are isomorphic linear spaces satisfying the conditions of Lemma~\ref{lemma:linearspace}(i), then any isomorphism $\varphi:\mathcal{S}\to \mathcal{S}'$  determines isomorphisms between the proper  partial linear spaces obtained from $\mathcal{S}$ and $\mathcal{S}'$ (in the sense of  Lemma~\ref{lemma:linearspace}). In the following, we show that the converse of this observation need not hold. 

\begin{remark}
\label{remark:linearspace}
Let $V:=V_2(q^2)$ where $q$ is a prime power, let $K:=\mathbb{F}_{r}$ where $r:=q^2$, and let $G$ be the rank~$3$ affine primitive permutation group $V{:}(\GL_2(q)\circ K^*){:}\Aut(K)$ from class (S1).  Let $\{v_1,v_2\}$ be a basis of $V$ on which $\GL_2(q)$ acts naturally over $\mathbb{F}_q$, and let $W:=\langle v_1,v_2\rangle_{\mathbb{F}_q}$ and $Y:=(v_1+\zeta_{r} v_2)^{G_0}$. Let $\mathcal{L}_U:=\{\langle w\rangle_K +v: w\in W^*,v\in V\}$, $\mathcal{L}_W:=\{\lambda W+v:\lambda\in K^*,v\in V\}$ and $\mathcal{L}_Y:=\{\langle w\rangle_K +v: w\in Y, v\in V\}$. Let $\mathcal{S}_U:=(V,\mathcal{L}_U)$, $\mathcal{S}_W:=(V,\mathcal{L}_W)$ and $\mathcal{S}_Y:=(V,\mathcal{L}_Y)$. Now $\mathcal{S}_U$, $\mathcal{S}_W$ and $\mathcal{S}_Y$ are $G$-affine  proper partial linear spaces, and the collinearity relation of $\mathcal{S}_U$ and $\mathcal{S}_W$ is disjoint from the collinearity relation of $\mathcal{S}_Y$. 
 Further, as we discussed    at the beginning of this section,   $\mathcal{S}_U$ and $\mathcal{S}_W$ are the two partial linear spaces  of  Example~\ref{example:tensor} (for $m=2$).
Thus  $\mathcal{S}_U\simeq \mathcal{S}_W$  by Lemma~\ref{lemma:tensorsimple}(iv), and of course $\mathcal{S}_Y$ is isomorphic to itself. However, the linear spaces $\mathcal{S}_1:=(V,\mathcal{L}_U\cup \mathcal{L}_Y)$ and $\mathcal{S}_2:=(V,\mathcal{L}_W\cup \mathcal{L}_Y)$ (see Lemma~\ref{lemma:linearspace}) are not isomorphic  for $q\geq 3$:  $\mathcal{S}_1$ is the affine plane $\AG_2(r)$, while $\mathcal{S}_2$ is the Hall plane of order $r$ (see~\cite[\S13]{Lun1980}). 
\end{remark}

\section{Subfield class (S0)}
\label{s:(S0)}

Class (S0) is similar to class (S1): it consists of those  affine permutation groups $G$ of rank~$3$ on $V_2(q^3)$ for which $\SL_2(q)\unlhd G_0$. Let $r:=q^3$, $K:=\mathbb{F}_r$ and $\zeta:=\zeta_r$ (see~\S\ref{ss:basicsvs}). Let $\{v_1,v_2\}$ be a basis of $V_2(r)$ on which $\GL_2(q)$ acts naturally over $\mathbb{F}_q$, and let $U:=\langle v_1,v_2\rangle_{\mathbb{F}_q}$. By assumption, $G_0\leq \GammaL_2(r)$, so  $G_0\leq (\GL_2(q)\circ K^*){:}\langle \sigma_r\rangle$ where  $\sigma_r$ (see~\S\ref{ss:basicsvs}) acts on $V_2(r)$  with respect to $\{v_1,v_2\}$.
Representatives for the  orbits of  $G_0$ are $v_1$ and $v_1+\zeta v_2$, and $v_1^{G_0}=\{\lambda v : \lambda\in K^*,v\in U^*\}$.  Let $\lambda U:=\{\lambda v : v \in U\}$ for $\lambda\in K^*$.  The orbits of $\SL_2(q)$ on $v_1^{G_0}$ are $(\lambda_1U)^*,\ldots,(\lambda_{s}U)^*$, where $\lambda_1,\ldots,\lambda_{s}$ is a transversal for $\mathbb{F}_q^*$ in~$K^*$ (so  $s=q^2+q+1$).

Here is an equivalent definition for the class (S0). As above, let $r:=q^3$, $K:=\mathbb{F}_{r}$ and $\zeta:=\zeta_r$. Let $\{v_1,v_2\}$ be a basis of $U:=V_2(q)$. The extension of scalars $U\otimes K$ is a $2$-dimensional $K$-vector space with basis $\{v_1\otimes 1,v_2\otimes 1\}$ on which $\GL_2(q)\otimes \GL_1(r){:}\langle \sigma_r\rangle$ acts naturally. Now class (S0) consists of those affine permutation groups $G$ of rank~$3$ on $U\otimes K$ for which $\SL_2(q)\otimes 1\unlhd G_0$. Representatives for the orbits of $G_0$ are $v_1\otimes 1$ and $v_1\otimes 1+v_2\otimes \zeta$. Further, we may view $U\otimes K$ as a $6$-dimensional $\mathbb{F}_q$-vector space with basis $\{v_1\otimes 1,v_2\otimes 1,v_1\otimes \zeta,v_2\otimes \zeta,v_1\otimes \zeta^2,v_2\otimes \zeta^2\}$, in which case $G_0$ lies in class (T6). (Note that in class (T6), there exists $t\in \GL_3(q)$ such that $t\sigma_q$ acts as the Frobenius automorphism of $K$.) 

The orbits of $\SL_2(q)\otimes 1$ on $(v_1\otimes 1)^{G_0}$ are $U^*\otimes\lambda_1,\ldots, U^*\otimes \lambda_{s}$ where $\lambda_1,\ldots,\lambda_{s}$ is a transversal for $\mathbb{F}_q^*$ in $K^*$. If we define $\mathcal{L}$ to be the set of translations of $U\otimes \lambda_i$ for $1\leq i\leq q^2+q+1$, then $(U\otimes K,\mathcal{L})$ is one of the proper partial linear spaces described in Example \ref{example:tensor} and Lemma~\ref{lemma:tensorsimple}(iii), and it is both $K$-independent and $\mathbb{F}_q$-independent. However, if we instead define $\mathcal{L}$ to be the set of translations of $v\otimes K=\langle v\otimes 1\rangle_K$ for $v\in U^*$, then $(U\otimes K,\mathcal{L})$ is not only the other proper partial linear space described in Example~\ref{example:tensor} and Lemma~\ref{lemma:tensorsimple}(iii), but also  a partial linear space from Example~\ref{example:AG} that is $K$-dependent and  $\mathbb{F}_q$-independent.  (In the notation of Example~\ref{example:AG}, for any $G_0$ in class (S0), the partial linear space is $\mathcal{S}_1$, where $\Delta_1$ is the orbit  $\{\langle v\otimes 1\rangle_K : v\in U^*\}$ of $G_0$ on the points of $\PG_{1}(q^3)$.)

By combining these two viewpoints of class (S0), we obtain the following as an immediate consequence of Lemmas~\ref{lemma:tensorsimple} and~\ref{lemma:tensorsimpleSL}(i).

\begin{prop}
\label{prop:(S0)good}
Let $G$ be an affine permutation group of rank~$3$ on $V:=V_2(q^3)$ where $\SL_2(q)\unlhd G_0$ and $q$ is a prime power. Let $K:=\mathbb{F}_{q^3}$.
 Let $\{v_1,v_2\}$ be a basis of $V$ on which $\GL_2(q)$ acts naturally over $\mathbb{F}_q$, and let $U:=\langle v_1,v_2\rangle_{\mathbb{F}_q}$.  \begin{itemize}
\item[(i)] The only  $K$-independent  $G$-affine proper partial linear space  for which $0$ and $v_1$ are collinear  has line set  $\mathcal{L}:=\{\lambda U+v:\lambda\in K^*,v\in V\}$.
\item[(ii)] The partial linear space $(V,\mathcal{L})$ is described in Example~\emph{\ref{example:tensor}}.
\end{itemize}
\end{prop}

Unfortunately, in contrast to class (S1), we are unable to provide a classification of the $K$-independent $G$-affine proper partial linear spaces for which $0$ is collinear with $v_1+\zeta v_2$ for groups $G$ in class (S0).  (Note that the conditions of  Theorem~\ref{thm:main}(iv)(c) hold under these assumptions.) Here the group $\GL_2(q)\circ K^*$ acts regularly on the orbit containing $v_1+\zeta v_2$, so  our usual techniques fail.  However, we can classify the $G$-affine proper partial linear spaces when $\GL_2(q)\circ K^*\leq G_0$. Since the affine permutation group with  stabiliser $\GL_2(q)\circ K^*$ is transitive of rank~$3$, any example that arises for $G_0$ also arises for $\GL_2(q)\circ K^*$, so we assume that $G_0=\GL_2(q)\circ K^*$. 

To state our classification, we need some notation. Let $F/E$ be a field extension of degree~$2$, and let $\End_E(F)$ denote the ring of $E$-endomorphisms of $F$. For 
$\varphi\in \End_E(F)$,  we denote the $2\times 2$ transformation matrix of $\varphi$ with respect to an $E$-basis $\mathcal{B}$ of $F$ by $[\varphi]_\mathcal{B}$. In other words, if $\mathcal{B}=\{b_1,b_2\}$ and  $v=\lambda_1b_1+\lambda_2 b_2$ for some $\lambda_1,\lambda_2\in E$, then  $(\lambda_1,\lambda_2)[\varphi]_{\mathcal{B}}=(\mu_1,\mu_2)$ if and only if  $v\varphi=\mu_1b_1+\mu_2b_2$ where $\mu_1,\mu_2\in E$.  Note  that in the following result, it is more convenient not to    make the distinction between  $K$-independent and $K$-dependent partial linear spaces. 

\begin{prop}
\label{prop:(S0)bad}
 Let $r:=q^3$, $K:=\mathbb{F}_r$ and $\zeta:=\zeta_r$ where $q$ is a prime power. Let $V:=K^2$ and $G:=V{:}(\GL_2(q)\circ K^*)$.  Let $x:=(1,\zeta)\in V$. Then $(V,\mathcal{L})$ is a $G$-affine proper partial linear space for which $0$ and $x$ are collinear if and only if $\mathcal{L}=L^G$  and one of the following holds.
\begin{itemize}
\item[(i)]  $L=\{\lambda x : \lambda\in F\}$,   where $F$ is a subfield of $K$ and $|F|>2$.
\item[(ii)]  $L=\{x \diag(\lambda,\lambda^\tau)^g: \lambda\in F\}$, where $F$ is a subfield of $\mathbb{F}_q$,  $\tau$ is a non-trivial element of $ \Aut(F)$, and $g\in\GL_2(q)$.
\item[(iii)] $L=\{x[\varphi_\lambda]_\mathcal{B}^g : \lambda\in F\}$, where $F$ is a field with $[F: F\cap \mathbb{F}_q]=2$,  the map  $\varphi_\lambda\in \End_{F\cap \mathbb{F}_q}(F)$ is defined by $v\mapsto \lambda v$ for all $v\in F$,   the set $\mathcal{B}$ is an $(F\cap \mathbb{F}_q)$-basis of $F$,  and $g\in \GL_2(q)$.
\end{itemize}
\end{prop}

 Before proving Proposition~\ref{prop:(S0)bad},  we make some remarks and prove some lemmas.  

\begin{remark}
\label{remark:S0exceptions}
Biliotti et al.~\cite{BilMonFra2015} consider the case where $G_0=\GL_2(q)\circ K^*$ for $2$-$(v,k,1)$ designs in~\cite[Theorem 32]{BilMonFra2015}, but they mistakenly assert that their result applies to all rank~$3$ groups $G$ in class (S0) with $K^*\leq G_0$ by stating  that any such group must contain $\GL_2(q)\circ K^*$ (see~\cite[p.148]{BilMonFra2015}). This is not true, as we now see. Let $K$ be a field of order $r:=q^3$ where $q=p^e$ and $p$ is prime. Let 
$S:=\SL_2(q)\circ K^*$ and $T:=\GL_2(q)\circ K^*$. Now $T$ is regular on $(1,\zeta_r)^{T}$ and $S$ has index $(2,q-1)$ in $T$, so $S$ has $(2,q-1)$ orbits on $(1,\zeta_r)^{T}$, say $X_1,\ldots,X_{(2,q-1)}$. Note also that $(1,0)^S=(1,0)^T$.  Let $g:=\diag(\zeta_q,1)$ and $\sigma:=\sigma_r$.
Then  $T=S\langle g\rangle$ and  $S\langle g,\sigma\rangle/S=\langle Sg\rangle\times \langle S\sigma\rangle\simeq C_{(2,q-1)}\times C_{3e}$. Now $G_0$ is the point stabiliser of a rank $3$ group in class (S0) that contains $K^*$ but not $T$ precisely when the following hold: $q$ is odd;  $S\leq G_0\leq S\langle g,\sigma\rangle$; $g\notin G_0$; and $X_1^h=X_2$ for some $h\in G_0\cap \langle g,\sigma\rangle$. Suppose then that $q$ is odd. The subgroups of $ S\langle g,\sigma\rangle/S$  not containing $Sg$ are cyclic, so it is routine to determine the possibilities for $G_0$: $e$ is even, and either   
$X_1^\sigma=X_1$, in which case $G_0$ is $S\langle g\tau\rangle$ for any   $\tau\in\langle\sigma\rangle$   with even order, or  $X_1^\sigma=X_2$, in which case  $G_0$ is  $S\langle g\tau\rangle$ for any  $\tau\in\langle\sigma^2\rangle$  with even order, or  $S\langle \tau\rangle$ for  any $\tau\in\langle\sigma\rangle\setminus \langle\sigma^2\rangle$. 
\end{remark}

\begin{remark}
\label{remark:S0bad}
Let $K$ be a field of order $q^3$ where $q=p^e$ and $p$ is prime. Let $V:=K^2$ and $G:=V{:}(\GL_2(q)\circ K^*)$.
Note that the fields $F$ for which $[F:F\cap \mathbb{F}_q]=2$ are precisely the fields of order $p^{2f}$ where $f$ divides $e$ but $2f$ does not divide $e$. Let $\mathcal{S}$ be any $G$-affine proper partial linear space that arises from  Proposition~\ref{prop:(S0)bad}(iii) with respect to the field $F$, where $|F|=p^{2f}$.  Now $F$ is not a subfield of $K$. Hence, if  $f< e$, then  $\mathcal{S}$ 
cannot be obtained from any $G$-affine $2$-$(q^6,p^{2f},1)$ design  by Propositions~\ref{prop:dep} and~\ref{prop:(S0)good} (in the sense of Remark~\ref{remark:PLSfromLinear} and Lemma~\ref{lemma:linearspace}).
\end{remark}
 
Next we have a lemma that applies to all groups in class (S0).

\begin{lemma}
\label{lemma:A5add}
Let $G$ be an affine permutation group of rank~$3$ on $V:=V_2(q^3)$ where $\SL_2(q)\unlhd G_0$ and $q$ is a power of a prime $p$. Let $r:=q^3$. Let $\{v_1,v_2\}$ be a basis of $V$ on which $\GL_2(q)$ acts naturally over $\mathbb{F}_q$.  Let  $(V,\mathcal{L})$ be a   $G$-affine proper partial linear space for which $0$ and $v_1+\zeta_{r} v_2$ are collinear.  If $L\in \mathcal{L}_0$, then $L$ is an $\mathbb{F}_p$-subspace of $V$ and $|L|$ divides~$q^3$.
\end{lemma}

\begin{proof}
Let $L\in \mathcal{L}_0$. Then $L$ is an $\mathbb{F}_p$-subspace of $V$ by Lemma~\ref{lemma:affine} since $-1\in \SL_2(q)$.
 If $\lambda v_1 +\mu v_2,\lambda v_1+\delta v_2\in L^*$ for some $\lambda,\mu,\delta\in  \mathbb{F}_r^*$, then $(\mu-\delta)v_2\in L$, so $\mu=\delta$. Thus $|L|\leq q^3$, and since $L$ is an $\mathbb{F}_p$-vector space, $|L|$ divides $q^3$.
\end{proof}

In the following, for a field of size $s$, we denote the ring of $2\times 2$ matrices over $\mathbb{F}_s$ by $\M_2(s)$. Note that if $F$ is a field and $\varphi:F\to \M_2(s)$ is an injective ring homomorphism, then $(F\varphi)^*\leq \GL_2(s)$.

\begin{lemma}
\label{lemma:(S0)bad}
 Let $r:=q^3$, $K:=\mathbb{F}_r$ and $\zeta:=\zeta_r$, where $q$ is a prime power. Let $V:=K^2$ and $G:=V{:}(\GL_2(q)\circ K^*)$.  Let $x:=(1,\zeta)\in V$. The following are equivalent.
\begin{itemize}
\item[(i)] $(V,\mathcal{L})$ is a $G$-affine proper partial linear space for which $0$ and $x$ are collinear.
\item[(ii)] $\mathcal{L}=L^G$ where $L=\{x(\lambda\varphi): \lambda\in F\}$ for some field $F$  with $|F|>2$ and injective ring homomorphism $\varphi:F\to \M_2(r)$ such that $(F\varphi)^*\leq \GL_2(q)$ or $K^*$.
\end{itemize}
\end{lemma}

\begin{proof}
Let $(V,\mathcal{L})$ be a $G$-affine proper partial linear space for which $0$ and $x$ are collinear. Let $L$ be the line containing $0$ and $x$. Since $G$ is transitive of rank~$3$ on $V$, the line set $\mathcal{L}=L^G$. By Lemma~\ref{lemma:A5add}, $L$ is an $\mathbb{F}_p$-subspace of $V$ and $|L|=p^m$ for some $m\leq 3e$ where $q=p^e$ and $p$ is prime.  Let $B:=L^*$. By Lemma~\ref{lemma:necessary}, $B$ is a block of $G_0$ on $x^{G_0}$. Let $H:=G_{0,B}$. Note that $|H|=p^m-1$ since $G_0$ is regular on $x^{G_0}$. In order to establish (ii), since $B=\{x^h :h \in H\}$, it suffices to prove that $H\cup\{0\}$ is a field (as a subring of $\M_2(r)$) and that $H\leq \GL_2(q)$ or $K^*$.

Observe that $H\cap K^*=\{\lambda\in K^* : \lambda x \in B\}$, for if $\lambda\in K^*$ and $\lambda x\in B$, then $x^\lambda\in B^\lambda\cap B$, and since $\lambda\in G_0$, it follows that $B^\lambda=B$, so $\lambda\in H$.   Since $L$ is closed under addition, it follows that $(H\cap K^*)\cup \{0\}$ is closed under addition. Clearly $H\cap K^*$ is
a subgroup of the abelian group  $K^*$, 
  so $(H\cap K^*)\cup \{0\}$  is a field and $|H\cap K^*|=p^f-1$ for some $f$ that divides $3e$. Since  $|H\cap K^*|$ divides $|H|$, it follows that $f$ divides $m$. If $f=m$, then $H\leq K^*$, in which case  $H\cup \{0\}$ is a field and (ii) holds, so we may assume that $f<m$.

The group $H/(H\cap K^*)$ is isomorphic to a subgroup of $\GL_2(q)/(\GL_2(q)\cap K^*)$, which has order $q(q^2-1)$, so $p^m-1$ divides $(p^f-1)(p^{2e}-1)$. We claim that $m$ divides $2e$. We may assume that $m>2$. If the pair $(p,m)\neq (2,6)$, then $p^m-1$ has a primitive prime divisor, say $t$ (see~\cite{Zsi1892}  or~\cite[Theorem 5.2.14]{KleLie1990}). Now $t$ divides $(p^f-1)(p^{2e}-1)$, but $f<m$, so $t$ divides $p^{2e}-1$.  Since $p$ has order $m$ in $(\mathbb{Z}/t\mathbb{Z})^*$, we conclude that  $m$ divides $2e$, as desired. Otherwise, $(p,m)=(2,6)$. Now $f\in \{1,2,3\}$, and $63$ divides $(2^f-1)(4^e-1)$. If $f=1$ or $2$, then $7$ divides $4^e-1$, and $4$ has order $3$ in $(\mathbb{Z}/7\mathbb{Z})^*$, so $3$ divides $e$. Thus $m$ divides $2e$. If instead  $f=3$, then $9$ divides $4^e-1$, and $4$ has order $3$ in $(\mathbb{Z}/9\mathbb{Z})^*$, so $m$ divides $2e$. Hence we have established the claim.

Next we claim that $H\leq \GL_2(q)$. Since $m$ divides $2e$, the order of $H$ divides $q^2-1$. Let $h\in H$. Now $h=\lambda g$ for some $g\in \GL_2(q)$ and $\lambda\in K^*$. Since $h$ has order dividing $q^2-1$, it follows that $g^{q^2-1}\lambda^{q^2-1}=1$.  Since $g\in \GL_2(q)$, either  $g$ has one Jordan block and the order of $g$ divides $p(q-1)$, or  $g$ has two Jordan blocks (over an extension field of $\mathbb{F}_q$ with degree at most $2$) and the order of $g$ divides $q^2-1$.  Thus $\lambda^{p(q^2-1)}=1$, so the order of $\lambda$ divides  $\gcd(p(q^2-1),q^3-1)=q-1$. Hence $\lambda\in \GL_2(q)$, and the claim follows.

It remains to show that $H\cup \{0\}$ is a field.  To do so, it suffices to prove that $\det(1+h)\neq 0$ for all $h\in H\setminus \{-1\}$. Indeed, suppose that this statement holds, and let $h_1,h_2\in H$. We claim that $h_1+h_2\in H\cup \{0\}$, in which case it will follow that $H\cup \{0\}$ is a field since  $H$ is a subgroup of $\GL_2(q)$ and any finite division ring is a field by Wedderburn's theorem. If $h_1+h_2=0$, then we are done, so assume otherwise. Now $h_1^{-1}h_2\neq -1$, so   $\det(1+h_1^{-1}h_2)\neq 0$, whence $\det(h_1+h_2)\neq 0$.  Since $H\leq\GL_2(q)$, it follows that $h_1+h_2\in \GL_2(q)$. Now $x^{h_1+h_2}\neq 0$, and $L$ is closed under addition, so   $x^{h_1+h_2}=x^{h_1}+x^{h_2}\in B$. Thus there exists $k\in H$ such that $x^{h_1+h_2}=x^k$, but $\GL_2(q)$ is semiregular on $x^{G_0}$, so $h_1+h_2=k\in H$, as desired.

Let $h\in H\setminus \{-1\}$, and suppose for a contradiction that $\det(1+h)=0$. Let  $\lambda$ and $\mu$ be the eigenvalues of $h$ in some field extension of $\mathbb{F}_q$.  Since $|H|=p^m-1$,  the order of $h$ is not divisible by $p$, so $h$ has two Jordan blocks. Thus $h^g=\diag(\lambda,\mu)$ for some $g\in \GL_2(\mathbb{F}_q(\lambda))$. Now $(1+h)^g=1+h^g$, so $\det(1+h^g)=0$. Thus $\lambda=-1$ or $\mu=-1$. Without loss of generality, we may assume that $\lambda=-1$. In particular,  $\mu\in \mathbb{F}_q^*$, so  $g\in \GL_2(q)$.  Let $y:=x^g$, and write $y=(\lambda_1,\lambda_2)$ where $\lambda_1,\lambda_2\in K^*$. Now $B^g$ is a block of $G_0$ containing $y$, and $h^g\in G_{0,B^g}$, so $(-\lambda_1 ,\mu \lambda_2)\in B^g$.  However, $L^g$ is an $\mathbb{F}_p$-subspace of $V$, so  $(-\lambda_1 ,-\lambda_2 )\in B^g$, but then $(0,\lambda_2(\mu +1) )\in L^g$. Since $B^g\subseteq x^{G_0}$, it follows that $\mu=-1$, but then $h=-1$, a contradiction. Thus we have established~(ii).

Conversely, suppose that $\mathcal{L}=L^G$ where $L=\{x(\lambda\varphi): \lambda\in F\}$ for some field $F$ with $|F|>2$ and injective ring homomorphism $\varphi:F\to \M_2(r)$ such that $(F\varphi)^*\leq \GL_2(q)$ or $K^*$. Let $H:=(F\varphi)^*$ and $B:=L^*=x^H$. Since $G_0$ acts regularly on $x^{G_0}$, it follows that $B$ is a block of $G_0$, and clearly  $L$ is an $\mathbb{F}_p$-subspace of $V$, so (i) holds by  Lemmas~\ref{lemma:transitive} and~\ref{lemma:sufficient}.
\end{proof}

Now we consider the possibilities for the field $F$ and the injective  ring homomorphism $\varphi$ in the case where $\varphi:F\to \M_2(q)$. There are two natural ways of defining such a ring homomorphism. First, let $F$ be a subfield of $\mathbb{F}_q$, let $\tau\in \Aut(F)$ and define $\varphi:F\to \M_2(q)$ by $\lambda\mapsto \diag(\lambda,\lambda^\tau)$. Clearly $\varphi$ is an injective ring homomorphism. Second, let $E$ be a subfield of $\mathbb{F}_q$, and let $F$ be a field extension of $E$ of degree $2$. Define $\varphi:F\to \End_E(F)$ by $\lambda\mapsto \varphi_\lambda $ where $\varphi_\lambda: v\mapsto \lambda v$ for all $v\in F$. This is an injective ring homomorphism, and by fixing a basis for the $E$-vector space $F$, we see that $\End_E(F)\simeq \M_2(E)$, and $\M_2(E)$ is a subring of $\M_2(q)$, so $\varphi$ naturally gives us an injective ring homomorphism from $F$ to $\M_2(q)$. It is well known that, up to conjugacy in $\GL_2(q)$,  there are no other ways of embedding $F$ in $\M_2(q)$. We state this result explicitly below, and since we could not find a reference for it in this form, we prove it  for completeness.

\begin{lemma}
\label{lemma:(S0)badmore}
Let $F$ be a field, and let $\varphi:F\to \M_2(q)$ be an injective ring homomorphism where $q$ is a prime power. Let $H:=(F\varphi)^*\leq \GL_2(q)$, and let $H=\langle h\rangle$. Let $\varepsilon$ be an eigenvalue of $h$ in some field extension of $\mathbb{F}_q$. Then $F^*\simeq \langle \varepsilon\rangle$, and we may view $F$ as a subfield of $\mathbb{F}_q(\varepsilon)$. Further, exactly one of the following holds, where $E:=F\cap \mathbb{F}_q$.
\begin{itemize}
\item[(i)] $\varepsilon\in\mathbb{F}_q$, and $h$ is conjugate in $\GL_2(q)$ to $\diag(\varepsilon,\varepsilon^\tau)$  for some $\tau\in \Aut(F)$.
\item[(ii)]  $[F:E]=2$, and  $h$ is conjugate in $\GL_2(q)$ to $[\varphi_\varepsilon]_{\mathcal{B}}$,    
where  $\varphi_\varepsilon\in \End_E(F)$ is defined by $  v\mapsto \varepsilon  v$ for all $v\in F$, and $\mathcal{B}:=\{1,\varepsilon\}$ is an $E$-basis of $F$.
\end{itemize}
\end{lemma}

\begin{proof}
Let $p$ be the characteristic of $\mathbb{F}_q$. 
Since $\varphi$ maps the identity of $F$ to the identity of $\M_2(q)$, it follows that $F$ also has  characteristic $p$. In particular, the order of $h$ is not divisible by $p$. Let $f(X)\in \mathbb{F}_q[X]$ be the characteristic polynomial of $h$. Then $f(X)=(X-\varepsilon)(X-\delta)$ for some $\delta\in \mathbb{F}_q(\varepsilon)$. Since the order of $h$ is not divisible by $p$,  it must have two  Jordan blocks, so  $h$ is conjugate in $\GL_2(\mathbb{F}_q(\varepsilon))$ to  $g:=\diag(\varepsilon,\delta)$. 

Let $\ell$ and $m$ be the order of $\varepsilon$ and $\delta$ in $\mathbb{F}_q(\varepsilon)^*$,  respectively. We may assume that $\ell\leq m$. Now $\langle g\rangle \cup \{0\}$ is conjugate  to $F\varphi$, so it is closed under addition. In particular, it contains the matrix $1-\diag(1,\delta^\ell)=\diag(0,1-\delta^\ell)$, so $1=\delta^\ell$, and it follows that $m=\ell=|F|-1$. Thus $F^*\simeq \langle \varepsilon\rangle\simeq \langle\delta\rangle$, and we may view $F$ as the subfield of $\mathbb{F}_q(\varepsilon)$ that contains $\varepsilon$ and $\delta$.  Now $\delta=\varepsilon^n$ for some integer $n$, so $\langle g\rangle=\{\diag(\lambda,\lambda^n):\lambda\in F^*\}$, and it follows that $(\lambda+\mu)^n=\lambda^n+\mu^n$ for all $\lambda,\mu\in F$. Thus there exists $\tau\in \Aut(F)$ such that $\delta=\varepsilon^\tau$, and (i) holds when $\varepsilon\in \mathbb{F}_q^*$.

Suppose instead that $\varepsilon\notin \mathbb{F}_q$. As in the statement of the lemma, let $E:=F\cap \mathbb{F}_q$.  Now $f(X)\in E[X]$ is irreducible, so $f(X)$ is the minimal polynomial of $\varepsilon$ over $E$, but $F=E(\varepsilon)$, so $[F:E]=2$. Thus $\mathcal{B}$ is an $E$-basis for $F$, and  there exist $a,b\in E$ such that $\varphi_\varepsilon$ maps $\varepsilon$ to $a\cdot 1+b\cdot \varepsilon$. Now $\varepsilon$ is a root of $X^2-bX-a\in E[X]$, so  $f(X)=X^2-bX-a$. 
Since $X^2-bX-a$ is also the characteristic polynomial of $[\varphi_\varepsilon]_{\mathcal{B}}$, we conclude that $[\varphi_\varepsilon]_{\mathcal{B}}$ is conjugate in $\GL_2(\mathbb{F}_q(\varepsilon))$ to $\diag(\varepsilon,\delta)$. Hence $h$ and $[\varphi_\varepsilon]_{\mathcal{B}}$ are conjugate in $\GL_2(\mathbb{F}_q(\varepsilon))$, but both matrices have entries in $\mathbb{F}_q$, so they are also conjugate in $\GL_2(q)$.
\end{proof}

\begin{proof}[Proof of Proposition~\emph{\ref{prop:(S0)bad}}]
Let $(V,\mathcal{L})$ be a $G$-affine proper partial linear space for which $0$ and $x$ are collinear. By Lemma~\ref{lemma:(S0)bad}, $\mathcal{L}=L^G$ where $L=\{x(\lambda\varphi): \lambda\in F\}$ for some field $F$  with $|F|>2$ and injective ring homomorphism $\varphi:F\to \M_2(r)$ such that $(F\varphi)^*\leq \GL_2(q)$ or $K^*$. If $(F\varphi)^*\leq K^*$, then (i) holds, so we may assume that $(F\varphi)^*\leq \GL_2(q)$. Let $(F\varphi)^*=\langle h\rangle$, and let $\varepsilon$ be an eigenvalue of $h$ in some field extension of $\mathbb{F}_q$. By Lemma~\ref{lemma:(S0)badmore},  $F^*\simeq \langle \varepsilon\rangle$, we may view $F$ as a subfield of $\mathbb{F}_q(\varepsilon)$,  and Lemma~\ref{lemma:(S0)badmore}(i) or (ii) holds.
First suppose that  Lemma~\ref{lemma:(S0)badmore}(i) holds. Then $\varepsilon\in \mathbb{F}_q$ and $h=\diag(\varepsilon,\varepsilon^\tau)^g$ for some $\tau\in \Aut(F)$ and $g\in \GL_2(q)$. If $\tau=1$, then (i) holds. Otherwise, (ii) holds. Thus we may assume that Lemma~\ref{lemma:(S0)badmore}(ii) holds. Let $E:=F\cap \mathbb{F}_q$. Then $[F:E]=2$, and $h$ is conjugate in $\GL_2(q)$ to $[\varphi_\varepsilon]_{\mathcal{B}}$,  where  
$\varphi_\varepsilon\in \End_E(F)$ is defined by $  v\mapsto \varepsilon v$ for all $v\in F$, and $\mathcal{B}=\{1,\varepsilon\}$ is an $E$-basis of $F$. Thus (iii) holds.

Conversely, suppose that $\mathcal{L}=L^G$ where $L$ is given by (i), (ii) or (iii). In each case, we claim that  that $L=\{x(\lambda \varphi): \lambda\in F\}$ for some injective ring homomorphism $\varphi:F\to \M_2(r)$ such that $(F\varphi)^*\leq \GL_2(q)$ or $K^*$.  If the claim holds, then since $|F|>2$ in (i)--(iii),  Lemma~\ref{lemma:(S0)bad} implies that $(V,\mathcal{L})$ is a $G$-affine proper partial linear space for which $0$ and $x$ are collinear, as desired. 
If (i) holds, then  we define $\varphi$ by $\lambda\mapsto \diag(\lambda,\lambda)$ for all $\lambda\in F$, and the claim holds.
If (ii) holds, then we define $\varphi$ by $\lambda\mapsto \diag(\lambda,\lambda^\tau)^g$ for all $\lambda\in F$, and the claim holds.
If (iii) holds, then we define $\varphi$ by $\lambda\mapsto [\varphi_\lambda]_\mathcal{B}^g$ for all $\lambda\in F$, and the claim holds.
\end{proof}

Proposition~\ref{prop:(S0)bad}  provides a classification of the $G$-affine proper partial linear spaces for which $0$ and $(1,\zeta)$ are collinear when $G_0$ lies in class (S0) and   $\GL_2(q)\circ \mathbb{F}_{q^3}^*\leq G_0$. However,  we suspect that the situation is much more complicated in general. 
To illustrate this, we now provide a complete classification of the partial linear spaces that arise when $q=4$. 

\begin{example}
\label{example:hardS0}
 Let $q:=4$, $r:=q^3$, $K:=\mathbb{F}_{r}$,  $\sigma:=\sigma_r$, $\zeta:=\zeta_{r}$,  $x:=(1,\zeta)$ and $V:=K^2$.
Let $S:=\GL_2(q)$. By a computation in {\sc Magma},    the point stabiliser of an   affine permutation group of rank~$3$ on $V$ in (S0) is 
either $S\langle\zeta,\sigma^i\rangle$ for $i\in \{1,2,3,6\}$, or conjugate to  $S\langle \zeta^3,\zeta^7\sigma^2\rangle$. Note that $S\langle \zeta^3,\zeta^7\sigma^2\rangle$ does not contain $S\circ K^*=S\langle\zeta\rangle$. Recall that we may view   $S\langle \zeta,\sigma\rangle$ as a member of class (T6) and therefore as a subgroup of $W_0:=(S\otimes \GL_3(q)){:}\Aut(\mathbb{F}_q)$. 
 
  We claim that Table~\ref{tab:S0examplesq=4} contains a complete list of representatives for the isomorphism classes of $X$-affine proper partial linear spaces for which $0$ and $x$ are collinear, where $X$ is an affine permutation group of rank~$3$ on $V$  in (S0). For each such partial linear space $\mathcal{S}$, we prove that $\Aut(\mathcal{S})=V{:}\Aut(\mathcal{S})_0$, and  we list  the following in Table~\ref{tab:S0examplesq=4}: the line-size $k$; $\Aut(\mathcal{S})_0$;
 and when  $S\langle\zeta\rangle\leq \Aut(\mathcal{S})_0$, a reference to 
 Proposition~\ref{prop:(S0)bad}. 
 
 \begin{table}[!h]
\centering
\begin{tabular}{ c c c c c  }
\hline
 $k$ & $\mathcal{S}$ &  $\Aut(\mathcal{S})_0$ & \ref{prop:(S0)bad} &   Notes \\
\hline 
$4$ & $\mathcal{S}_{4,1}$ & $W_0$ & (i) & $K$-dependent\\
& $\mathcal{S}_{4,2}$ & $S\langle \zeta,\sigma\rangle$ & (ii) & \\
& $\mathcal{S}_{4,3}$ & $S\langle \zeta,\sigma^3\rangle$ & (ii) & \\
& $\mathcal{S}_{4,4}$ & $S\langle \zeta\rangle$ & (ii) &  \\
& $\mathcal{S}_{4,\ell}$ & $S\langle \zeta^3,\zeta^7\sigma^2\rangle$ & n/a  &  $5\leq \ell\leq 7$ \\
\hline 
$8$ & $\mathcal{S}_{8,1}$ & $S\langle \zeta,\sigma\rangle$ & (i) & $K$-dependent \\
\hline 
$16$ & $\mathcal{S}_{16,1}$ & $S\langle \zeta\rangle$ & (iii) & \\
& $\mathcal{S}_{16,\ell}$ & $S\langle \zeta^3,\zeta^7\sigma^2\rangle$ & n/a & $2\leq \ell\leq 3$ \\
\hline 
$64$ & $\mathcal{S}_{64,1}$ & $S\langle \zeta,\sigma\rangle$ & (i) & $K$-dependent \\
\hline 
\end{tabular}
\caption{Partial linear spaces $\mathcal{S}$ when $q=4$ and $0$ is collinear with $(1,\zeta_{64})$}
\label{tab:S0examplesq=4}
\end{table}

 Observe in Table~\ref{tab:S0examplesq=4} that there are  five pairwise non-isomorphic partial linear spaces $\mathcal{S}$ for which $\Aut(\mathcal{S})_0=S\langle \zeta^3,\zeta^7\sigma^2\rangle$, none of which are  isomorphic to a partial linear space described by  Proposition~\ref{prop:(S0)bad}. Further,  there are three pairwise non-isomorphic partial linear spaces  that are described by Proposition~\ref{prop:(S0)bad}(ii), necessarily with $F=\mathbb{F}_4$ and $\tau=\sigma_4$.

 Now we prove the claim. Let $G:=V{:}S\langle\zeta\rangle$, $H:=V{:}S\langle \zeta^3,\zeta^7\sigma^2\rangle$, $N:=V{:}S\langle \zeta,\sigma\rangle$ and $W:=V{:}W_0$. Then $N_{N_0}(H_0)=H_0\langle \sigma^2\rangle=G_0\langle\sigma^2\rangle$ and $|G_0|=|H_0|$.   Let $\mathcal{G}:=\{H,G,G\langle\sigma^2\rangle,G\langle\sigma^3\rangle,N\}$.
For $X\in \{G,H\}$, let $\mathcal{B}_X$ denote the set of non-trivial blocks $B$ of $X_0$ on $x^{N_0}$ such that $x\in B$ and  $B\cup \{0\}$ is an $\mathbb{F}_2$-subspace of $V$. Let $Y\in \mathcal{G}$, and recall
from Lemmas~\ref{lemma:necessary} and~\ref{lemma:A5add} that if  $(V,\mathcal{L})$ is a $Y$-affine proper partial linear space and $L\in \mathcal{L}_0\cap \mathcal{L}_x$, then $L^*\in \mathcal{B}_X$ for some $X\in \{G,H\}$. Conversely, if $X\in \{G,H\}$ and 
 $B\in \mathcal{B}_X$, then $(V,(B\cup \{0\})^X)$ is an $X$-affine proper partial linear space by Lemmas~\ref{lemma:transitive} and~\ref{lemma:sufficient}. Therefore, in order to find all of the relevant partial linear spaces, it suffices to determine the sets $\mathcal{B}_X$. By a computation in {\sc Magma}, $\mathcal{B}_G=\mathcal{B}_H$, and $\mathcal{B}_G$ consists of 19 blocks. Three of these correspond to $K$-dependent $N$-affine partial linear spaces with line-sizes $4$, $8$ and $64$, which we denote by  $\mathcal{S}_{4,1}$, $\mathcal{S}_{8,1}$ and $\mathcal{S}_{64,1}$, respectively. Note that these are described by Proposition~\ref{prop:(S0)bad}(i) with $F=\mathbb{F}_k$ where $k$ is the line-size. By a computation in {\sc Magma}, $\Aut(\mathcal{S}_{8,1})=\Aut(\mathcal{S}_{64,1})=N$. We will determine $\Aut(\mathcal{S}_{4,1})$ below.

 Of the remaining sixteen blocks in $\mathcal{B}_G$, six have size $15$ and ten have size $3$. 
 First we consider the blocks with size $15$.  By a computation in {\sc Magma}, the  group $\langle\sigma\rangle$ permutes the six  $G$-affine partial linear spaces with line-size $16$ transitively, so all of these examples are isomorphic; let $\mathcal{S}_{16,1}$ denote one of these. Further, the group $\langle\sigma\rangle$ does not act on the six  $H$-affine partial linear spaces with line-size $16$, while   $\langle\sigma^2\rangle$ has two orbits of size $3$; let $\mathcal{S}_{16,2}$ and $\mathcal{S}_{16,3}$ denote  orbit representatives of these. By a computation in {\sc Magma}, $\Aut(\mathcal{S}_{16,1})=G$, while $\Aut(\mathcal{S}_{16,2})=\Aut(\mathcal{S}_{16,3})=H$, and the partial linear spaces $\mathcal{S}_{16,1}$, $\mathcal{S}_{16,2}$ and $\mathcal{S}_{16,3}$ are pairwise non-isomorphic. Note that  $\mathcal{S}_{16,1}$ must have the  form of  Proposition~\ref{prop:(S0)bad}(iii) with $F=\mathbb{F}_{16}$.  
 
 It remains to consider the blocks with size $3$. 
 Recall that $\mathcal{S}_{4,1}$ denotes the $K$-dependent $N$-affine partial linear space with line-size $4$. 
 By a computation in {\sc Magma}, the group $\langle\sigma\rangle$ has orbit sizes $1$, $3$ and $6$ on the ten $G$-affine $K$-independent  partial linear spaces with line-size $4$. In particular, $N$ is an automorphism group of the partial linear space  in  the orbit of size $1$; let $\mathcal{S}_{4,2}$ denote this partial linear space, and let $\mathcal{S}_{4,3}$ and $\mathcal{S}_{4,4}$ be  representatives of the orbits with size  $3$ and $6$, respectively. Note that $\mathcal{S}_{4,j}$ must have the form of Proposition~\ref{prop:(S0)bad}(ii) with $F=\mathbb{F}_4$ and $\tau=\sigma_4$ for $2\leq j\leq 4$. By a computation in {\sc Magma}, $\langle\sigma\rangle$ does not act on the nine $H$-affine $K$-independent partial linear spaces with line-size $4$ that are not equal to $\mathcal{S}_{4,2}$, while     $\langle\sigma^2\rangle$ has orbit sizes  $3$, $3$ and $3$; let $\mathcal{S}_{4,5}$, $\mathcal{S}_{4,6}$ and $\mathcal{S}_{4,7}$  be orbit  representatives of these. 
 
 Let $j\in \{1,\ldots,7\}$.  We wish to determine $\Aut(\mathcal{S}_{4,j})$, but we are unable to do so directly  using {\sc Magma}.
 By Lemma~\ref{lemma:Taut}, $\Aut(\mathcal{S}_{4,j})\leq W$. In particular, $\Aut(\mathcal{S}_{4,1})=W$, so we assume that $j\neq 1$.  
Recall that $S\unlhd W_0$ and   $S\langle \zeta^3\rangle\leq \Aut(\mathcal{S}_{4,j})_0$.  
 By a computation in {\sc Magma}, 
  there are exactly four  subgroups of $W_0$ (up to conjugacy)  that contain $S$ and have two orbits on $V^*$ but are not conjugate to a subgroup of $N_0$. 
  Further, if  $B$ is a block of one of these four groups on~$x^{N_0}$ for which $x\in B$ and $B\cup \{0\}$ is an $\mathbb{F}_2$-subspace of $V$ with size~$4$, then $B\cup \{0\}=\langle x\rangle_{\mathbb{F}_4}$. 
However, $j\neq 1$, so $\Aut(\mathcal{S}_{4,j})_0$ is conjugate in $W_0$ to a subgroup of $N_0$. By  a computation in  {\sc Magma}, $N_0$ is the only conjugate of $N_0$ in $W_0$ that contains $\zeta^3$. Thus  $\Aut(\mathcal{S}_{4,j})_0\leq N_0$. Now we can deduce from  the actions of $\langle\sigma\rangle$ given above  that  $\Aut(\mathcal{S}_{4,j})$ is as described in Table~\ref{tab:S0examplesq=4}.  
By a computation in {\sc Magma}, $G\langle\sigma^2\rangle$ is the normaliser of $H$ in $S_{4096}$, and  $G$ and $H$ are not permutation isomorphic. Thus the partial linear spaces $\mathcal{S}_{4,j}$ are pairwise non-isomorphic for $1\leq j\leq 7$.
 \end{example}

\section{The imprimitive class (I)}
\label{s:(I)}

Let $V:=V_n(p)\oplus V_n(p)$ where $n\geq 1$ and $p$ is prime. 
We adopt the following notation throughout this section. 
  The stabiliser of the decomposition of $V$ in $\GL_{2n}(p)$ is $\GL_n(p)\wr\langle \tau\rangle $, where $\tau$ is the involution in $\GL_{2n}(p)$  defined by  $(u_1,u_2)^\tau =(u_2,u_1)$ for all $u_1,u_2\in V_n(p)$. We  write $V_1:=\{(u,0): u \in V_n(p)\}$ and $V_2:=\{(0,u):u\in V_n(p)\}$, and for $G_0\leq \GL_n(p)\wr\langle \tau\rangle$ and $i\in \{1,2\}$, we write
$G_0^i$ for the image of the projection of $G_{0,V_1}=G_0\cap (\GL_n(p)\times \GL_n(p))$ onto the $i$-th factor of $\GL_n(p)\times \GL_n(p)$. 

If $G$ is an affine permutation group of rank~$3$  on $V$, then the orbits of $G_0$ on $V^*$ are $V_1^*\cup V_2^*$ and $V_n(p)^*\times V_n(p)^*$. In particular,  $G$ is primitive if and only if $p^n>2$ by Lemma~\ref{lemma:primitive}.
 Further, $G_0^i$ is transitive on $V_n(p)^*$ for both $i$, and $G_{0,V_1}$ is an index $2$ subgroup of $G_0$. Now $G_0=G_{0,V_1}\langle (t,s)\tau\rangle$ for some $t,s\in\GL_n(p)$. Conjugating $G_0$ by $(s^{-1},1)$ if necessary, we may assume that $s=1$, in which case $G_0^1=G_0^2$ and $G_0\leq G_0^1\wr \langle\tau\rangle$ since $((t,1)\tau)^2=(t,t)$. 

 If $H_0\leq \GammaL_m(q)$ where $q^m=p^n$ and $H_0$ is transitive on $V_n(p)^*$, then $V{:}(H_0\wr\langle \tau\rangle)$ is a rank~$3$ subgroup of $V{:}(\GL_n(p)\wr \langle \tau\rangle)$.  Note that  $\GammaL_m(q)\wr\langle\tau\rangle$ is not a subgroup of $\GammaL_{2m}(q)$ when $q\neq p$. Indeed, 
 $(\GammaL_m(q)\wr\langle\tau\rangle)\cap \GammaL_{2m}(q)=\{(g_1\sigma,g_2\sigma):g_1,g_2\in \GL_m(q),\sigma\in\Aut(\mathbb{F}_q)\}\langle \tau\rangle$,  which is isomorphic to $(\GL_m(q)\wr \langle \tau\rangle){:}\Aut(\mathbb{F}_q)$.
 
In previous sections, the rank~$3$ groups under consideration were all subgroups of $\AGammaL_m(q)$ for some field extension $\mathbb{F}_q$ of $\mathbb{F}_p$, in which case it was logical to consider  $\mathbb{F}_q$-dependent and $\mathbb{F}_q$-independent partial linear spaces separately. However, since this does not naturally occur for all groups in a particular subclass of (I), we will not make this distinction here. Nevertheless, we will see that there are two infinite families of $\mathbb{F}_q$-dependent proper partial linear spaces for each field $\mathbb{F}_q$ such that $q^m=p^n$. In fact,  one of these examples admits $V{:}(\GammaL_m(q)\wr\langle\tau\rangle)$ as an automorphism group, even though $\GammaL_m(q)\wr\langle\tau\rangle$ is not $\mathbb{F}_q$-semilinear  when $q\neq p$.
 
 In this section, we prove the following two propositions. For the first, recall the definition of the cartesian product of two partial linear spaces that was given in \S\ref{ss:plsnew}, and observe that the partial linear space in (i) below is that of Example~\ref{example:grid}.

 \begin{prop}
\label{prop:(I)good}
Let $G$ be an affine permutation group of rank~$3$ on $V:=V_n(p)\oplus V_n(p)$ where $n\geq 1$, $p$ is prime and $G_0\leq \GL_n(p)\wr \langle \tau\rangle$. 
Then $\mathcal{S}$ is a $G$-affine proper partial linear space  in which $\mathcal{S}(0)=V_1^*\cup V_2^*$ if and only if one of the following holds.
\begin{itemize}
\item[(i)] $\mathcal{S}$ is the  $p^n\times p^n$ grid where $p^n>2$.
\item[(ii)] $\mathcal{S}=\AG_m(q)\cprod\AG_m(q)$ and $G_0\leq \GammaL_m(q)\wr\langle \tau\rangle$, where $q^m=p^n$, $m\geq 2$ and $q>2$.
\item[(iii)] $\mathcal{S}\simeq \mathcal{S}'\cprod\mathcal{S}'$ and one of the following holds.
\begin{itemize}
\item[(a)]  $\mathcal{S}'$ is  the  nearfield plane of order $9$ and $G_0\leq N_{\GL_4(3)}(D_8\circ Q_8)\wr\langle\tau\rangle$. Here $p^n=3^4$.
\item[(b)]  $\mathcal{S}'$ is the Hering plane of order $27$ or either of the two Hering spaces  with line-size $9$, and $G_0=\SL_2(13)\wr\langle\tau\rangle$. Here $p^n=3^6$.
\end{itemize}
\end{itemize}
\end{prop}

\begin{remark}
\label{remark:Hexample}
The partial linear space  $\mathcal{S}:=\AG_m(q)\cprod\AG_m(q)$ in Proposition~\ref{prop:(I)good}(ii) has line set $\{\langle u\rangle_{\mathbb{F}_q}+v : u\in V_1^*\cup V_2^*,v\in V\}$, so $\mathcal{S}$ is an $\mathbb{F}_q$-dependent partial linear space. In particular, $\mathcal{S}$ is described in Example~\ref{example:AG} with respect to the triple $(\AGL_m(q)\wr \langle\tau\rangle,2m,q)$, which satisfies Hypothesis~\ref{hyp:AGgroups}(iii). However, $\AGammaL_m(q)\wr\langle \tau\rangle\leq \Aut(\mathcal{S})$, and if $q$ is not prime, then   the triple $(\AGammaL_m(q)\wr\langle \tau\rangle,2m,q)$ does not satisfy Hypothesis~\ref{hyp:AGgroups}  by Remark~\ref{remark:hypfail}.
\end{remark}

Recall that for a vector space $V$ and $H\leq \GL(V)$, we write $C_V(H)$ for the subspace of $V$ consisting of those vectors in $V$ that are fixed by every element of $H$.

\begin{prop}
\label{prop:(I)bad}
Let $G$ be an affine permutation group of rank~$3$ on $V:=V_n(p)\oplus V_n(p)$ where $n\geq 1$, $p$ is prime, $G_0\leq \GL_n(p)\wr \langle \tau\rangle$ and $G_0$  is not a subgroup of $\GammaL_1(p^n)\wr\langle\tau\rangle$. If $\mathcal{S}$ is a $G$-affine proper partial linear space  in which $\mathcal{S}(0)=V_n(p)^*\times V_n(p)^*$, then $\mathcal{S}$ is isomorphic to a partial linear space $(V,\mathcal{L})$ where  one of the following holds.
\begin{itemize}
\item[(i)] $\mathcal{L}=\{\langle v\rangle_{\mathbb{F}_q} +w: v\in V_n(p)^*\times V_n(p)^*,w\in V\}$  and $G_0\leq (\GammaL_m(q)\wr \langle\tau\rangle)\cap\GammaL_{2m}(q)$, where $q^m=p^n$, $m\geq 2$ and $q>2$.
\item[(ii)] $\mathcal{L}=\{\{(v,v^r) +w : v \in V_n(p)\} : r\in R, w\in V\}$ 
for $R\leq \GL_n(p)$,  where $n=2$ and 
$(p,R)$ is one of 
$(3,Q_8)$, 
$(5,\SL_2(3))$, 
$(7,2\nonsplit S_4^-)$, 
$(11,\SL_2(3)\times C_5)$, 
$(23,2\nonsplit S_4^-\times C_{11})$, 
$(11,\SL_2(5))$, 
$(29,\SL_2(5)\times C_7)$ or 
$(59,\SL_2(5)\times C_{29})$. 

\item[(iii)] $\mathcal{L}=\{(v,v^\alpha): v\in M\}^{V{:}H_0\times H_0}$, where $p^n=3^4$, $M=C_{V_4(3)}(H_{0,u})\simeq V_2(3)$ for some $u\in V_4(3)^*$,  $\alpha\in \{1,\beta\}$   where    $\beta$ is  chosen to be one of the two involutions in $\GL(M)_u\setminus (N_{\GL_4(3)}(H_0))_M^M$, and $H_0\leq \GL_4(3)$   is  one of  $2\nonsplit S_5^-{:}2$ or $(D_8\circ Q_8).C_{5}$.

\item[(iv)] $\mathcal{L}=\{(v,v^\alpha): v\in M\}^{V{:}H_0\times H_0}$, where $p^n=3^6$, $M=C_{V_6(3)}(H_{0,u})\simeq V_2(3)$ for some $u\in V_6(3)^*$,  $\alpha\in \GL(M)_u\simeq S_3$ such that $\alpha^2=1$, and $H_0=\SL_2(13)\leq \GL_6(3)$.

\item[(v)] $\mathcal{L}=\{\{(v,v^{g})+w : v\in V_n(p)\} : g\in \SL_2(3)\alpha\cup \SL_2(3)\alpha^2r,\ w\in V\}$, where $p^n=7^2$, $2\nonsplit S_4^-=\SL_2(3)\langle r\rangle$, and $\alpha$ is chosen to be one of the non-trivial elements of  $N_{\GL_2(7)}(2\nonsplit S_4^-)_u\simeq C_3$ for some $u\in V_2(7)^*$.

\end{itemize}
Conversely, if one of (i)--(v) holds, then  $(V,\mathcal{L})$ is an affine proper partial linear space.
\end{prop}

 We are unable to  provide a classification of the $G$-affine proper partial linear spaces~$\mathcal{S}$ for which $\mathcal{S}(0)=V_n(p)^*\times V_n(p)^*$ for groups $G$ in class (I0), which is why we assume that $G_0$ is not a subgroup of $ \GammaL_1(p^n)\wr \langle\tau\rangle$ in Proposition~\ref{prop:(I)bad}. Note that these conditions on $G$ and $\mathcal{S}$  are precisely those of  Theorem~\ref{thm:main}(iv)(b).

Proposition~\ref{prop:(I)good} is a direct consequence of Proposition~\ref{prop:(I)goodmore}, proved below, while Proposition~\ref{prop:(I)bad} will follow (with some work) from Proposition~\ref{prop:(I)badmore}, where detailed information about the groups involved is given. See also Examples~\ref{example:(I)dep}--\ref{example:(I)spor}. 
First we consider the orbit $V_1^*\cup V_2^*$.

\begin{lemma}
\label{lemma:A2good}
Let $G$ be an affine permutation group of rank~$3$ on $V:=V_n(p)\oplus V_n(p)$ where $n\geq 1$, $p$ is prime and $G_0=G_{0,V_1}\langle(t,1)\tau\rangle\leq \GL_n(p)\wr\langle\tau\rangle$.  The following are equivalent.
\begin{itemize}
\item[(i)] $\mathcal{S}$ is a $G$-affine proper partial linear space  in which $\mathcal{S}(0)=V_1^*\cup V_2^*$.
\item[(ii)] $\mathcal{S}=\mathcal{S}'\cprod\mathcal{S}'$ where $\mathcal{S}'$ is a linear space with point set $V_n(p)$ and line-size at least $3$ admitting  $V_n(p){:}G_0^1$ as a $2$-transitive group of automorphisms.
\end{itemize}
\end{lemma}

\begin{proof}
Suppose that (i) holds and let $\mathcal{L}$ be the line set of $\mathcal{S}$. Let $L\in\mathcal{L}_0$ and $B:=L^*$. Without loss of generality, we may assume that $(u,0)\in L$ for some $u\in V_n(p)^*$.  If there exists $(0,v)\in B$ for some $v\in V_n(p)^*$, then $(u,-v)\in V_1^*\cup V_2^*$ by Lemma~\ref{lemma:basic}(i), a contradiction. Thus $L\subseteq V_1$, so $L=\{(v,0):v \in L'\}$ for some $L'\subseteq V_n(p)$.  Let $\mathcal{S}':=(V_n(p),\mathcal{L}')$ where $\mathcal{L}'$ is defined to be the set of subsets $M'$ of $V_n(p)$ for which $\{(v,0): v\in M'\}\in\mathcal{L}$. Then $\mathcal{S}'$ is a  linear space with line-size at least $3$ admitting $V_n(p){:}G_0^1$ as a $2$-transitive  group of automorphisms. It remains to show that $\mathcal{S}=\mathcal{S}'\cprod\mathcal{S}'$. Let $M\in \mathcal{L}$. Then $M=L^g+(v_1,v_2)$ for some $g\in G_0$ and $(v_1,v_2)\in V$. Now $g=(g_1,g_2)\tau^i$ for some $g_1,g_2\in\GL_n(p)$ and $i\in \{1,2\}$. If $i=2$, then $M=\{(v,v_2): v \in (L')^{g_1}+v_1\}$, and if $i=1$, then $M=\{(v_1,v): v \in (L')^{g_1}+v_2\}$. Now $g_1\in G_0^1$ since $G_0\leq G_0^1\wr\langle\tau\rangle$, so  $(L')^{g_1}+v_j\in\mathcal{L}'$ for $j=1,2$. Thus $M\in\mathcal{L}'\cprod\mathcal{L}'$. On the other hand, a typical line of $\mathcal{S}'\cprod\mathcal{S}'$ has the form $\{(v,w):v\in M'\}$ or $\{(w,v):v\in M'\}$  for some $w\in V_n(p)$ and $M'\in\mathcal{L}'$. Now $\{(v,0): v\in M'\}\in\mathcal{L}$, so $\{(v,w):v\in M'\}=\{(v,0): v\in M'\}+(0,w)\in \mathcal{L}$. Further, $g:=(1,t^{-1})\tau\in G_0$, so  
 $\{(w,v):v\in M'\}= \{(v,w^t):v\in M'\}^g\in\mathcal{L}$. Thus (ii) holds. 

Conversely, suppose that (ii) holds. By Lemma~\ref{lemma:cartesian}, $\mathcal{S}$ is a proper partial linear space with point set $V$ in which $\mathcal{S}(0)=V_1^*\cup V_2^*$, and $\Aut(\mathcal{S})\geq (V_n(p){:}G_0^1)\wr\langle\tau\rangle\simeq  V{:}(G_0^1\wr\langle\tau\rangle)\geq G$. Thus (i) holds. 
\end{proof}

Using Kantor's classification~\cite{Kan1985} of the linear spaces with a $2$-transitive group of automorphisms (see Theorem~\ref{thm:Kantor}), we prove the following, from which Proposition~\ref{prop:(I)good} immediately follows. Note that this result includes the groups in class (I0); in fact, its proof does not require Theorem~\ref{thm:A2groups}. For the following, recall the definition of   $\zeta_9$ and $\sigma_9$ from  \S\ref{ss:basicsvs}.

 \begin{prop}
\label{prop:(I)goodmore}
Let $G$ be an affine permutation group of rank~$3$ on $V:=V_n(p)\oplus V_n(p)$ where $n\geq 1$, $p$ is prime and $G_0=G_{0,V_1}\langle(t,1)\tau\rangle\leq \GL_n(p)\wr \langle \tau\rangle$. 
Then $\mathcal{S}:=(V,\mathcal{L})$ is a $G$-affine proper partial linear space in which $\mathcal{S}(0)=V_1^*\cup V_2^*$ if and only if  one of the following holds.
\begin{itemize}
\item[(i)] $\mathcal{L}=\{V_i+v: i\in \{1,2\}, v\in V\}$  where $p^n\neq 2$.
\item[(ii)] $\mathcal{L}=\{\langle u\rangle_F +v : u\in V_1^*\cup V_2^*, v\in V\}$ for some subfield $F$ of $\mathbb{F}_q$ where $q^m=p^n$,  $|F|\geq 3$ and $G_0\leq \GammaL_m(q)\wr\langle \tau\rangle$. If $m=1$, then $F\neq \mathbb{F}_q$.
\item[(iii)] $\mathcal{L}=\mathcal{L}'\cprod\mathcal{L}'$ where $(V_n(p),\mathcal{L}')$ is a linear space and one of the following holds.
\begin{itemize}
\item[(a)] $G_0\leq N_{\GL_4(3)}(D_8\circ Q_8)\wr\langle\tau\rangle$ and $p^n=3^4$, and if $\mathcal{N}$ denotes the line set of the nearfield plane of order $9$, then either $\mathcal{L}'=\mathcal{N}$, or $G_0\leq (\SL_2(5)\langle \zeta_9\sigma_9\rangle)\wr\langle\tau\rangle$ where $\SL_2(5)\leq \SL_2(9)$, in which case  $\mathcal{L}'=\mathcal{N}^g$ for some $g\in \{1,\zeta_9^2\}$ and $G_0$ is $(\SL_2(5)\langle \zeta_9\sigma_9)\rangle\wr\langle\tau\rangle$ or $(\SL_2(5)\times \SL_2(5))\langle (\zeta_9\sigma_9,1)\tau\rangle$. 

\item[(b)]  $\mathcal{L}'$ is the line set of the Hering plane of order $27$ or either  of the two Hering spaces  with line-size $9$, and $G_0=\SL_2(13)\wr\langle\tau\rangle$ where $p^n=3^6$.
\end{itemize}
\end{itemize}
\end{prop}

\begin{proof}
Recall that $t\in G_0^1=G_0^2$ and $G_0\leq G_0^1\wr\langle\tau\rangle$. If one of (i)--(iii) holds, then $(V,\mathcal{L})$ is a $G$-affine proper partial linear space by Lemma~\ref{lemma:A2good} and Theorem~\ref{thm:Kantor}. 

Conversely, let $\mathcal{S}:=(V,\mathcal{L})$ be a $G$-affine proper partial linear space in which $\mathcal{S}(0)=V_1^*\cup V_2^*$. By Lemma~\ref{lemma:A2good}, 
$\mathcal{L}=\mathcal{L}'\cprod\mathcal{L}'$ where $\mathcal{S}':=(V_n(p),\mathcal{L}')$ is a linear space with  line-size at least $3$ admitting  $V_n(p){:}G_0^1$ as a $2$-transitive group of automorphisms. Thus $\mathcal{S}'$ is given by Theorem~\ref{thm:Kantor}. If Theorem~\ref{thm:Kantor}(i) holds for $\mathcal{S}'$, then  $\mathcal{L}'=\{V_n(p)\}$, and (i) holds.

Suppose that Theorem~\ref{thm:Kantor}(ii) holds for $\mathcal{S}'$. Then $G_0^1\leq \GammaL_m(q)$ where $q^m=p^n$ and $\mathcal{L}'=\{\langle u\rangle_F+v : u\in V_n(p)^*, v\in V_n(p)\}$ where $F$ is a subfield of $\mathbb{F}_q$ with $|F|\geq 3$. If $m=1$ and $F=\mathbb{F}_q$, then (i) holds; otherwise, (ii) holds. 

Suppose that Theorem~\ref{thm:Kantor}(iv) holds for $\mathcal{S}'$. Then $G_0^1=\SL_2(13)$ where $p^n=3^6$ and $\mathcal{S}'$ is the Hering plane of order $27$ or either of the two Hering spaces with line-size $9$. It is straightforward to verify that $G_0= \SL_2(13)\wr\langle\tau\rangle$, so (iii)(b) holds. 

Suppose that Theorem~\ref{thm:Kantor}(iii) holds for $\mathcal{S}'$. Now $G_0^1\leq N_{\GL_4(3)}(D_8\circ Q_8)$ where $p^n=3^4$. Let $\mathcal{N}$ be the line set of the nearfield plane of order $9$. By Theorem~\ref{thm:Kantor}, either $D_8\circ Q_8\unlhd G_0^1$ and $\mathcal{L}'=\mathcal{N}$, or $G_0^1=\SL_2(5)\langle \zeta_9\sigma_9\rangle\leq \GammaL_2(9)$ and $\mathcal{L}'=\mathcal{N}^g$ for some $g\in \{1,\zeta_9^2\}$. In either case, $G_0\leq N_{\GL_4(3)}(D_8\circ Q_8)\wr\langle\tau\rangle$. If $G_0\leq (\SL_2(5)\langle \zeta_9\sigma_9)\rangle\wr\langle\tau\rangle$, then since $G$ has rank~$3$, it can be verified that $G_0=(\SL_2(5)\langle \zeta_9\sigma_9\rangle)\wr\langle\tau\rangle$ or $(\SL_2(5)\times \SL_2(5))\langle (\zeta_9\sigma_9,1)\tau\rangle$. 
Thus (iii)(a) holds. 
\end{proof}

\begin{example}
\label{example:(I)goodnearfield}
Let $\mathcal{N}$ be the nearfield plane of order $9$ (see~\S\ref{ss:nearfield}). By Propositions~\ref{prop:(I)good}(iii)(a) or~\ref{prop:(I)goodmore}(iii)(a),  $\mathcal{N}\cprod \mathcal{N}$ is a $G$-affine proper partial linear space with point set $V:=V_4(3)\oplus V_4(3)$ where $G:=V{:}(N_{\GL_4(3)}(D_8\circ Q_8)\wr\langle \tau\rangle)$. By Remark~\ref{remark:2trans},  $N_{\GL_4(3)}(D_8\circ Q_8)=D_8\circ Q_8\circ 2\nonsplit S_5^-$. Using {\sc Magma}, we determine that $\Aut(\mathcal{N}\cprod\mathcal{N})=G$. Note  also that the lines of $\mathcal{N}\cprod \mathcal{N}$ are affine subspaces of $V$.
\end{example}

We wish to determine the full automorphism groups of the partial linear spaces in Proposition~\ref{prop:(I)good}(iii)(b), but we are unable to do this using {\sc Magma}.  Instead, we use the following consequence of Theorem~\ref{thm:primrank3plus}. We will also use this result to determine the full automorphism groups of  some of the partial linear spaces from Proposition~\ref{prop:(I)bad}. 

\begin{lemma}
\label{lemma:(I)aut}
Let $V:=V_n(p)\oplus V_n(p)$ where $n\geq 2$ and $p$ is an odd prime. Let $G$ be an insoluble  affine permutation group of rank~$3$ on $V$  where $G_0\leq \GL_n(p)\wr\langle\tau\rangle$. Let 
$\mathcal{S}$ be a $G$-affine proper partial linear space with point set $V$ that is not isomorphic to the $p^n\times p^n$ grid. Then $\Aut(\mathcal{S})$ is an affine  permutation group  on $V$, and $\Aut(\mathcal{S})_0 \leq \GL_n(p)\wr\langle\tau\rangle$. 
\end{lemma}

\begin{proof}
Let $H:=\Aut(\mathcal{S})$. 
By Theorem~\ref{thm:primrank3plus}, $H$ is an affine permutation group on $V$. Since $G_0$ is not soluble, $H_0$ does not belong to class (R0). If $H_0$ belongs to one of the classes (T), (S) or (R1)-(R5), then $p$ divides one of the subdegrees of $H$ (see Table~\ref{tab:subdegree}), but $H$ has subdegrees $2(p^n-1)$ and $(p^n-1)^2$, so $p=2$, a contradiction. By   Tables~\ref{tab:E} and \ref{tab:AS}, $H$ belongs to class (I), so  $H_0$ stabilises some (internal) direct sum decomposition $V=U_1\oplus U_2$ where $U_1$ and $U_2$ are $n$-dimensional subspaces of $V$. Now $G_0$ also stabilises this decomposition. In particular,  $V_1^*\cup V_2^*$ and $U_1^*\cup U_2^*$ are both orbits of $G_0$ on $V^*$ with order $2(p^n-1)$, and $2(p^n-1)\neq (p^n-1)^2$ since $n\geq 2$, so $V_1^*\cup V_2^*=U_1^*\cup U_2^*$. If $x,y\in U_1^*$ where $x\in V_1$ and $y\in V_2$, then $x+y\in U_1$, so without loss of generality, $x+y\in V_1$, but then $y\in V_1$, a contradiction. Thus $\{U_1,U_2\}=\{V_1,V_2\}$, so $H_0$ stabilises the decomposition of $V$. 
\end{proof}

\begin{example}
\label{example:(I)goodHering}
Let $\mathcal{H}$ be the Hering plane of order $27$ or either of the two Hering spaces with line-size $9$  (see \S\ref{ss:2trans}). By Propositions~\ref{prop:(I)good}(iii)(b) or~\ref{prop:(I)goodmore}(iii)(b),  $\mathcal{H}\cprod\mathcal{H}$ is a $G$-affine proper partial linear space with point set $V:=V_6(3)\oplus V_6(3)$ where $G:=V{:}(\SL_2(13)\wr\langle\tau\rangle)$. 
By  Lemma~\ref{lemma:(I)aut}, $\Aut(\mathcal{\mathcal{H}\cprod\mathcal{H}})$ is an affine permutation group of rank~$3$ on $V$, and $\Aut(\mathcal{\mathcal{H}\cprod\mathcal{H}})_0\leq \GL_6(3)\wr\langle\tau\rangle$. Note that $\SL_2(13)$ is not a subgroup of $\GammaL_3(9)$ or $\GammaL_2(27)$. Thus $\Aut(\mathcal{\mathcal{H}\cprod\mathcal{H}})=G$ by Proposition~\ref{prop:(I)good}. Using  {\sc Magma}, we determine that $G$ is self-normalising in $\Sym(V)$, so the  two partial linear spaces with line-size $9$ are not isomorphic. Note that the lines of $\mathcal{H}\cprod \mathcal{H}$ are affine subspaces of $V$ for all possible  $\mathcal{H}$, and $G_0$ has no proper subgroups with two orbits on $V^*$.
\end{example}

Now we consider the orbit $V_n(p)^*\times V_n(p)^*$.  We begin by describing a method of constructing a partial linear space  on $V_n(p)\oplus V_n(p)$  from a linear space  on $V_n(p)$  that has a $2$-transitive affine group of automorphisms for which the stabiliser of any line $M$  induces a sharply $2$-transitive group  on $M$. There are several examples of linear spaces with such groups, including  $\AG_m(q)$ with its group of automorphisms  $\AGL_m(q)$ (where $q^m=p^n$), as well as $(V_2(p),\{V_2(p)\})$ with its  group of automorphisms $V_2(p){:}R$,   where $R\leq \GL_2(p)$ and $R$ is regular on $V_2(p)^*$; for example, we may take  $p=3$ and $R=Q_8$. The reader may find it helpful to remember these examples throughout the following exposition. 

\begin{definition}
\label{defn:A2badtriple}
Let $\mathcal{S}:=(V_n(p),\mathcal{M})$ be a linear space where $n\geq 1$ and $p$ is  prime. The pair $(H_0,M)$  \textit{sharply generates}  $\mathcal{S}$ if  the following two conditions hold.
\begin{itemize}
\item[(i)] $M\in \mathcal{M}_0$ and $H_0\leq \GL_n(p)$.
\item[(ii)]  $H:=V_n(p){:}H_0$  is  a  $2$-transitive subgroup of $\Aut(\mathcal{S})$,  and $H_M^M$ is sharply $2$-transitive.
\end{itemize}
 Let $R:=H_{0,M}^M$ and $N:=N_{\GL_n(p)}(H_0)$. Then $R\unlhd  N_M^M\leq   \Sym(M)_0$, so 
 $ N_M^M\leq   N_{\Sym(M)_0}(R)$. 
The triple $(H_0,M,\alpha)$ is  \textit{compatible} with $\mathcal{S}$ if
 $(H_0,M)$ sharply generates $\mathcal{S}$ and the following two conditions hold.
\begin{itemize}
\item[(a)] $\alpha\in N_{\Sym(M)_0}(R)$ and $\alpha^2\in N_M^M$.
\item[(b)]  There exist  $w\in M^*$ and  $h_1,h_2\in H_0$ such that $M^{h_1}+w=M$ and $(v^{h_1}+w)^\alpha=v^{\alpha h_2}+w^\alpha$ for all $v\in M$.
\end{itemize}
The triple $(H_0,M,\alpha)$ is \textit{$u$-compatible} with  $\mathcal{S}$  if $(H_0,M,\alpha)$ is compatible with  $\mathcal{S}$  and  $u^\alpha=u\in M^*$. The triple $(H_0,M,\alpha)$ is \textit{linearly ($u$-)compatible} with  $\mathcal{S}$ if $(H_0,M,\alpha)$ is ($u$-)compatible with  $\mathcal{S}$, the line $M$ is a subspace of $V_n(p)$ (in which case $N_M^M\leq \GL(M,\mathbb{F}_p)$), and $\alpha\in \GL(M,\mathbb{F}_p)$.  
\end{definition}

Suppose that $(H_0,M)$ sharply generates a linear space $\mathcal{S}:=(V_n(p),\mathcal{M})$.  Let $H:=V_n(p){:}H_0$,    $R:=H_{0,M}^M$, $N:=N_{\GL_n(p)}(H_0)$ and $ ^-:N_M\to N_M^M$  be the natural homomorphism. Here are some important observations. First,  the linear space $\mathcal{M}$ can be recovered from $(H_0,M)$ since $\mathcal{M}=M^{H}$. Second, if $M$ is a subspace of $V_n(p)$, then $H_M^M=M{:}R$. Lastly, for any $M$, the group $R$ is regular on $M^*$, so for $u\in M^*$,  the subspace  $C_{V_n(p)}(H_{0,u})$ contains $M$, and 
if $R\unlhd X\leq \Sym(M)_0$, then 
 $X=R{:}X_u$ for $u\in M^*$. 

Here are some natural instances of triples that are compatible with $\mathcal{S}$. 
First, $(H_0,M,\overline{s})$ is compatible with $\mathcal{S}$ for $s\in N_M$: clearly (a) holds,  and  for any $w\in M^*$, there exists $h_1\in H_0$ such that $M^{h_1}+w=M$ since $H_M$ is transitive on $M$, so  (b) holds with $h_2:=s^{-1}h_1s$. In particular, $(H_0,M,1)$ is compatible with $\mathcal{S}$.  
 Second, for  any triple  $(H_0,M,\alpha)$ that is compatible with $\mathcal{S}$ and any $s\in H_{0,M}$, the triple $(H_0,M,\alpha \overline{s})$ is also compatible with $\mathcal{S}$:  (a) holds for $\alpha\overline{s}$ since  $\alpha$ normalises $R$, and (b) holds since   there exist  $w\in M^*$ and  $h_1,h_2\in H_0$ such that $M^{h_1}+w=M$ and $(v^{h_1}+w)^\alpha=v^{\alpha h_2}+w^\alpha$ for all $v\in M$, in which case $(v^{h_1}+w)^{\alpha\overline{s}}=v^{(\alpha\overline{s})(s^{-1}h_2s)}+w^{\alpha\overline{s}}$ for all $v\in M$. 
(However,  we caution the reader that there are instances where $(H_0,M,\alpha)$  is a compatible triple and $s\in N_M$, but $(\alpha\overline{s})^2\notin N_M^M$: see Examples~\ref{example:(I)sl25} and~\ref{example:(I)E}.) 
Finally, if   $M$ is a subspace of $V_n(p)$ and  $\alpha\in N_{\GL(M,\mathbb{F}_p)}(R)$ such that $\alpha^2\in N_M^M$, then  $(H_0,M,\alpha)$ is  linearly compatible with $\mathcal{S}$ since (b) holds for any $w\in M^*$ with $h_1=h_2=1$.  

\begin{definition}
\label{defn:A2badfamily}
Let $V:=V_n(p)\oplus V_n(p)$ where $n\geq 1$ and $p$ is prime, and let $(H_0,M,\alpha)$ be a triple  that is  compatible with a linear space on $V_n(p)$. 
 Let $N:=N_{\GL_n(p)}(H_0)$ and $ ^-:N_M\to N_M^M$  be the natural homomorphism.  Recall that the stabiliser of the decomposition of $V$ in $\GL_{2n}(p)$ is $\GL_n(p)\wr\langle \tau\rangle $,  where  $(u_1,u_2)^\tau =(u_2,u_1)$ for all $u_1,u_2\in V_n(p)$.
Define 
\begin{align*}
\mathcal{L}(H_0,M,\alpha)&:=\{(v,v^\alpha) : v \in M \}^{V:H_0\times H_0},\\
\mathcal{S}(H_0,M,\alpha)&:=(V,\mathcal{L}(H_0,M,\alpha)),\\
\mathcal{N}(H_0,M,\alpha)&:=\{(s_1,s_2) \in N_M\times N_M : \overline{s}_2 = \overline{s}_1^\alpha\}, \\
\mathcal{G}(H_0,M,\alpha)&:=(H_0\times H_0) \mathcal{N}(H_0,M,\alpha)\langle \{(s,1)\tau : s\in N_M, \overline{s}=\alpha^2 \}\rangle.
\end{align*}
\end{definition}

Clearly $\mathcal{N}(H_0,M,\alpha)$ is a group, and if $s\in N_M$ such that $\overline{s}=\alpha^2$, then $(s,1)\tau$ normalises $\mathcal{N}(H_0,M,\alpha)$, so $\mathcal{N}(H_0,M,\alpha)\langle \{(s,1)\tau : s\in N_M, \overline{s}=\alpha^2 \}\rangle$ is a group. Thus $\mathcal{G}(H_0,M,\alpha)$ is a group with normal subgroups $H_0\times H_0$ and $(H_0\times H_0)\mathcal{N}(H_0,M,\alpha)$. 
Observe that for $t\in N_M$ such that $\overline{t}=\alpha^2$, $(H_0\times H_0)\langle (t,1)\tau \rangle$ has two orbits on $V^*$. In particular, $V{:}\mathcal{G}(H_0,M,\alpha)$ is a rank~$3$ group. Observe also that $\mathcal{G}(H_0,M,\alpha)_{V_1}=(H_0\times H_0)\mathcal{N}(H_0,M,\alpha)$ since $(s,s)\in \mathcal{N}(H_0,M,\alpha)$ for all $s\in N_M$ such that $\overline{s}=\alpha^2$. 

Here are some basic properties of $\mathcal{G}(H_0,M,\alpha)$ and $\mathcal{S}(H_0,M,\alpha)$.

\begin{lemma}
\label{lemma:A2badgeneric}
Let $V:=V_n(p)\oplus V_n(p)$ where $n\geq 1$ and $p$ is prime, and let $(H_0,M,\alpha)$ be a triple  that is  compatible with a linear space on $V_n(p)$.  Let $N:=N_{\GL_n(p)}(H_0)$ and $ ^-:N_M\to N_M^M$ be the natural homomorphism.   Let $R:=H_{0,M}^M$ and $S:=N_M^M$. Let $t\in N_M$ be such that $\overline{t}=\alpha^2$.
 \begin{itemize}
\item[(i)] $\mathcal{N}(H_0,M,\alpha)\langle (t,1)\tau \rangle=\mathcal{N}(H_0,M,\alpha)\langle \{(s,1)\tau : s\in N_M, \overline{s}=\alpha^2 \}\rangle$.

\item[(ii)] 
$(H_0\times H_0)\mathcal{N}(H_0,M,\alpha)=(H_0\times H_0)\{(s_1,s_2)\in N_M\times N_M : u^{s_1}=u,\overline{s}_2=\overline{s}_1^\alpha\}$ for $u\in M^*$.

\item[(iii)] For $s\in N_{M}$, $\mathcal{G}(H_0,M,\overline{s})=\mathcal{G}(H_0,M,1)^{(1,s)}$ and $\mathcal{L}(H_0,M,\overline{s})=\mathcal{L}(H_0,M,1)^{(1,s)}$.

\item[(iv)]  For $r\in R$, $\mathcal{G}(H_0,M,\alpha r)=\mathcal{G}(H_0,M,\alpha)$ and $\mathcal{L}(H_0,M,\alpha r)=\mathcal{L}(H_0,M,\alpha)$.

\item[(v)] Suppose that $\alpha$ normalises $S$ and fixes $u\in M^*$. Then  $\mathcal{G}(H_0,M,1)=\mathcal{G}(H_0,M,\alpha)$ if and only if $\alpha^2=1$ and $\alpha$ centralises $S_u$.

\item[(vi)] $\mathcal{S}(H_0,M,\alpha)$ is a $(V{:}\mathcal{G}(H_0,M,\alpha))$-affine partial linear space.
\end{itemize}
\end{lemma}

\begin{proof}
(i) If $s\in N_M$ such that $\overline{s}=\alpha^2$, then $(st^{-1},1)\in \mathcal{N}(H_0,M,\alpha)$, so (i) holds.
 
(ii) Let $u\in M^*$. Recall that $R$ is regular on $M^*$. Let $(s_1,s_2)\in \mathcal{N}(H_0,M,\alpha)$. There exist $h_1,h_2\in H_{0,M}$ such that $t_1:=h_1^{-1}s_1$ fixes $u$ and $t_2:=h_2^{-1}s_2$ fixes $u^\alpha$, in which case $t_1,t_2\in N_M$. Note that $\overline{t}_1^\alpha$ and $\overline{t}_2$ both fix $u^\alpha$. 
 Now $\overline{h}_2\overline{t}_2=\overline{s}_2=\overline{s}_1^\alpha=\overline{h}_1^\alpha\overline{t}_1^\alpha$, and  $R_{u^\alpha}=1$, so $\overline{t}_2=\overline{t}_1^\alpha$. It follows that (ii) holds.

(iii)
Let $s\in N_M$. Recall that $(H_0,M,\overline{s})$ and $(H_0,M,1)$ are  compatible with $\mathcal{S}$. Since $(1,s)$ normalises $V$ and $H_0\times H_0$, it follows that $\mathcal{L}(H_0,M,1)^{(1,s)}=\mathcal{L}(H_0,M,\overline{s})$.
 It is routine to verify that $\mathcal{N}(H_0,M,1)^{(1,s)}=\mathcal{N}(H_0,M, \overline{s})$. 
Further, $((s,s)\tau)^{(1,s)}=(s^2,1)\tau$ and $(s,s)\in \mathcal{N}(H_0,M,1)$, so $(\mathcal{N}(H_0,M,1)\langle \tau\rangle)^{(1,s)}=\mathcal{N}(H_0,M, \overline{s})\langle (s^2,1)\tau\rangle$. Thus (iii) holds by (i). 

(iv) Let $s\in H_{0,M}$. Recall that $(H_0,M,\alpha\overline{s})$ is compatible with $\mathcal{S}$. Since 
$(1,s)\in H_0\times H_0$, it  follows that $\mathcal{L}(H_0,M,\alpha\overline{s})=\mathcal{L}(H_0,M,\alpha)$. It is routine to verify that $\mathcal{N}(H_0,M,\alpha)^{(1,s)}=\mathcal{N}(H_0,M,\alpha \overline{s})$. 
Further,  $(\alpha\overline{s})^2=\alpha^2 \overline{s}^\alpha\overline{s}=\overline{th}$ for some $h\in H_{0,M}$.  Hence by (i),   
 $\mathcal{G}(H_0,M,\alpha \overline{s})=(H_0\times H_0)\mathcal{N}(H_0,M,\alpha)^{(1,s)}\langle (th,1)\tau\rangle=\mathcal{G}(H_0,M,\alpha)$ since  $(1,s),(1,h)\in H_0\times H_0$. 

(v) Suppose that $\alpha$ normalises $S$ and fixes $u\in M^*$. First suppose that $\mathcal{G}(H_0,M,1)=\mathcal{G}(H_0,M,\alpha)$. Then $(t,1)\tau \in  \mathcal{G}(H_0,M,1)$, so $(t,1)\tau=(h_1,h_2)(s_1,s_2)\tau$ for some $h_1,h_2\in H_0$ and $s_1,s_2\in N_M$ such that $\overline{s}_2=\overline{s}_1$. Thus $h_1,h_2\in H_{0,M}$, so $\alpha^2=\overline{t}=\overline{h}_1\overline{s}_2=\overline{h}_1\overline{h}_2^{-1}\in R_u=1$.  Let $\overline{t}_1\in S_u$ where $t_1\in N_M$. Now $\overline{t}_1^\alpha=\overline{t}_2\in S_u$ for some $t_2\in N_M$ since $\alpha$ normalises $S$ 
 and fixes $u$, so $(t_1,t_2)\in \mathcal{N}(H_0,M,\alpha)\leq (H_0\times H_0)\mathcal{N}(H_0,M,1)$. Since  $t_1$ and $t_2$  fix $u$, and since $H_{0,u}\times H_{0,u}\leq \mathcal{N}(H_0,M,1)$, (ii) implies that 
  $(t_1,t_2)\in \mathcal{N}(H_0,M,1)$, so $\overline{t}_1=\overline{t}_2=\overline{t}_1^\alpha$. Thus $\alpha$ centralises $S_u$, as desired. Conversely, if $\alpha^2=1$ and $\alpha$ centralises $S_u$, then $\overline{t}=\alpha^2=\overline{1}$,  so $\mathcal{G}(H_0,M,\alpha) =\mathcal{G}(H_0,M,1)$ by (i) and  (ii). 

(vi) Let $L:=\{(v,v^\alpha) : v \in M\}$, and observe that $\mathcal{N}(H_0,M,\alpha)\langle (t,1)\tau\rangle$ fixes $L$. By Lemma~\ref{lemma:sufficient}, (vi) holds if $B:=L^*$ is a block of $H_0\times H_0$ and $(V{:}(H_0\times H_0))_L$ is transitive on $L$. Suppose that $(v,v^\alpha)^{(h_1,h_2)}=(w,w^\alpha)$ for some $v,w\in M^*$ and $(h_1,h_2)\in H_0\times H_0$. Now $h_1,h_2\in H_{0,M}$ and $v^{\alpha h_2\alpha^{-1}}=w=v^{h_1}$, and $R$ is regular on $M^*$, so $\alpha \overline{h}_2=\overline{h}_1\alpha$. Thus $(h_1,h_2)$ fixes $B$, so $B$ is a block of  $H_0\times H_0$. By assumption, there exist $w\in M^*$ and $h_1,h_2\in H_0$ such that $M^{h_1}+w=M$ and $(v^{h_1}+w)^\alpha=v^{\alpha h_2}+w^\alpha$ for all $v\in M$. Let $v_1,v_2\in V_n(p)$ be such that $v_1^{h_1}=w$ and $v_2^{h_2}=w^\alpha$, and let $g:=\tau_{(v_1,v_2)}(h_1,h_2)$,   where   $\tau_{(v_1,v_2)}$ denotes the translation of $V$ by $(v_1,v_2)$. Now $L^g=L$ and $0^g=(w,w^\alpha)\in B$. It follows that $(V{:}(H_0\times H_0))_L$ is transitive on $L$.
\end{proof}

 By  Lemma~\ref{lemma:A2badgeneric}(iv), since $R$ is transitive on $M^*$,  there is no loss in assuming  that $\alpha$ fixes some $u\in M^*$. In other words, we may restrict our attention to triples that are $u$-compatible. 

It turns out that $\AG_m(q)$, the nearfield plane of order $9$ and exactly one of the Hering spaces on $3^6$ points with line-size $9$ have  sharply generating pairs. Further, for $n=2$ and  $p\in \{3,5,7,11,23,29,59\}$, the linear space with line set $\{V_2(p)\}$ has sharply generating pairs.
 Note that the lines of these linear spaces are all  affine subspaces of $V_n(p)$.  We now describe five families of examples that arise under this framework, all of which are built from linearly $u$-compatible triples. 
    
\begin{notation}
Let $V:=V_n(p)\oplus V_n(p)$ where $n\geq 1$ and $p$ is prime. 
 Recall that the stabiliser of the decomposition of $V$ is $\GL_n(p)\wr\langle \tau\rangle $,  where  $(u_1,u_2)^\tau =(u_2,u_1)$ for all $u_1,u_2\in V_n(p)$, and 
recall Definitions~\ref{defn:A2badtriple} and~\ref{defn:A2badfamily}. 
Given a  pair $(H_0,M)$ that sharply generates some linear space on $V_n(p)$ where $M$ is a subspace of $V_n(p)$, we  let $N:=N_{\GL_n(p)}(H_0)$,  $R:=H_{0,M}^M$, $S:=N_M^M$, $T:=N_{\GL(M,\mathbb{F}_p)}(R)$ and $ ^-:N_M\to S$  be the natural homomorphism.   Recall that $R\leq S\leq T$ and $T=R{:}T_u$ for all $u\in M^*$.  Further,  $(H_0,M,\alpha)$ is a linearly $u$-compatible triple if and only if $\alpha\in T_u$  and $\alpha^2\in S_u$.  Lastly, for such a triple, $\mathcal{S}(H_0,M,\alpha)$ is a $(V{:}\mathcal{G}(H_0,M,\alpha))$-affine partial linear space   by Lemma~\ref{lemma:A2badgeneric}(vi).
\end{notation}
 
\begin{example}
\label{example:(I)dep}
Let $U:=V_m(q)$ where $q^m=p^n$, and let $u\in U^*$.  Let $F$ be a subfield of $\mathbb{F}_q$, let $M:=\langle u\rangle_F$, and let $H_0:=\GL_m(q)$. Now $(U,M^{U{:}H_0})$ is isomorphic to the linear space $\AG_a(r)$ where $r=|F|$ and $r^a=p^n$, and $(U{:}H_0)_M^M\simeq \AGL_1(F)$, so $(U{:}H_0)_M^M$ is sharply $2$-transitive. Thus $(H_0,M)$ sharply generates $(U,M^{U{:}H_0})$. 
 Further, $N=\GammaL_m(q)$, $R\simeq \GL_1(F)$, and $S=T\simeq \GammaL_1(F)$.
 Let $\alpha\in T_u\simeq \Aut(F)$.  There exists $\theta\in \Aut(F)$ such that $(\lambda u)^\alpha=\lambda ^\theta u$ for all $\lambda\in F$; with some abuse of notation, we may write $\theta$ for $\alpha$, in which case
$$
\mathcal{L}(\GL_m(q),\langle u\rangle_F,\theta)=\{\{(\lambda u_1+v_1,\lambda^\theta u_2+v_2) : \lambda\in F\} : u_1,u_2\in U^*,v_1,v_2\in U\}.
$$
In addition,  there exists $\pi\in \Aut(\mathbb{F}_q)$ such that  $\pi|_F=\theta^2$. We may assume that   $\Aut(\mathbb{F}_q)$ acts on $U$ in such a way that $u$ is fixed, so by Lemma~\ref{lemma:A2badgeneric}(i) and (ii), $\mathcal{G}(\GL_m(q),\langle u\rangle_F,\theta)$ equals
$$(\GL_m(q)\times \GL_m(q))\{(\sigma_1,\sigma_2)\in \Aut(\mathbb{F}_q)\times\Aut(\mathbb{F}_q): \sigma_1|_F=\sigma_2|_F\}\langle (\pi,1)\tau\rangle,$$ 
which is a subgroup of $\GammaL_m(q)\wr\langle\tau\rangle$. Observe that $\mathcal{S}(\GL_m(q),M,1)\simeq \mathcal{S}(\GL_m(q),M,\theta)$ by Lemma~\ref{lemma:A2badgeneric}(iii). Further,   $\mathcal{S}(\GL_m(q),M,1)$ is an $\mathbb{F}_q$-dependent partial linear space that contains the line $\langle (u,u)\rangle_F$. In fact, if $F=\mathbb{F}_r$ and $r^a=q^m$, then by viewing $U$ as $V_a(r)$, we see that 
$\mathcal{L}(\GL_m(q),M,1)=\mathcal{L}(\GL_a(r),M,1)$ and  
$$\mathcal{G}(\GL_m(q),M,1)\leq \mathcal{G}(\GL_a(r),M,1)=(\GammaL_a(r)\wr\langle\tau\rangle) \cap\GammaL_{2a}(r).$$ 
Since $\Aut(F)\simeq S_u=T_u$ is abelian, parts (iii) and (v) of Lemma~\ref{lemma:A2badgeneric} imply that for $\theta_1,\theta_2\in\Aut(F)$, the group $\mathcal{G}(\GL_m(q),M,\theta_1)=\mathcal{G}(\GL_m(q),M,\theta_2)$ if and only if $\theta_1^2=\theta_2^2$, and this occurs if and only if $\theta_2=\theta_1\iota$ where $\iota\in \Aut(F)$ and either $\iota=1$, or, for $|\Aut(F)|$  even, $\iota$ is the unique involution in $\Aut(F)$. In particular, in this latter case, $\mathcal{G}(\GL_m(q),M,1)$ admits both $\mathcal{S}(\GL_m(q),M,1)$ and $\mathcal{S}(\GL_m(q),M,\iota)$ as examples.
\end{example}

\begin{example}
\label{example:(I)reg}
Let $M:=V_n(p)$ and $H_0\leq \GL_n(p)$ be regular on $V_n(p)^*$. Now $(H_0,M)$ sharply generates the linear space $(V_n(p),\{V_n(p)\})$. 
Further, $R=H_0$ and $S=N=T$. Let $u\in V_n(p)^*$ and $\alpha\in T_u$. Since $\alpha$ normalises $R$, and since $N_M=N$,
\begin{align*}
\mathcal{L}(H_0,M,\alpha)&=\{\{(v+v_1,v^{g}+v_2): v\in V_n(p)\}: g\in R\alpha,v_1,v_2\in V_n(p)\},\\
\mathcal{G}(H_0,M,\alpha)&=(R\times R)\{ (s,s^\alpha): s \in T_u\} \langle (\alpha^2,1)\tau\rangle.
\end{align*}
By Lemma~\ref{lemma:A2badgeneric}(iii), $\mathcal{S}(H_0,M,1)$ is isomorphic to $\mathcal{S}(H_0,M,\alpha)$ for all $\alpha\in T_u$. Moreover, since $S=T$, Lemma~\ref{lemma:A2badgeneric}(iii) and (v) imply that for $\alpha_1,\alpha_2\in T_u$, the groups $\mathcal{G}(H_0,M,\alpha_1)$ and $\mathcal{G}(H_0,M,\alpha_2)$ are equal if and only if $\alpha_1^2=\alpha_2^2$ and $\alpha_2\alpha_1^{-1}\in Z(T_u)$.   

The possibilities for $H_0$ are given by  Zassenhaus's  classification~\cite{Zas1936}
 of the sharply $2$-transitive permutation groups (see~\cite[\S 7.6]{DixMor1996}); in particular, either $H_0$ is a subgroup of $\GammaL_1(p^n)$,
or $n=2$ and $(p,H_0,T_u)$ is given by Table~\ref{tab:sharp} (cf.~Theorem~\ref{thm:2trans}), in which case $\mathcal{S}(H_0,M,1)$ is  listed in Table~\ref{tab:main}. Note that $Q_8\leq \GammaL_1(9)$, but we include this case in Table~\ref{tab:sharp} since $T_u\nleq \GammaL_1(9)$. Note also that if $H_0=\GL_1(p^n)$, then $\mathcal{S}(H_0,V_n(p),\alpha)=\mathcal{S}(H_0,\langle u\rangle_F,\alpha)$ where   $F=\mathbb{F}_{p^n}$, a partial linear space from  Example~\ref{example:(I)dep}. For any possible $H_0$, the linear space $\mathcal{N}$ with point set $V$ and line set $\{V_i+v : i \in \{1,2\}, v\in V\}\cup \mathcal{L}(H_0,M,1)$ is a nearfield plane (see~\S\ref{ss:nearfield}); in particular, if $n=2$ and $(p,H_0)$ is given by Table~\ref{tab:sharp}, then either $p=3$ and $\mathcal{N}$ is the nearfield plane of order~$9$, or $p\neq 3$ and $\mathcal{N}$ is an irregular nearfield plane of order $p^2$.
\begin{table}[!h]
\centering
\begin{tabular}{ c | c c c c c c c c  }
\hline 
 $p$ & $3$ & $5$ & $7$ & $11$ & $23$ & $11$ & $29$ & $59$\\
 $H_0$ & $Q_8$ & $\SL_2(3)$ & $2\nonsplit S_4^-$ & $\SL_2(3)\times C_5$ & $2\nonsplit S_4^-\times C_{11}$ & $\SL_2(5)$ & $\SL_2(5)\times C_7$ & $\SL_2(5)\times C_{29}$ \\
 $T_u$ & $S_3$ & $C_4$ & $C_3$ & $C_2$ & $1$ & $C_5$ & $C_2$ & $1$ \\
 \hline 
\end{tabular}
\caption{Exceptional sharply $2$-transitive groups}
\label{tab:sharp}
\end{table}

We claim that for each  $(p,H_0)$ in Table~\ref{tab:sharp}, the group  $V{:}\mathcal{G}(H_0,M,1)$ is the full automorphism group of $\mathcal{S}(H_0,M,1)$. If $(p,H_0)$ is not $(29,\SL_2(5)\times C_7)$ or $(59,\SL_2(5)\times C_{29})$, then we verify this using {\sc Magma}.  Let $K:=\Aut(\mathcal{S}(H_0,M,1))$ where $(p,H_0)$ is $(29,\SL_2(5)\times C_7)$ or 
$(59,\SL_2(5)\times C_{29})$.   By  Lemma~\ref{lemma:(I)aut}, $K$ is an affine permutation group on $V$, and $K_0\leq \GL_2(p)\wr \langle\tau\rangle$.  If $K$ belongs to (I1),  then $\{(v,v): v\in V_2(p)^*\}$ is a block of $\SL_2(p)\times \SL_2(p)$, a contradiction. Thus $K$ belongs to (I4), and it follows from  Lemma~\ref{lemma:A2badequiv} below (and its proof) that  $K=V{:}\mathcal{G}(H_0,M,1)$.
\end{example}

\begin{example}
\label{example:(I)sl213}
Here $p^n=3^6$. Let $H_0:=\SL_2(13)\leq \GL_6(3)$.  Let  $u\in V_6(3)^*$ and $M:=C_{V_6(3)}(H_{0,u})\simeq V_2(3)$. Using {\sc Magma}, we determine that $M$ is a line of one of the  Hering  spaces $\mathcal{H}$ with line-size $9$  (see \S\ref{ss:2trans}). Now $V_6(3){:}H_0$ is a $2$-transitive  subgroup of $\Aut(\mathcal{H})$, so $(H_0,M)$  sharply generates  $\mathcal{H}$.  
 Further, $N=H_0$, $S=R= Q_8$,  $T=\GL_2(3)$ and $T_u\simeq S_3$. If $\alpha\in T_u$  such that $\alpha^2\in S$, then $\alpha^2\in R_u=1$, so there are four possibilities for $\alpha$, namely $1$ and the three involutions in $T_u$. Now $H_0\wr\langle \tau\rangle=\mathcal{G}(H_0,M,\alpha)$ for all such $\alpha$, so $H_0\wr\langle \tau\rangle$ admits  $\mathcal{S}(H_0,M,\alpha)$ for all such $\alpha$. Let $K:=\Aut(\mathcal{S}(H_0,M,\alpha))$ for such an $\alpha$. By Lemma~\ref{lemma:(I)aut},  $K$ is an affine permutation group on $V$, and $K_0\leq \GL_6(3)\wr \langle\tau\rangle$. Note that $\SL_2(13)$ is not a subgroup of $\GammaL_3(9)$ or $\GammaL_2(27)$. If $K$  belongs to class (I1) or (I2),  then $\{(v,v^\alpha): v\in M^*\}$ is  a block of $\SL_6(3)\times \SL_6(3)$ or $\Sp_6(3)\times \Sp_6(3)$,  a contradiction. Thus  $\Aut(\mathcal{S}(H_0,M,\alpha))=V{:}(H_0\wr\langle \tau\rangle)$ for each $\alpha\in T_u$ such that $\alpha^2=1$. Using  {\sc Magma}, we determine that $V{:}(H_0\wr\langle \tau\rangle)$ is self-normalising in $\Sym(V)$, so the four partial linear spaces $\mathcal{S}(H_0,M,\alpha)$  are pairwise non-isomorphic.
\end{example}

\begin{example}
\label{example:(I)sl25}
Here $p^n=3^4$. Let $U:=V_2(9)$,  $u\in U^*$, $\zeta:=\zeta_9$  and $\sigma:=\sigma_9$ where $u^\sigma=u$ (see \S\ref{ss:basicsvs} for the definition of $\zeta_9$ and $\sigma_9$). Now $\SL_2(5)\leq \SL_2(9)$ and $N_{\GL_4(3)}(\SL_2(5))=\SL_2(5)\langle \zeta,\sigma\rangle$. Let $H_0:=\SL_2(5)\langle \zeta^2, \zeta\sigma\rangle\simeq 2\nonsplit S_5^-{:}2$ and  $M:=\langle u\rangle_{\mathbb{F}_9}$. Note that $\SL_2(5)$ has two orbits on $U^*$, and if $\Omega$ is one such orbit, then $\zeta\Omega$ is the other. Now  $H_0$ is transitive on $U^*$ and has order $480$.  Further,  
$H_{0,u}\leq \GL_2(9)$, so  $(H_0,M)$ sharply generates $\AG_2(9)$, and   $N=\SL_2(5)\langle \zeta,\sigma\rangle$. 
 (Note that the pair $(\SL_2(5)\langle  \zeta^i\sigma\rangle,M)$     sharply generates $\AG_2(9)$  for $i\in \{1,3\}$, but it is more convenient to work only with $H_0$ since  $\SL_2(5)\langle  \zeta^i\sigma\rangle$ is not normal in $N$.)
With some abuse of notation, we may write $\zeta$ for $\overline{\zeta}$ and $\sigma$ for $\overline{\sigma}$, in which case 
 $R=\langle \zeta^2,\zeta\sigma\rangle  =Q_8$, $S=\langle \zeta,\sigma\rangle\simeq \GammaL_1(9)$, $T= \GL_2(3)$ and  $\langle \sigma\rangle=S_u\leq T_u\simeq S_3$. If $\alpha\in T_u$ such that $\alpha^2\in S_u$, then $\alpha^2=1$, so there are four possibilities for $\alpha$, namely $1$, $\sigma$, and the two remaining involutions in $T_u$. Using Lemma~\ref{lemma:A2badgeneric},  we verify that $\mathcal{G}(H_0,M,1)=\mathcal{G}(H_0,M,\sigma)=H_0\wr\langle\tau\rangle \langle (\sigma,\sigma)\rangle$, while for an involution $\alpha\in T_u\setminus\{1,\sigma\}$, $\mathcal{G}(H_0,M,\alpha)=H_0\wr\langle \tau\rangle$. In particular, any subgroup of $H_0\wr\langle \tau\rangle \langle (\sigma,\sigma)\rangle$ admits  $\mathcal{S}(H_0,M,1)$ and $\mathcal{S}(H_0,M,\sigma)$, while any subgroup of $H_0\wr\langle\tau\rangle$ admits  $\mathcal{S}(H_0,M,\alpha)$ for  $\alpha\in T_u$ with $\alpha^2=1$.   Note that $\mathcal{S}(H_0,M,1)\simeq \mathcal{S}(H_0,M,\sigma)$ by Lemma~\ref{lemma:A2badgeneric}(iii), and $\mathcal{S}(H_0,M,\alpha)\simeq \mathcal{S}(H_0,M,\beta)$ for involutions $\alpha,\beta\in T_u\setminus \{1,\sigma\}$ since   $\mathcal{L}(H_0,M,\alpha^\sigma)=\mathcal{L}(H_0,M,\alpha)^{(\sigma,\sigma)}$.  By a computation  in {\sf GAP}   using~\cite{NautyTraces,Grape}, 
the full automorphism group of $\mathcal{S}(H_0,M,\alpha)$ is $V{:}\mathcal{G}(H_0,M,\alpha)$ for $\alpha\in T_u$ such that $\alpha^2=1$. In particular, $\mathcal{S}(H_0,M,1)$ is not isomorphic to $\mathcal{S}(H_0,M,\alpha)$ for any involution $\alpha\in T_u\setminus \{1,\sigma\}$. Further, $\mathcal{S}(H_0,M,1)\not\simeq \mathcal{S}(\GL_2(9),M,1)$. 
\end{example}

\begin{example}
\label{example:(I)E}
Here $p^n=3^4$. Let  $H_0:=(D_8\circ Q_8).C_{5}\leq \GL_4(3)$.  Let $u\in V_4(3)^*$ and $M:=C_{V_4(3)}(H_{0,u})\simeq V_2(3)$. The nearfield plane  of order $9$  (see  \S\ref{ss:nearfield}) admits the $2$-transitive group $V_4(3){:}H_0$ as an automorphism group by Remark~\ref{remark:2trans} and contains the line $M$, so $(H_0,M)$ is a sharply generating pair for this linear space.  Further, $N=(D_8\circ Q_8).\AGL_1(5)$, $R= Q_8$, $S\simeq \GammaL_1(9)$, $T= \GL_2(3)$, $S_u\simeq C_2$, $T_u\simeq S_3$   and $ C_8\simeq N_u\leq N_M$ since $V_4(3){:}N$ is an automorphism group of the nearfield plane of order $9$   by Remark~\ref{remark:2trans}. Write $N_u=\langle s\rangle$. Then $H_{0,u}=\langle s^4\rangle$ and $S_u=\langle \overline{s}\rangle$, so $N=H_0\langle s\rangle$ and $s^2$ fixes $M$ pointwise.  Let $K_0:=H_0\langle s^2\rangle$, so $H_0\leq K_0\leq N$ and   $K_0=(D_8\circ Q_8).D_{10}$. Then $K_{0,u}=\langle s^2\rangle$ and  $M=C_{V_4(3)}(K_{0,u})$, so $(K_0,M)$ is also a sharply generating pair for the nearfield plane of order $9$. Note that $N=N_{\GL_4(3)}(K_0)$.  If $\alpha\in T_u$ such that $\alpha^2\in S_u$, then $\alpha^2=1$, so there are four possibilities for $\alpha$, namely $1$, $\overline{s}$, and the two remaining involutions in $T_u$.  
Observe that $\mathcal{L}(H_0,M,\alpha)=\mathcal{L}(K_0,M,\alpha)$ for any  $\alpha\in T_u$ with $\alpha^2=1$ since  $s^2$ fixes $M$ pointwise.  
Using Lemma~\ref{lemma:A2badgeneric}, we verify that 
$$\mathcal{G}(H_0,M,1)=\mathcal{G}(H_0,M,\overline{s})=\mathcal{G}(K_0,M,1)=\mathcal{G}(K_0,M,\overline{s})=K_0\wr\langle \tau\rangle \langle (s,s)\rangle,
$$ 
while $\mathcal{G}(H_0,M,\alpha)=\mathcal{G}(K_0,M,\alpha)=K_0\wr\langle\tau\rangle$ for an involution $\alpha\in T_u\setminus  \{1,\overline{s}\}$. In particular, any subgroup of 
$K_0\wr\langle \tau\rangle \langle (s,s)\rangle$  admits  $\mathcal{S}(H_0,M,1)$ and $\mathcal{S}(H_0,M,\overline{s})$, while any subgroup of $K_0\wr\langle\tau\rangle$  admits $\mathcal{S}(H_0,M,\alpha)$ for $\alpha\in T_u$ with $\alpha^2=1$. Note that $\mathcal{S}(H_0,M,1)\simeq \mathcal{S}(H_0,M,\overline{s})$ by  Lemma~\ref{lemma:A2badgeneric}(iii), and $\mathcal{S}(H_0,M,\alpha)\simeq \mathcal{S}(H_0,M,\beta)$ for involutions $\alpha,\beta\in T_u\setminus \{1,\overline{s}\}$ since $\mathcal{L}(H_0,M,\alpha^{\overline{s}})=\mathcal{L}(H_0,M,\alpha)^{(s,s)}$. By a computation in {\sf GAP}   using~\cite{NautyTraces,Grape}, 
the full automorphism group of $\mathcal{S}(H_0,M,\alpha)$ is $V{:}\mathcal{G}(H_0,M,\alpha)$ for $\alpha\in T_u$ such that $\alpha^2=1$. In particular, $\mathcal{S}(H_0,M,1)$ is not isomorphic to $\mathcal{S}(H_0,M,\alpha)$ for any involution $\alpha\in T_u\setminus \{1,\overline{s}\}$.
\end{example}

Now we describe a sporadic example  that is very similar to Example~\ref{example:(I)reg} but is not isomorphic to $\mathcal{S}(H_0,M,\alpha)$ for any $u$-compatible triple $(H_0,M,\alpha)$.

\begin{example}
\label{example:(I)spor}
Let $R:=2\nonsplit S_4^-\leq \GL_2(7)$ and $T:=N_{\GL_2(7)}(R)$. Let $u\in V_2(7)^*$ and $V:=V_2(7)\times V_2(7)$.  Now $R$ is regular on $V_2(7)^*$, so $T=R{:}T_u$. Further, the following hold: $T_u\simeq C_3$,  $R$ has a  unique index $2$ subgroup $K:=\SL_2(3)=R'$, and $T=N_{\GL_2(7)}(K)$. Write $R=K\langle r\rangle$. Let 
 $X_0:=(K\times K) \langle (r,1)\tau\rangle$ and $Y_0:=X_0\{(s,s): s \in T_u\}$. Now $X_0$ has  index $2$ in $R\wr\langle\tau\rangle$, while $Y_0$ has index $2$ in  $\mathcal{G}(R,V_2(7),1)=R\wr\langle \tau\rangle\{(s,s) : s \in T_u\}$ (see~Figure~\ref{fig:sporexp}). 
 \begin{figure}
\begin{center}
\includegraphics[height=3cm]{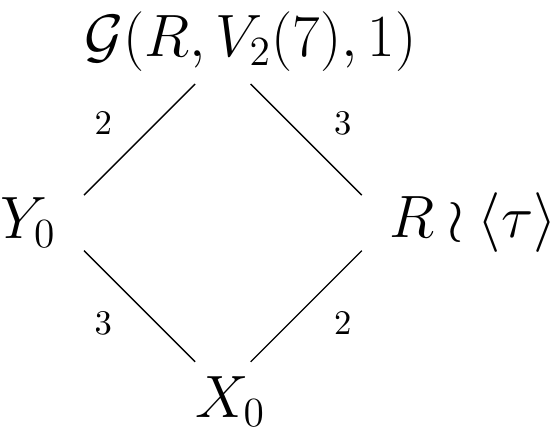}
\end{center}
\caption{Groups of Example~\ref{example:(I)spor}}
\label{fig:sporexp}
\end{figure} 
For  $g\in \{1,r\}$, 
let $\Omega(g):=\{(v,v^{h}) : v\in V_2(7)^*, h\in Kg\}$. Now  $X_{0,V_1}=(K\times K)\langle (r,r)\rangle $, and the orbits of $X_{0,V_1}$ on $V_2(7)^*\times V_2(7)^*$ are $\Omega(1)$ and $\Omega(r)$, so  $V{:}X_0$ and $V{:}Y_0$ are rank~$3$ groups.  
For $\alpha\in T_u$,  let
 $$
\mathcal{L}(\alpha):=\{\{(v,v^{g}) : v\in V_2(7)\}+w : g\in K\alpha\cup K\alpha^2r,w\in V\}.
$$
Observe that  $\mathcal{L}(1)=\mathcal{L}(R,V_2(7),1)$ but $\mathcal{L}(\alpha) \neq \mathcal{L}(R,V_2(7),\alpha)$ for $\alpha\in T_u\setminus \{1\}$.  
We claim that $\mathcal{S}(\alpha):=(V,\mathcal{L}(\alpha))$ is a proper $(V{:}Y_0)$-affine partial linear space for $\alpha\in T_u$. 
Let $L_\alpha:=\{(v,v^\alpha):v\in V_2(7)\}$ and $B_\alpha:=L_\alpha^*$. Since $\mathcal{L}(\alpha)_0=L_\alpha^{X_0}$ and $\{(s,s): s\in T_u\}$ fixes $L_\alpha$,   it suffices to show that $B_\alpha$ is a block of $X_0$ by Lemmas~\ref{lemma:transitive} and~\ref{lemma:sufficient}. Now $B_\alpha$ is a block of $R\times R$ and therefore of   $X_{0,V_1}$. Recall that $(r,1)\tau$ maps the $X_{0,V_1}$-orbit $\Omega(1)$ to $\Omega(r)$, and observe that $B_\alpha\subseteq \Omega(1)$ since $v^\alpha\in v^{K}$ for all $v\in V_2(7)^*$. Thus $B_\alpha$ is a block of $X_0$, and the claim holds.  Note that $\mathcal{S}(\alpha)$ and $\mathcal{S}(\alpha^2)$ are isomorphic for $\alpha\in T_u$ since $\mathcal{L}(\alpha^2)=\mathcal{L}(\alpha)^{(1,r)}$. Let $\alpha\in T_u\setminus\{1\}$. Now $B_\alpha$ is not a block of $R\wr\langle\tau\rangle$  since  $\tau$ fixes $(u,u)$ but not $B_\alpha$, so neither $R\wr\langle\tau\rangle$ nor $\mathcal{G}(R,V_2(7),1)$ are automorphism groups of $\mathcal{S}(\alpha)$. Indeed, we verify using {\sc Magma}  that $V{:}Y_0$ is the full automorphism group of $\mathcal{S}(\alpha)$; 
in particular, $\mathcal{S}(\alpha)$ is not isomorphic to $\mathcal{S}(1)=\mathcal{S}(R,V_2(7),1)$. 

Let  $\mathcal{A}(\alpha)$ be the affine plane of order $49$  on $V$ with  line set $\{V_i+v : i \in \{1,2\}, v\in V\}\cup \mathcal{L}(\alpha)$ for $\alpha\in T_u$. 
By~\cite[Theorem~19.10]{Lun1980}, there are (up to isomorphism) two affine planes of order~$49$ of type F$\ast$7 (as defined in~\cite[p.95]{Lun1980}):  
the irregular nearfield plane $\mathcal{A}(1)$  (cf.   Example~\ref{example:(I)reg}) and the \textit{exceptional L\"uneburg plane of order $49$}.  We claim that $\mathcal{A}(\alpha)$ is  isomorphic to the  latter   when  $\alpha\neq 1$. Now $\mathcal{A}(\alpha)\not\simeq\mathcal{A}(1)$ when $\alpha\neq 1$, or else  $\mathcal{S}(\alpha)\simeq \mathcal{S}(1)$  by Remark~\ref{remark:lineartopls}, a contradiction. Thus it  suffices to show that $\mathcal{A}(\alpha)$  is of type F$\ast$7. To do so, we must show that $\mathcal{A}(\alpha)$ has properties~(j) and~(ij) of~\cite[p.95]{Lun1980}. Property (j) requires $\Aut(\mathcal{A}(\alpha))_L$ to be $2$-transitive on $L$ for all lines $L$, which holds by Lemma~\ref{lemma:necessary}(ii). 
For~(ij), let $P$ and $Q$ be the points at infinity  of the  lines $V_1$ and $V_2$, respectively, and let $O$ be the point $0\in V$. Let $S_0:=K\times 1$ and $S_\infty:=1\times K$. Then $S_0\simeq S_\infty\simeq \SL_2(3)$, and $S_0$ and $S_\infty$ have the same four orbits on $\ell_\infty$. Further,  $S_0$ fixes $P$ and the line on $O$ and $Q$ pointwise (i.e., $V_2\cup \{Q\}$), while $S_\infty$ fixes $Q$ and the line on $O$ and $P$ pointwise (i.e., $V_1\cup \{P\}$). Since~(ij) requires the existence of $P$, $Q$, $O$,  $S_0$ and $S_\infty$ with precisely these properties,  $\mathcal{A}(\alpha)$ is of type F$\ast$7, and the claim holds.
\end{example}

We wish to prove that no other examples arise; the following lemma is a basic but crucial observation. 

\begin{lemma}
\label{lemma:A2badbasic}
Let $G$ be an affine permutation group of rank~$3$ on $V:=V_n(p)\oplus V_n(p)$ where $n\geq 1$, $p$ is prime and $G_0\leq \GL_n(p)\wr\langle\tau\rangle$. Let  $\mathcal{S}:=(V,\mathcal{L})$ be a $G$-affine proper partial linear space in which $\mathcal{S}(0)=V_n(p)^*\times V_n(p)^*$. Let $L\in\mathcal{L}_0$. Then $L=\{(v,v^\alpha) : v\in M\}$ for some $M\subseteq V_n(p)$ and  injective map $\alpha:M\to V_n(p)$.    
In particular, if  $(x,x)\in L$, then $\alpha$ fixes $x$.  
\end{lemma}

\begin{proof}
 If $(v,w_1),(v,w_2)\in L^*$ for some $v,w_1,w_2\in V_n(p)$, then $w_1=w_2$ by Lemma~\ref{lemma:basic}(i). The desired result then follows from the symmetry of this argument.
\end{proof}

In the following, we prove that if $H_0\times H_0\unlhd  G_0$ where $H_0$ is transitive on $V_n(p)^*$, then every $G$-affine proper partial linear space $\mathcal{S}$ with $\mathcal{S}(0)=V_n(p)^*\times V_n(p)^*$ has the form $\mathcal{S}(H_0,M,\alpha)$ for some $u$-compatible triple $(H_0,M,\alpha)$. In order to give an elementary proof of this fact, we do not assume that the triple is linearly compatible. 
Note that if $H_0\times H_0\leq G_0=G_{0,V_1}\langle (t,1)\tau\rangle\leq \GL_n(p)\wr\langle \tau\rangle$ where $H_0\leq \GL_n(p)$ is transitive on $V_n(p)^*$, then for any $u\in V_n(p)^*$, there exists $h\in H_0$ such that $u^{ht}=u$, and $G_0=G_{0,V_1}\langle (ht,1)\tau\rangle$. Thus, for such $G_0$, there is no loss of generality in assuming that $t$ fixes some non-zero vector.

\begin{lemma}
\label{lemma:A2badequiv}
Let $G$ be an affine permutation group of rank~$3$ on $V:=V_n(p)\oplus V_n(p)$ where $n\geq 1$, $p$ is prime and $H_0\times H_0\unlhd G_0=G_{0,V_1}\langle (t,1)\tau\rangle\leq \GL_n(p)\wr\langle\tau\rangle$ such that 
 $H_0\leq \GL_n(p)$ is transitive on $V_n(p)^*$ and $t\in \GL_n(p)_u$ for some $u\in V_n(p)^*$. The following are equivalent.
\begin{itemize}
\item[(i)]  $\mathcal{S}$ is a $G$-affine proper partial linear space  in which $\mathcal{S}(0)=V_n(p)^*\times V_n(p)^*$.
\item[(ii)] 
$\mathcal{S}=\mathcal{S}(H_0,M,\alpha)$ and $G_0\leq \mathcal{G}(H_0,M,\alpha)$ for some triple  $(H_0,M,\alpha)$ that is $u$-compatible with a  linear space on $V_n(p)$ such that $|M|\geq 3$ and $\alpha^2=t|_M$.
\end{itemize}
\end{lemma}

\begin{proof}
If (ii) holds, then (i) holds by Lemma~\ref{lemma:A2badgeneric}(vi). Conversely, suppose that (i) holds, and let $\mathcal{L}$ be the line set of $\mathcal{S}$. Recall that $t\in G_0^1=G_0^2$, and note that $G_0^1\leq N:=N_{\GL_n(p)}(H_0)$ since $H_0\unlhd G_0^1$. 
Let $L\in \mathcal{L}_0$ be such that $(u,u)\in L$. By Lemma~\ref{lemma:A2badbasic}, $L=\{(v,v^\alpha): v\in M\}$ for some $M\subseteq V_n(p)$ and injective map $\alpha:M\to V_n(p)$, where $\alpha$ fixes $0$ and $u$. Also, $|M|=|L|\geq 3$. Let  $ ^-:N_M\to N_M^M$ be the natural homomorphism and $R:=H_{0,M}^M
$. 
Let $B:=L^*$. By Lemma~\ref{lemma:necessary}, $B$ is a block of $G_0$.  Then $(t,1)\tau\in G_{0,B}$ since $(t,1)\tau$ fixes $(u,u)$. Thus $L=\{(v^{\alpha},v^t):v\in M\}$, so $\alpha\in \Sym(M)$ and $t\in N_M$. Further  $\alpha^2=\overline{t}=t|_M$.

Let $H:=V_n(p){:}H_0$. We claim that $(V_n(p),M^H)$ is a linear space. By Lemma~\ref{lemma:sufficient}, it suffices to show that $C:=M^*$ is a block of $H_0$ on $V_n(p)^*$ and that $H_M$ is transitive on $M$. 
 Now $C$ is a block of $H_0$ on $V_n(p)^*$, for if $v^h=w$ for some $v,w\in C$ and $h\in H_0$, then there exists $h'\in H_0$ such that $(v^\alpha)^{h'}=w^\alpha$, so $(h,h')\in G_{0,B}$, in which case   $C^h=C$, as desired. 
 Recall that $(t,1)\tau$ fixes $(u,u)$ and $L$, and note that $V{:}(H_0\times H_0)\langle(t,1)\tau\rangle$   is a  rank~$3$ subgroup of $G$.
Now, by Lemma~\ref{lemma:necessary}, there exist $(v_1,v_2)\in V$ and $(h_1,h_2)\in H_0\times H_0$ such that $g:=\tau_{(v_1,v_2)}(h_1,h_2)$ fixes $L$ and maps $(0,0)$ to $(u,u)$,   where  $\tau_{(v_1,v_2)}$ denotes the translation of $V$  by $(v_1,v_2)$. Then    $L=L^g=\{(v^{h_1}+u,v^{\alpha h_2}+u) : v \in M\}$. In particular, $\tau_{v_1}h_1$ fixes $M$ and maps $0$ to $u$,   where $\tau_{v_1}$ denotes the translation of $V_n(p)$ by $v_1$. Thus $H_M$ is transitive on $M$, and the claim holds. 

Next we claim that that $H_M^M$ is sharply $2$-transitive, in which case $(H_0,M)$ sharply generates $(V_n(p),M^H)$.  Since $H_M^M$ is $2$-transitive, it suffices to prove that $H_{0,u}$ fixes $C=M^*$ pointwise. Let $h\in H_{0,u}$. Now $(h,1)\in G_{0,B}$, so   $(v^h,v^\alpha)\in B$ for all $v\in C$, in which case $(v^h)^\alpha=v^\alpha$ for all $v\in C$.  Since $\alpha$ is injective, $v^h=v$ for all $v\in C$, as desired.

Now we claim that $(H_0,M,\alpha)$ is $u$-compatible with  $(V_n(p),M^H)$. For condition (a) of Definition~\ref{defn:A2badtriple}, since $\alpha^2=\overline{t}\in N_M^M$, 
it remains to prove that $\alpha$ normalises $R$. Let $h\in H_{0,M}$. There exists $h'\in H_{0,M}$ such that $u^{h'}=(u^h)^\alpha$. Now $(h,h')\in G_{0,B}$, so $L=\{(v^h,v^{\alpha h'}) : v\in M\}$, in which case  $\overline{h}\alpha=\alpha\overline{h'}$. Thus $\alpha$ normalises $R$.
We saw above that there exist $h_1,h_2\in H_0$ such that $L=\{(v^{h_1}+u,v^{\alpha h_2}+u) : v \in M\}$. In particular, $M^{h_1}+u=M$ and $(v^{h_1}+u)^\alpha=v^{\alpha h_2}+u^\alpha$ for all $v\in M$, so condition (b) of Definition~\ref{defn:A2badtriple} holds, and the claim follows. 

It remains to prove that $G_0\leq \mathcal{G}(H_0,M,\alpha)$ and $\mathcal{L}=\mathcal{L}(H_0,M,\alpha)$. By definition,  $\mathcal{L}(H_0,M,\alpha)\subseteq L^G=\mathcal{L}$, so it suffices to show that $G_0\leq \mathcal{G}(H_0,M,\alpha)$. Since $\overline{t}=\alpha^2$ where $t\in N_M$ and $G_0=G_{0,V_1}\langle (t,1)\tau\rangle$, we  only need to show that $G_{0,V_1}\leq (H_0\times H_0)\mathcal{N}(H_0,M,\alpha)$. Let $(g_1,g_2)\in G_{0,V_1}$. There exist $h_1,h_2\in H_0$ such that $u^{h_1g_1}=u=u^{h_2g_2}$. Let $s_1:=h_1g_1$ and $s_2:=h_2g_2$. Now  $(s_1,s_2)\in G_{0,B}$, so $L=\{(v^{s_1},v^{\alpha s_2}) : v\in M\}$, and  $s_1,s_2\in N_M$ since $H_0\unlhd G_0^1\leq N$. Further, $\overline{s}_1\alpha=\alpha \overline{s}_2$, so $\overline{s}_2=\overline{s}_1^\alpha$. Thus $(g_1,g_2)=(h_1^{-1},h_2^{-1})(s_1,s_2)\in (H_0\times H_0)\mathcal{N}(H_0,M,\alpha)$.
\end{proof}

Now we prove that Examples~\ref{example:(I)dep}--\ref{example:(I)spor} are the only partial linear spaces $\mathcal{S}$ on $V_n(p)\times V_n(p)$ with an affine automorphism group in class (I) but not (I0) for which $\mathcal{S}(0)=V_n(p)^*\times V_n(p)^*$.  
In particular, we see that with the exception of  Example~\ref{example:(I)spor}, every such partial linear space has the form of Definition~\ref{defn:A2badfamily}. We also provide more information about the possible rank~$3$ groups that arise.  Although the groups $\mathcal{G}(H_0,M,1)$ and $\mathcal{G}(H_0,M,\alpha)$ are often conjugate, we work with arbitrary $\alpha$ since this is not true in general (see Examples~\ref{example:(I)sl25} and~\ref{example:(I)E}). For the following, recall the definition of   $\zeta_q$ and $\sigma_q$ from  \S\ref{ss:basicsvs}.

\begin{prop}
\label{prop:(I)badmore}
Let $G$ be an affine permutation group of rank~$3$ on $V:=V_n(p)\oplus V_n(p)$ where $n\geq 2$, $p$ is prime, $G_0=G_{0,V_1}\langle (t,1)\tau\rangle \leq \GL_n(p)\wr \langle \tau\rangle$ 
and $G_0$ is not a subgroup of $\GammaL_1(p^n)\wr \langle\tau\rangle$. Let $u\in V_n(p)^*$. Then $\mathcal{S}:=(V,\mathcal{L})$ is a $G$-affine proper partial linear space in which $\mathcal{S}(0)=V_n(p)^*\times V_n(p)^*$ if and only if one of the following holds  (up to conjugacy in $\GL_n(p)\wr\langle\tau\rangle$).
\begin{itemize}
\item[(i)] $\mathcal{L}=\mathcal{L}(H_0,M,\alpha)$ and $G_0\leq \mathcal{G}(H_0,M,\alpha)$  for some triple $(H_0,M,\alpha)$ that is linearly $u$-compatible with a  linear space on $V_n(p)$. Here $\alpha\in T_u$ where $T:=N_{\GL(M,\mathbb{F}_p)}(H_{0,M}^M)$, and one of the following holds. 
\begin{itemize}
\item[(a)] $H_0=\GL_m(q)$ and $M=\langle u\rangle_F$ for some subfield $F$ of $\mathbb{F}_q$ where $q^m=p^n$, $m\geq 2$ and $|F|\geq 3$. Here there exists $\theta\in\Aut(F)$ such that $(\lambda u)^\alpha=\lambda^\theta u$ for all $\lambda\in F$, and  $\theta^2=\pi|_F$ where $\pi\in \Aut(\mathbb{F}_q)$ is such that $t$ is $\pi$-semilinear.
\item[(b)] $H_0$ is regular on $V_n(p)^*$ and $M=V_n(p)$ where $n=2$ and $(p,H_0,T_u)$ is given by Table~\emph{\ref{tab:sharp}}. Either $H_0\times H_0\unlhd G_0$ and $\alpha^2t^{-1}\in H_0$, or  one of the following holds. 
\begin{itemize}
\item[(1)]   $p=7$ and either  $G_0=X_0$ and $\alpha=1$,  
or $G_0=X_0^{g_\alpha}\{(s,s): s \in T_u\}$  where $g_\alpha=(1,\alpha)$, $X_0=(\SL_2(3)\times \SL_2(3))\langle (r,1)\tau\rangle$ and $H_0=2\nonsplit S_4^-=\SL_2(3)\langle r\rangle$.
\item[(2)]  $p=23$ and  $G_0=(K\times K)\langle (t,1)\tau\rangle$ where  $K:=\SL_2(3)\times C_{11}$ is the unique index $2$  subgroup of  $H_0=T=2\nonsplit S_4^-\times C_{11}$,  and $t\in H_0\setminus  K$.
\end{itemize}

\item[(c)] $H_0=\SL_2(13)$ and $M=C_{V_6(3)}(H_{0,u})\simeq V_2(3)$ where $p^n=3^6$. Here $T_u\simeq S_3$, $\alpha^2=1$ and 
$G_0=\SL_2(13)\wr\langle \tau\rangle
=\mathcal{G}(H_0,M,\alpha)$.

\item[(d)]  $H_0=\SL_2(5)\langle \zeta_9^2,\zeta_9\sigma_9\rangle\leq\GammaL_2(9)$ and $M=\langle u\rangle_{\mathbb{F}_9}$ where $p^n=3^4$  and $u^{\sigma_9}=u$.  Here $T_u\simeq S_3$, $\alpha^2=1$,
$\SL_2(5)\times \SL_2(5)\unlhd G_0\leq H_0\wr\langle\tau\rangle\langle (\sigma_9,\sigma_9)\rangle=\mathcal{G}(H_0,M,1)$,  and if $G_0\nleq H_0\wr\langle\tau\rangle$, then $\alpha\in\{1,\sigma_9\}$.

\item[(e)]  
$H_0=(D_8\circ Q_8).C_{5}$ and $M=C_{V_4(3)}(H_{0,u})\simeq V_2(3)$ where $p^n=3^4$. Here $T_u\simeq S_3$, $\alpha^2=1$, $H_0\times H_0\unlhd G_0\leq (H_0\langle s^2\rangle)\wr\langle\tau\rangle\langle (s,s)\rangle=\mathcal{G}(H_0,M,1)$,   $H_0\langle s\rangle=N_{\GL_4(3)}(H_0)$, $u^s=u$, $M^s =M$,
 and if $G_0\nleq (H_0\langle s^2\rangle)\wr\langle\tau\rangle$, then $\alpha\in\{1,s|_M\}$.  
\end{itemize}

\item[(ii)] $\mathcal{L}=\{\{(v,v^{g})+w : v\in V_n(p)\} : g\in \SL_2(3)\alpha\cup \SL_2(3)\alpha^2r, w\in V\}$ and $G_0=X_0$ or $X_0\{(s,s): s \in T_u\}$,  where $p^n=7^2$,
$T=N_{\GL_n(p)}(2\nonsplit S_4^-)$,   $X_0=(\SL_2(3)\times \SL_2(3))\langle (r,1)\tau\rangle$,  $2\nonsplit S_4^-=\SL_2(3)\langle r\rangle$, $\alpha\in T_u\setminus \{1\}$ and $T_u\simeq C_3$.
\end{itemize}
\end{prop}

\begin{proof}
If (i) holds, then $\mathcal{S}$ is a $G$-affine proper partial linear space by Lemma~\ref{lemma:A2badgeneric} and the information given in Examples~\ref{example:(I)dep}--\ref{example:(I)E}. If (ii) holds, then $\mathcal{S}$ is a $G$-affine proper partial linear space by the information given in Example~\ref{example:(I)spor}.

Conversely, suppose that $\mathcal{S}:=(V,\mathcal{L})$ is a $G$-affine proper partial linear space in which $\mathcal{S}(0)=V_n(p)^*\times V_n(p)^*$. There exists  $L\in\mathcal{L}_0$  such that $(u,u)\in L$. Let  $B:=L^*$. By Lemma~\ref{lemma:necessary}, $B$ is a  block of $G_0$ on $V_n(p)^*\times V_n(p)^*$. By Lemma~\ref{lemma:A2badbasic}, $L=\{(v,v^\alpha) : v\in M\}$ for some $M\subseteq V_n(p)$ with $|M|\geq 3$ and  injective map $\alpha:M\to V_n(p)$, where $\alpha$ fixes $0$ and $u$.   By assumption,  $G_0$ belongs to one of the classes (I1)--(I8), so that (in the notation of Theorem~\ref{thm:rank3}),  $G_0\leq \GammaL_m(q)\wr \langle\tau\rangle$ where $q^m=p^n$ and $m\geq 2$.   Recall that $t\in G_0^1=G_0^2\leq\GammaL_m(q)$. Let $\sigma:=\sigma_q$.  We may assume without loss of generality that $u^\sigma=u$. Let $\zeta:=\zeta_q$. 
  Let $\pi\in\Aut(\mathbb{F}_q)$ be such that $t$ is $\pi$-semilinear. 

Assume for this paragraph that  $H_0\times H_0\unlhd  G_0$ where $H_0\leq \GammaL_m(q)$ is transitive on $V_n(p)^*$. Now  $u^{ht}=u$ for some $h\in H_0$, so by Lemma~\ref{lemma:A2badequiv}, $\mathcal{L}=\mathcal{L}(H_0,M,\alpha)$, $G_0\leq \mathcal{G}(H_0,M,\alpha)$ and $(H_0,M,\alpha)$ is $u$-compatible with a  linear space on $V_n(p)$.  
 In particular,  $M\subseteq C_{V_n(p)}(H_{0,u})$.
  Further, $\alpha^2=\overline{ht}$ and $ht\in N_M$, where  $N:=N_{\GL_n(p)}(H_0)$ and $ ^-:N_M\to N_M^M$ is the natural homomorphism. Note that  the line $M$  is  described by Theorem~\ref{thm:Kantor}, and  $H_0$ is regular on $V_n(p)^*$ when $M=V_n(p)$. Note also that if $L$ is an $\mathbb{F}_p$-subspace of $V$, then $M$ is an $\mathbb{F}_p$-subspace of $V_n(p)$ and $\alpha\in\GL(M,\mathbb{F}_p)$, in which case $\alpha\in T_u$ where $T:=N_{\GL(M,\mathbb{F}_p)}(H_{0,M}^M)$ and $(H_0,M,\alpha)$ is linearly $u$-compatible. 
Suppose now that $H_0\leq \GL_m(q)$ and  $M=\langle u\rangle_{F}$ for some subfield $F$ of $\mathbb{F}_q$. We claim that 
(i)(a) holds. Note that $H_{0,M}^M=\GL_m(q)_M^M\simeq \GL_1(F)$.  There exists $g\in H_0$ such that $u^g=-u$. Now $L\subseteq M\oplus M$ since $\alpha\in\Sym(M)$, and $g\in \GL_m(q)$, so $L$ is an $\mathbb{F}_p$-subspace of $V$ by Lemma~\ref{lemma:affine}. Thus  $\alpha\in T_u\simeq \Aut(F)$, and we have already seen that $\alpha^2=\overline{ht}\in \GammaL_m(q)_M^M$, so  $(\GL_m(q),M,\alpha)$ is linearly $u$-compatible. Further, there exists $\theta\in \Aut(F)$ such that $(\lambda u)^\alpha=\lambda^\theta u$ for all $\lambda\in F$, and $\theta^2=\pi|_F$. Now $L^{\GL_m(q)\times \GL_m(q)}=L^{H_0\times H_0}$, so $\mathcal{L}(\GL_m(q),M,\alpha)=\mathcal{L}(H_0,M,\alpha)=\mathcal{L}$. Recall that $G_0=G_{0,V_1}\langle (ht,1)\tau\rangle$, and $G_{0,V_1}\leq (H_0\times H_0)\{ (s_1,s_2)\in N_M\times N_M : \overline{s}_2=\overline{s}_1^\alpha\}$. Since $G_{0,V_1}\leq \GammaL_m(q)\times \GammaL_m(q)$ and $ht$ is $\pi$-semilinear, it follows that $G_0\leq \mathcal{G}(\GL_m(q),M,\alpha)$ (cf. Example~\ref{example:(I)dep}), and the claim holds.

We now  consider the various possibilities for $G_0$.  Whenever $H_0\times H_0\unlhd  G_0$ such that  $H_0\leq \GammaL_m(q)$ is transitive on $V_n(p)^*$, we make use of the above observations and notation.

Suppose that $G_0$ belongs to one of (I1)--(I3). We may take $H_0$ to be $\SL_m(q)$, $\Sp_m(q)$ or $G_2(q)'$,  respectively (by~\cite[p.125]{Wil2009} when $H_0=G_2(q)'$).  Now $H_0$ is not regular on $V_n(p)^*$, so $M\neq V_n(p)$, in which case
 $M=\langle u\rangle_{F}$ for some subfield $F$ of $\mathbb{F}_q$ by Theorem~\ref{thm:Kantor}. Thus (i)(a) holds.

Suppose that $G_0$ belongs to (I5). We may take $H_0$ to be $A_6$ or $A_7$. Now $H_0$ is not regular on $V_n(p)^*$, so $M\neq V_n(p)$, but this contradicts Theorem~\ref{thm:Kantor}.

Suppose that $G_0$ belongs to (I6). We may take $H_0$ to be $\SL_2(13)$, which is self-normalising in $\GL_6(3)$.  Since $-1\in H_0$, Lemma~\ref{lemma:affine} implies that $L$ is a subspace of $V_6(3)$, so  $M$ is a subspace of $C_{V_6(3)}(\SL_2(13)_u)\simeq V_2(3)$.  Either  $M=\langle u\rangle_{\mathbb{F}_3}$, in which case (i)(a) holds, or $M=C_{V_6(3)}(\SL_2(13)_u)$, in which case     (i)(c) holds by Example~\ref{example:(I)sl213}.  

Suppose that $G_0$ belongs to (I7), and let $N^*:=N_{\GL_m(q)}(\SL_2(3))$. Then $G_0\leq  N^*\wr\langle\tau\rangle$. Note that $m=2$ and $q=p$. Using {\sc Magma}, we determine that $-1\in G_0$,  so $L$ is an $\mathbb{F}_p$-subspace of $V$ by Lemma~\ref{lemma:affine}. In particular, $M$ is an $\mathbb{F}_p$-subspace of $V_2(p)$ and $|L|$ divides $p^2$.  If $|L|=p$, then $L=\langle (u,u)\rangle_{\mathbb{F}_p}$ and $G_0\leq \GL_2(p)\wr\tau = \mathcal{G}(\GL_2(p),\langle u\rangle_{\mathbb{F}_p},1)$, so (i)(a) holds. Otherwise, $|L|=p^2$ and $M=V_2(p)$. 
 Using {\sc Magma}, we determine that $N^*$ has a subgroup $R$ that is regular on $V_2(p)^*$ and  isomorphic to a group in Table~\ref{tab:sharp}. Moreover,   $N^*=N_{\GL_2(p)}(R)$, so $G_0$ normalises $R\times R$. If $R\times R\subseteq  G_0$, then we may take $H_0$ to be $R$, in which case $\alpha^2t^{-1}=h\in H_0$ and  $T_u$ is  given by Table~\ref{tab:sharp}, so  (i)(b) holds. Suppose instead that $R\times R$ is not a subgroup of $G$. Since $G_0\nleq \GammaL_1(p^2)\wr\langle\tau\rangle$, it can be checked using {\sc Magma} that either $p=23$ and (i)(b)(2) holds, or $p=7$, in which case either (i)(b)(1) holds,  or (ii) holds  (up to conjugacy in $N^*\times N^*$).

Suppose that $G_0$ belongs to (I4). Then $S\times S\unlhd G_0\leq N^*\wr\langle \tau\rangle$ where $S:=\SL_2(5)$, $N^*:=N_{\GammaL_m(q)}(S)$ and $m=2$. 
 Now $-1\in G_0$, so $L$ is an $\mathbb{F}_p$-subspace of $V$ by Lemma~\ref{lemma:affine}. In particular, $M$ is an $\mathbb{F}_p$-subspace of $V_n(p)$ and $|L|$ divides $q^2$. 
Suppose that $q\in \{11,19,29,59\}$ (we treat the case $q=9$ below). Now $q=p$, so if $|L|=p$, then $L=\langle (u,u)\rangle_{\mathbb{F}_p}$ and $G_0\leq \GL_2(p)\wr\tau = \mathcal{G}(\GL_2(p),\langle u\rangle_{\mathbb{F}_p},1)$, so (i)(a) holds. Otherwise, $|L|=p^2$ and $M=V_2(p)$.  By a computation in {\sc Magma}, $N^*=S\langle \zeta\rangle$, so $|N^*_u|=60/(q+1)$. Further, if $q=19$, then $G_0=N^*\wr \langle \tau\rangle$, so we may take $H_0$ to be $N^*$, but $N^*$ is not regular on $V_2(p)^*$, a contradiction. Thus $q\in \{11,29,59\}$. Using {\sc Magma}, we determine that $N^*$ has a subgroup $R$ that is regular on $V_2(p)^*$ and  isomorphic to a group in Table~\ref{tab:sharp}. Moreover,   $N^*=N_{\GL_2(p)}(R)$ and $R\times R\unlhd G_0$. Thus we may take $H_0$ to be $R$, in which case $\alpha^2t^{-1}=h\in H_0$ and  $T_u$ is  given by Table~\ref{tab:sharp}, so  (i)(b) holds.

Still assuming that $G_0$ lies in (I4), suppose that $q=9$. 
If $|L|=3$, then $L=\langle (u,u)\rangle_{\mathbb{F}_3}$ and $G_0\leq \GammaL_2(9)\wr\langle\tau\rangle=\mathcal{G}(\GL_2(9),\langle u\rangle_{\mathbb{F}_3},1)$, so (i)(a) holds. Otherwise, $|L|\in \{9,81\}$ since $|L|$ divides $81$ and $|B|$ divides 6400. 
By a computation in {\sc Magma} and  Examples~\ref{example:(I)dep} and~\ref{example:(I)sl25}, 
if   $G_0\leq \mathcal{G}(\GL_2(9),\langle u\rangle_{\mathbb{F}_9},1)=\GL_2(9)\wr \langle\tau\rangle\langle (\sigma,\sigma)\rangle$, then   (i)(a) holds, and if not,  then 
(i)(d) holds.

Suppose that $G_0$ belongs to (I8). Then $(D_8\circ Q_8)\times (D_8\circ Q_8)\unlhd G_0\leq N^*\wr\langle\tau\rangle$ where $N^*:=N_{\GL_m(q)}(D_8\circ Q_8)$, $m=4$ and $q=p=3$. Now $-1\in G_0$, so $L$ is an $\mathbb{F}_3$-subspace of $V$ by Lemma~\ref{lemma:affine}. In particular, $M$ is a subspace of $V_4(3)$ and $|L|$ divides $3^4$. If $|L|=3$, then $L=\langle (u,u)\rangle_{\mathbb{F}_3}$ and $G_0\leq \GL_4(3)\wr\tau = \mathcal{G}(\GL_4(3),\langle u\rangle_{\mathbb{F}_3},1)$, so (i)(a) holds.  Otherwise, $|L|>3$. By  Example~\ref{example:(I)E},  $\GL_4(3)$ has a subgroup $H_0:=(D_8\circ Q_8).C_5$ that is transitive on $V_4(3)^*$. Further, $N:=N_{\GL_4(3)}(H_0)=(D_8\circ Q_8).\AGL_1(5)=H_0\langle s\rangle$, where $N_u=\langle s\rangle\simeq C_8$ and  $H_{0,u}=\langle s^4\rangle$.  By a computation in {\sc Magma}, there exists $g\in N^*\times N^*$ such that $H_0\times H_0\unlhd G_0^g$. 
(Note that $H_0$ is not normal in $N^*$ since  $N<(D_8\circ Q_8).S_5=N^*$ by Remark~\ref{remark:2trans}.) 
Thus we may assume that  
 $H_0\times H_0\unlhd G_0=G_{0,V_1}\langle(t,1)\tau\rangle$, in which case  $M=C_{V_4(3)}(H_{0,u})\simeq V_2(3)$ and (i)(e) holds. (The additional statements  can be verified using {\sc Magma} and  Example~\ref{example:(I)E}.)
\end{proof}

For Proposition~\ref{prop:(I)badmore}(ii) and  cases (a), (c), (d) and (e) of Proposition~\ref{prop:(I)badmore}(i), the number of examples for a particular group $G$ can be determined directly from the conditions on $\alpha$, as described in Examples~\ref{example:(I)spor}, 
\ref{example:(I)dep},~\ref{example:(I)sl213},~\ref{example:(I)sl25} and~\ref{example:(I)E}, respectively. This can also be done for Proposition~\ref{prop:(I)badmore}(i)(b)  when $H_0\times H_0\unlhd G_0$ and $\alpha^2t^{-1}\in H_0$, as follows. Observe that if $T_u$ is abelian and $G_0\leq \mathcal{G}(H_0,V_2(p),\alpha_1)\cap\mathcal{G}(H_0,V_2(p),\alpha_2)$ for some $\alpha_1,\alpha_2\in T_u$, then $\alpha_1^2=r_1t$ and $\alpha_2^2=r_2t$ for some $r_1,r_2\in H_0$,  so $\alpha_1^2=\alpha_2^2$  since $H_{0,u}=1$,  in which case  $\mathcal{G}(H_0,V_2(p),\alpha_1)=\mathcal{G}(H_0,V_2(p),\alpha_2)$ by Lemma~\ref{lemma:A2badgeneric} (cf. Example~\ref{example:(I)reg}). In particular, if $T_u$ is cyclic of odd order, then $G_0$ admits a unique example, and if $T_u=C_2$ or $C_4$, then $G_0$ admits exactly two examples. Otherwise, $T_u=S_3$ and the groups $\mathcal{G}(H_0,V_2(3),\alpha)$ are pairwise distinct for $\alpha\in T_u$. In particular,  both $\mathcal{G}(H_0,V_2(3),1)$ and its subgroup $H_0\wr\langle\tau\rangle\{(s,s): s \in A_3\}$ admit exactly one example,
 while $H_0\wr\langle\tau\rangle\{(s,s^{-1}): s\in A_3\}=\bigcap_{\alpha\in S_3\setminus A_3} \mathcal{G}(H_0,V_2(3),\alpha)$  admits exactly three examples. Up to conjugacy, these are all of the rank~$3$ groups. (Note that $H_0\wr\langle\tau\rangle=\bigcap_{\alpha\in (S_3\setminus A_3)\cup\{1\}} \mathcal{G}(H_0,V_2(3),\alpha)$   admits exactly four examples, but $H_0=Q_8\leq \GammaL_1(9)$, so this group is not considered in Proposition~\ref{prop:(I)badmore}.)
 
 \begin{proof}[Proof of Proposition~\emph{\ref{prop:(I)bad}}]
 This follows from Proposition~\ref{prop:(I)badmore}, Lemma~\ref{lemma:A2badgeneric}(iii), and the information given in  Examples~\ref{example:(I)dep}--\ref{example:(I)spor}.
 \end{proof}

\section{The extraspecial class (E)}
 \label{s:(E)}
 
  We begin by describing the exceptional partial linear spaces that arise for class (E). Recall the definition of   $\zeta_q$ from  \S\ref{ss:basicsvs}.

\begin{example}
\label{example:(E)walker}
Let $E:=D_8\circ Q_8$, and let $G:=V{:}N_{\GL_4(5)}(E)$ where  $V:=V_4(5)$. 	Now $G$ has rank~$3$ with subdegrees   $240$ and $384$. By a computation in {\sc Magma},  $G_0=(\SL_2(5)\circ E \circ \langle \zeta_5\rangle)\nonsplit 2$, and the proper subgroups of $G_0$ with two orbits on $V^*$ are $\SL_2(5)\circ E$ and $\SL_2(5)\circ E\circ \langle \zeta_5\rangle$.  The exceptional Walker plane of order $25$ is an affine plane with group of automorphisms $G$, and $G$ has  $\ell_\infty$-orbit lengths $(10,16)$   (see~\cite{Ost1981,Walk1979} and~\cite[Corollary~8]{BilJoh2001}). Now $G$ has two orbits $\mathcal{L}_1$ and $\mathcal{L}_2$ on the lines of this affine plane, so   $\mathcal{S}_1:=(V,\mathcal{L}_1)$ and  $\mathcal{S}_2:=(V,\mathcal{L}_2)$ are  $G$-affine proper partial linear spaces by Lemma~\ref{lemma:linearspace}. Using {\sc Magma}, we determine that $\Aut(\mathcal{S}_1)=\Aut(\mathcal{S}_2)=G$. Further, the lines of $\mathcal{S}_1$ and $\mathcal{S}_2$ are affine subspaces of $V$.\end{example}

\begin{example}
\label{example:(E)mason-ostrom}
Let $E:=D_8\circ Q_8$, and let $G:=V{:}N_{\GL_4(7)}(E)$ where $V:=V_4(7)$. Now $G$ has  rank~$3$ with subdegrees $480$ and $1920$. By a computation in {\sc Magma},   $G_0=2\nonsplit S_5^+\circ E\circ \langle \zeta_7\rangle$, and the proper subgroups of $G_0$ with two orbits on $V^*$ are $2\nonsplit S_5^+\circ E$ or a conjugate of $(E\circ \langle \zeta_7\rangle)\nonsplit \AGL_1(5)$. The Mason-Ostrom plane of order $49$ is an affine plane with group of automorphisms $G$, and $G$ has $\ell_\infty$-orbit lengths $(10,40)$ (see~\cite{MasOst1985} and~\cite[Corollary~8]{BilJoh2001}). Now $G$ has two orbits $\mathcal{L}_1$ and $\mathcal{L}_2$ on the lines of this affine plane, so   $\mathcal{S}_1:=(V,\mathcal{L}_1)$ and  $\mathcal{S}_2:=(V,\mathcal{L}_2)$ are  $G$-affine proper partial linear spaces by Lemma~\ref{lemma:linearspace}. Using {\sc Magma}, we determine that $\Aut(\mathcal{S}_1)=\Aut(\mathcal{S}_2)=G$. Further, the lines of $\mathcal{S}_1$ and $\mathcal{S}_2$ are affine subspaces of $V$.
\end{example}

 \begin{example}
 \label{example:(E)nonplane}
 Let $E:=Q_8\circ Q_8\circ Q_8$, and let $G:=V{:}N_{\GL_8(3)}(E)$ where $V:=V_8(3)$. By~\cite[p.485]{Lie1987}, $G_0=E\nonsplit \SO^-_6(2)\leq \GSp_8(3)$, and $G$ has rank~$3$ with subdegrees $1440$ and $5120$. By a computation in {\sc Magma}, the only proper subgroup of $G_0$ with two orbits on $V^*$ is $E\nonsplit \Omega_6^-(2)$. Now  $G_0$ has an orbit $\Delta$ on the points of $\PG_7(3)$ with size $720$. By a computation in {\sf GAP} using~\cite{FinInG}, there is a partition $\Sigma$ of $\Delta$ into hyperbolic lines---that is,  projective lines arising from non-degenerate $2$-subspaces of $V$ with respect to the symplectic form preserved by $\Sp_8(3)$---such that $G_0$ acts transitively on the elements of $\Sigma$. 
Embed $\PG_7(3)$ as a hyperplane $\Pi$ in $\PG_8(3)$. View $V$ as the set of points in $\AG_8(3)$ (i.e., the set of points in $\PG_8(3)$ that are not in $\Pi$), 
 and let $\mathcal{L}$ be the set of affine  $2$-subspaces of $\AG_8(3)$   whose completions meet $\Pi$ in an element of  $\Sigma$. Then $\mathcal{S}:=(V,\mathcal{L})$ is a $G$-affine proper partial linear space with line-size $9$. By a computation in {\sc Magma}, $\Aut(\mathcal{S})=G$.
  \end{example}	

 \begin{prop}
 \label{prop:(E)}
Let $G$ be an affine permutation group of rank~$3$ on $V:=V_d(p)$  where $p$ is prime, $E\unlhd G_0$,  $E$ is irreducible on $V$  and $(E,p^d,G)$ is given by Table~\emph{\ref{tab:E}}. Then $\mathcal{S}$ is a
$G$-affine proper partial linear space  if and only if one of the following holds.
\begin{itemize}
\item[(i)] $\mathcal{S}=(V,\{\langle u\rangle_{\mathbb{F}_p}+v : u\in X, v\in V\})$ where $X$ is an orbit of $G_0$ on $V^*$ and $p\neq 2$.
\item[(ii)] $(E,p^d)=(3^{1+2},2^6)$ and $\mathcal{S}=(V,\{\langle u\rangle_{\mathbb{F}_4}+v : u\in X, v\in V\})$ where $V=V_3(4)$ and $X$ is an orbit of $G_0$ on $V^*$. Here $\mathcal{S}$ has an automorphism group in class~\emph{(R1)}.
\item[(iii)] $(E,p^d)=(D_8\circ Q_8,5^4)$
 and $\mathcal{S}$ is described in Example~\emph{\ref{example:(E)walker}}. 
\item[(iv)] $(E,p^d)=(D_8\circ Q_8,7^4)$ and $\mathcal{S}$ is described in Example~\emph{\ref{example:(E)mason-ostrom}}. 
\item[(v)] $(E,p^d)=(Q_8\circ Q_8\circ Q_8,3^8)$ and $\mathcal{S}$ is described in Example~\emph{\ref{example:(E)nonplane}}. 
\end{itemize}
\end{prop}

\begin{proof}
 Note that by Corollary~\ref{cor:remark3}, $\langle x\rangle_{\mathbb{F}_p}^*\subseteq x^{G_0}$ for all $x\in V^*$. Further, if $(E,p^d)=(3^{1+2},2^6)$, then $G_0\leq\GammaU_3(2)$ and $V=V_3(4)$, so $\langle x\rangle_{\mathbb{F}_4}^*\subseteq x^{G_0}$ for all $x\in V^*$. In particular, if (i) or (ii) holds, then $\mathcal{S}$ is a $G$-affine proper partial linear space, and when (ii) holds, $V{:}\GammaU_3(2)$ is an automorphism group of $\mathcal{S}$ in class (R1). Moreover, if one of (iii)--(v) holds, then we have already seen in Examples~\ref{example:(E)walker}--\ref{example:(E)nonplane} that $\mathcal{S}$ is a $G$-affine proper partial linear space.
 
 Suppose that $\mathcal{S}$ is a $G$-affine proper partial linear space. Let $L\in\mathcal{L}_0$, let $B:=L^*$, and let $x\in B$. Since $E$ is irreducible,  $G$ is primitive by Lemma~\ref{lemma:primitiveirreducible}, so  $B$ is a non-trivial block of $G_0$ on $x^{G_0}$ by  Lemma~\ref{lemma:necessary}. If $p$ is odd, then since $E$ is irreducible, its  central involution is $-1$. 
Thus $L$ is an $\mathbb{F}_p$-subspace of $V_d(p)$ for all possible $p$ by Lemma~\ref{lemma:affine}. If $|L|=p$, then (i) holds, so we assume that $|L|\geq p^2$. In particular,  $d>2$.  If $(E,p^d)=( D_8\circ Q_8\circ \langle\zeta_{5}\rangle,5^4)$ and $G_0\nleq N_{\GL_4(5)}(D_8\circ Q_8)$, then there are no examples by a computation in {\sc Magma}. Otherwise, $(E,p^d)$ is one of $(3^{1+2},2^6)$, $(D_8\circ Q_8,5^4)$, $(D_8\circ Q_8,7^4)$ or   $(Q_8\circ Q_8\circ Q_8,3^8)$, in which case one of (ii)--(v) holds   by a computation in {\sc Magma}.
 \end{proof}

 \section{The almost simple class (AS)}
 \label{s:(AS)}
 
 We begin by describing the exceptional partial linear spaces that arise for class (AS). Note that when computing with {\sc Magma} in these examples, we often use~\cite{web-atlas} to construct the representation of the appropriate quasisimple group. Recall the definition of   $\zeta_q$ and $\sigma_q$ from  \S\ref{ss:basicsvs}.

 \begin{example}
 \label{example:(AS)nearfield}
 Let $S:=2\nonsplit A_5\simeq \SL_2(5)$. By~\cite[Lemma 11.2]{Fou1964}, the normaliser $N$ of $S$ in~$\GammaL_2(9)$  is $S\langle \zeta,\sigma\rangle$ where $\zeta:=\zeta_9$ and $\sigma:=\sigma_9$, and $N$ is transitive on $V^*$ where $V:=V_2(9)$. Moreover, by~\cite[Theorem 5.3]{FouKal1978}, the  subgroups of $N$ containing $S$ with two orbits on $V^*$  are $S$,  $S\langle \zeta^2\rangle$, $S\langle \sigma\rangle$, $S\langle \zeta^2\sigma\rangle$ and  $S\langle \zeta^2,\sigma\rangle$; these all have orbit sizes $40$ and $40$ on $V^*$. Note that $S\langle \sigma\rangle^\zeta=S\langle \zeta^2\sigma\rangle\simeq 2\nonsplit S_5^+$.  
 The nearfield plane of order $9$  is an affine plane whose full automorphism group is the $2$-transitive group $V{:}(D_8\circ Q_8\circ S\langle \zeta\sigma\rangle)$ (see~\S\ref{ss:nearfield} and   Remark~\ref{remark:2trans}), and $V{:}S$ has $\ell_\infty$-orbit lengths $(5,5)$ (see~\cite[Corollary~8]{BilJoh2001}). Now $V{:}S$ has two orbits $\mathcal{L}_1$ and $\mathcal{L}_2$ on the lines of  the nearfield plane, so   $\mathcal{S}_1:=(V,\mathcal{L}_1)$ and  $\mathcal{S}_2:=(V,\mathcal{L}_2)$ are  $(V{:}S)$-affine proper partial linear spaces by Lemma~\ref{lemma:linearspace}. Using {\sc Magma}, we determine that, without loss of generality,  $\Aut(\mathcal{S}_1)=V{:}S\langle \sigma\rangle $ and $\Aut(\mathcal{S}_2)=V{:}S\langle \zeta^2\sigma\rangle$. 
  Thus $\mathcal{S}_1$ is a $(V{:}G_0)$-affine proper partial linear space for $G_0\in \{S,S\langle \sigma \rangle\}$, and $\mathcal{S}_2$ is a $(V{:}G_0)$-affine proper partial linear space for $G_0\in \{S,S\langle \zeta^2\sigma \rangle\}$.  
  Note that $\mathcal{S}_1\simeq \mathcal{S}_2$ since $\zeta\sigma$ maps $\mathcal{L}_1$ to $\mathcal{L}_2$.
 Note also that the lines of $\mathcal{S}_1$ are affine $\mathbb{F}_3$-subspaces of $V$.
  \end{example}

\begin{example}
\label{example:(AS)korch}
Let $S:=2\nonsplit A_5\simeq \SL_2(5)$ and $q:=49$. By a computation in {\sc Magma}, the normaliser $N$ of $S$ in $\GammaL_2(q)$ is $S\langle \zeta,\sigma\rangle$ where $\zeta:=\zeta_q$ and $\sigma:=\sigma_q$, and $N$ has
orbits $X$ and $Y$ on $V^*$ with sizes  $960$ and $1440$,  respectively,  where $V:=V_2(q)$. Moreover, the proper subgroups of $N$ with two orbits on $V^*$ are  $S\circ \langle \zeta\rangle$ and $S\langle \zeta\sigma, \zeta^2\rangle=2\nonsplit S_5^-\circ \langle \zeta^2\rangle$. Let $G:=V{:}S\langle \zeta\sigma, \zeta^2\rangle$. The Korchm{\'a}ros plane of order $49$ is an affine plane with group of automorphisms $G$, and $G$ has $\ell_\infty$-orbit lengths $(20,30)$~\cite{Kor1985}. Now $G$ has two orbits on the lines of this affine plane; let $\mathcal{L}$ be the orbit for which $|\mathcal{L}_0|=20$. By Lemma~\ref{lemma:linearspace}, $\mathcal{S}:=(V,\mathcal{L})$ is a  $G$-affine proper partial linear space with $\mathcal{S}(0)=X$, and using {\sc Magma}, we determine that $\Aut(\mathcal{S})=G$ and $\mathcal{L}$ consists of  affine $\mathbb{F}_7$-subspaces of $V$. We do not consider  the other partial linear space arising from the Korchm{\'a}ros plane here, for it is $\mathbb{F}_q$-dependent with 
 line set $\{\langle u\rangle_{\mathbb{F}_q}+v : u\in Y, v\in V\}$. Further, $\mathcal{S}$ is not isomorphic to the $\mathbb{F}_q$-dependent partial linear space  with line set $\{\langle u\rangle_{\mathbb{F}_q}+v : u\in X, v\in V\}$, for the  automorphism group of this latter partial linear space  contains $V{:}S\langle \zeta,\sigma\rangle$. 
\end{example}

\begin{example}
\label{example:(AS)walker}
Let $S:=2\nonsplit A_6$. By a computation in {\sc Magma}, the normaliser $N$ of $S$ in $\GL_4(5)$ is $(S\circ \langle \zeta_5\rangle)\nonsplit 2$, and $N$ has orbit sizes $144$ and $480$ on $V^*$ where $V:=V_4(5)$. Moreover, $N$ has no proper subgroups with two orbits on $V^*$. Let $G:=V{:}N$. The exceptional Walker plane of order $25$, or  Hering plane, is an affine plane  with group of automorphisms $G$, and $G$ has $\ell_\infty$-orbit lengths $(6,20)$ (see~\cite{Walk1979} and~\cite[Corollary~8]{BilJoh2001}).
Now $G$ has two orbits $\mathcal{L}_1$ and $\mathcal{L}_2$ on the lines of this affine plane, so   $\mathcal{S}_1:=(V,\mathcal{L}_1)$ and  $\mathcal{S}_2:=(V,\mathcal{L}_2)$ are  $G$-affine proper partial linear spaces by Lemma~\ref{lemma:linearspace}. Using {\sc Magma}, we determine that $\Aut(\mathcal{S}_1)=\Aut(\mathcal{S}_2)=G$. Further, the lines of $\mathcal{S}_1$ and $\mathcal{S}_2$ are affine subspaces of~$V$.
\end{example}

\begin{example}
\label{example:(AS)A9}
Let $S:=A_9$. 
By~\cite[p.509]{Lie1987}, 
the group $G:=V{:}S$ has rank~$3$, where $V:=V_8(2)$ and  $S\leq \Omega_8^+(2)$. The orbits of $S$ on $V^*$ are the sets of  non-singular and singular vectors with respect to the quadratic form preserved by $\Omega_8^+(2)$,  with sizes $120$ and $135$, respectively. By a computation in {\sc Magma}, $S$   has no proper subgroups with two orbits on $V^*$. 
The following definition is made in~\cite[\S 2.3.3]{BueVanBook1994}. Let $\Delta$ be the set of  points of $\PG_7(2)$ that are singular with respect to the quadratic form preserved by $\Omega_8^+(2)$. By~\cite{Dye1977}, there exists a  spread $\Sigma$ of $\Delta$---that is, a partition of the points of $\Delta$ into maximal totally singular subspaces---such that $S$ acts $2$-transitively on the elements of $\Sigma$.  Embed $\PG_7(2)$ as a hyperplane $\Pi$ in $\PG_8(2)$. View $V$ as the set of points in $\AG_8(2)$ (i.e. the set of points in $\PG_8(2)$ that are not in $\Pi$), and let $\mathcal{L}$ be the set of affine $4$-subspaces of $\AG_8(2)$ whose completions  meet $\Pi$ in an element of $\Sigma$. Then $\mathcal{S}:=(V,\mathcal{L})$ is a $G$-affine proper partial linear space with line-size $16$. By~\cite[\S 2.3.3]{BueVanBook1994}, $\Aut(\mathcal{S})=G$.
\end{example}

\begin{example}
\label{example:(AS)M11}
Let $S:=M_{11}$. Let $V$ be the irreducible $\mathbb{F}_3S$-module of dimension  $5$ such that $S$ has orbit sizes $110$ and $132$ on $V^*$.   Let $G:=V{:}S$.  Let $x$ be an element of the orbit of $S$ of size $132$ on $V^*$. Then $S_x\simeq A_5$, and there are exactly two maximal subgroups $H$ of $S$  for which $S_x\leq H\simeq \PSL_2(11)$. Let  $B:=x^H$. Now  $B$ is a block of $S$. Let $L:=B\cup \{0\}$ and $\mathcal{L}:=L^G$. By a computation in {\sc Magma}, $S_L$ is transitive on $L$, so by Lemma~\ref{lemma:sufficient}, $\mathcal{S}:=(V,\mathcal{L})$ is a $G$-affine proper partial linear space with line-size $12$ and point-size $12$.                                                                                                                                                                                                    
By a computation in {\sc Magma}, $\Aut(\mathcal{S})=G$.
Note that if we instead choose $H$ to be the other  maximal subgroup  of $S$  for which $S_x\leq H\simeq \PSL_2(11)$, then we obtain another  $G$-affine proper partial linear space, but we will see in Proposition~\ref{prop:(AS)} that these partial linear spaces are isomorphic. Note also that $S$ has no proper subgroups with two orbits on $V^*$.
\end{example}

\begin{example}
\label{example:(AS)G24}
Let $S:=2\nonsplit G_2(4)$. 
By a computation in {\sc Magma}, the normaliser $N$ of $S$ in $\GL_{12}(3)$ is $S.2^-$, and $N$ has orbit sizes $65520$ and $465920$ on $V^*$ where $V:=V_{12}(3)$. Moreover, the only proper subgroup of $N$ with two orbits on $V^*$ is $S$.  Recall from  Table~\ref{tab:AS} that   $S\leq \Sp_{12}(3)$, so $N\leq \GSp_{12}(3)$ by~\cite[Corollary~2.10.4(i)]{KleLie1990}. 
Let $G:=V{:}N$. 
Now $G_0$ has an orbit $\Delta$ on the points of $\PG_{11}(3)$ with size $32760$.   
  By a computation in {\sf GAP} using~\cite{FinInG}, there is a partition $\Sigma$ of $\Delta$ into hyperbolic lines---that is,  projective lines arising from non-degenerate $2$-subspaces of $V$ with respect to the symplectic form preserved by $\Sp_{12}(3)$---such that $G_0$ acts transitively on the elements of $\Sigma$. 
 Embed $\PG_{11}(3)$ as a hyperplane $\Pi$ in $\PG_{12}(3)$. View $V$ as the set of points in $\AG_{12}(3)$, and let $\mathcal{L}$ be the set of affine  $2$-subspaces of $\AG_{12}(3)$ whose completions  meet $\Pi$ in an element of $\Sigma$. Then $\mathcal{S}:=(V,\mathcal{L})$ is a $G$-affine proper partial linear space with line-size~$9$. 
We claim that $\Aut(\mathcal{S})=G$. By Theorem~\ref{thm:primrank3plus}, $\Aut(\mathcal{S})$ is an affine permutation group on $V$. A consideration of the subdegrees in Tables~\ref{tab:E}--\ref{tab:subdegree}  shows that $\Aut(\mathcal{S})$ belongs to class (AS) with $S=2\nonsplit G_2(4)$ or $2\nonsplit \Suz$. It then follows from Proposition~\ref{prop:(AS)} below that $\Aut(\mathcal{S})=G$.
\end{example}

\begin{example}
\label{example:(AS)J2} 
Let $S:=2\nonsplit J_2$. By a computation in {\sc Magma}, the normaliser $N$ of $S$ in $\GL_6(5)$ is $(S\circ \langle \zeta_5\rangle)\nonsplit 2$, and $N$ has orbit sizes $7560$ and $8064$ on $V^*$ where $V:=V_6(5)$. Moreover, the proper subgroups of $N$ with two orbits on $V^*$ are $S$ and $S\circ \langle \zeta_5\rangle$.  Recall from  Table~\ref{tab:AS} that $S\leq \Sp_{6}(5)$, so $N\leq \GSp_{6}(5)$ by~\cite[Corollary~2.10.4(i)]{KleLie1990}. 
Let $G:=V{:}N$. Now $G_0$ has an orbit $\Delta$ on the points of $\PG_5(5)$ with size $1890$.      By the proof of~\cite[Lemma~5.1]{Lie1987},   there is a partition $\Sigma$ of $\Delta$ into hyperbolic lines---that is,  projective lines arising from non-degenerate $2$-subspaces of $V$ with respect to the symplectic form preserved by $\Sp_6(5)$---such that $G_0$ acts transitively on the elements of $\Sigma$. (This partition arises from the $315$ quaternionic reflections described in~\cite[p.42]{Atlas}: these reflections form a conjugacy class $C$ of involutions of $S$, and for each $r\in C$, there is a unique non-degenerate $2$-subspace of $V$ on which $r$ acts as $-1$, and $C_{S}(r)$ acts as $\SL_2(5)$ on this $2$-subspace.) 
Embed $\PG_5(5)$ as a hyperplane $\Pi$ in $\PG_6(5)$. View $V$ as the set of points in $\AG_6(5)$, and let $\mathcal{L}$ be the set of affine $2$-subspaces of $\AG_6(5)$ whose completions  meet $\Pi$ in an element of $\Sigma$. Then $\mathcal{S}:=(V,\mathcal{L})$ is a $G$-affine proper partial linear space with line-size $25$. By  a computation in {\sc Magma}, $\Aut(\mathcal{S})=G$. 
\end{example}

\begin{prop}
\label{prop:(AS)}
Let $G$ be an affine permutation group of rank~$3$ on $V:=V_n(q)$ where $q^n=p^d$ for $p$ prime, $S\unlhd G_0$, 
  $S$ is irreducible on $V_d(p)$ and  $(S,p^d,G,q)$   is given by Table~\emph{\ref{tab:AS}}. Let $\zeta:=\zeta_q$. Then $\mathcal{S}$ is a
$G$-affine proper partial linear space  if and only if one of the following holds.
\begin{itemize}
\item[(i)] $\mathcal{S}=(V,\{\langle u\rangle_{\mathbb{F}_p}+v : u\in X, v\in V\})$ where $X$ is an orbit of $G_0$ on $V^*$ and $p\neq 2$.
\item[(ii)] $(S,p^d)$ is one of  $(2\nonsplit A_5,7^4)$, $(3\nonsplit A_6,2^6)$ or $(J_2,2^{12})$ and  $\mathcal{S}=(V,\{\langle u\rangle_{\mathbb{F}_{q}}+v : u\in X, v\in V\})$ where  $X$ is an orbit of $G_0$ on $V^*$.  Here $q=p^2$. 
\item[(iii)] $(S,p^d)=(2\nonsplit A_5,3^4)$
and  $G_0$ and $\mathcal{S}$ or $\mathcal{S}^{\zeta^2}$   are described in Example~\emph{\ref{example:(AS)nearfield}}.
\item[(iv)] $(S,p^d)=(2\nonsplit A_5,7^4)$ and  $G_0$ and $\mathcal{S}$ or $\mathcal{S}^{\zeta}$ are described in Example~\emph{\ref{example:(AS)korch}}.
\item[(v)] $(S,p^d)=(2\nonsplit A_6,5^4)$ and $\mathcal{S}$ is described in Example~\emph{\ref{example:(AS)walker}}. 
\item[(vi)] $(S,p^d)=(A_9,2^8)$ and $\mathcal{S}$ is described in Example~\emph{\ref{example:(AS)A9}}.
\item[(vii)] $(S,p^d)=(M_{11},3^5)$ and    $G_0$ and $\mathcal{S}$ or $\mathcal{S}^{\zeta}$ are described in Example~\emph{\ref{example:(AS)M11}}.
\item[(viii)] $(S,p^d)=(2\nonsplit G_2(4),3^{12})$ and $\mathcal{S}$ is described in Example~\emph{\ref{example:(AS)G24}}.
\item[(ix)] $(S,p^d)=(2\nonsplit J_2,5^6)$ and $\mathcal{S}$ is described in Example~\emph{\ref{example:(AS)J2}}. 
\end{itemize}
\end{prop}

\begin{proof}
By assumption, $S$ is  irreducible on $V_d(p)$,  so $G$ is primitive by Lemma~\ref{lemma:primitiveirreducible}. Let $\sigma:=\sigma_q$.
By Corollaries~\ref{cor:remark3} and~\ref{cor:AGgroups}, $\langle x\rangle_{\mathbb{F}_q}^*\subseteq x^{G_0}$ for all $x\in V^*$, except when $S=2\nonsplit A_5$ and $q=9$, in which case $\langle x\rangle_{\mathbb{F}_3}^*\subseteq x^{G_0}$ for all $x\in V^*$. In particular, if (i) or (ii) holds, then $\mathcal{S}$ is a $G$-affine proper partial linear space. Moreover, if one of (iii)--(ix) holds, then we have already seen in Examples~\ref{example:(AS)nearfield}--\ref{example:(AS)J2} that $\mathcal{S}$ is a $G$-affine proper partial linear space.

Suppose that $\mathcal{S}$ is a $G$-affine proper partial linear space. 
Let $L\in\mathcal{L}_0$, let $B:=L^*$, and let $x\in B$. By Lemma~\ref{lemma:necessary},  $B$ is a non-trivial block of $G_0$ on $X:=x^{G_0}$. 
 First suppose that $S\neq M_{11}$. Then  $-1\in G_0$ (including the case $p=2$), so $L$ is an $\mathbb{F}_p$-subspace of $V$ by Lemma~\ref{lemma:affine}. If $|L|=p$, then   $p\neq 2$ and (i) holds. Thus  we may assume that $|L|>p$, so $d>2$. Using {\sc Magma} and~\cite{web-atlas}, we verify  that one of  (ii)-(vi), (viii) or (ix) holds.

Thus  $(S,p^d)=(M_{11},3^5)$ and $q=3$. If $B=\{x,-x\}$, then (i) holds, so we assume otherwise. By a computation in {\sc Magma}, either $|X|=220$,  $G_0\in \{S,S\times  \mathbb{F}_3^*\}$ and $|B|=4$ where $B$ is the same block  for $S$ and $S\times  \mathbb{F}_3^*$;  or $|X|=132$, $G_0=S$  and $B$ is one of two blocks with size $11$. In the latter case, (vii) holds. In the former case, $|L|$ does not divide $3^5$, but  $-1\in (S\times \mathbb{F}_3^*)\setminus S$,  so we obtain a contradiction using Lemmas~\ref{lemma:necessary},~\ref{lemma:sufficient} and~\ref{lemma:affine}.
\end{proof}

Note that the partial linear space in Proposition~\ref{prop:(AS)}(ii) with $S=3\nonsplit A_6$ and $|X|=18$ is the unique generalised quadrangle of order $(3,5)$; see \cite{BueVan1994} for more details. 

\section{Proof of Theorem~\ref{thm:main}}
\label{s:proof}

Let $\mathcal{S}$ be a finite proper partial linear space, and let $G\leq \Aut(\mathcal{S})$ where $G$ is an affine  primitive permutation group of rank~$3$ with socle $V:=V_d(p)$ where $d\geq 1$ and $p$ is prime. We may assume that $\mathcal{S}$ has point set $V$. Let $k$ be the line-size of $\mathcal{S}$,  and let $\ell$ be the point-size of $\mathcal{S}$. By Lemma~\ref{lemma:lines}, $|\mathcal{S}(0)|=\ell(k-1)$.

Suppose that $\mathcal{S}$ is described in one of Examples~\ref{example:(R2)},~\ref{example:(I)goodnearfield},~\ref{example:(I)goodHering},~\ref{example:(I)reg}--\ref{example:(I)spor},~\ref{example:(E)walker}--\ref{example:(E)nonplane} or~\ref{example:(AS)nearfield}--\ref{example:(AS)J2}. By the information  in these examples, in order to prove that (iii) holds, it remains to justify the comments  in the column ``Notes" for Examples~\ref{example:(R2)},~\ref{example:(I)goodHering} (with line-size $27$),~\ref{example:(E)nonplane} and~\ref{example:(AS)A9}--\ref{example:(AS)J2}, where we claim that $\mathcal{S}$ cannot be obtained from any  $2$-$(p^d,k,1)$ design admitting a rank~$3$ group (by the method of Lemma~\ref{lemma:linearspace}), as well as Examples~\ref{example:(I)goodnearfield},~\ref{example:(I)goodHering} (with line-size~$9$) and~\ref{example:(I)sl213}--\ref{example:(I)E}, where we claim that $\mathcal{S}$ can be obtained from a rank~$3$ affine $2$-$(p^d,k,1)$ design.  By joining the lines of a partial linear space of line-size $9$  from  Example~\ref{example:(I)goodHering} with the lines of  a partial linear space from Example~\ref{example:(I)sl213}, both of which admit  $3^{12}{:}(\SL_2(13)\wr S_2)$ as an automorphism group, we obtain a rank~$3$ affine $2$-$(3^{12},9,1)$ design by  Lemma~\ref{lemma:linearspace}. Similarly, we obtain a rank~$3$ affine $2$-$(3^{8},9,1)$ design by joining the lines of a partial linear space from Example~\ref{example:(I)goodnearfield} with the lines of a partial linear space from Example~\ref{example:(I)sl25} (respectively,~\ref{example:(I)E}), both of which admit $3^8{:}(2\nonsplit S_5^-\wr S_2)$ (respectively, $3^8{:}(((D_8\circ Q_8).D_{10})\wr S_2)$) as an automorphism group. 

Now suppose that $\mathcal{S}$ is described in Examples~\ref{example:(R2)},~\ref{example:(I)goodHering} (with line-size $27$),~\ref{example:(E)nonplane} or~\ref{example:(AS)A9}--\ref{example:(AS)J2}, and suppose that there exists a permutation group $H$ on $V$ that is transitive of rank~$3$ and an automorphism group of a $2$-$(p^d,k,1)$ design $\mathcal{D}$ with point set $V$ such that $H$ has orbits $\mathcal{L}_1$ and $\mathcal{L}_2$ on the lines of $\mathcal{D}$, where $\mathcal{L}_1$ is the line set of $\mathcal{S}$. Now $H\leq \Aut(\mathcal{S})$, and $\Aut(\mathcal{S})$ is an affine primitive permutation group with socle $V$.  We claim that $V\leq H$. Note that $|H|=|V||H_0|$ and $H_0$ has two orbits on $V^*$. In particular, by the information given in the examples, either $H_0=\Aut(\mathcal{S})_0$, or  $H_0$ is one of the following: $(Q_8\circ Q_8\circ Q_8)\nonsplit \Omega_6^-(2)$ when $p^d=3^8$; $2\nonsplit G_2(4)$ when $p^d=3^{12}$; or one of $2\nonsplit J_2$ or $2\nonsplit J_2 \circ 4$ when $p^d=5^6$. In particular, $[\Aut(\mathcal{S})_0:H_0]=2^i$ for some $i$, and  $i=0$ when $p=2$. However $[V:V\cap H]=[HV:H]$, and this divides $[\Aut(\mathcal{S}):H]=[\Aut(\mathcal{S})_0:H_0]$, so  the claim holds.  Now $(V,\mathcal{L}_2)$ is an $H$-affine proper partial linear space with line-size $k$ by Lemma~\ref{lemma:linearspace}, but no such partial linear space exists by Propositions~\ref{prop:(R2)msmall},~\ref{prop:(I)bad},~\ref{prop:(E)} and~\ref{prop:(AS)} since $\SL_2(13)$ is not a subgroup of $\GammaL_2(27)$. Thus  (iii) does indeed hold.

For the remainder of this proof, we say that a triple  $(H,a,r)$ is \textit{good} if it satisfies Hypothesis~\ref{hyp:AGgroups} and $H$ has socle $V$, in which case $H$ is primitive on $V$ and $r^a=p^d$, and by Corollary~\ref{cor:AGgroups}, $H_0\leq \GammaL_a(r)$ and $H_0$ has two orbits on the points of $\PG_{a-1}(r)$. 

By Theorem~\ref{thm:rank3}, $G_0$ belongs to one of the classes (R0)--(R5), (T1)--(T3), (S0)--(S2), (I0)--(I8), (E) or (AS).  If $G_0$ belongs to the class (R0), then (iv)(a) holds. If $G_0$ belongs to class (E), then by Proposition~\ref{prop:(E)}, either $\mathcal{S}$ is described in Examples~\ref{example:(E)walker}--\ref{example:(E)nonplane} and (iii) holds, or $\mathcal{S}$ is described in Example~\ref{example:AG} with respect to the   good triple $(G,d,p)$ when $p$ is odd  or $(G,3,4)$ when $p=2$,  in which case (i) holds. Similarly, if $G_0$ belongs to class (AS), then by Proposition~\ref{prop:(AS)}, either  $\mathcal{S}$ is described in Examples~\ref{example:(AS)nearfield}--\ref{example:(AS)J2} and  (iii) holds, or $\mathcal{S}$ is described in Example~\ref{example:AG} with respect to the good triple $(G,d,p)$ when $p\neq 2$ or the good triple $(G,d/2,p^2)$ when  $S\unlhd G_0$ and $(S,p^d)$ is one of $(2\nonsplit A_5,7^4)$, $(3\nonsplit A_6,2^6)$ or $(J_2,2^{12})$, in which case (i) holds. 
 
Suppose that $G_0$ belongs to one of the classes (R1)--(R5), (T1)--(T3) or  (S0)--(S2). Now $V=V_c(s)$  and $G_0\leq \GammaL_c(s)$ where  $(c,s)$ is given by 
Table \ref{tab:ar} of Corollary~\ref{cor:remark3}, and $\langle x\rangle_{\mathbb{F}_s}^*\subseteq x^{G_0}$ for all $x\in V^*$.  If $\mathcal{S}$ is $\mathbb{F}_s$-dependent, then by Proposition~\ref{prop:dep}, $\mathcal{S}$ has line set $\{\langle u\rangle_{\mathbb{F}_r}+v : u\in \mathcal{S}(0),v\in V\}$ for some subfield $\mathbb{F}_r$  of $\mathbb{F}_s$ with $r>2$, so $\mathcal{S}$ is described in  Example~\ref{example:AG} with respect to the good triple $(G,\log_r(p^d),r)$  (see the discussion before Proposition~\ref{prop:dep}).    Otherwise, $\mathcal{S}$ is $\mathbb{F}_s$-independent. If $G_0$ belongs to one of the classes (R1)--(R5), then by Propositions~\ref{prop:(R1)},~\ref{prop:(R2)mbig},~\ref{prop:(R2)msmall},~\ref{prop:(R3)},~\ref{prop:(R4)},~\ref{prop:(R5)10} and~\ref{prop:(R5)7}, $\mathcal{S}$ is described in Example~\ref{example:(R2)}, so (iii) holds. If $G_0$ belongs to one of the classes (T1)--(T3), then by Proposition~\ref{prop:(T1)--(T3)}, $\mathcal{S}$ is described in Example~\ref{example:tensor}, and (ii) holds. If $G_0$ belongs to (S1) or (S2), then by Propositions~\ref{prop:(S1)--(S2)good} and~\ref{prop:(S1)--(S2)bad}, either $\mathcal{S}$ is described in Example~\ref{example:tensor} and (ii) holds,  or  $s=9$ and $\mathcal{S}$ is isomorphic to a partial linear space from Example~\ref{example:AG}
that is defined with respect to the good triple $(G,c,s)$, so  (i) holds. 
 Lastly, suppose that $G_0$ belongs to (S0). If $|\mathcal{S}(0)|=q(q^3-1)(q^2-1)$, then (iv)(c) holds. Otherwise, by Proposition~\ref{prop:(S0)good}, $\mathcal{S}$ is described in Example~\ref{example:tensor}, so (ii) holds. 
  
We may therefore assume that $G_0$ belongs to one of the classes (I0)--(I8). Now $V=V_n(p)\oplus V_n(p)$. Let $V_1:=\{(u,0): u \in V_n(p)\}$ and $V_2:=\{(0,u):u\in V_n(p)\}$, and recall that $\mathcal{S}(0)=V_1^*\cup V_2^*$ or $V_n(p)^*\times V_n(p)^*$. First  suppose that $\mathcal{S}(0)=V_1^*\cup V_2^*$. By Proposition~\ref{prop:(I)good},  one of the following holds: $\mathcal{S}$ is the $p^n\times p^n$ grid  of Example~\ref{example:grid}, in which case (ii) holds; $\mathcal{S}$ is isomorphic to a partial linear space in Example~\ref{example:(I)goodnearfield} or~\ref{example:(I)goodHering}, in which case (iii) holds; or $\mathcal{S}=\AG_b(r)\cprod \AG_b(r)$ where $r^b=p^n$, $b\geq 2$ and $r>2$, in which case $\mathcal{S}$  is described in Example~\ref{example:AG} with respect to the good triple $(V{:}((\GammaL_b(r)\wr S_2)\cap \GammaL_{2b}(r)),2b,r)$ by Lemma~\ref{lemma:primitive}, and (i) holds. Suppose instead that $\mathcal{S}(0)=V_n(p)^*\times V_n(p)^*$. If $G_0$ lies in class (I0), then (iv)(b) holds, so we assume otherwise. By Proposition~\ref{prop:(I)bad}, either $\mathcal{S}$ is isomorphic to  a partial linear space described in Examples~\ref{example:(I)reg}--\ref{example:(I)spor}, in which case (iii) holds, or $G_0\leq (\GammaL_b(r)\wr S_2)\cap \GammaL_{2b}(r)$ and 
$\mathcal{S}$ is isomorphic to a partial linear space with line set $\{\langle u\rangle_{\mathbb{F}_r} +v: u\in V_n(p)^*\times V_n(p)^*,v\in V\}$ where $r^b=p^n$, $b\geq 2$ and $r>2$, in which case (i) holds since this partial linear space is described in Example~\ref{example:AG} with respect to the good triple $(G,2b,r)$.  \qed

\section{Proof of Corollary~\ref{cor:affinesub}}
 \label{s:affinesub}
 
First we require the following result. Recall the definition of   $\zeta_q$ and $\sigma_q$ from  \S\ref{ss:basicsvs}.
 
\begin{lemma}
\label{lemma:onedim}
Let  $G_0\leq \GammaL_1(p^d)$ where $d\geq 1$ and $p$ is an odd  prime. Suppose that $G_0$ has $t$ orbits on $V_1(p^d)^*$ where   $t\leq 2$, and suppose that $|G_0|$ is even when $t=2$. Then $-1\in G_0$.
\end{lemma}

\begin{proof}
 Let $q:=p^d$, $\zeta:=\zeta_q$ and $\sigma:=\sigma_q$. Let $H:=G_0\cap \langle \zeta\rangle$. By~\cite[\S 3]{FouKal1978}, $G_0=\langle \zeta^n,\zeta^k\sigma^s\rangle $ for some integers $n$, $k$ and $s$ such that  $H=\langle \zeta^n\rangle$, $|H|=(q-1)/n$, $|G_0|=d(q-1)/sn$ and $s$ divides $d$.   Since $-1\in G_0$ if and only if $|H|$ is even, it suffices to prove that $(q-1)/n$ is even.  Since $q$ is odd, we may assume that $n$ is even, and since $|G_0|$ is even and $s$ divides $d$, we may  assume that $d$ is even. Write $d=2^ed'$ and $n=2^fn'$ where $d'$ and $n'$ are odd. If $t=2$, then since $n$ is even, the orbits of $G_0$ on $V^*$ have the same size  by~\cite[Corollary 3.8]{FouKal1978}, in which case $n/2$ divides $d$ by~\cite[Theorem 3.10]{FouKal1978}. If $t=1$, then $n$ divides $d$ by~\cite[Proposition 15.3]{Fou1964}. Thus, in either case, $f\leq e+1$, so it suffices to prove that $2^{e+2}$ divides $q-1$.
 Let $r:=p^{d'}$. Since $e\geq 1$,	 
$$ q-1=r^{2^e}-1=(r-1)(r^{2^0}+1)(r^{2^1}+1)\cdots (r^{2^{e-1}}+1).$$
Since $4$ divides $r-1$ or $r+1$, 
it follows that  $2^{e+2}$ divides $q-1$, as desired.
\end{proof}

 Note that $\GammaL_1(q)$ may contain odd-order  subgroups with  two orbits on $V_1(q)^*$: if $q\equiv 3\mod 4$, then $G_0:=\{\lambda^2:\lambda \in \mathbb{F}_q^*\}$ has odd order, and the orbits of $G_0$ on $\mathbb{F}_q^*$ are $G_0$ and $-G_0$.
 
 \begin{proof}[Proof of Corollary~\emph{\ref{cor:affinesub}}]
We may assume that $\mathcal{S}$ has point set $V$.  Suppose that the lines of $\mathcal{S}$ are not affine subspaces of $V$. Now  Theorem~\ref{thm:main}(iii) or (iv) holds. If Theorem~\ref{thm:main}(iii) holds, then by the descriptions of $\mathcal{S}$ given in Examples~\ref{example:(R2)},~\ref{example:(I)goodnearfield},~\ref{example:(I)goodHering},~\ref{example:(I)reg}--\ref{example:(I)spor},~\ref{example:(E)walker}--\ref{example:(E)nonplane} and~\ref{example:(AS)nearfield}--\ref{example:(AS)J2}, either  (i) or (ii) holds. 
 Thus we may assume that Theorem~\ref{thm:main}(iv) holds. By Lemma~\ref{lemma:affine}, $p$ is odd and $G_0$ does not contain $-1$. In particular, Theorem~\ref{thm:main}(iv)(c) does not hold. If Theorem~\ref{thm:main}(iv)(a) holds, then $G_0\leq \GammaL_1(p^d)$ and $G_0$ has two orbits on $V^*$, so by Lemma~\ref{lemma:onedim}, $|G_0|$ is odd. However, this contradicts the fact that the orbitals of $G$ are self-paired  by Remark~\ref{remark:PLSrank3}. 
 
 Hence Theorem~\ref{thm:main}(iv)(b) holds. Now $V=V_n(p)\oplus V_n(p)$ and $G_0\leq \GammaL_1(p^n)\wr S_2$ where $\mathcal{S}(0)=V_n(p)^*\times V_n(p)^*$. In order to show that (iii) holds, it remains to show that (2)--(7) hold and $n\geq 2$. Now (2) holds by Lemma~\ref{lemma:notaffine},  (3) holds by Lemma~\ref{lemma:affineplus}, (4) holds by Lemma~\ref{lemma:A2badbasic}, and (5) holds by Lemma~\ref{lemma:affine}.   If $H\times K\leq G_0$ for some $H,K\leq \GammaL_1(p^n)$  and if $H$ and $K$ are both transitive on $V_n(p)^*$, then $-1\in H$ and $-1\in  K$ by Lemma~\ref{lemma:onedim}, so $-1\in G_0$, a contradiction. Thus (7) holds.  Let $V_1:=\{(u,0): u \in V_n(p)\}$ and $V_2:=\{(0,u):u\in V_n(p)\}$, and for  $i\in \{1,2\}$, let
$G_0^i$ be the image of the projection of $G_{0,V_1}=G_0\cap (\GammaL_1(p^n)\times \GammaL_1(p^n))$ onto the $i$-th factor of $\GammaL_1(p^n)\times \GammaL_1(p^n)$. Since $G_0$ is transitive on $V_1^*\cup V_2^*$, it follows that $G_0^1$ and $G_0^2$ are transitive on $V_n(p)^*$. By Lemma~\ref{lemma:onedim}, $-1\in G_0^1$ and $-1\in G_0^2$, so (6) holds. 

Now suppose for a contradiction that $n=1$. Then $G_0^1=\GL_1(p)=G_0^2$. For $i\in \{1,2\}$, let $K_i$ be the kernel of the projection map of $G_{0,V_1}=G_0\cap (\GL_1(p)\times \GL_1(p))$ onto $G_0^i$. Since $G_0=G_{0,V_1}\langle (t,s)\tau\rangle$ for some $t,s\in\GL_1(p)$, there exists $H\unlhd G_0^1$ such that $K_1=1\times H$ and $K_2=H\times 1$, and $|G_0|=2|G_0^1||H|$. By assumption, $G_0$ is transitive on $V_1(p)^*\times V_1(p)^*$, so $(p-1)^2$ divides $|G_0|$. Thus   $|H|=p-1$ or $(p-1)/2$, so $\langle \zeta^2\rangle\leq H$, where $\zeta:=\zeta_p$. Since $\zeta\in G_0^1=G_0^2$, it follows that  $(\zeta,\zeta)\in G_0$, but then $-1\in G_0$, a contradiction.
 \end{proof}
 
 \section{Example~\ref{example:tensor} and  $2$-$(v,k,1)$ designs}
 \label{s:extra}
 
 In this section, we consider  when a  partial linear space  from Example~\ref{example:tensor}  can be obtained from a $2$-$(v,k,1)$ design using a rank~$3$ group (in the sense of Remark~\ref{remark:PLSfromLinear} and Lemma~\ref{lemma:linearspace}).
 
Let $V:=U\otimes W$ where $U:=V_2(q)$, $W:=V_m(q)$, $m\geq 2$, and $q$ is a prime power.  Let $\mathcal{S}_U:=(V,\mathcal{L}_U)$ where $\mathcal{L}_U:=\{(U\otimes w) +v : w\in W^*,v\in V\} $, and define $\mathcal{S}_W$ and $\mathcal{L}_W$ similarly. Recall that $\mathcal{S}_U\simeq \mathcal{S}_W$ when $m=2$ by Lemma~\ref{lemma:tensorsimple}(iv).  Let $K:=\mathbb{F}_{q^2}$.  In \S\ref{s:(S1)}, we saw that  $\mathcal{S}_U$ may be viewed as a  $K$-dependent proper partial linear space from  Example~\ref{example:AG} with respect to the rank~$3$ affine primitive group $G:=V{:}(\GL_m(q)\circ K^*){:}\Aut(K)$ from class (S1).  We may define $\mathcal{L}_Y:=\{\langle u\rangle_K +v : u \in Y,v\in V\}$, where $Y$ is the orbit of $G_0$ on $V^*$ consisting of those vectors that are not collinear with $0$ in $\mathcal{S}_U$, in which case $\mathcal{S}_Y:=(V,\mathcal{L}_Y)$ is a $K$-dependent $G$-affine proper partial linear space whose collinearity relation is disjoint from that of $\mathcal{S}_U$. Observe that the incidence structure $(V,\mathcal{L}_U\cup \mathcal{L}_Y)$ is the linear space $\AG_m(q^2)$. Thus the partial linear space $\mathcal{S}_U$ can be obtained from a $2$-$(q^{2m},q^2,1)$ design using a rank $3$ group  (in the sense of Remark~\ref{remark:PLSfromLinear} and Lemma~\ref{lemma:linearspace}).
Similarly, in \S\ref{s:(S0)} we saw that for $K:=\mathbb{F}_{q^3}$ and $m=3$, $\mathcal{S}_W$ may be viewed as a  $K$-dependent proper partial linear space from  Example~\ref{example:AG} with respect to the rank~$3$ affine primitive group  $G:=V{:}(\GL_2(q)\circ K^*){:}\Aut(K)$ from class (S0), in which case $\mathcal{S}_W$ can be obtained from the $2$-$(q^{6},q^3,1)$ design $\AG_2(q^3)$ using a rank~$3$ group. However, when $m\geq 4$, it turns out that $\mathcal{S}_W$ cannot be obtained from any $2$-$(v,k,1)$ design using a rank~$3$ group   unless $(m,q)=(5,2)$, which we now prove.
 
\begin{prop}
\label{prop:tensornodesign}
Let $V:=U\otimes W$ where $U:=V_2(q)$, $W:=V_m(q)$, $m\geq 4$ and $q$ is a prime power. Let  $\mathcal{L}_W:=\{(u\otimes W) +v : u\in U^*,v\in V\} $. The  following are equivalent.
\begin{itemize}
\item[(i)] $H$ is a  permutation group of rank~$3$  on $V$ and there is a $2$-$(q^{2m},q^m,1)$ design $\mathcal{D}:=(V,\mathcal{L})$ such that  $H\leq \Aut(\mathcal{D})$ and $H$ has two orbits on $\mathcal{L}$, one of which is $\mathcal{L}_W$.
\item[(ii)] $H=V{:}(\GammaL_1(2^2)\otimes  \GammaL_1(2^5))\leq \AGammaL_1(2^{10})$ and  $(m,q)=(5,2)$.
\end{itemize}
\end{prop}

\begin{proof}
Suppose that (i) holds, and 
let $\mathcal{S}_W$ denote the partial linear space $(V,\mathcal{L}_W)$.  Note that  $H\leq \Aut(\mathcal{S}_W)$. By Proposition~\ref{prop:tensoraut}, 
$\Aut(\mathcal{S}_W)= V{:}(\GL_2(q)\otimes \GL_m(q)){:}\Aut(\mathbb{F}_q)$, so $\Aut(\mathcal{S}_W)$ is primitive on $V$ by Lemma~\ref{lemma:primitive}. Thus $H$ is primitive on $V$ by Lemma~\ref{lemma:primrank3}. 
If $H$ is not affine, then 
by Proposition~\ref{prop:primrank3}, either $H$ is almost simple, or $H$ has subdegrees $2(q^m-1)$ and $(q^m-1)^2$. However, the former is not possible by Theorem~\ref{thm:Guralnick} since $H$ is not $2$-transitive, and the 
latter is not possible since $m\geq 4$ and  $\Aut(\mathcal{S}_W)$ has subdegrees $(q+1)(q^m-1)$ and $q(q^m-1)(q^{m-1}-1)$. Thus $H$ is affine, so the socles of $H$ and $\Aut(\mathcal{S}_W)$ are equal by ~\cite[Proposition~5.1]{Pra1990}. 

Hence $H$ is an affine permutation group of rank~$3$ on $V$    and  $H_0$ stabilises the tensor decomposition of $V$. Since $m\geq 4$, $H$ belongs to one of the classes (R0), (I0) or (T1)--(T5) by Theorem~\ref{thm:tensorgroups}. Recall the definition of the group $H_0^W$ 
 from the proof of Theorem~\ref{thm:tensorgroups}. Now $H_0^W$ is transitive on the lines of $\PG(W)$ since $H_0$ is transitive on the set of non-simple tensors of $V$,  so  either $H_0^W$ is $2$-transitive on $\PG(W)$, or  $H_0^W=\GammaL_1(2^5)$ with $(m,q)=(5,2)$~\cite{Kan1973}. If $H_0^W=\GammaL_1(2^5)$ and $(m,q)=(5,2)$, then (ii) holds by a computation in {\sc Magma}, so we assume otherwise. By~\cite{CamKan1979}, either $\SL_m(q)\unlhd H_0^W$, or $H_0^W=A_7$ and $(m,q)=(4,2)$. In particular, $H_0$ is not soluble, so $H_0$ does not belong to (R0) or (I0). 
 Let $\mathcal{L}'$ be the orbit of $H$ on  $\mathcal{L}$ that is not $\mathcal{L}_W$. 
By Lemma~\ref{lemma:linearspace}, $\mathcal{S}':=(V,\mathcal{L}')$ is an $H$-affine proper partial linear space with line-size $q^m$ for which $\mathcal{S}'(0)$ is the set of non-simple tensors of $V$. 
If $H_0$ belongs to one of the classes (T1)--(T3), then by Proposition~\ref{prop:(T1)--(T3)}, $\mathcal{S}'$ is $\mathbb{F}_q$-dependent, but then $\mathcal{S}'$ has line-size at most $q$, a contradiction. Otherwise,  $H_0$ belongs to (T4) or (T5), in which case $H_0$ belongs to (S2) or (S1), respectively. By Proposition~\ref{prop:(S1)--(S2)bad}, $\mathcal{S}'$ is isomorphic to an $\mathbb{F}_{q^2}$-dependent partial linear space, but then $\mathcal{S}'$ has line-size at most $q^2$, a contradiction.

Conversely, suppose that (ii) holds. Using {\sc Magma} and Lemmas~\ref{lemma:transitive}, \ref{lemma:sufficient} and  \ref{lemma:linearspace}, we verify that (i) holds. 
\end{proof}

 \bibliographystyle{acm}
\bibliography{jbf_references}

\end{document}